\documentclass{article}

\usepackage{etex}
\usepackage{graphicx}
\usepackage{latexsym}
\usepackage{amsfonts}
\usepackage{amssymb}
\usepackage{amsmath}
\usepackage{verbatim}
\usepackage{url}
\usepackage[standard, thmmarks]{ntheorem}
\usepackage[margin=3cm]{geometry}
\usepackage{fancyhdr}

\usepackage{rotating}
\usepackage{multirow}
\usepackage{easybmat}

\input xy
\xyoption{all}

\newtheorem{thm}{Theorem}[section]
\newtheorem{cor}[thm]{Corollary}
\newtheorem{lem}[thm]{Lemma}
\newtheorem{prop}[thm]{Proposition}
\newtheorem{alg}[thm]{Algorithm}

\newtheorem{defn}[thm]{Definition}
\newtheorem{rem}[thm]{Remark}

\newcommand{\Z}{\mathbb Z}

\newcommand{\M}{\mathcal M}
\newcommand{\Int}{\text{Int\,}}
\DeclareMathOperator{\Spin}{Spin}

\DeclareMathOperator{\Up}{Up}
\DeclareMathOperator{\Down}{Down}
\newcommand{\Ob}{\text{Ob\,}}
\newcommand{\Mor}{\text{Mor\,}}

\newcommand{\To}{\longrightarrow}

\begin{document}

\title{Chord diagrams, contact-topological quantum field theory, and contact categories}

\author{Daniel V. Mathews}%

\date{}

\maketitle

\begin{abstract}

We consider contact elements in the sutured Floer homology of solid tori with longitudinal sutures, as part of the (1+1)-dimensional topological quantum field theory defined by Honda--Kazez--Mati\'{c} in \cite{HKM08}. The $\Z_2$ $SFH$ of these solid tori forms a ``categorification of Pascal's triangle'', and contact structures correspond bijectively to chord diagrams, or sets of disjoint properly embedded arcs in the disc. Their contact elements are distinct and form distinguished subsets of $SFH$ of order given by the Narayana numbers. We find natural ``creation and annihilation operators'' which allow us to define a QFT-type basis of each $SFH$ vector space, consisting of contact elements. Sutured Floer homology in this case reduces to the combinatorics of chord diagrams. We prove that contact elements are in bijective correspondence with comparable pairs of basis elements with respect to a certain partial order, and in a natural and explicit way. The algebraic and combinatorial structures in this description have intrinsic contact-topological meaning.

In particular, the QFT-basis of $SFH$ and its partial order have a natural interpretation in pure contact topology, related to the contact category of a disc: the partial order enables us to tell when the sutured solid cylinder obtained by ``stacking'' two chord diagrams has a tight contact structure. This leads us to extend Honda's notion of contact category to a ``bounded'' contact category, containing chord diagrams and contact structures which occur within a given contact solid cylinder. We compute this bounded contact category in certain cases. Moreover, the decomposition of a contact element into basis elements naturally gives a triple of contact structures on solid cylinders which we regard as a type of ``distinguished triangle'' in the contact category. We also use the algebraic structures arising among contact elements to extend the notion of contact category to a 2-category. 

\end{abstract}

\tableofcontents

\section{Introduction}

\subsection{Fun with chord diagrams}

\label{fun}

This paper contains elementary combinatorial results about chord diagrams which have applications to contact topology and sutured Floer homology. 

\begin{defn}
A \emph{chord diagram} $\Gamma$ is a set of disjoint properly embedded arcs (chords) in a disc $D^2$, considered up to homotopy relative to endpoints.
\end{defn}

Consider a chord diagram with $n$ chords; it has $2n$ marked points on the boundary of the disc, connected in pairs by disjoint chords. We declare one of those marked points on the boundary a base point; rotating a chord diagram will generally give a distinct chord diagram.

The chords of a chord diagram divide the disc $D$ into regions, which we alternately label as positive or negative. The labelling is induced from a labelling on the arcs of $\partial D^2$ between marked points; we declare that the arc immediately clockwise of the base point is positive, and the arc immediately anticlockwise is negative. See figure \ref{fig:1}.

\begin{rem}
The base point is always denoted by a solid red dot.
\end{rem}

\begin{figure}
\centering
\includegraphics[scale=0.3]{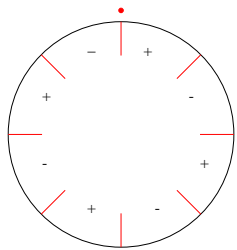}
\caption{Base point and sign of regions.} \label{fig:1}
\end{figure}

\begin{defn}[Euler class of chord diagram]
The \emph{(relative) euler class} $e$ of a chord diagram $\Gamma$ is the sum of the signs of the regions of $D - \Gamma$.
\end{defn}
That is, a $+$ region counts as $+1$ and a $-$ region counts as $-1$. It's not difficult to see that $e$ has opposite parity to $n$, and $|e| \leq n-1$.

We consider a certain vector space generated by chord diagrams.
\begin{defn}[Combinatorial $SFH$]
The $\Z_2$-vector space generated by chord diagrams of $n$ chords and euler class $e$, subject to the \emph{bypass relation} in figure \ref{fig:4}, is called $SFH_{comb}(T, n, e)$. The $\Z_2$-vector space generated by all chord diagrams of $n$ chords, subject to the same relation, is called $SFH_{comb}(T,n)$.
\begin{figure}
\centering
\includegraphics[scale=0.5]{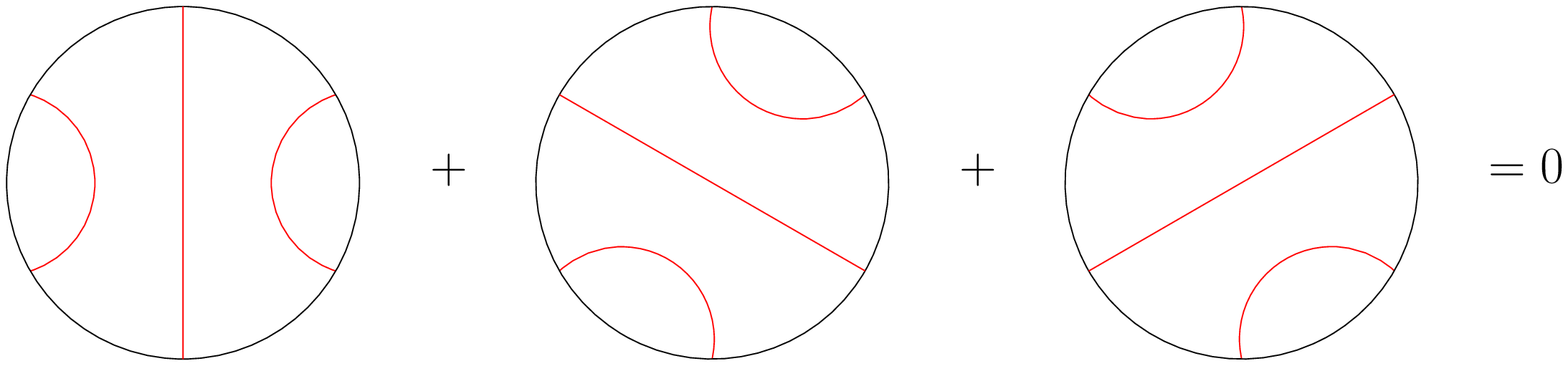}
\caption{The bypass relation.} \label{fig:4}
\end{figure}
\end{defn}

We will show that these combinatorial objects are isomorphic to $SFH(T,n,e)$ and $SFH(T,n)$, the sutured Floer homology of certain sutured manifolds.

The bypass relation means that if we have three chord diagrams $\Gamma_1, \Gamma_2, \Gamma_3$ which are all identical, except in a sub-disc $D' \subset D$, on which each of $\Gamma_1, \Gamma_2, \Gamma_3$ contains three arcs, respectively in the three arrangements shown in figure \ref{fig:4}, then we consider them to sum to zero.

The terminology ``bypass'' comes from contact geometry, but the idea of ``bypasses'' here can be considered purely as a type of surgery on a chord diagram.
\begin{defn}[Arc of attachment]
An \emph{arc of attachment}, or \emph{attaching arc} in a chord diagram $\Gamma$ is an embedded arc which intersects the chords of $\Gamma$ at precisely three points, namely, its two endpoints, and one interior point.
\end{defn}

We consider attaching arcs equivalent if they are homotopic through attaching arcs.

\begin{defn}[Bypass surgery / move]
Let $c$ be an attaching arc in $\Gamma$.
\begin{enumerate}
\item 
\emph{Upwards bypass surgery} $\Up_c$ along $c$ on $\Gamma$ removes a small disc neighbourhood of $c$ and replaces it with another disc with chords as shown in figure \ref{fig:2a} (right).
\item
\emph{Downwards bypass surgery} $\Down_c$ along $c$ on $\Gamma$ also removes a small neighbourhood of $c$, but now replaces it as shown in figure \ref{fig:2a} (left).
\end{enumerate}
We also call bypass surgeries \emph{bypass moves}.
\end{defn}

\begin{figure}
\centering
\includegraphics[scale=0.5]{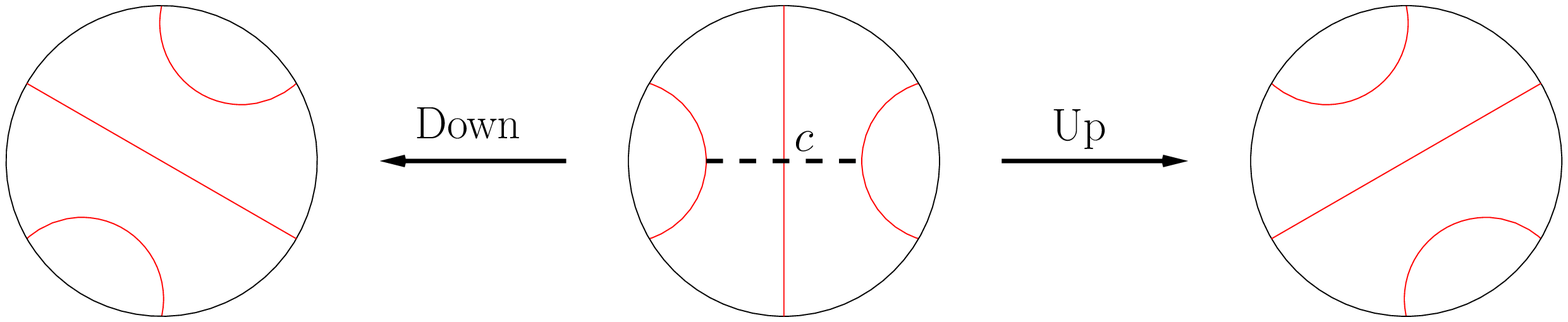}
\caption{Upwards and downwards bypass moves.} \label{fig:2a}
\end{figure}

Clearly, chord diagrams related by bypass surgery naturally come in triples, and such triples are defined to sum to $0$ in $SFH_{comb}$. In particular, in $\Gamma' = \Up_c \Gamma$, there is an attaching arc $c'$ such that $\Down_{c'} \Gamma' = \Gamma$, and in $\Gamma'' = \Down_c \Gamma$, there is an attaching arc $c''$ such that $\Up_{c''} \Gamma'' = \Gamma$. Bypass moves perform a local $60^\circ$ rotation on part of a chord diagram (see figure \ref{fig:49}); ``three bypass moves is the identity''. This observation, as we will see, is the source of much interesting algebraic and categorical structure.
\begin{defn}[Bypass triple]
\label{def_bypass_triple}
Three chord diagrams $\Gamma$, $\Gamma'$, $\Gamma''$ form a \emph{bypass triple} if there exists an attaching arc $c$ on $\Gamma$ such that
\[
\Gamma' = \Up_c \; \Gamma, \quad \Gamma'' = \Down_c \; \Gamma.
\]
\end{defn}
Clearly the existence of such an attaching arc on $\Gamma$ is equivalent to existence of such arcs on $\Gamma'$ or $\Gamma''$. If two distinct chord diagrams are related by a bypass move, then there is a unique third chord diagram forming a bypass triple.

\begin{figure}
\centering
\includegraphics[scale=0.5]{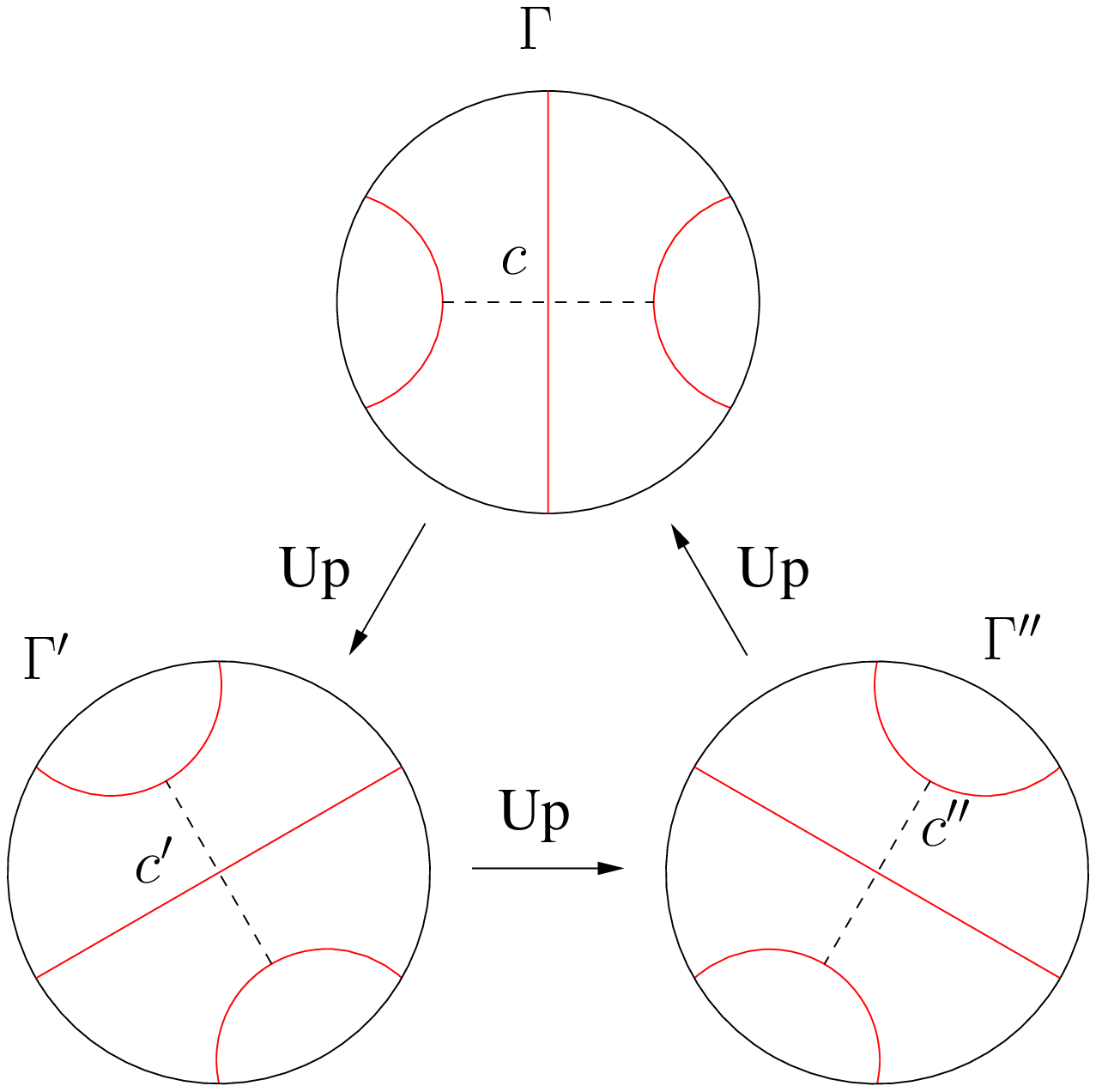}
\caption{Bypass triple.} \label{fig:49}
\end{figure}

In general a bypass move on a chord diagram need not produce a chord diagram; it may produce a closed loop. A diagram with a closed loop is considered to be zero in $SFH_{comb}$. In this case, the effect of the bypass move in the opposite direction leaves the chord diagram unchanged; the bypass relation still holds and is of the form $x + x + 0 = 0$.

We will give a nice basis for each vector space $SFH_{comb}(T, n, e)$ and show that when chord diagrams are decomposed into a sum of basis elements, this decomposition has certain nice properties. There will be a partial order on this basis, and chord diagrams will correspond bijectively with pairs of basis elements which are comparable with respect to this partial order.

For instance, consider $SFH_{comb}(T, 4, -1)$. This vector space is spanned by the $6$ chord diagrams which have $4$ chords and relative euler class $-1$: see figure \ref{fig:5}.
\begin{figure}
\centering
\includegraphics[scale=0.4]{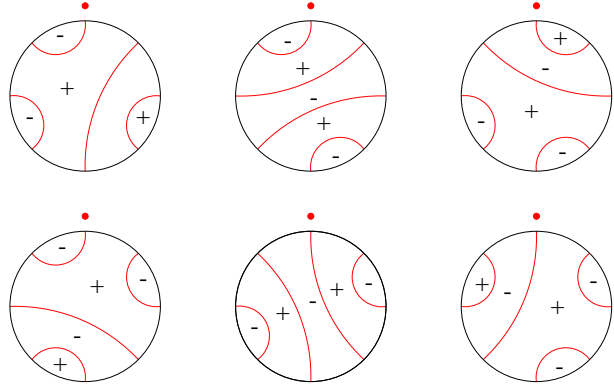}
\caption{Chord diagrams in $SFH(T,4,-1)$.} \label{fig:5}
\end{figure}

We will show, and it was essentially known previously in \cite{HKM08}, that
\[
 SFH_{comb}(T,4,-1) = \Z_2^3.
\]
Our basis will consist of the three chord diagrams in the top row of figure \ref{fig:5}; we will see later that they are naturally labelled with words 
\[
 --+, \quad -+-, \quad +--.
\]
In general, the basis for $SFH_{comb}(T, n+1, e)$ will be labelled by words on $\{-,+\}$ of length $n$ whose symbols sum to $e$, and the number of such words is $\binom{n}{k}$, where $k = (n+e)/2$. We will write $v_w$ to denote the basis element labelled by $w$.

On this set of words, there is a partial order defined by ``all minus signs move right (or stay where they are)''. (In this simple case, it is actually a total order; but this will not be true for words of longer length. For instance, $-++-$ and $+--+$ are not comparable.) Thus,
\[
 --+ \preceq -+- \preceq +--.
\]
In terms of our basis, the 6 chord diagrams (arranged as in figure \ref{fig:5}) are
\[
\begin{array}{ccc}
 v_{--+} & v_{-+-} & v_{+--} \\ v_{--+} + v_{-+-} & v_{--+} + v_{+--} & v_{-+-} + v_{+--}.
\end{array}
\]
Moreover, there are six pairs of words $w_1, w_2$ which are \emph{comparable} with respect to $\preceq$, namely three ``doubles''
\[
 (--+, --+), \quad (-+-, -+-), \quad (+--, +--)
\]
and three less trivial pairs
\[
 (--+, -+-), \quad (--+, +--), \quad (-+-, +--).
\]
And in fact, for each pair, there is precisely one chord diagram having that pair as its first and last basis element. That is, there is a bijection
\[
 \left\{ \text{Chord diagrams} \right\} \leftrightarrow \left\{ \text{Comparable pairs of words} \right\}
\]
given by taking a chord diagram to the first and last basis elements in its basis decomposition.

This is a general fact, and our main theorem. Moreover, this bijection, and its inverse, can be described explicitly. Given a chord diagram, we can algorithmically extract its first and last basis elements, and they are comparable. Conversely, given two comparable words, we can algorithmically produce the chord diagram for which those words give its first and last basis elements. We will also say more about the set of basis chord diagrams that occur in a given chord diagram; as well as relationships between the various $SFH_{comb}(T, n, e)$.

An information-theoretic note from this result is that a chord diagram of $4$ chords can be encoded in $6$ bits, with the redundancy that the first $3$ bits form a word lesser than the second $3$ bits, with respect to $\preceq$. In general, a chord diagram of $n+1$ chords can be encoded in $2n$ bits, with a similar redundancy.

This particular example, with $4$ chords and $e=-1$, is actually the essence of Honda's octahedral axiom (see \cite{HonCat}, also section \ref{sec_octahedral} below).

While this result is combinatorial, the motivation, notation, and applications come from the theory of sutured Floer homology, with its connections to topological quantum field theory and contact topology.

\subsection{Contact elements in $SFH$ of solid tori}

\subsubsection{Computation of $SFH$}

\label{sec_intro_SFH}

We study the contact elements in the sutured Floer homology of a very simple sutured manifold: the solid torus $(T,n)$ with $2n$ longitudinal sutures. We now give an overview of this aspect of our results.

\label{sec_SFH_contact_intro}
\label{sec_SFH_overview}

The theory of sutured Floer homology was introduced by Juh{\'a}sz in \cite{Ju06}. It is an extension of Heegaard Floer homology, developed in \cite{OS04Prop, OS04Closed, OSContact, OS06}. We refer to those papers for background, and recall some facts.

Sutured Floer homology is an invariant of a \emph{balanced sutured manifold}. A sutured manifold $(M, \Gamma)$ is \emph{balanced} if it satisfies the following conditions: $M$ has no closed components; $\chi(R_+(\Gamma)) = \chi(R_-(\Gamma))$; and every boundary component of $M$ has an annular suture. (In particular, there are no toric sutures.)

\label{sec_spin-c}

Throughout this paper, we take $\Z_2$ coefficients, so that $SFH(M, \Gamma)$ is a $\Z_2$-vector space.

The vector space $SFH(M, \Gamma)$ splits as a direct sum over spin-c structures, as in \cite{Ju06}:
\[
 SFH(M, \Gamma) = \bigoplus_{\mathfrak{s} \in \Spin^c(M, \Gamma)} SFH(M, \Gamma, \mathfrak{s}).
\]

\label{sec_ct_elts_TQFT}

A contact structure $\xi$ on $(M, \Gamma)$ gives rise to a \emph{contact element} or \emph{contact class} $c(\xi)$ in $SFH(-M, -\Gamma)$. (Here the minus signs refer to reversed orientation.) A contact structure on a sutured 3-manifold $(M, \Gamma)$ is required to be compatible with the sutures $\Gamma$: $\partial M$ must be convex with dividing set $\Gamma$, and the positive/negative regions of $\partial M$ as a convex surface must agree with the positive/negative regions arising from the sutures. With $\Z$-coefficients, there is a $(\pm 1)$ ambiguity; but with $\Z_2$ coefficients the contact element is a well-defined single element. Here we follow the definition of Honda--Kazez--Mati\'{c} in \cite{HKM06ContClass} and only consider $\Z_2$ coefficients. The contact class in this case is an extension of the definition of a contact class in the Heegaard Floer homology of a closed manifold, as defined in \cite{OSContact} and reformulated in \cite{HKMContClass}.

The contact class is known to satisfy various properties, also noted in \cite{HKM06ContClass}: for instance, $c(\xi) = 0$ when $\xi$ is overtwisted, or when the partial monodromy of a corresponding partial open book is not ``right-veering'' (see \cite{HKM05, HKM06RV}).

In \cite{HKM08}, Honda--Kazez--Mati\'{c} proved that $SFH$ has some of the properties of a topological quantum field theory (TQFT). In particular, an inclusion of sutured manifolds, together with an ``intermediate'' contact structure, induces a map on $SFH$. We give a $\Z_2$ version.
\begin{thm}[Honda--Kazez--Mati\'{c} \cite{HKM08}]
Let $(M', \Gamma')$ be a sutured submanifold of $(M, \Gamma)$ lying in $\Int M$, and let $\xi$ be a contact structure on $(M - \Int M', \Gamma \cup \Gamma')$. Let $M - \Int M'$ have $m$ components which are isolated, i.e. components which do not intersect $\partial M$. Then $\xi$ induces a natural map
\[
 \Phi_\xi: SFH(-M', -\Gamma') \To SFH(-M, -\Gamma) \otimes V^m,
\]
where $V = \Z_2 \oplus \Z_2 = \widehat{HF}(S^1 \times S^2)$. This map has the property that for any contact structure $\xi'$ on $(M', \Gamma')$,
\[
 \Phi_\xi \left( c(\xi') \right) = c(\xi' \cup \xi) \otimes x^{\otimes m},
\]
where $x$ is the contact class of the standard tight contact structure on $S^1 \times S^2$.
\end{thm}

\label{sec_SFH_solid_tori}

A contact structure on $(T,n)$ can be described by examining the dividing set on a convex meridional disc, which is a chord diagram of $n$ chords. The tight contact structures on $(T,n)$, up to isotopy rel boundary, are in bijective correspondence with chord diagrams of $n$ chords (see \cite{Hon00I}, but note \cite{Hon01}, also \cite{Hon00II, Hon02, Gi00, GiBundles}; we also prove this as part of our study of bypasses, as proposition \ref{prop_tight_ct_str_solid_torus}). The relative euler class of a contact structure on $(T,n)$ (evaluated on a meridian disc) is the relative euler class of the corresponding chord diagram. And, in a notationally-executed blatant cover-up of the unpleasant reversals of orientation, we will write $SFH(T,n)$ to denote the $SFH$ of the appropriately orientation-reversed manifold. The orientation reversal is never an issue in the following, so hopefully the abuse of notation will not cause too much confusion.

Juh{\'a}sz \cite{Ju08} and Honda--Kazez--Mati\'{c} \cite{HKM06ContClass, HKM08} have proved theorems to calculate $SFH$ when two sutured manifolds are glued together in certain ways.  One immediate corollary of these theorems, given in \cite{HKM08}, is a computation of $SFH(T,n)$. Moreover, euler classes of contact elements correspond to spin-c structures.  
\begin{thm}[Honda--Kazez--Mati\'{c} \cite{HKM08}, Juh{\'a}sz \cite{Ju08}]
$SFH(T, n+1) = \Z_2^{2^{n}}$ and splits as a direct sum over spin-c structures
\[
 SFH(T,n+1) = \Z_2^{\binom{n}{0}} \oplus \cdots \oplus \Z_2^{\binom{n}{n}}.
\]
If $\xi$ is a contact structure on $(T, n+1)$ with relative euler class $e$, then its contact element $c(\xi)$ lies in the summand $\Z_2^{\binom{n}{k}}$, where $k = (e+n)/2$.
\end{thm}
We therefore denote the summand $\Z_2^{\binom{n}{k}}$ as $SFH(T,n+1,e)$. The upshot is that \emph{we may regard chord diagrams of $n$ chords as contact elements in $SFH(T, n)$.}

We study the question: \emph{How do contact elements lie in sutured Floer homology?}

A first proposition in this direction was known to Honda--Kazez--Mati\'{c} in \cite{HKM08}; we will prove it again.
\begin{prop}[Contact elements distinct]
\label{contact_distinct}
Distinct tight contact structures (up to isotopy) on $(T,n)$, or equivalently, distinct chord diagrams, give distinct nonzero contact elements of $SFH(T,n)$.
\end{prop}
Thus chord diagrams of $n$ chords may be identified with contact elements in $SFH(T,n)$.

A second proposition is that the meaning of addition in $SFH$ corresponds to bypass moves on chord diagrams. This was probably known to the authors of \cite{HKM08}, although the whole of this result was not made explicit. The set of contact elements in $SFH(T,n,e)$ is not a subgroup under addition, but the extent to which it is closed under addition is described by bypass moves.
\begin{prop}[Addition means bypass moves]
\label{addition_bypasses}
Suppose $a,b$ are contact elements in $SFH(T, n, e)$. Then $a+b$ is a contact element if and only if $a,b$ are related by a bypass move. If so, then $a+b$ is the third element of their bypass triple.
\end{prop}
(Note here we identify chord diagrams with contact elements, and we continue this abuse of notation throughout.) 

The $SFH_{comb}$ described in section \ref{fun} seems to have been known in \cite{HKM08}, though not made explicit; it also appears to be the origin of Honda's ``contact category'' \cite{HonCat}. In any case, $SFH_{comb}$ is indeed a combinatorial version of $SFH$.
\begin{prop}[$SFH$ is combinatorial]
\label{combinatorial_SFH}
There is an isomorphism
\[
 SFH_{comb}(T,n,e) \stackrel{\cong}{\To} SFH(T,n,e).
\]
This isomorphism takes a chord diagram to the contact element of the tight contact structure on $(T,n)$ with that chord diagram as its dividing set on a meridional disc. 
\end{prop}

\subsubsection{Categorification of Pascal's triangle}

If we consider all the sutured Floer homology groups $SFH(T,n+1)$ and their decompositions into direct sums of $SFH(T,n+1,e)$, over all possible $n$ and $e$, we can arrange these in a triangle. (Recall $-n \leq e \leq n$ and $e \equiv n$ mod $2$.)
\[
\begin{array}{ccccc}
		&		&SFH(T,1,0)	&		&		\\
 		&SFH(T,2,-1)	&\oplus		&SFH(T,2,1)	&	\\
SFH(T,3,-2)	&\oplus		&SFH(T,3,0)	&\oplus		&SFH(T,3,2)	\\
\end{array}
\]

These are isomorphic respectively to a ``categorification of Pascal's triangle''. 
\[
\begin{array}{ccccccc}
			&			&			&\Z_2^{\binom{0}{0}}	&			&			& \\
			&			&\Z_2^{\binom{1}{0}}	&\oplus			&\Z_2^{\binom{1}{1}}	&			& \\
			&\Z_2^{\binom{2}{0}}	&\oplus			&\Z_2^{\binom{2}{1}}	&\oplus			&\Z_2^{\binom{2}{2}}	& \\
\end{array}
\]

There are various maps between these vector spaces. There are maps denoted
\[
\begin{array}{cccc}
 B_\pm : & SFH(T, n) & \To & SFH(T, n+1) \\
		& \Z_2^{2^{n-1}} & \To & \Z_2^{2^n}
\end{array}
\]
which we call \emph{creation} maps. They are defined by the picture in figure \ref{fig:6} of ``creating a chord, adding a $\pm$ outermost region near the base point''. 

\begin{figure}
\centering
\includegraphics[scale=0.5]{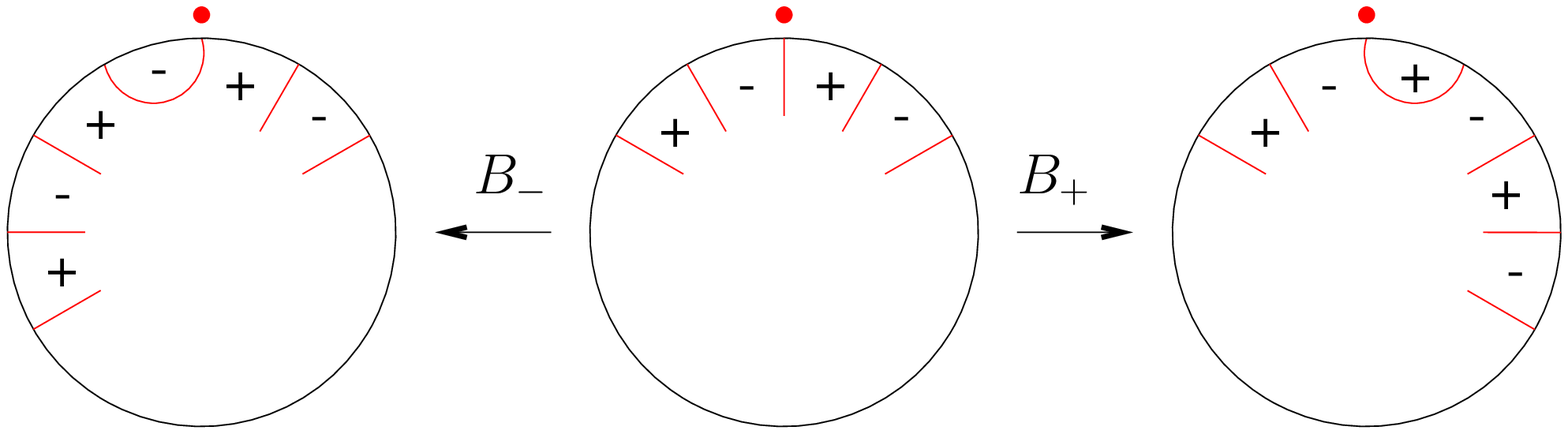}
\caption{Creation maps $B_\pm$.} \label{fig:6}
\end{figure}

In $SFH_{comb}$, it is clear that they are linear maps. The fact that they are linear in $SFH$ comes from the TQFT-inclusion property, as we see more precisely in section \ref{sec_observations_SFH}.

The maps $B_\pm$ add $\pm 1$ to the relative euler class of the chord diagram / contact structure; so that, restricting to particular summands, $B_\pm$ define maps
\[
\begin{array}{ccc}
SFH(T,n,e) & \stackrel{B_\pm}{\To} & SFH(T,n+1,e \pm 1) \\
\Z_2^{\binom{n-1}{k}} & \stackrel{B_-}{\To} & \Z_2^{\binom{n}{k}}  \\
\Z_2^{\binom{n-1}{k}} & \stackrel{B_+}{\To} & \Z_2^{\binom{n}{k+1}} 
\end{array}
\]
where $k = (n+e-1)/2$. (Strictly speaking, a $B_\pm$ is defined on each $SFH(T,n,e)$ or $SFH(T,n)$, but we denote them all by $B_\pm$; alternatively, $B_\pm$ may be considered to act on the direct sum of all the $SFH(T,n)$.)

\begin{prop}[Categorification of Pascal recursion]
\label{categorification}
There are linear maps
\[
 B_\pm: SFH(T, n, e) \To SFH(T, n+1, e \pm 1)
\]
which correspond to ``creating'' a chord as in figure \ref{fig:6} above. These are injective and
\[
SFH(T, n+1, e) = B_+ \left( SFH(T, n, e-1) \right) \oplus B_- \left( SFH(T, n, e+1) \right).
\]
\end{prop}

Similarly, there are two maps 
\[
\begin{array}{cccc}
 A_\pm : & SFH(T, n+1) & \To & SFH(T, n) \\
		& \Z_2^{2^n} & \To & \Z_2^{2^{n-1}}
\end{array}
\]
which we call \emph{annihilation} maps, defined by ``closing off an outermost $\pm$ region near the base point''. See figure \ref{fig:7}.

\begin{figure}
\centering
\includegraphics[scale=0.5]{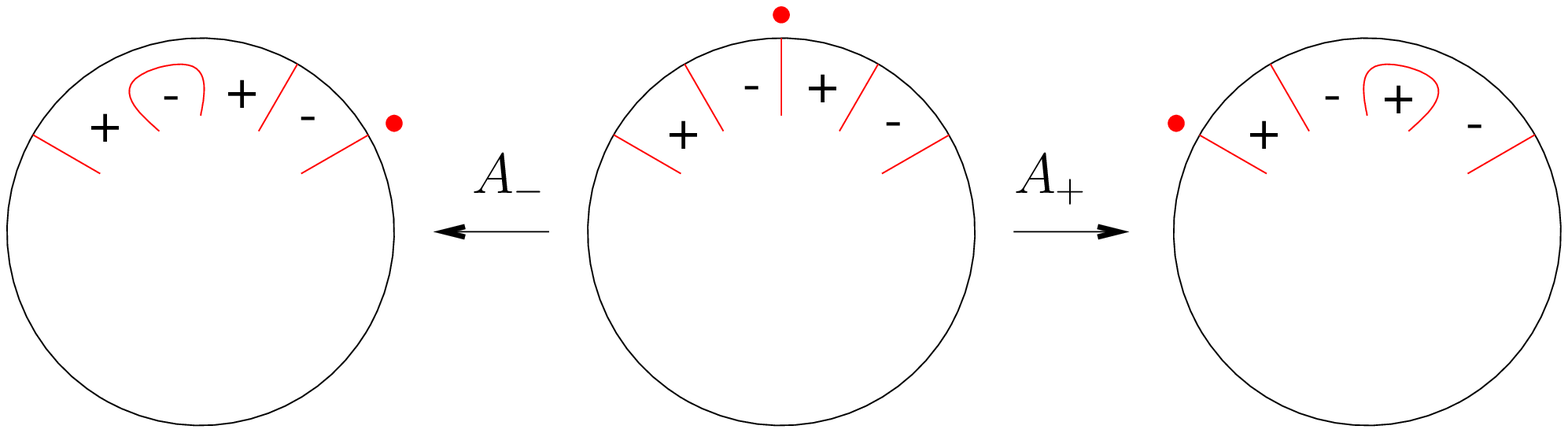}
\caption{Annihilation maps $A_\pm$.} \label{fig:7}
\end{figure}

The maps $A_\pm$ add $\pm 1$ to the relative euler class of the chord diagram / contact structure; restricting to summands again, for $k=(n+e)/2$,
\[
\begin{array}{ccc}
SFH(T,n+1,e) & \stackrel{A_-}{\To} & SFH(T,n,e \pm 1) \\
\Z_2^{\binom{n}{k}} & \stackrel{A_-}{\To} & \Z_2^{\binom{n-1}{k-1}} \\ 
\Z_2^{\binom{n}{k}} & \stackrel{A_+}{\To} & \Z_2^{\binom{n-1}{k}}.
\end{array}
\]

\begin{prop}[Annihilation operators]
\label{annihilation}
There are maps
\[
 A_\pm: SFH(T,n+1,e) \To SFH(T,n,e \pm 1)
\]
which correspond to ``annihilating'' a chord as in figure \ref{fig:7} above.
These are surjective and satisfy
\[
A_+ \circ B_- = A_- \circ B_+ = 1 \quad \text{and} \quad A_+ \circ B_+ = A_- \circ B_- = 0.
\]
\end{prop}

Thus, the $A_\pm, B_\pm$ operators give a categorification of Pascal's triangle, in the sense of the following diagram:
\[
\xymatrix{
&&& \Z_2^{\binom{0}{0}} \ar@/^/[dr]^{B_+} \ar@/_/[dl]_{B_-} &&& \\
&& \Z_2^{\binom{1}{0}} \ar@/_/[ur]_{A_+} \ar@/^/[dr]^{B_+} \ar@/_/[dl]_{B_-} && \Z_2^{\binom{1}{1}} \ar@/^/[ul]^{A_-} \ar@/^/[dr]^{B_+} \ar@/_/[dl]_{B_-} && \\
& \Z_2^{\binom{2}{0}} \ar@/_/[ur]_{A_+} && \Z_2^{\binom{2}{1}} \ar@/_/[ur]_{A_+} \ar@/^/[ul]^{A_-} && \Z_2^{\binom{2}{2}} \ar@/^/[ul]^{A_-} &
}
\]
(The effect of a composition $B_\pm \circ A_\pm$ is also easily understood, with the basis described in the next section.)

\subsubsection{Basis, words, orderings, and quantum field theory}

\label{sec_intro_basis}

Denote by $v_\emptyset$ the nonzero element of $SFH(T,1) = \Z_2$ (the ``vacuum''), which corresponds to the unique chord diagram with $1$ chord (lemma \ref{lem_vacuum}). Then in 
\[
SFH(T,n+1,e) \cong \Z_2^{\binom{n}{k}}
\]
there are $\binom{n}{k}$ contact elements of the form
\[
 B_\pm \; B_\pm \; \cdots \; B_\pm \; v_\emptyset
\]
where there are $n_\pm$ of the $B_\pm$'s, satisfying $k = n_+$, $n = n_+ + n_-$ and $e = n_+ - n_-$.

Denote by $W(n_-, n_+)$ the set of all words on $\{-,+\}$ of length $n = n_+ + n_-$, with $n_-$ minus signs and $n_+$ plus signs; equivalently, which sum to $e = n_+ - n_-$. For every word $w \in W(n_-, n_+)$ there is a corresponding element $v_w = B_w v_\emptyset$ in $SFH(T, n+1, e)$; $B_w$ denotes the string of $B_\pm$'s corresponding to $w$.  Each $v_w \in SFH(T, n+1, e)$ is a contact element, corresponding to a chord diagram $\Gamma_w$ of $n+1$ chords and relative euler class $e$.
\begin{rem}[Conventions for variables]
\label{rem_variables}
Unless mentioned otherwise, we will assume that the variables $n_-, n_+, n, e, k$ are related as above: $SFH(T, n+1,e) = \Z_2^{\binom{n}{k}}$ contains the $v_w$ with $w \in W(n_-, n_+)$, i.e.
\[
 k = (e+n)/2, \quad e = 2k-n, \quad n_+ = k, \quad n = n_+ + n_-, \quad e = n_+ - n_-.
\]
\end{rem}

The set $W(n_-, n_+)$ has some orderings.
\begin{defn}[Lexicographic order]
There is a total order on $W(n_-, n_+)$ obtained from regarding $-$ as coming before $+$ in the dictionary. This also induces a total order on the elements $v_w \in SFH(T, n+1,e)$ and the chord diagrams $\Gamma_w$.
\end{defn}

\begin{defn}[Partial order $\preceq$]
There is a partial order $\preceq$ on $W(n_-, n_+)$ defined by: $w_1 \preceq w_2$ if and only if, for all $i = 1, \ldots, n_-$, the $i$'th $-$ sign in $w_1$ occurs to the left of (or in the same position as) the $i$'th $-$ sign in $w_2$. This also induces a partial order, also denoted $\preceq$, on the $v_w \in SFH(T, n+1, e)$ and chord diagrams $\Gamma_w$.
\end{defn}
Thus $\preceq$ essentially says ``all minus signs move right (or stay where they are)'', or equivalently ``all $+$ signs move left (or stay where they are)''. It is clear that this is a partial order, and a sub-order of the lexicographic total order.

Note that $|W(n_-,n_+)| = \binom{n}{n_+} = \binom{n}{k} = \dim SFH(T,n+1,e)$. Even better:
\begin{prop}[QFT basis]
\label{QFT_basis}
The set of $v_w$, for $w \in W(n_-, n_+)$, forms a basis for $SFH(T,n+1,e)$.
\end{prop}

The terminology is by analogy with operators for the creation and annihilation of particles in quantum field theory. We think of particles with charge (spin?) $\pm 1$, and consider $SFH(T,n+1,e)$ as the space generated by $n$-particle states of charge $e$. Each chord diagram with $n+1$ chords and relative euler class $e$ becomes an ``$n$-particle state of charge $e$''; the chord diagram with $1$ chord, ``the vacuum''. The bypass relation says ``the superposition of two bypass-related states is the third state in their triple''.

Proposition \ref{QFT_basis} then says ``the space of states has a basis obtained by applying creation operators to the vacuum''. This is usual in quantum field theory. However, for bosons, creation operators commute; for fermions, they anti-commute. Our case however is completely non-commutative.

\subsubsection{Catalan and Narayana numbers}
\label{sec_intro_Catalan_Narayana}

A simple first question is: \emph{How many contact elements are there in $SFH(T,n)$?}

The number of distinct chord diagrams of $n$ chords is given by $1, 1, 2, 5, 14, \ldots$, i.e. the ubiquitous \emph{Catalan numbers} $C_n$. Recall $C_n$ can be defined recursively by $C_0 = 1$, $C_1 = 1$ and
\[
 C_n = C_0 C_{n-1} + C_1 C_{n-2} + \cdots + C_{n-1} C_0;
\]
also $C_n = \frac{1}{n+1} \binom{2n}{n}$. Chord diagrams of $n$ chords are bijective with bracketings of $n$ pairs of brackets; it's well known that there are $C_n$ of these. Since each chord diagram gives a distinct contact element (proposition \ref{contact_distinct}), the nonzero contact elements form a distinguished subset of size $C_n$ in $SFH(T, n) \cong \Z_2^{2^{n-1}}$.

Refining, we ask: \emph{How many contact elements are there in $SFH(T,n+1,e)$?} 

Let this number, which is simply the number of chord diagrams with $n+1$ chords and relative euler class $e$, be $C_{n+1}^e$. It will also be useful to define $C_{n+1,k} = C_{n+1}^{2k-n} = C_{n+1}^e$, following our convention in remark \ref{rem_variables}; so that $k$ is an integer, $0 \leq k \leq n$.

From counting chord diagrams of various relative euler classes, we have
\[
C_{n+1} = C_{n+1,0} + C_{n+1,1} + \cdots + C_{n+1,n} = C_{n+1}^{-n} + C_{n+1}^{-n+2} + \cdots + C_{n+1}^n.
\]

The numbers $C_{n+1}^e$ form a triangle, which is known as the \emph{Catalan triangle}. Its entries are known as the \emph{Narayana numbers}.
\[
\begin{array}{cccccccccccccccccccc}
	&	&	&	&C_1^0	&	&	&	&	&	&	&	&	&	&1	&	&	&	& \\
	&	&	&C_2^{-1}&	&C_2^{1}&	&	&	&	&	&	&	&1	&	&1	&	&	& \\
	&	&C_3^{-2}&	&C_3^0	&	&C_3^{2}&	&	&=	&	&	&1	&	&3	&	& 1	& 	&\\
	&C_4^{-3}&	&C_4^{-1}&	&C_4^{1}&	&C_4^3	&	&	&	&1	&	&6	&	& 6	&	& 1 	& \\
C_5^{-4}&	&C_5^{-2}&	&C_5^0	&	&C_5^2	&	&C_5^4	&	&1	&	&10	&	&20	&	&10	&	&1 \\
\end{array}
\]
The Narayana numbers are usually given as $N_{n,k} = C_{n,k-1}$; we have shifted them for our purposes. They have an explicit formula, although we shall not use it:
\[
 C_{n+1}^e = C_{n+1,k} = N_{n+1, k+1} = \frac{1}{n+1} \binom{n+1}{k+1} \binom{n+1}{k} .
\]
There is a substantial literature on the Narayana numbers (e.g. \cite{Aigner, Benchekroun, Bona-Sagan, DSV, Fomin-Reading, HNT, Hwang-Mallows, Novelli-Thibon, Sulanke, Williams05, Yano-Yoshida}); we will restate some of their properties.
\begin{prop}[Narayana numbers]
\label{Narayana_recursion}
The Narayana numbers give the number of chord diagrams $C_{n+1}^e$ in $SFH(T,n+1,e)$, and satisfy the following relations:
\begin{enumerate}
\item $C_{n+1} = C_{n+1,0} + C_{n+1,1} + \cdots + C_{n+1,n}$.
\item
\[
 C_{n+1,k} = C_{n,k} + C_{n,k-1} + \sum_{\substack{ n_1 + n_2 = n \\ k_1 + k_2 = k-1}} C_{n_1, k_1} \; C_{n_2, k_2},
\]
or equivalently,
\[
 C_{n+1}^e = C_n^{e-1} + C_n^{e+1} + \sum_{\substack{n_1 + n_2 = n \\ e_1 + e_2 = e}} C_{n_1}^{e_1} \; C_{n_2}^{e_2}.
\]
\end{enumerate}
\end{prop}
In a certain tenuous sense, there is a ``categorification'' of this recursion also.
\begin{prop} [Categorification of Catalan recursion]
\label{Catalan_categorification}
There is an operator
\[
M: SFH(T, n_1, e_1) \otimes SFH(T, n_2, e_2) \To SFH(T, n_1 + n_2 + 1, e_1 + e_2)
\]
which, applied to contact elements $c(\xi_1) \otimes c(\xi_2)$, gives the contact element obtained by ``merging'' the corresponding chord diagrams. The operator $M$ reduces to a creation operator $B_\pm$ in the case $n_1 = 0$ or $n_2 = 0$. Every contact element in $SFH(T, n+1, e)$ can then be written uniquely as $M(c(\xi_1), c(\xi_2))$, where $c(\xi_i)$ is a contact element in $SFH(T, n_i, e_i)$, and $n_i, e_i$ satisfy $n_1 + n_2 = n$ and $e_1 + e_2 = e$ (possibly $n_i = 0$, regarding $SFH(T,0)$ as trivial). That is,
\begin{align*}
 \left\{ \begin{array}{c} \text{Contact el'ts in} \\ \text{$SFH(T,n+1,e)$} \end{array} \right\} &=
	\bigsqcup_{\substack{n_1 + n_2 = n \\ e_1 + e_2 = e}}  M \left( \left\{ \begin{array}{c} \text{Contact el'ts in} \\ \text{$SFH(T, n_1, e_1)$} \end{array} \right\} , \left\{ \begin{array}{c} \text{Contact el'ts in} \\ \text{$SFH(T, n_2, e_2)$} \end{array} \right\} \right).
\end{align*}
\end{prop}
The precise nature of ``merging'' will be made clear in section \ref{sec_Catalan_Narayana}.

We also have a crucial enumerative result for our main theorem (section \ref{sec_enumerative}). Recall that the partial order $\preceq$ of $W(n_-, n_+)$ indexes the basis elements of $SFH(T,n+1,e)$.
\begin{prop}[Number of comparable pairs]
\label{enumerative_bijection}
The number of pairs $w_0, w_1$ in $W(n_-, n_+)$ with $w_0 \preceq w_1$ is $C_{n+1}^e$.
\end{prop}
That is, for given $n$ and $e$, contact elements and comparable pairs of words are equal in number.

\subsubsection{Contact elements and comparable pairs}

\label{sec_main_theorem}

Our main theorem fleshes out the enumerative proposition \ref{enumerative_bijection}, giving an explicit bijection between contact elements and comparable pairs of words. A general contact element is determined by decomposing it in terms of basis elements and looking at the first and last basis elements among them. We can think of every state as a morphism from a first state to a last state.
\begin{thm}[Contact elements and comparable pairs]
\label{main_theorem_maxmin}
\label{main_theorem}
Consider a contact element $v \in SFH(T,n+1,e)$. Writing $v$ as a sum of basis vectors $v_w$, where $w \in W(n_-, n_+)$, there is a lexicographically first $v_{w_-}$ and last $v_{w_+}$ basis vector amongst them. Then for every basis vector $v_w$ occurring in the sum, $w_- \preceq w \preceq w_+$. In particular, $w_- \preceq w_+$. Moreover, the map
\[
\Phi: \left\{ \begin{array}{c} \text{Contact} \\ \text{elements in} \\ SFH(T,n+1,e) \end{array} \right\} 
\cong
\left\{ \begin{array}{c} \text{Chord diagrams} \\ \text{with } n+1 \text{ chords,} \\ \text{euler class } e \end{array} \right\}
\To
\left\{ \begin{array}{c} \text{Comparable pairs} \\ \text{of words } w_1 \preceq w_2 \\ \text{in } W(n_-, n_+) \end{array} \right\}
\]
given by $v \mapsto (w_-, w_+)$ is a bijection.

That is, given any comparable pair $w_1 \preceq w_2$, there is precisely one contact element which, when written as a sum of basis elements, has $v_{w_1}$ as its first and $v_{w_2}$ as its last.
\end{thm}

We will denote the unique contact element with first basis element $v_{w_-}$ and last basis element $v_{w_+}$ by $[w_-, w_+]$ or $[v_{w_-}, v_{w_+}]$ or $[\Gamma_{w_-}, \Gamma_{w_+}]$, and throughout we will abuse notation, often identifying contact elements with chord diagrams and basis contact elements with words; hopefully this will not cause too much confusion.

\subsubsection{Moves on chord diagrams and words}
\label{sec_moves_diags_words}

The proof of the main theorem is by explicit construction. Given $w_- \preceq w_+$, we construct a chord diagram whose decomposition has $v_{w_-}$ as its first and $v_{w_+}$ as its last element. Then by the enumerative proposition \ref{enumerative_bijection}, this is shown to be a bijection.

\begin{defn}
\label{def_bypass_system}
A \emph{bypass system} is a finite set of disjoint arcs of attachment.
\end{defn}
We will build up a method for performing bypass surgery on bypass systems on basis chord diagrams, taking $\Gamma_{w_1}$ to $\Gamma_{w_2}$, whenever $w_1 \preceq w_2$, by upwards surgery on the \emph{bypass system of the pair} $(w_1, w_2)$. Conversely, we can take $\Gamma_{w_2}$ to $\Gamma_{w_1}$ by downwards moves. This method will be explicitly analogous to certain combinatorial ``word-processing'' moves on the corresponding words. 

Then, we will show that performing all these bypass moves \emph{in the opposite direction}, gives us a chord diagram whose decomposition has $w_1, w_2$ as first and last elements.
\begin{prop}[Bypass system of a comparable pair]
\label{bypass_system_one_to_other}
\label{bypass_system_other_way}
Suppose $\Gamma_1 \preceq \Gamma_2$ are basis chord diagrams.
On $\Gamma_1$, there exists a bypass system $FBS(\Gamma_1, \Gamma_2)$, and on $\Gamma_2$, there exists a bypass system $BBS(\Gamma_1, \Gamma_2)$ such that:
\begin{enumerate}
\item
performing \emph{upwards} bypass moves on $FBS(\Gamma_1, \Gamma_2)$ gives $\Gamma_2$;
\item
performing \emph{downwards} bypass moves on $BBS(\Gamma_1, \Gamma_2)$ gives $\Gamma_1$;
\item
performing \emph{downwards} bypass moves on $FBS(\Gamma_1, \Gamma_2)$ or \emph{upwards} bypass moves on $BBS(\Gamma_1, \Gamma_2)$ gives a chord diagram whose basis decomposition contains $\Gamma_1$ and $\Gamma_2$, and for every basis element $\Gamma_w$ in this decomposition, $\Gamma_1 \preceq \Gamma_w \preceq \Gamma_2$. That is, $\Gamma_1$ is a total minimum and $\Gamma_2$ a total maximum, with respect to $\preceq$, among all the basis elements occurring in the decomposition.
\end{enumerate}
\end{prop}

The proof of this proposition is based on correspondences between the following notions, which we will define in due course.
\[
\begin{array}{rcl}
\left\{ \begin{array}{c} \text{elementary move} \\ \text{on a word} \end{array} \right\} &\leftrightarrow& \left\{ \begin{array}{c} \text{bypass move} \\ \text{on an attaching arc} \end{array} \right\} \\
\left\{ \begin{array}{c} \text{generalised elementary} \\ \text{move on a word} \end{array} \right\} &\leftrightarrow& \left\{ \begin{array}{c} \text{bypass moves on the} \\ \text{bypass system of a} \\ \text{generalised attaching arc} \end{array} \right\} \\
\left\{ \begin{array}{c} \text{nicely ordered sequence of} \\ \text{generalised elementary moves} \\ \text{on a word} \end{array} \right\} &\leftrightarrow& \left\{ \begin{array}{c} \text{bypass moves on the} \\ \text{bypass system of a} \\ \text{nicely ordered sequence of} \\ \text{generalised attaching arcs} \end{array} \right\}.
\end{array}
\]
The final correspondence is strong enough to give the constructions in the above two propositions, explicitly; which in turn gives the main theorem.

\subsubsection{Contact elements are tangled}

If a contact element is determined by the first and last basis elements in its decomposition, and (by theorem \ref{main_theorem_maxmin}) every other basis element lies between the first and last with respect to $\preceq$, then a natural question arises: what are the other basis elements?

First, we can prove results about \emph{the number of basis elements} in a contact element. For a basis element, this answer is clear: one --- itself. Otherwise, we have:
\begin{prop}[Size of basis decomposition]
\label{even_number_decomposition}
Every chord diagram which is not a basis element has an even number of basis elements in its decomposition.
\end{prop}

Second, we can show that basis elements occurring in the decomposition of a contact element are ``tangled up'', in some sense. We have said that the first and last basis elements of a contact element $\Gamma$ are comparable to all others in the decomposition. We show that no other basis element has this property: we cannot ``untangle them''. 

\begin{thm}[Not much comparability]
\label{not_much_comparability}
Suppose $v_w$ occurs in the basis decomposition of $v = [v_{w_-}, v_{w_+}]$ and is comparable, with respect to $\preceq$, with every other basis element occurring in the decomposition. Then $w = w_-$ or $w_+$.
\end{thm}

More generally, proposition \ref{prop_prec_follow_even} will show that for any basis element $v_w$ in the decomposition of $v$, other than $v_{w_\pm}$, the number of basis elements $v_{w'}$ of $\Gamma$ such that $w' \preceq w$ (resp. $w \preceq w'$) is even (resp. even also). This implies the above theorem.

Third, the presence of $\pm$ symbols in certain positions in both $w_-$ and $w_+$ implies the presence of certain symbols in similar positions in all $w$ occurring in $[w_-, w_+]$. Such symbols also tell us about the corresponding chord diagram:
\begin{itemize}
\item (Lemma \ref{symbolic_basepoint}) A chord diagram $\Gamma = [\Gamma_-, \Gamma_+]$ has an outermost region at the base point, iff the words for $\Gamma_-, \Gamma_+$ begin with the same symbol, iff all basis elements of $\Gamma$ have words which begin with the same symbol. 
\item (Section \ref{symbolic_outermost}) Similarly, for various locations on the disc, a chord diagram has an outermost region at that location, iff the words for $\Gamma_-, \Gamma_+$ both possess a certain property (ending with the same symbol; having the $j$'th $-$ sign not the first in its block; etc.), iff each basis element of $\Gamma$ has the same property.
\end{itemize}

Fourth, we note it is possible to give an algorithm to write down the basis decomposition of any $[v_{w_-}, v_{w_+}]$ with $w_- \preceq w_+$. However, this basically just replicates the construction of bypass systems in the construction of the chord diagram, in combinatorial language (or writes a computer program to manipulate chord diagrams!).

\subsubsection{Computation by rotation}

None of the above gives a good way to \emph{compute} all contact elements. One way is to use \emph{rotation} of a chord diagram, giving a linear operator $R$ on $SFH(T,n,e)$. We may rotate any chord diagram until there is an outermost region adjacent to the base point; then it lies in the image of $B_\pm$. We will give a recursive formula (proposition \ref{prop_recursive_rotation}) for $R$, and describe it explicitly (proposition \ref{prop_explicit_rotation}). There is interesting combinatorics in the matrix of $R$; see section \ref{sec_rotation}. We wonder if it has other applications.

\subsubsection{Simplicial structure}

We will also show that there is a simplicial structure on the $SFH$ vector spaces forming the various diagonals of Pascal's triangle. We note that our creation and annihilation operators $A_\pm, B_\pm$ were defined at a particular point, namely the base point, but there are $2n$ marked points on the boundary of the disc. Choosing other points gives more creation and annihilation operators, which, as it turns out, obey the same relations as face and degeneracy maps in simplicial structures. The associated boundary maps make the categorified Pascal's triangle into a double chain complex.
\begin{prop}[Simplicial structure]
\label{prop_pascal_double_complex}
On each diagonal of Pascal's triangle, there are face and degeneracy maps giving it a simplicial structure, with boundary maps making each diagonal into a chain complex with trivial homology, and the whole triangle into a double complex.
\end{prop}

\subsection{Contact categories, stacking}

\label{sec_contact_overview}

Studying the $SFH$ of the solid torus leads to considerable algebraic and combinatorial structure, detailed above. Much of this structure actually has direct contact-geometric meaning; it has applications independent of sutured Floer homology. ``Bypass moves'', regarded as actual bypass attachments, relate these algebraic and combinatorial structures to contact structures on solid cylinders $D \times I$. This leads us to consider the ``contact category'' of Honda \cite{HonCat}, and various extensions and generalisations of it. Indeed, we seem to be led in the direction of a ``categorification of contact geometry''.

\subsubsection{Contact ``cobordisms'' and stackability}
\label{sec_stackability}

Bypass moves arise from the contact-geometric construction of \emph{bypass attachment} \cite{Hon00I}. A bypass is half an overtwisted disc (thickened), and an elementary contact-geometric building block. We analyse bypasses in detail in section \ref{sec_bypasses}. Attachment of bypasses on a disc gives a cylinder $D^2 \times I$ with distinct dividing sets on $D^2 \times \{0\}$, $D^2 \times \{1\}$.

This motivates a construction we call \emph{stacking}. Given two chord diagrams $\Gamma_0, \Gamma_1$, we form a sutured solid cylinder $\M(\Gamma_0, \Gamma_1)$, which is $D^2 \times I$ with sutures $\Gamma_i$ on $D^2 \times \{i\}$. We ask whether there is a tight contact structure on this sutured manifold: if so, we say $\Gamma_1$ is \emph{stackable} on $\Gamma_0$. We think of such a contact structure as a ``cobordism'' between the two convex discs given by $\Gamma_0, \Gamma_1$. Details will be given in section \ref{sec_discontents}.

The question of whether $\Gamma_1$ is stackable on $\Gamma_0$ is a linear question in $SFH$.
\begin{prop}[Stackability map]
\label{mexists}
There is a linear map
\[
 m: SFH(T,n) \otimes SFH(T,n) \To \Z_2
\]
which takes pairs of contact elements, corresponding to pairs of chord diagrams $\Gamma_0$, $\Gamma_1$, to $1$ or $0$ respectively as $\Gamma_1$ is stackable on $\Gamma_0$ or not.
\end{prop}

Thus $m$ is the boolean question ``Is $\Gamma_1$ stackable on $\Gamma_0$?'' Moreover, the summands $SFH(T,n,e)$ of $SFH(T,n)$ are ``orthogonal'' with respect to this question:
\begin{prop}[Relative euler class orthogonality]
\label{lem_euler_orthogonality}
Let $\Gamma_0$ and $\Gamma_1$ be chord diagrams with $n$ chords. If $\Gamma_0, \Gamma_1$ have distinct relative euler class then $m(\Gamma_0, \Gamma_1) = 0$.
\end{prop}

We give a complete description of $m$, intimately related to the partial order $\preceq$; in fact, on basis chord diagrams $\Gamma_w$, $m$ \emph{is} $\preceq$ (regarded as a boolean function). 
\begin{prop}[Contact interpretation of $\preceq$]\
\label{contact_interp_partial_order}\
$m(\Gamma_{w_0}, \Gamma_{w_1})=1$ iff $w_0 \preceq w_1$.
\end{prop}

Then we can use this to obtain a result for general chord diagrams.
\begin{prop}[General stackability]
\label{general_stackability}
Let $\Gamma_0, \Gamma_1$ be chord diagrams of $n$ chords with relative euler class $e$. Then $\Gamma_1$ is stackable on $\Gamma_0$ if and only if the cardinality of the following set is odd:
\[
 \left\{ (w_0, w_1) \; : \; w_0 \preceq w_1, \; \Gamma_{w_i} \text{ occurs in the decomposition of } \Gamma_i \right\}.
\]
\end{prop}

\begin{rem}
We write $\Gamma_w \in \Gamma$ to denote that $\Gamma_w$ occurs in the basis decomposition of $\Gamma$. After all, mod 2 arithmetic is boolean addition.
\end{rem}

We will show various other properties of $m$ and $\M$:
\begin{enumerate}
\item (Lemma \ref{MGammaGamma}) $m(\Gamma, \Gamma) = 1$, i.e. $\Gamma$ is stackable on itself.
\item (Lemma \ref{cancel_outermost}) If $\Gamma_0, \Gamma_1$ respectively have outermost chords $\gamma_0, \gamma_1$ in the same position, then $m(\Gamma_0, \Gamma_1) = m(\Gamma_0 - \gamma_0, \Gamma_1 - \gamma_1)$.
\item (Lemma \ref{bypass-related}) If $\Gamma_0, \Gamma_1$ are related by a bypass move then, when placed in the right order, they are stackable. (The order is given in lemma \ref{bypass-related}.)
\end{enumerate}

Lemma \ref{bypass-related} relates stackability to bypass moves; in fact, bypass triples naturally give triples of tight contact cobordisms. When $\Gamma_1$ can be obtained from $\Gamma_0$ by attaching bypasses on top of $\Gamma_0$, we have a construction of a contact structure on $\M(\Gamma_0, \Gamma_1)$. 

Our explicit construction of bypass moves from $\Gamma_{w_1}$ to $\Gamma_{w_2}$, via a bypass system $FBS(w_1, w_2)$ for any $w_1 \preceq w_2$, thus gives a contact structure on $\M(\Gamma_{w_1}, \Gamma_{w_2})$. We show this is tight (lemma \ref{MGamma+-}), and obtain a generalisation of bypass triples. We consider various possible generalisations of bypass triples as we proceed.

\subsubsection{Contact categories}

The question of which dividing sets are stackable on which others is the essence of the \emph{contact category} defined by Honda \cite{HonCat}. Honda shows that this category possesses certain properties of a triangulated category, and behaves functorially with respect to $SFH$. Essentially, objects in this category are dividing sets on a surface $\Sigma$, and morphisms are contact structures on $\Sigma \times I$ (a rigorous definition is given in \ref{defn_contact_category}). A nontrivial (tight) morphism $\Gamma_0 \To \Gamma_1$ precisely means that $\Gamma_1$ is stackable on $\Gamma_0$. Our map $m$ describes the morphisms in the contact category of a disc, $\mathcal{C}(D^2)$.

Further, we can start from a given cobordism $\M(\Gamma_0, \Gamma_1)$ with tight contact structure, and ask what chord diagrams $\Gamma$ occur as dividing sets of discs \emph{inside} this cobordism (definition \ref{def_existence_chord_diagram}). We give a criterion for when $\Gamma$ occurs (lemma \ref{lem_existence_of_chord_diagram}). Using this, we obtain easily that the only chord diagram existing in $\M(\Gamma, \Gamma)$ is $\Gamma$ itself (lemma \ref{lem_CbGammaGamma}).

This leads us to the notion of \emph{bounded contact category} (definition \ref{defn_bonded_ct_cat}) $\mathcal{C}^b (\Gamma_0, \Gamma_1)$: the ``subcategory of $\mathcal{C}(D^2)$ which is contained in $\M(\Gamma_0, \Gamma_1)$'', or the ``subcategory of $\mathcal{C}(D^2)$ bounded by $\Gamma_0$ and $\Gamma_1$''. Its objects are those dividing sets $\Gamma$ which occur in a tight $\M(\Gamma_0, \Gamma_1)$, and its morphisms are those cobordisms $\M(\Gamma, \Gamma')$ which occur in $\M(\Gamma_0, \Gamma_1)$. We prove (lemma \ref{lem_Cb_is_category}) that $\mathcal{C}^b (\Gamma_0, \Gamma_1)$ is indeed a category.

Lemma \ref{lem_CbGammaGamma} says that $\mathcal{C}^b (\Gamma, \Gamma)$ is trivial.

Any partially ordered set (such as $W(n_-, n_+)$ with $\preceq$) can be considered as a category; conversely, under certain conditions a category can be considered a partially ordered set (lemma \ref{lem_category_partial_order}). We can then prove the following.
\begin{prop}
\label{lem_Cb_partial_order}
The bounded contact category $\mathcal{C}^b (\Gamma_0, \Gamma_1)$ is partially ordered.
\end{prop}
 
In fact, if a tight contact structure on $\M(\Gamma_0, \Gamma_1)$ is obtained by attaching bypasses to $\Gamma_0$ along a bypass system $c$, then the power set $\mathcal{P}(c)$, considered a partially ordered set under inclusion, possesses a functor to $\mathcal{C}^b (\Gamma_0, \Gamma_1)$ (lemma \ref{lem_up_down}).

Contact categories have structures resembling ``exact triangles'' and ``cones'', analogously to a triangulated category. Bypass triples resemble exact triangles: the composition of two morphisms in the triangle is overtwisted, since a bypass is half an overtwisted disc. Two bypass-related chord diagrams determine a third one (their sum in $SFH$), which can be regarded as the cone of the morphism between them. 

When $\M(\Gamma_0, \Gamma_1)$ can be described by bypass attachments along a bypass system, we have generalised versions of bypass triples, exact triangles and cones. This includes triples $(\Gamma, \Gamma_-, \Gamma_+)$ where $\Gamma = [\Gamma_-, \Gamma_+]$. However, we believe these notions are still in an unsatisfactory state: not every tight contact cylinder can be constructed from attachments along a bypass system (lemma \ref{lem_not_every_cob_elementary}); and there is no general ``cone'' or ``exact triangle'' for a morphism. We discuss these and related issues in sections \ref{sec_discontents} and \ref{sec_conseq_main_results}.

\subsubsection{Computation of bounded contact categories}
\label{sec_computation_Cb}

We can compute some bounded contact categories, making use of the partial order $\preceq$. First, we can compute $\mathcal{C}^b (\Gamma_{w_0}, \Gamma_{w_1})$ for any basis chord diagrams $\Gamma_{w_0}, \Gamma_{w_1}$ corresponding to words $w_0, w_1 \in W(n_-, n_+)$. For a tight cobordism, we assume $w_0 \preceq w_1$.

\begin{defn}[Partially ordered set $W(w_0, w_1)$]
Given $w_0 \preceq w_1$ in $W(n_-,  n_+)$, let
\[
 	W(w_0, w_1) = \left\{ w \in W(n_-, n_+ ) \; : \; w_0 \preceq w \preceq w_1 \right\},
\]
endowed with the partial order (hence category structure) inherited from $W(n_-, n_+)$.
\end{defn}

\begin{prop}[Bounded contact category of basis cobordism]
\label{prop_Cb_basis_cob}
For $w_0 \preceq w_1$ corresponding to basis chord diagrams $\Gamma_{w_0}, \Gamma_{w_1}$,
\[
 \mathcal{C}^b \left( \Gamma_{w_0}, \Gamma_{w_1} \right) \cong W(w_0, w_1).
\]
The word $w \in W(w_0, w_1)$ corresponds to the basis chord diagram $\Gamma_w$.
\end{prop}
That is, the chord diagrams occurring in $\M(\Gamma_{w_0}, \Gamma_{w_1})$ are precisely the basis chord diagrams $\Gamma_w$ with $w_0 \preceq w \preceq w_1$; and convex discs in the cobordism with dividing sets $\Gamma_w$, $\Gamma_{w'}$ can be separated, $\Gamma_w$ below $\Gamma_{w'}$, if and only if $w \preceq w'$.

Taking $w_0$ and $w_1$ to be the total minimum and maximum in $W(n_-, n_+)$, i.e. $(-)^{n_-} (+)^{n_+}$ and $(+)^{n_+} (-)^{n_-}$ respectively, the bounded contact category is precisely $W(n_-, n_+)$; the chord diagrams occurring are precisely all the basis chord diagrams. This leads us to consider a sort of ``universal cobordism'' $\mathcal{U}(n_-, n_+)$, defined as
\[
\mathcal{U}(n_-, n_+) = \M \left( \Gamma_{(-)^{n_-} (+)^{n_+}}, \Gamma_{(+)^{n_+} (-)^{n_-}} \right).
\]
We denote its bounded contact category by
\[
 \mathcal{C}^b \left( \mathcal{U}(n_-, n_+) \right) = \mathcal{C}^b \left( \Gamma_{(-)^{n_-} (+)^{n_+}}, \Gamma_{(+)^{n_+} (-)^{n_-}} \right).
\]
Thus as a special case of the preceding proposition, the bounded contact category of a universal cobordism is given by
\[
\mathcal{C}^b \left( \mathcal{U}(n_-, n_+) \right) \cong W(n_-, n_+).
\]

We may regard $\mathcal{U}(n_-, n_+)$ as a ``geometric realisation'' of the category $W(n_-, n_+)$, in a moral (not technical) sense; similarly, $\M(\Gamma_{w_0}, \Gamma_{w_1})$ ``realises'' $W(w_0, w_1)$.

Although in a sense $\mathcal{U}(n_-, n_+)$ is the ``most complicated'' bounded contact category for given $n_\pm$, it is just a \emph{bypass cobordism}: a single bypass attachment on $\Gamma_{(-)^{n_-} (+)^{n_+}}$ gives $\Gamma_{(+)^{n_+} (-)^{n_-}}$ (see section \ref{sec_single_bypass_el_moves}). In effect, the computation of $\mathcal{C}^b(\mathcal{U}(n_-, n_+))$ tells us what ``bypasses exist inside the bypass''. The presence of extra chords near the attaching arc allows for extra ``intermediate'' bypasses.

As it turns out, $\mathcal{U}(n_-, n_+)$ actually describes the ``bypasses inside \emph{any} bypass''. We can compute the bounded contact category of \emph{any} cobordism $\M(\Gamma_0, \Gamma_1)$ obtained from attaching a single bypass above $\Gamma_0$ (here $\Gamma_0, \Gamma_1$ need not be basis chord diagrams): $\mathcal{C}^b(\Gamma_0, \Gamma_1)$ is isomorphic to the ``contact category of the largest universal cobordism that can be embedded into $\M(\Gamma_0, \Gamma_1)$''. The presence of other chords makes no difference. The precise statement is theorem \ref{thm_Cb_bypass_cob}.

\subsubsection{Categorical meaning of main theorem}
\label{sec_intro_categorical_meaning}

Our main theorem can be interpreted in this language of contact categories. This largely amounts to saying the same thing with fancier words, but may still be of interest.

The theorem, for given $w_- \preceq w_+$, furnishes a bypass system $FBS(w_-, w_+)$ such that $\Up_{FBS(w_-, w_+)} \Gamma_{w_-} = \Gamma_{w_+}$ and $\Down_{FBS(w_-, w_+)} \Gamma_{w_-} = \Gamma = [\Gamma_{w_-}, \Gamma_{w_+}]$. Moreover, attaching bypasses along $FBS(w_-, w_+)$ actually gives the tight contact structure on $\M(\Gamma_{w_-}, \Gamma_{w_+})$ (lemma \ref{MGamma+-}); that is, it is \emph{elementary} (definition \ref{def_elementary_cob}).
\begin{prop}[Tight basis cobordisms elementary]
\label{prop_tight_basis_cob_elementary}
Let $\Gamma_0$ and $\Gamma_1$ be basis chord diagrams, and suppose $\M(\Gamma_0, \Gamma_1)$ is tight. Then $\M(\Gamma_0, \Gamma_1)$ is elementary.
\end{prop}

We may therefore consider $\Gamma, \Gamma_{w_-}, \Gamma_{w_+}$ as a ``generalised bypass triple'' or ``exact triangle''. In fact, $\Gamma = [\Gamma_{w_-}, \Gamma_{w_+}]$ can be regarded as the ``cone'', the third element in an exact triangle arising from the morphism $\Gamma_{w_-} \To \Gamma_{w_+}$. Chord diagrams of $n$ chords and euler class $e$ are in bijective correspondence with morphisms of the bounded contact category $\mathcal{C}^b(\mathcal{U}(n_-, n_+))$ of the universal cylinder, and may be regarded as their cones. We make this precise in proposition \ref{prop_chord_diagrams_cones}.

\subsubsection{A contact 2-category}
\label{sec_2-category_simplicial}

We also consider generalisations of Honda's contact category in another direction.

The abstract nonsense version of our main theorem (proposition \ref{prop_chord_diagrams_cones}) says that in our simple case (i.e. a disc), the objects of the contact category can themselves be viewed as morphisms. For chord diagrams are described by pairs $w_0 \preceq w_1$, and $\preceq$ describes morphisms in $W(n_-, n_+)$. In this spirit, we define a \emph{contact 2-category} $\mathcal{C}(n+1,e)$: the objects of Honda's contact category become its 1-morphisms, and the morphisms of Honda's category become 2-morphisms, or ``morphisms between morphisms''.
\begin{prop}[Contact 2-category]
\label{contact_2_category}
There is a 2-category $\mathcal{C}(n+1,e)$ such that:
\begin{enumerate}
\item objects are words $w \in W(n_-, n_+)$, equivalently basis chord diagrams $\Gamma_w$;
\item 1-morphisms are chord diagrams of $n+1$ chords and euler class $e$;
\item 2-morphisms $\Gamma_0 \rightarrow \Gamma_1$ are contact structures on $\M(\Gamma_0, \Gamma_1)$, with (vertical) composition given by stacking contact structures.
\end{enumerate}
\end{prop}

Note that $\mathcal{C}(n+1,e) \cong \mathcal{C}^b(\mathcal{U}(n_-,n_+)) \cong W(n_-, n_+)$ as a $1$-category; so $\mathcal{C}(n+1,e)$ can be regarded as a ``2-category'' structure on $W(n_-, n_+)$ or $\mathcal{U}(n_-, n_+)$. It may be that considering all values of $n$ and $e$, we obtain a 3-category.

\subsection{Structure of this paper, acknowledgments}

\label{sec_intro_str_thesis}

This introduction gives an overview of our results, in a narrative order; results relating to contact elements in $SFH$ are separated from results about contact categories and cobordisms. However, the body of this paper presents results in logical order.

Section \ref{ch_first_steps} contains preliminary steps. We build up our picture of $SFH$ vector spaces categorifying Pascal's triangle, and establish our basis. We prove a $TQFT$-version of the Narayana recursion, and the relationship of Narayana numbers to the order $\preceq$.

In section \ref{ch:considerations}, we return to contact geometry with a thorough study of bypasses and contact categories and ``cobordisms'' (section \ref{sec_discontents}).

In section \ref{ch_basis_interpretations} we consider basis chord diagrams. We show how to construct them from words, how to decompose into this basis, and prove ``the stackability map $m$ is $\preceq$''.

In section \ref{ch_bypass_systems_basis_diagrams} we consider bypass systems on basis chord diagrams. We detail the possible bypass moves on basis diagrams, giving a bypass system taking $\Gamma_{w_0}$ to $\Gamma_{w_1}$, for any $w_0 \preceq w_1$; we use these to compute bounded contact categories. 

Section \ref{ch_ct_elts} then turns to a study of contact elements. We complete the main theorem and prove various properties of contact elements. We discuss how the main theorem describes a generalised bypass triple with contact and categorical implications.

Finally, in section \ref{ch_further_considerations} we compute the operator for rotation; we give the simplicial and double complex structure on the categorified Pascal's triangle; we introduce our contact 2-category; and make some remarks about and extending our results beyond discs.

As may already be clear, we assume familiarity with basic 3-dimensional contact geometry (see e.g. \cite{Et02, Hon3Dim, HonDimThree, Kaz}), including convex surfaces \cite{Gi91} and bypasses \cite{Hon00I}.

This paper is an abridged version of the author's PhD thesis at Stanford, completed during a postdoctoral fellowship at the Universit\'{e} de Nantes, supported by the ANR grant ``Floer power''. I would like to thank Yasha Eliashberg and Steve Kerckhoff.

\section{First steps}

\label{ch_first_steps}

\subsection{First observations in $SFH$}

\label{sec_observations_SFH}

\subsubsection{The vacuum}

\label{sec_vacuum}

Begin with $(T, 1)$, the solid torus with one pair of longitudinal sutures; $SFH(T,1) = SFH(T,1,0) = \Z_2$. There are not many tight contact structures on $(T,1)$ --- not many chord diagrams with $1$ chord! The unique tight contact structure on this sutured manifold is a standard neighbourhood of a closed legendrian curve.

It's not difficult to see that the contact element for this contact structure must be nonzero: it can be embedded into the (Stein fillable) standard contact $S^3$, for instance.
\begin{lem}
\label{lem_vacuum}
The contact element of the unique tight contact structure on $(T,1)$ is the nonzero element $v_\emptyset \in \Z_2 = SFH(T,1)$.
\qed
\end{lem}

We call $v_\emptyset$ \emph{the vacuum}. The vacuum state in quantum field theory is not zero.

\subsubsection{Creation and annihilation}
\label{sec_creation_annihilation}

To define creation operators, consider the following embedding $(T,n) \hookrightarrow (T,n+1)$ and contact structure on $(T,n+1) - (T,n)$. Embed a disc inside a larger disc, all times $S^1$. The intermediate manifold $(T,n+1) - (T,n)$ is an annulus times $S^1$, with $2n+2$ longitudinal sutures ``on the outside'', and $2n$ longitudinal sutures ``on the inside''. Specify an $S^1$-invariant contact structure by drawing a dividing set on the annulus, as in figure \ref{fig:11}, with base points marked as shown.

\begin{figure}[tbh]
\centering
\includegraphics[scale=0.6]{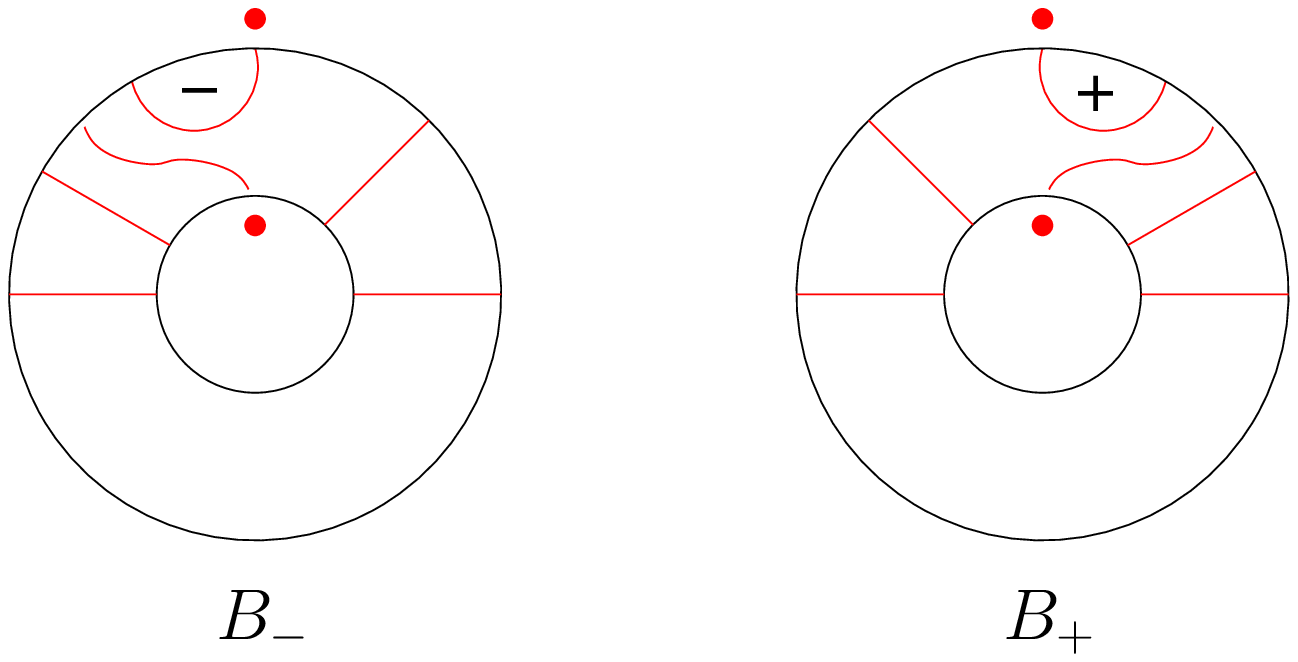}
\caption{Inclusion of sutured manifolds for $B_\pm$.} \label{fig:11}
\end{figure}

\begin{defn}[Creation operators]
The \emph{creation operators} are the maps
\[
 B_-, B_+ : SFH(T, n) \To SFH(T, n+1)
\]
given by TQFT-inclusion, from $(T,n) \hookrightarrow (T,n+1)$ together with the contact structures on $(T,n+1) - (T,n)$ described by the dividing sets in figure \ref{fig:11}.
\end{defn}

Given a chord diagram $\Gamma$ of $n$ chords, applying $B_\pm$ to its contact element gives the contact element described by the chord diagram with $(n+1)$ chords, adding a chord enclosing an outermost $\pm$ region near the base point, as described in the introduction. Applying $B_\pm$ adds $\pm 1$ to the euler class of this contact structure, so takes contact elements in $SFH(T,n,e)$ to $SFH(T,n+1,e\pm 1)$.

Define annihilation maps similarly, from an embedding $(T,n+1) \hookrightarrow (T,n)$.

\begin{figure}[tbh]
\centering
\includegraphics[scale=0.6]{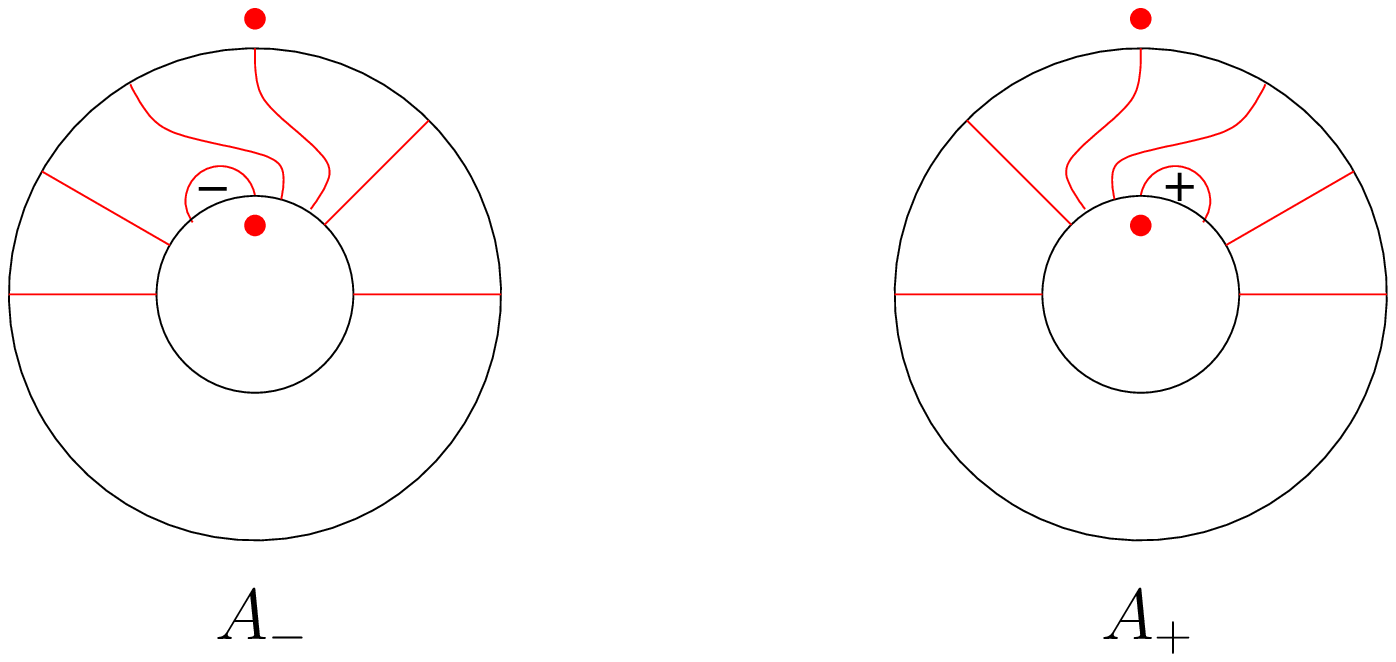}
\caption{Inclusion of sutured manifolds for $A_\pm$.} \label{fig:13}
\end{figure}

\begin{defn}[Annihilation operators]
The \emph{annihilation operators} are maps
\[
A_+, A_- : SFH(T, n+1) \To SFH(T, n)
\]
given by TQFT-inclusion, from $(T,n+1) \hookrightarrow (T,n)$ together with the contact structures on $(T,n) - (T,n+1)$ described by the dividing sets in figure \ref{fig:13}.
\end{defn}

It's clear that $A_\pm$ has the effect on contact elements described in the introduction, taking contact elements in $SFH(T,n+1,e)$ to $SFH(T,n,e \pm 1)$.

Note that $A_\pm$ may produce a diagram with a closed loop. The corresponding contact structure is overtwisted, and the contact element is zero.

Clearly, we could consider ``creation'' and ``annihilation'' not just near the base point, but at any specific location. We will use these later; for now we note they exist.

It's also now clear that the creation and annihilation effects have the relations
\[
A_+ \circ B_- = A_- \circ B_+ = 1 \quad \text{and} \quad A_+ \circ B_+ = A_- \circ B_- = 0,
\]
when applied to contact elements. Just place these figures of annuli together. Restricting to each summand $SFH(T, n+1, e)$, we have proved proposition \ref{annihilation}.

\subsubsection{Nontriviality and uniqueness}

\label{sec_ct_elt_uniqueness}

We prove contact classes are distinct and nonzero; arguments also appearing in \cite{HKM08}.

\begin{lem}
Any tight contact structure $\xi$ on $(T,n)$, corresponding to a chord diagram $\Gamma$, has nonzero contact element $c(\xi)$.
\end{lem}

\begin{proof}
For any such $\xi$ and $\Gamma$, at most one of $A_\pm$ can create a closed loop. Therefore, repeatedly applying $A_\pm$ we may reduce $\Gamma$ to one chord. The composition of these annihilation operators is a linear map which takes $c(\xi)$ to $v_\emptyset \neq 0$.
\end{proof}

\begin{proof}[Proof of proposition \ref{contact_distinct}]
Let $\Gamma_1, \Gamma_2$ be two distinct dividing sets of $n$ chords. There is a sequence of annihilation operators which reduces the contact element of $\Gamma_1$ to the vacuum state but which, when applied to the contact element of $\Gamma_2$, at some point creates a closed curve. These annihilation operators might not be applied in the positions of $A_+$, $A_-$, but may be at other positions; there is nothing special about annihilating at the base point. The composition of these operators takes the contact element of $\Gamma_1$ to $v_\emptyset$ but takes that of $\Gamma_2$ to $0$; hence they cannot be equal.
\end{proof}

This establishes a bijective correspondence; the same is true for refinements by $e$.
\[
\left\{ 
\begin{array}{c}
\text{Tight contact} \\ \text{structures on } (T,n)
\end{array}
\right\}
\cong
\left\{
\begin{array}{c}
\text{Chord diagrams} \\ \text{of } n \text{ chords}
\end{array}
\right\}
\cong
\left\{
\begin{array}{c}
\text{Nonzero contact} \\ \text{elements in } SFH(T,n)
\end{array}
\right\}
\]

\begin{rem}[Lax notation]
We often denote by $\Gamma$ a chord diagram, or its corresponding contact element, and drop the notation $c(\xi)$. The meaning should be clear.
\end{rem}

\subsubsection{Bypasses and addition}

\label{sec_byp_add}

The smallest $n,e$ for which there is more than one chord diagram is $n=3$ and $e=0$. Since $SFH(T,3,0) = \Z_2^2$ and $C_3^0 = 3$, there are 3 chord diagrams giving 3 distinct elements of $\Z_2^2$: see figure \ref{fig:14}. They form the simplest nontrivial bypass triple.
\begin{figure}
\centering
\includegraphics[scale=0.35]{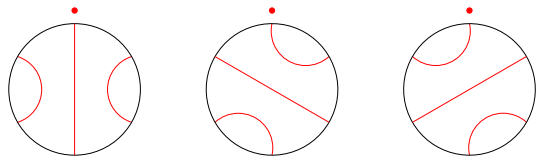}
\caption{Chord diagrams in $SFH(T,3,0)$.} \label{fig:14}
\end{figure}

The 3 nonzero elements of $\Z_2^2$ have the property that they sum to zero. The following proposition is a reformulation of \ref{addition_bypasses}.
\begin{prop}
\label{prop_sum_to_zero}
If three chord diagrams $\Gamma_1, \Gamma_2, \Gamma_3$ form a bypass triple, then $\Gamma_1 + \Gamma_2 + \Gamma_3 = 0$. Conversely, if three chord diagrams $\Gamma_1, \Gamma_2, \Gamma_3$ satisfy $\Gamma_1 + \Gamma_2 + \Gamma_3 = 0$, then $\Gamma_1, \Gamma_2, \Gamma_3$ form a bypass triple.
\end{prop}

\begin{proof}
Suppose $\Gamma_1, \Gamma_2, \Gamma_3$ form a bypass triple in $SFH(T,n,e)$. Then they are obtained from the three diagrams in figure \ref{fig:14} by adding an annulus, with fixed arcs on it, to their outsides. Using TQFT-inclusion we obtain a linear map
\[
\Z_2^2 \cong SFH(T,3,0) \To SFH(T,n,e).
\]
The three contact classes in $SFH(T,3,0)$ sum to zero; hence so too do $\Gamma_1, \Gamma_2, \Gamma_3$.

For the converse: proof by induction on the number of chords $n$. For $n=3$ it is clear. Suppose three chord diagrams sum to zero. Then they all have the same relative euler class, and (since all contact elements of chord diagrams are nonzero) are all distinct.

We use the following fact: given any two distinct chord diagrams with the same $n$ and $e$, there exists an annihilation operator, annihilating at some location (possibly not the base point), that creates no closed curves on either. If annihilating at every position creates a closed curve on at least one of the diagrams, then the two chord diagrams consist entirely of outermost chords, enclosing all positive regions on one diagram, and all negative regions on the other. Thus the diagrams have distinct $e$, a contradiction.

Applying this to $\Gamma_1, \Gamma_2$, we find an annihilation operator $A$ such that $A \Gamma_1, A \Gamma_2$ are nonzero. If $A \Gamma_1 \neq A \Gamma_2$ then $A \Gamma_3 \neq 0$; thus we have reduced to a smaller case and are done by induction. If $A \Gamma_1 = A \Gamma_2$ then the situation must be as in figure \ref{fig:50}; and hence $\Gamma_1, \Gamma_2$ are related by a bypass move.  Then $\Gamma_3 = \Gamma_1 + \Gamma_2$ (by the first part of the proposition) is the third diagram in their bypass triple.
\end{proof}

\begin{figure}
\centering
\includegraphics[scale=0.5]{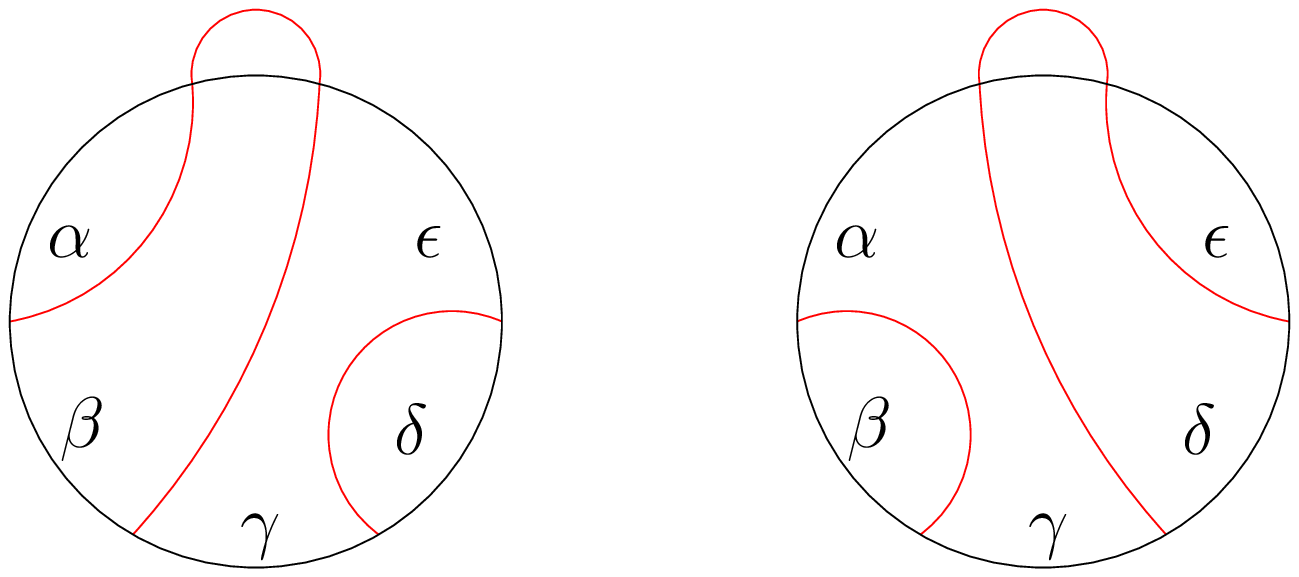}
\caption[Distinct $\Gamma_1, \Gamma_2$ for which $A \Gamma_1 = A \Gamma_2$.]{Distinct $\Gamma_1, \Gamma_2$ for which $A \Gamma_1 = A \Gamma_2$. Here $\alpha, \beta, \gamma, \delta, \epsilon$ denote that the two diagrams contain identical chords in these regions.} \label{fig:50}
\end{figure}

\subsubsection{The basis}
\label{sec_basis}

We now show that $\{ v_w, \; w \in W(n_-,n_+)\}$ forms a basis, proving proposition \ref{QFT_basis}. Recall (section \ref{sec_intro_basis}) $v_w$ is obtained from applying $B_\pm$ to $v_\emptyset$ according to the word $w$. 

\begin{proof}[Proof of proposition \ref{QFT_basis}]
To see the $v_w$ are linearly independent, suppose some $v_{w_1} + \cdots + v_{w_j} = 0$. To this sum apply annihilation operators which undo the creation operators in the definition of $v_{w_1}$. The composition $A$ of these operators takes $v_{w_1}$ to $v_\emptyset \neq 0$ and every other $v_{w_i}$ to $0$; hence $A(v_{w_1} + \cdots + v_{w_j}) = v_\emptyset = 0$, a contradiction.

The number of $v_w$ is $|W(n_-,n_+)| = \binom{n}{k} = \dim SFH(T,n+1,e)$. Hence they form a basis.
\end{proof}

\begin{proof}[Proof of proposition \ref{categorification}]
It remains to prove that $B_\pm$ are injective and that
\[
SFH(T,n+1,e) = B_+ SFH(T,n,e-1) \oplus B_- SFH(T,n,e+1).
\]
A basis of $SFH(T,n+1,e)$ consists of elements $v_w$. If $w$ begins with a $+$ (resp. $-$), $w=+w'$ (resp. $-w'$), then $v_w = B_+ v_{w'} \in B_+ SFH(T,n,e-1)$ (resp. $B_- v_{w'} \in B_- SFH(T,n,e+1)$. This proves the recursion, and injectivity is clear.
\end{proof}

There are simple algorithms to decompose a contact element into this basis, detailed in section \ref{sec_decomposition}. Essentially, there is either an outermost region at the base point, or there is not. If there is, we factor out a $B_\pm$ and reduce to a smaller chord diagram. If not, we perform bypass surgery near the base point to write our diagram as a sum of two other diagrams, each containing an outermost region at the base point. We proceed until we reach the vacuum. See figure \ref{fig:15} for an example.
\begin{figure}
\centering
\includegraphics[scale=0.35]{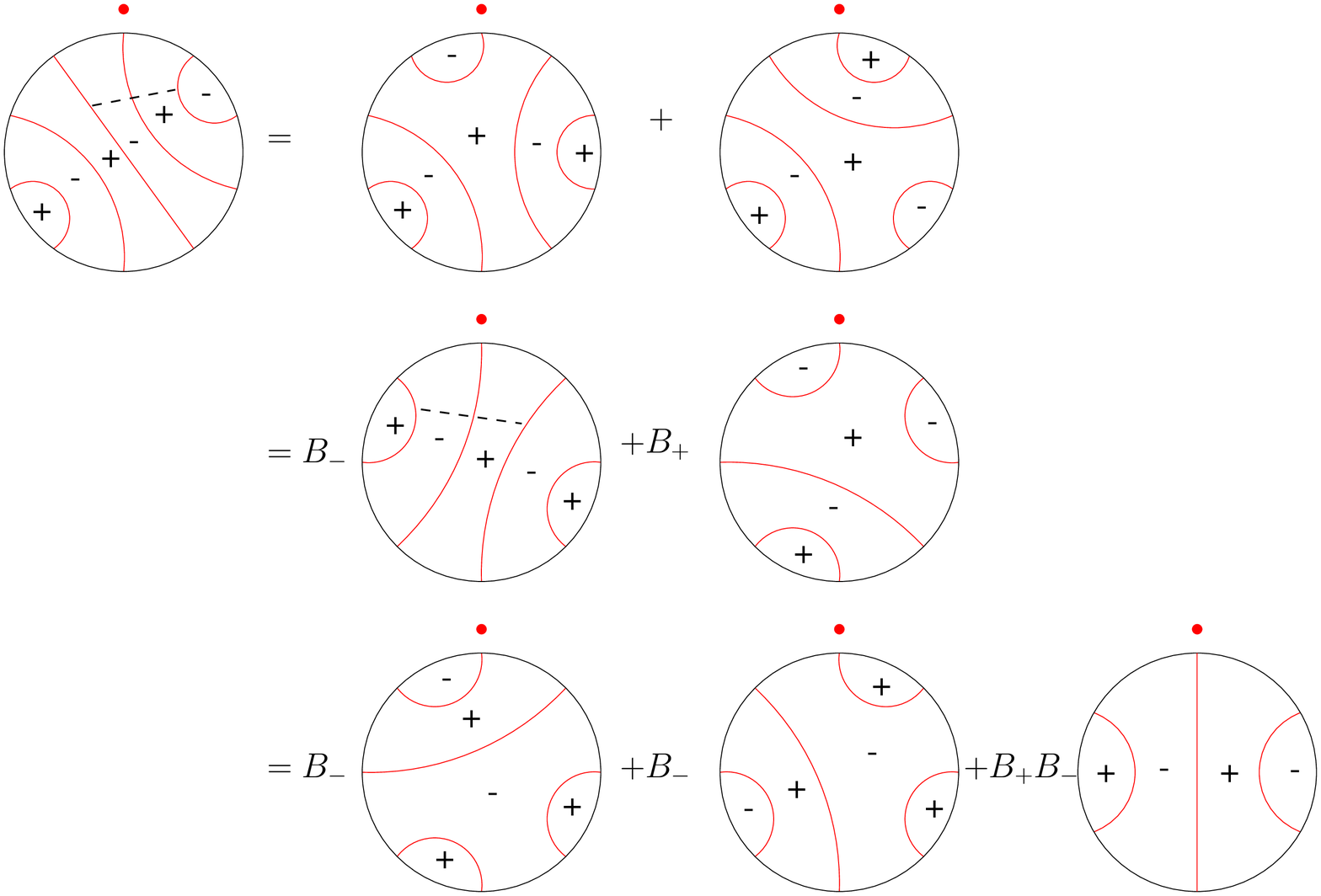}
\caption[Basis decomposition.]{Basis decomposition. A couple more steps shows that the original diagram is $v_{--++} + v_{-++-} + v_{+--+} + v_{+-+-}$.} \label{fig:15}
\end{figure}

We can now prove proposition \ref{combinatorial_SFH}, ``SFH is combinatorial''.
\begin{proof}[Proof of proposition \ref{combinatorial_SFH}]
Recall $SFH_{comb}(T,n,e) = \Z_2 \langle V \rangle / \mathcal{B}$ where $\Z_2 \langle V \rangle$ is the free $\Z_2$-vector space generated by chord diagrams, and $\mathcal{B}$ is the subspace generated by bypass relations. There is certainly a linear map $\Z_2 \langle V \rangle \To SFH(T,n,e)$, taking chord diagrams to the corresponding contact elements, and taking $\mathcal{B} \mapsto \{0\}$; hence it descends to a map $\phi: SFH_{comb} \To SFH$. Since $SFH$ is spanned by chord diagrams, $\phi$ is surjective and $\dim SFH_{comb} \geq \binom{n}{k}$. But $SFH_{comb}(T,n,e)$ is spanned by the $\Gamma_w$, of which there are $\binom{n}{k}$, hence $\phi$ is an isomorphism.
\end{proof}

\subsubsection{The octahedral axiom}

\label{sec_octahedral}

We briefly return to the example $SFH(T,4,-1) = \Z_2^3$ from section \ref{fun}. We have $6$ chord diagrams as described there, namely $3$ basis elements and sums of them in pairs. Consider elements of $\Z_2^3$ as vertices of a cube with coordinates $(x,y,z)$. Each 2-dimensional subspace generated by two basis elements contains 3 contact elements which form a bypass triple; and also the subspace $x+y+z=0$. Thus, the 6 vertices of the cube which are contact elements contain between them 4 triangles which are bypass triples. We can arrange these 6 vertices as an octahedron, with 4 of the 8 faces exact. This is the arrangement which appears in the octahedral axiom of Honda \cite{HonCat}.

Every $SFH(T,n,e)$ can be considered a higher-order version of the octahedral axiom.

\subsection{Enumerative combinatorics}
\label{sec_enumerative}

\subsubsection{Catalan, Narayana, and merging}
\label{sec_Catalan_Narayana}

Recall the discussion of Catalan and (shifted) Narayana numbers in section \ref{sec_intro_Catalan_Narayana}; take the recursive definitions
\[
 C_{n+1} = \sum_{n_1 + n_2 = n} C_{n_1} C_{n_2}, \quad 
 C_{n+1}^e = C_n^{e-1} + C_n^{e+1} + \sum_{\substack{n_1 + n_2 = n \\ e_1 + e_2 = e}} C_{n_1}^e \; C_{n_2}^e.
\]
To see that the number of chord diagrams with $n$ chords (resp. and with relative euler class $e$) satisfies the Catalan (resp. Narayana) recursion, consider the \emph{merging operation} shown in figure \ref{fig:16}. Given two chord diagrams $\Gamma_1, \Gamma_2$ with $n_1, n_2$ chords and relative euler classes $e_1, e_2$, we obtain a chord diagram with $n_1 + n_2 + 1$ chords and relative euler class $e_1 + e_2$. Note the specification of base points. When one $n_i = 0$, the operation reduces to $B_\pm$. 
\begin{figure}
\centering
\includegraphics[scale=0.35]{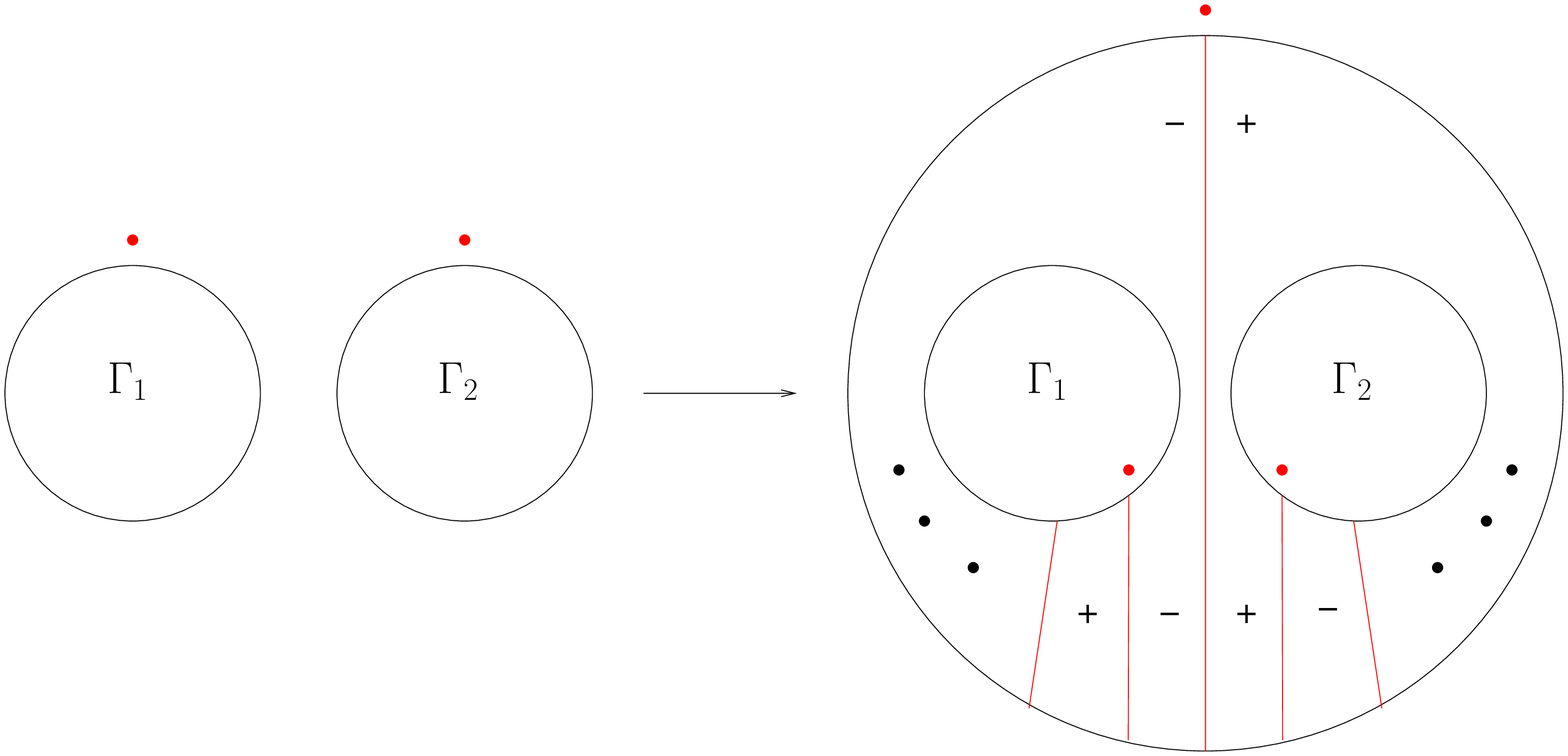}
\caption{Merging operation.} \label{fig:16}
\end{figure}

Any chord diagram can be expressed as the merge of two (possibly null) chord diagrams, in precisely one way. Counting the number of chord diagrams of $n$ chords gives the Catalan recursion. Keeping track of relative euler class gives the Narayana recursion. It is then clear that $C_n = \sum_e C_n^e$. This proves proposition \ref{Narayana_recursion}.

The ``merging'' operation precisely describes a sutured inclusion $(T, n_1) \sqcup (T, n_2) \hookrightarrow (T, n_1 + n_2 + 1)$, together with a contact structure on $(T, n_1 + n_2 + 1) - \left( (T, n_1) \sqcup (T, n_2) \right)$. Thus TQFT-inclusion applies. Note $SFH( (T,n_1) \sqcup (T, n_2) ) = SFH(T, n_1) \otimes SFH(T, n_2)$.
\begin{defn}[Merge operator]
The linear map
\[
 M: SFH(T, n_1) \otimes SFH(T, n_2) \To SFH(T, n_1 + n_2 + 1)
\]
arises from the merging operation on chord diagrams described above. It restricts to a map on summands
\[
 M: SFH(T, n_1, e_1) \otimes SFH(T, n_2, e_2) \To SFH(T, n_1 + n_2 + 1, e_1 + e_2).
\]
\end{defn}

When one $n_i$ is $0$, this definition naturally extends as $B_\pm$. Every contact element lies in the image of $M$, applied to contact elements; this proves proposition \ref{Catalan_categorification}.

\subsubsection{Counting comparable pairs}
\label{sec_counting_comparable}

We now prove proposition \ref{enumerative_bijection}: the number of pairs $w_0 \preceq w_1$ in $W(n_-, n_+)$ is $C_{n+1}^e$.

We give a ``baseball'' interpretation of the partial order $\preceq$. The $m$'th symbol in a word $w$ is the $m$'th \emph{inning}. The sum of the first $m$ symbols the \emph{score after $m$ innings}. The relation $w_0 \preceq w_1$ means precisely that after every inning, $w_1$ is not losing.

(Note, this is \emph{low-scoring} baseball: every inning, each team scores $\pm 1$ run. It is also \emph{fixed}: the end result is tied. The lead changes precisely when words are not comparable; comparable words are \emph{uninteresting} as spectator sport. Two words are comparable iff they describe a low-scoring, fixed, and uninteresting baseball game.)

\begin{proof}[Proof of proposition \ref{enumerative_bijection}]
First, there is a bijection between pairs of comparable words of length $n$ with $k$ plus signs, and monotone increasing functions 
\[
f: \{1, \ldots, n+1\} \To \{1, \ldots, n+1\}, \quad \text{i.e. } \quad f: [n+1] \To [n+1],
\]
satisfying $f(i) \leq i$ for all $i$ and taking $k+1$ distinct values. The bijection is as follows. Given a comparable pair $w_0 \preceq w_1$, we know that for all $j$, the $j$'th $+$ sign in $w_0$ is to the right of the $j$'th $+$ sign in $w_1$. Insert a $+$ at the start of $w_0$ and $w_1$ to obtain $w'_0, w'_1$, so these are words of length $n+1$ with $k+1$ plus signs. For $i \in \{1, \ldots, n+1\}$, let the number of $+$ signs up to and including the $i$'th symbol of $w_0$ be $j(i)$; then define $f(i)$ to be the position of the $j(i)$'th $+$ sign in $w_1$.

Conversely, given such a function, we can easily reconstruct the words $w_0, w_1$. The positions of the $+$ signs in $w'_1$ are precisely the values of $f$. And the positions of the $+$ signs in $w'_0$ are precisely those $i$ for which $f(i)$ jumps, $f(i) > f(i-1)$.

The number of such functions $f: [n+1] \To [n+1]$ with $k+1$ distinct values is well known to be $N_{n+1,k+1} = C_{n+1,k} = C_{n+1}^e$.

To see this, we can show that $F_{n,k}$, the number of increasing $f: [n] \To [n]$ with $f(i) \leq i$ taking $k$ values, satisfies the Narayana recursion; clearly $F_{n,k} = N_{n,k}$ for small values. Clearly any such function has the fixed point $f(1) = 1$. The number with no other fixed points is $F_{n-1,k}$. The number with a fixed point $f(2) = 2$ is $F_{n-1,k-1}$. Otherwise let $j$ be the least fixed point $\geq 2$. We can then ``break the function into two'' at that fixed point, and the number of such functions is given by $F_{j-2,k_1} F_{n-j+1,k_2}$ over the possible $k_1, k_2$ where $k_1 + k_2 = k$.
\end{proof}

\section[Contact considerations]{Contact considerations}

\label{ch:considerations}

\subsection{Edge rounding}

\label{ch_preliminaries}
\label{sec_contact_prelim}
\label{sec_fundamental_contact}
\label{sec_convex_surfaces}

Suppose two convex surfaces $\Sigma_1, \Sigma_2$ with dividing sets $\Gamma_1, \Gamma_2$, meet along a common legendrian boundary $C$, forming a corner. In this case $\Gamma_1, \Gamma_2$ ``interleave'' along $C$ (figure \ref{fig:8}(left)) and $|\Gamma_1 \cap C| = |\Gamma_2 \cap C| = \frac{1}{2} |tb(C)|$. Rounding the corner gives a smooth surface; dividing curves behave as in figure \ref{fig:8}(right). See \cite{Hon00I} for details.

\begin{figure}
\centering
\includegraphics[scale=0.3]{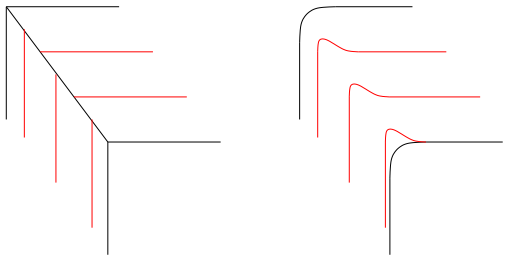}
\caption{Edge rounding of convex surfaces.} \label{fig:8}
\end{figure}

\label{sec_ct_strs_balls_tori}

\subsection{On Bypasses}

\label{sec_bypasses}
\label{sec_bypass_manifold}

The notion of bypass was introduced by Honda in \cite{Hon00I}: see there for further details. A thickened convex overtwisted disc $D \times I$ has dividing set on its boundary consisting of 3 closed loops; one on each of $D \times \{0\}$, $D \times \{1\}$, and $\partial D \times \{ \frac{1}{2} \}$. A bypass is half of this object, slicing through a diameter of $D$, times $I$; on this ``sliced'' part, thought of as the ``base'', the dividing set consists of three arcs of the form $\{\cdot\} \times I$.

We may attach a bypass to a convex surface along an attaching arc, above or below. Rounding and flattening then gives upwards or downwards bypass surgery on dividing sets. Clearly, adding two bypasses, above and below a convex surface, along the same attaching arc, attaches an overtwisted disc.

Bypasses are the ``smallest building blocks'' of contact structures; on the other hand, a bypass is half an overtwisted disc. The smallest step in contact geometry is half way to oblivion; such is the precariousness of all tight contact life.

\subsubsection{When do bypasses exist?}
\label{sec_bypasses_exist}

Suppose we have a contact 3-manifold $(M, \xi)$ with convex boundary, and an attaching arc $c$ on $\partial M$. We ask whether there exists a bypass inside $M$ along $c$. Sometimes there is an easy answer to this question.
\begin{enumerate}
\item 
If $\xi$ is tight, and bypass surgery along $c$, inwards into $M$, would result in a convex surface with overtwisted neighbourhood (easily detected by looking at the dividing set), then no such bypass exists.
\item
If bypass surgery along $c$, inwards into $M$, would result in a convex surface with dividing set isotopic to the original, then a bypass exists there. This principle has been mischievously named the ``right to life'' principle \cite{Hon02, HKM01}.
\item
One existing bypass may imply the existence of others. After attaching one bypass, other arcs of attachment may become trivial, so that bypasses exist by the right-to-life principle: this is ``bypass rotation'' \cite{HKMPinwheel}. See figure \ref{fig:22} (left). 
\end{enumerate}

\begin{figure}
\centering
\mbox{
\includegraphics[scale=0.35]{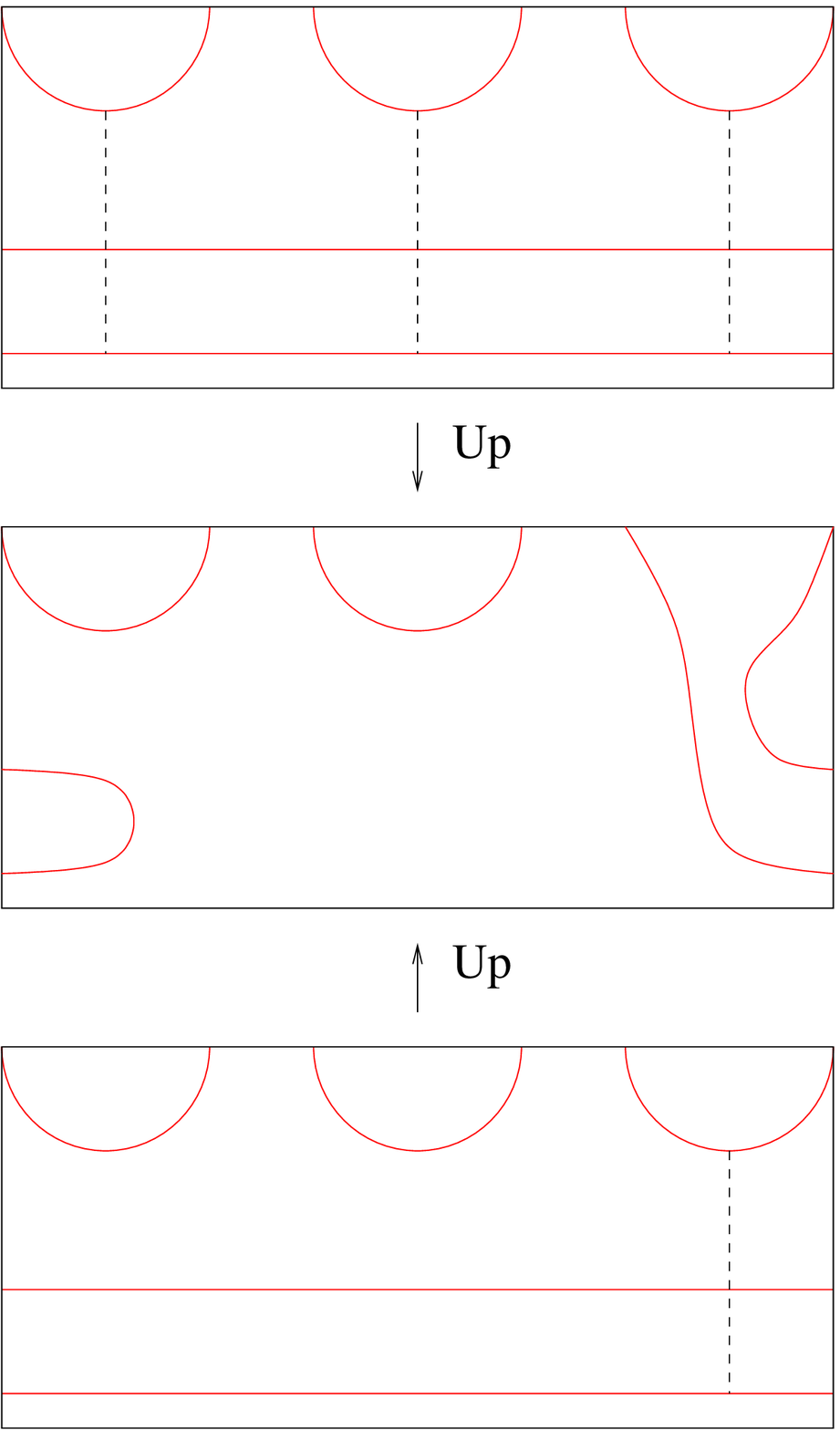}
\quad \quad \quad \quad \quad
\includegraphics[scale=0.5]{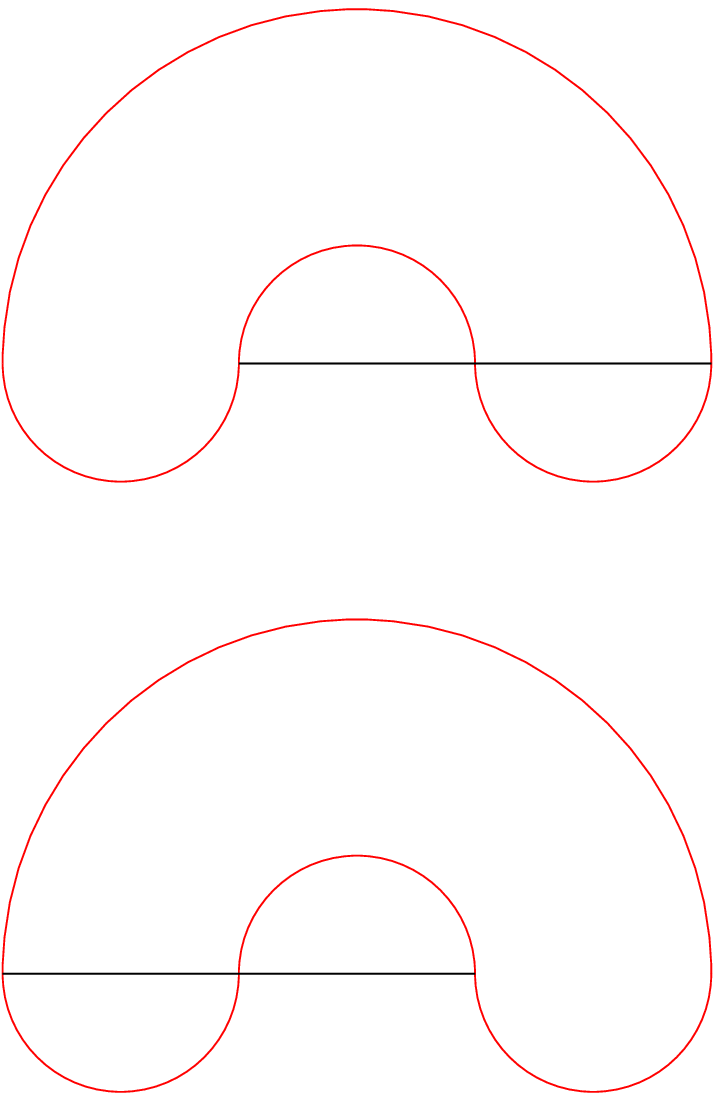}
}
\caption{Left: Redundancy of bypasses, or bypass rotation. Right: Possible attaching arcs on a tight $\partial B^3$. View this as $\partial B^3$ from the outside.}
\label{fig:22}
\label{fig:9}
\end{figure}

\subsubsection{Bypasses on a tight 3-ball}

\label{sec_bypasses_3-ball}

We consider bypasses along attaching arcs on a tight $B^3$ with convex boundary. Only two topologically distinct attaching arcs exist on a tight $\partial B^3$: see figure \ref{fig:9} (right). 

In the first case, after adding a bypass to the outside of the ball, the manifold becomes overtwisted; the dividing set is disconnected. But there exists a bypass inside the ball along this attaching arc, by the right-to-life principle. In the second case we obtain precisely the opposite answers, for similar reasons: adding a bypass outside the ball, the contact structure remains tight; but no bypass exists inside.

In the first case, call the arc of attachment \emph{inner}; in the second case \emph{outer}. The above applies to any attaching arc on a tight contact $D \times I$.

\subsubsection{Bypasses are building blocks}

\label{sec_building_blocks}

We now show how a tight contact solid cylinder can be constructed out of bypasses. The proof is in essence a version of Honda's \emph{imbalance principle} (see \cite{Hon00I}), with complications arising from the corners and boundary.

\begin{lem}[Cobordisms are constructed out of bypasses]
\label{lem_cobs_constructed_bypasses}
Suppose that on the cylinder $D \times I$ there is a tight contact structure $\xi$ with dividing sets $\Gamma_0, \Gamma_1$ on $D \times \{0\}$, $D \times \{1\}$ and with vertical dividing set along $\partial D \times I$. Then $(D \times I, \xi)$ is contactomorphic to the thickened convex surface $D \times \{0\}$ with some finite set of bypass attachments.
\end{lem}

\begin{proof}
Obviously $\Gamma_0, \Gamma_1$ must be chord diagrams with the same number $n$ of chords. Proof by induction on $n$. For $n=1$ or $2$ we have $\Gamma_0 = \Gamma_1$; if $\Gamma_0 \neq \Gamma_1$ then after edge rounding we have an overtwisted 3-ball.

Suppose $\Gamma_0, \Gamma_1$ have a common chord $\gamma$ enclosing an outermost region $R$. Then consider another arc $\delta$ in $D$ running close and parallel to $\gamma$, enclosing it and $R$. Legendrian realise $\delta$ and (possibly after perturbing) consider the convex surface $\delta \times I$. After edge rounding, we can legendrian realise $\partial(\delta \times I)$, intersecting the dividing set on $\partial(D \times I)$ in two points. Since the contact structure is tight, there is only one possible dividing set on $\delta \times I$. Indeed, cutting along $\delta \times I$ gives two solid cylinders, both of which must be tight, and both of which (after re-sharpening corners) have vertical sutures on $\partial D \times I$. One of these has dividing set $\gamma$ on top and bottom, hence has an $I$-invariant contact structure. The other cylinder has dividing sets on both ends with $n-1$ chords; hence by induction is obtained by attaching bypasses to the base; hence so is the original cylinder.

Now suppose that $\Gamma_1$ has an outermost chord $\gamma$ which does not occur in $\Gamma_0$. Let its endpoints, labelled clockwise, be $p$ and $q$; let the next marked point clockwise be $r$. Then on $\Gamma_0$, there is no outermost chord joining $p$ and $q$ (by assumption), nor joining $q$ and $r$ (which after edge rounding would give a closed dividing set component along with $\gamma$). Thus the situation must be as shown in figure \ref{fig:52} (left), and we may take an arc $\delta$ on $D$ as shown, intersecting $\Gamma_0$ in $3$ points and $\Gamma_1$ in $1$ point. 

After perturbing if necessary, consider a convex $\delta \times I$; we determine the dividing set on $\delta \times I$. By the interleaving property of dividing sets, the dividing set on $\delta \times I$ has six boundary points ($3$ from $\Gamma_0$, $1$ from $\Gamma_1$, and $2$ from the vertical sutures after rounding), hence contains 3 arcs. Cutting along $\delta \times I$ gives two smaller cylinders. The smaller cylinder containing $\{p,q,r\} \times I$ has boundary dividing set as shown in figure \ref{fig:53} (centre).  Since this is tight, there is only one possible dividing set on $\delta \times I$, shown in figure \ref{fig:54} (right). 

\begin{figure}
\centering
\mbox{
\includegraphics[scale=0.4]{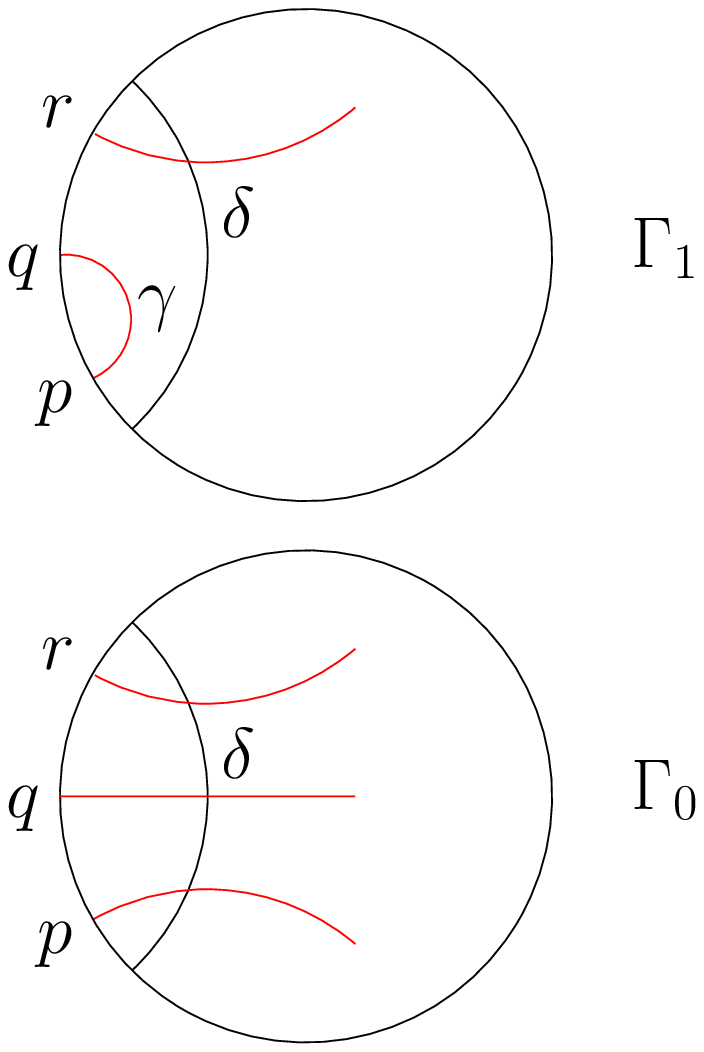}
\quad \quad \quad
\includegraphics[scale=0.3]{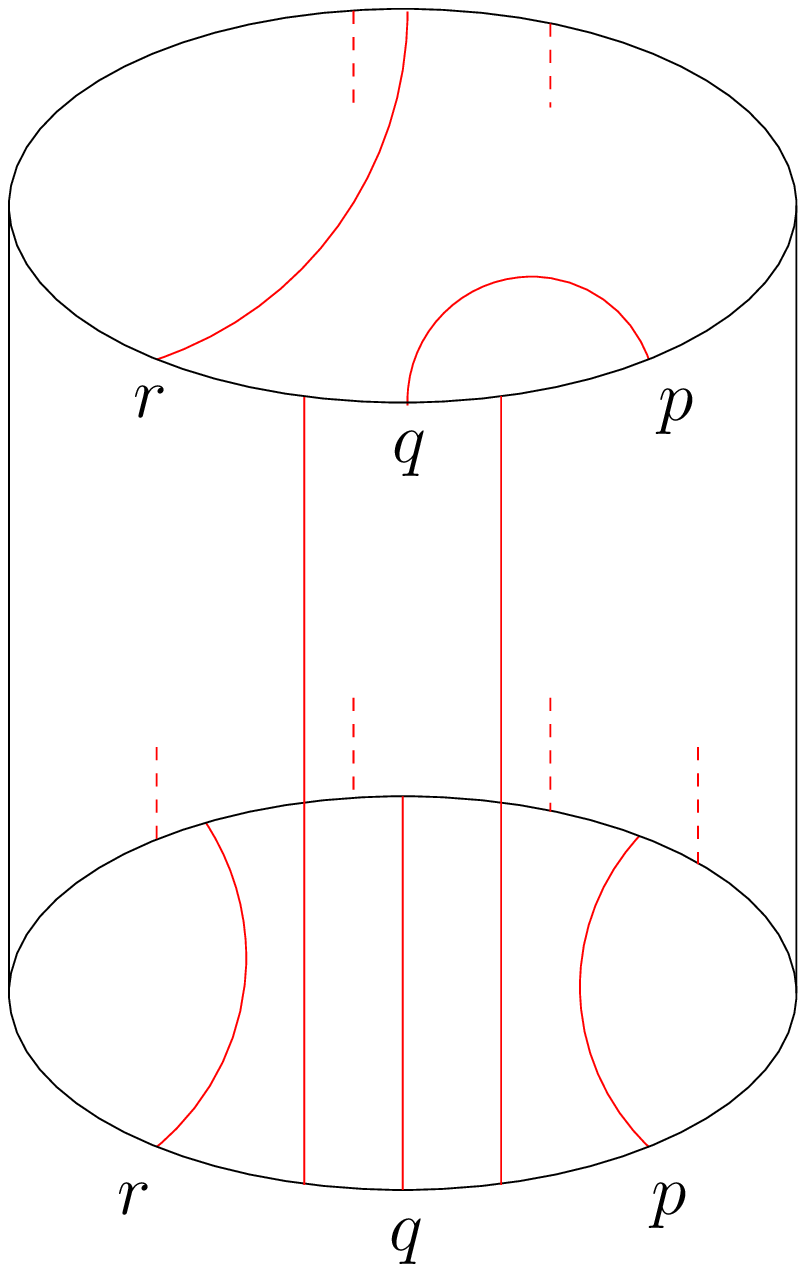}
\quad \quad \quad
\includegraphics[scale=0.3]{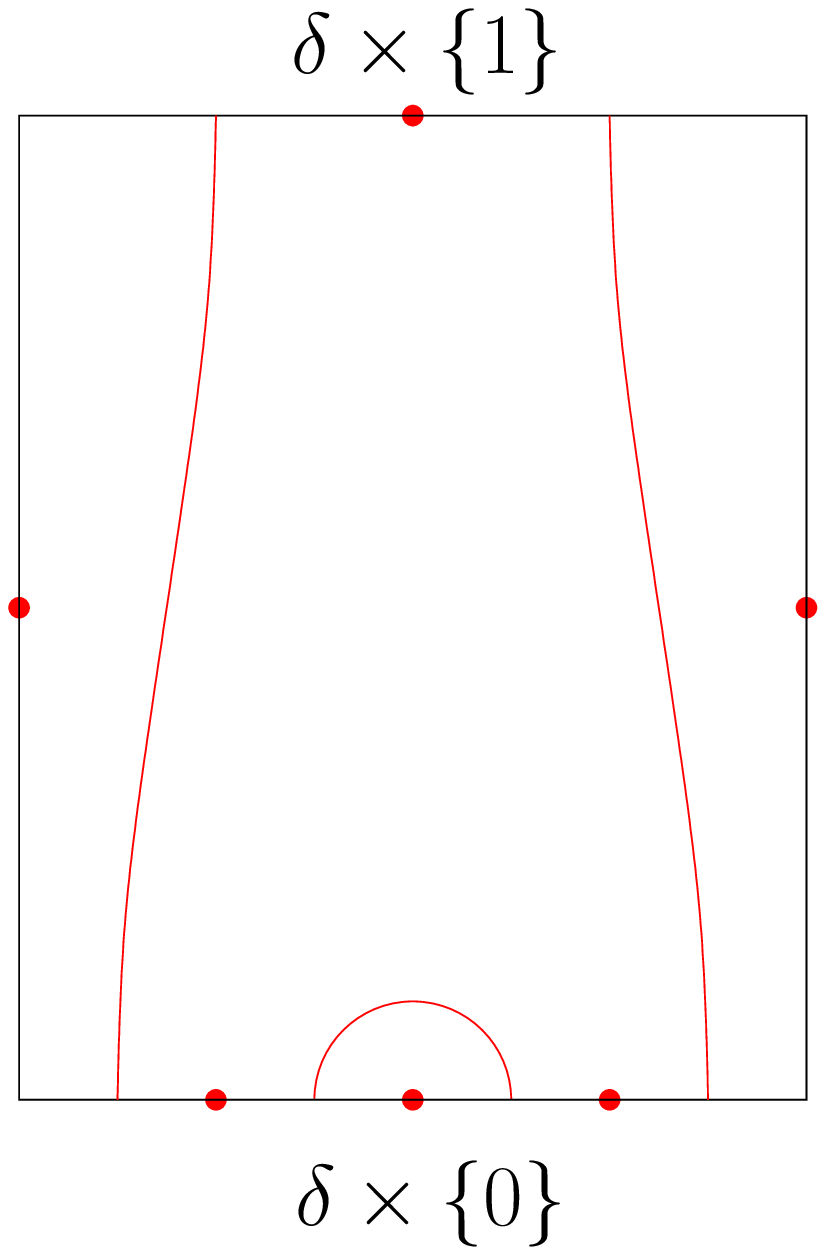}
}
\caption{Left: Arc $\delta$ on $\Gamma_0, \Gamma_1$. Centre: Dividing set on cylinder obtained by cutting along $\delta \times I$. Here $\delta \times I$ forms the back of the picture; its dividing set is to be determined but its boundary points are shown. Right: Dividing set on $\delta \times I$. Red dots show interleaving intersections with the dividing set of $\partial(D \times I)$.} 
\label{fig:52}
\label{fig:53}
\label{fig:54}
\end{figure}

Hence there is a bypass above $\Gamma_0$ along a sub-arc of $\delta$. Attaching this bypass to $\Gamma_0$ gives a dividing set with an outermost chord in the same position as $\Gamma_1$. After removing a layer containing this bypass, we reduce to the previous case and are done.
\end{proof}

\subsubsection{Pinwheels}
\label{sec_pinwheels}

Consider attaching several bypasses to a convex disc along a bypass system (definition \ref{def_bypass_system}). Is the resulting manifold tight? By \cite{HKMPinwheel}, the key indicator is a \emph{pinwheel}.

\begin{defn}[Pinwheel]
\label{def_pinwheel}
An (upwards) \emph{pinwheel} is an embedded polygonal region $P$ on a convex surface satisfying the following conditions.
\begin{enumerate}
\item 
The boundary of $P$ consists of $2k$ ($k \geq 1$) consecutive sides
\[
 \gamma_1, \alpha_1, \gamma_2, \alpha_2, \ldots, \gamma_k, \alpha_k,
\]
labelled anticlockwise, where $\gamma_i$ is an arc on a chord of the dividing set $\Gamma$, and $\alpha_i$ is half of an arc of attachment $c_i$.
\item
For each $i$, $c_i$ extends beyond $\alpha_i$ in the direction shown in figure \ref{fig:42}, and does not again intersect $P$.
\end{enumerate}
\end{defn}

\begin{figure}
\centering
\includegraphics[scale=0.6]{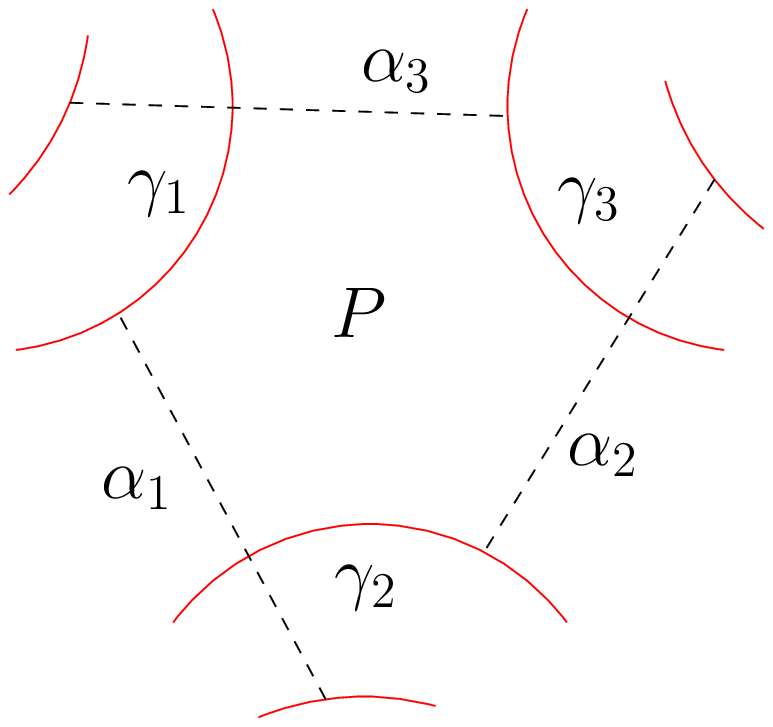}
\caption{An (upwards) pinwheel.} \label{fig:42}
\end{figure}

Attaching bypasses above a surface along the attaching arcs of a pinwheel clearly results in an overtwisted contact structure. The converse is also true.
\begin{thm}[Honda--Kazez--Mati\'{c} \cite{HKMPinwheel}]
\label{thm_pinwheel}
Let $D$ be a convex disc with legendrian boundary and $c$ a bypass system. The contact manifold obtained by attaching bypasses above a thickened $D$ along $c$ is tight iff there are no pinwheels in $D$.
\qed
\end{thm}
A similar result obviously holds, attaching bypasses downwards; the orientation of a pinwheel is reversed. Hence we speak of \emph{upwards} and \emph{downwards pinwheels}.

\subsection{Contact cobordisms}

\label{sec_discontents}

Despite the heading, by ``contact cobordisms'' we mean manifolds $D \times I$, considered as ``cobordisms'' between $D \times \{0\}$ and $D \times \{1\}$. Moreover, we consider only vertical sutures $F \times I \subset \partial D \times I$, where $F \subset \partial D$ is finite. This is enough for our purposes; some of the following clearly applies in far greater generality.

\subsubsection{Stackability}
\label{sec_Mandm}

We now formalise the stacking construction of section \ref{sec_stackability}. Take two chord diagrams $\Gamma_0, \Gamma_1$ of $n$ chords, and the cylinder $D \times I$. Align the marked points on $\Gamma_0, \Gamma_1$ at $\{p_i\} \times \{0,1\}$, and the base points at $\{p_0\} \times \{0,1\}$. Take $2n$ points $\{q_i\}$ on $\partial D$, evenly spaced between successive $p_i$.
\begin{defn}
With $\Gamma_i$ as above, the sutured manifold $\M(\Gamma_0, \Gamma_1)$ is $D \times I$ with sutures
\[
 \left( \Gamma_0 \times \{0\} \right) \; \cup \; \left( \left( \bigcup \{q_i\} \right) \times I \right) \; \cup \; \left( \Gamma_1 \times \{1\} \right).
\]
\end{defn}
See figure \ref{fig:39} (left). We think of the $I$ factor as ``vertical'': positive up, negative down.

This $\M(\Gamma_0, \Gamma_1)$ is a sutured 3-ball with corners. Regarding the boundary as convex, and corners legendrian, $\M(\Gamma_0, \Gamma_1)$ has a natural contact structure near the boundary. If, after rounding corners, the dividing set is disconnected, then this contact structure is overtwisted; otherwise, it extends to the unique tight contact structure on the ball. 
\begin{defn}[Tight/overtwisted $\M(\Gamma_0, \Gamma_1)$]
The sutured manifold $\M(\Gamma_0, \Gamma_1)$ is \emph{tight} if it admits a tight contact structure, i.e. if after rounding corners, the sutures of $\M(\Gamma_0, \Gamma_1)$ are connected. Otherwise $\M(\Gamma_0, \Gamma_1)$ is \emph{overtwisted}.
\end{defn}

\begin{defn}[Stackable]
A chord diagram $\Gamma_1$ is \emph{stackable} on another chord diagram $\Gamma_0$ if $\M(\Gamma_0, \Gamma_1)$ is tight.
\end{defn}

We can now define the \emph{stackability map}. Consider $\partial \M(\Gamma_0, \Gamma_1)$, which after rounding is $S^2$. Remove a small neighbourhood of an interior point on one of the vertical dividing curves $\{q_i\} \times I$, to obtain a convex disc. Taking a product of this disc with $S^1$, with $S^1$-invariant contact structure, gives a solid torus $(T,1)$, which contains two solid tori ($T,n)$, i.e. $(T, n) \cup (T,n) \hookrightarrow (T, 1)$, and with a specified contact structure on the intermediate $(\text{pants}) \times S^1$. Hence there is a map
\[
 m: SFH(T, n) \otimes SFH(T, n) \To SFH(T,1) = \Z_2.
\]
(Note that on $\partial B^3 = S^2$, the two discs with chord diagrams $\Gamma_0, \Gamma_1$ are oriented; one agrees with the orientation of $S^2$, the other does not. This is in addition to the orientation issues discussed in section \ref{sec_SFH_contact_intro}! Nonetheless, $m$ still exists.)

From two contact elements in $SFH(T,n)$, $m$ gives a contact element in $SFH(T,1)$. If $\M(\Gamma_0, \Gamma_1)$ has disconnected sutures, this is an overtwisted contact structure, and we obtain $0$. Otherwise, it is the unique tight contact structure on $(T,1)$, and we obtain $1$. That is, $m(\Gamma_0, \Gamma_1) \in \Z_2$ is $0$ or $1$, respectively as $\M(\Gamma_0, \Gamma_1)$ is overtwisted or tight. In fact, $m$ could be defined purely combinatorially in $SFH_{comb}$. We have proved proposition \ref{mexists}.

Clearly $m$ restricts to various summands $SFH(T,n,e)$, giving
\[
 m: SFH(T,n,e_0) \otimes SFH(T,n,e_1) \To SFH(T,1,e_0 - e_1);
\]
but $SFH(T,1,e) = 0$ unless $e=0$. Hence $m$ is trivial when $e_0 \neq e_1$, and proposition \ref{lem_euler_orthogonality} is proved; $SFH(T,n) = \bigoplus_e SFH(T,n,e)$ is orthogonal with respect to $m$.

\begin{figure}
\centering
\mbox{
\includegraphics[scale=0.4]{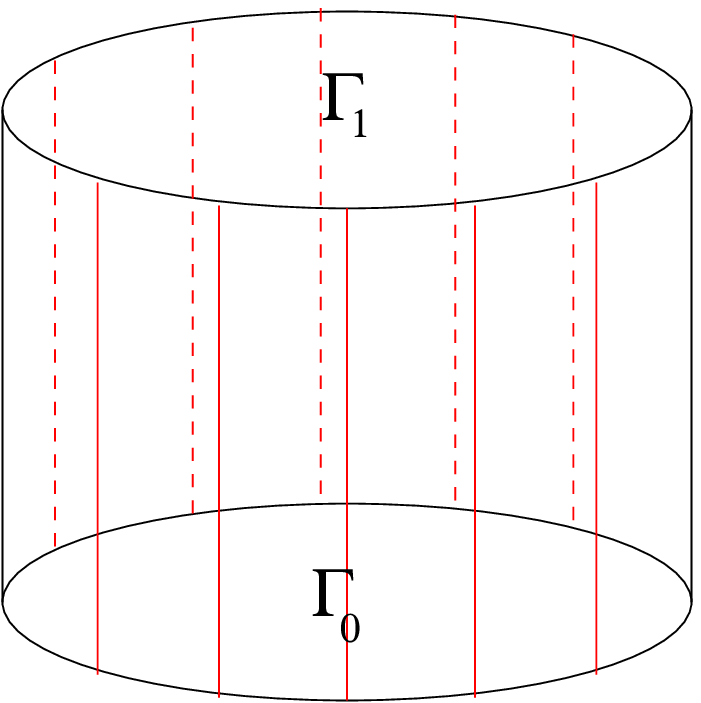}
\quad \quad \quad \quad \quad
\includegraphics[scale=0.4]{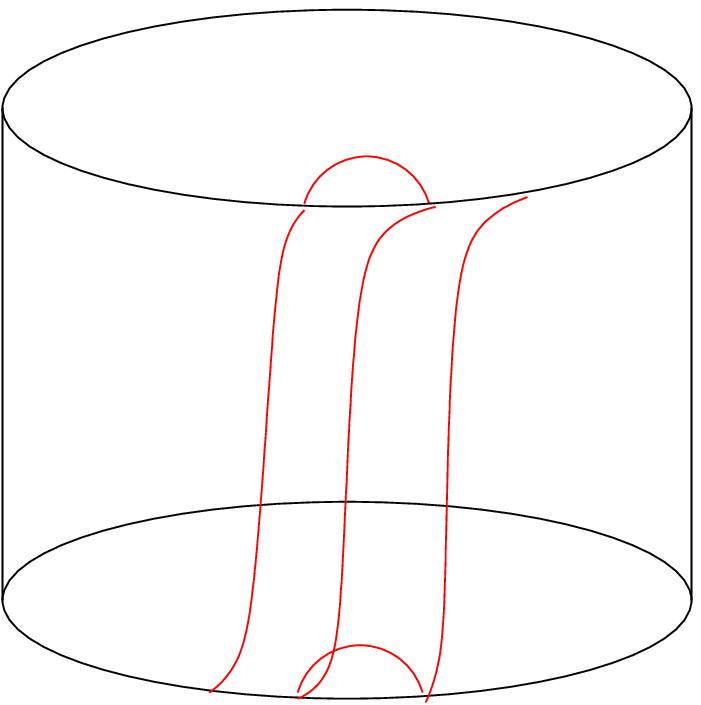}
}
\caption{Left: $\M(\Gamma_0, \Gamma_1)$. Right: $\M(\Gamma, \Gamma)$ with some edge-rounding.} 
\label{fig:39}
\label{fig:40}
\end{figure}

\subsubsection{First properties of $\M$ and $m$}
\label{sec_properties_Mandm}

\begin{lem}
\label{MGammaGamma}
For any chord diagram $\Gamma$, $\M(\Gamma,\Gamma)$ is tight: $m(\Gamma, \Gamma) = 1$.
\end{lem}

We give two proofs.  The second is less elegant but similar in spirit to subsequent proofs (e.g. section \ref{sec_contact_partial_order} proving proposition \ref{contact_interp_partial_order}), and leads to a useful lemma.

\begin{proof}[Proof \# 1]
Since a chord diagram has no closed curves, there is a tight $I$-invariant contact structure on $D^2 \times I$ with boundary dividing set given by $\M(\Gamma, \Gamma)$.
\end{proof}

\begin{proof}[Proof \# 2]
Proof by induction on the number of chords in the chord diagram. For $1$ chord it is clear. Now consider general $\Gamma$. Let $\gamma$ be an outermost chord of $\Gamma$. We reduce the case of $\M(\Gamma, \Gamma)$ to that of $M(\Gamma - \gamma, \Gamma - \gamma)$.

Note that $\gamma \times \{1\}$ has two endpoints; after rounding corners, one of these is connected to an endpoint of $\gamma \times \{0\}$. Thus on the rounded ball we have a connected arc $c$, part of the dividing set, of the form $c = c_1 \cup (\gamma \times \{1\}) \cup c_2 \cup (\gamma \times \{0\}) \cup c_3$, where the $c_i$ are (rounded versions of) arcs $q_i \times [0,1]$. But now we can perform a ``finger move'' on the dividing set and see that this is equivalent to the dividing set of $\M(\Gamma - \gamma, \Gamma - \gamma)$; where all of $c$ becomes one of the $q_i \times [0,1]$ arcs. See figure \ref{fig:40} (right).
\end{proof}

The argument of this proof is useful in its own right.
\begin{lem}[Cancelling outermost chords]
\label{cancel_outermost}
Suppose $\Gamma_0, \Gamma_1$ each have an outermost chord $\gamma$ in the same position. Then $\M(\Gamma_0, \Gamma_1)$ is tight iff $\M(\Gamma_0 - \gamma, \Gamma_1 - \gamma)$ is tight. 
\qed
\end{lem}

As $m$ is bilinear, is ``positive definite'', and satisfies an orthogonality relation, it behaves something like a metric. However, $m$ is not symmetric; nor antisymmetric; in fact, there exist $\Gamma_0, \Gamma_1$ with same $n$ and $e$ giving any of
\[
 \left( m(\Gamma_0, \Gamma_1), m(\Gamma_1, \Gamma_0) \right) = (0,0), \; (0,1), \; (1,1).
\]
The pair $(1,1)$ is possible even when $\Gamma_0 \neq \Gamma_1$. See figure \ref{fig:55}.

\begin{figure}
\centering
\includegraphics[scale=0.25]{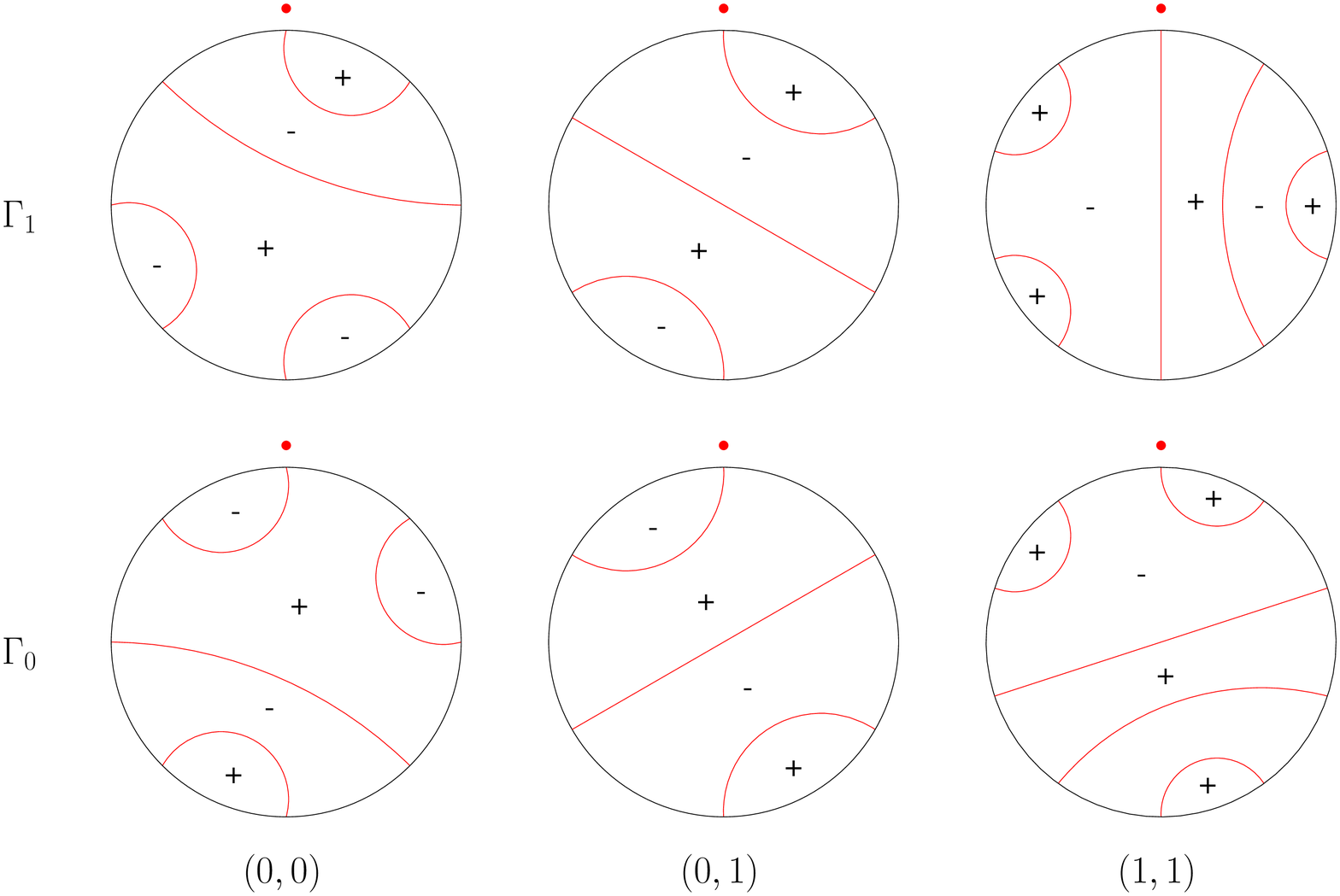}
\caption{Pairs $(\Gamma_0, \Gamma_1)$ giving various values for $\left( m(\Gamma_0, \Gamma_1), m(\Gamma_1, \Gamma_0) \right)$.} \label{fig:55}
\end{figure}

\subsubsection{Bypass cobordisms and bypass triples}
\label{sec_bypass_cobordisms} 

The simplest nontrivial cobordism is a \emph{bypass cobordism}: a disc with a bypass attached. By lemma \ref{lem_cobs_constructed_bypasses}, every tight $\M(\Gamma_0, \Gamma_1)$ can be decomposed into bypass cobordisms.

As long as we attach a single bypass along a nontrivial attaching arc $c$ (i.e. $c$ intersects three different components of $\Gamma$), there can be no pinwheel, and we obtain a tight contact structure on $D \times I$. Attaching above or below, then, $\M(\Gamma, \Up_c(\Gamma))$ and $\M(\Down_c(\Gamma), \Gamma)$ are tight. Moreover, bypass-related chord diagrams naturally come in a triples $\Gamma, \Up_c(\Gamma)$ and $\Down_c(\Gamma)$.
\begin{lem}
\label{bypass-related}
Let $c$ be a nontrivial arc of attachment on a chord diagram $\Gamma$. Then
\begin{enumerate}
\item 
$\M(\Gamma, \Up_c(\Gamma))$, $\M(\Up_c(\Gamma), \Down_c(\Gamma))$ and $\M(\Down_c(\Gamma), \Gamma)$ are all tight, with tight contact structure given by a single bypass attachment;
\item
$\M(\Gamma, \Down_c(\Gamma))$, $\M(\Down_c(\Gamma), \Up_c(\Gamma))$, $\M(\Up_c(\Gamma), \Gamma)$ are all overtwisted.
\end{enumerate}
\end{lem}

\begin{proof}
We proved (1) above; by symmetry, (2) follows from showing $\M(\Gamma, \Down_c(\Gamma))$ is overtwisted. If $\Gamma$ has only $3$ chords, the result is true by inspection; if there are more, applying lemma \ref{cancel_outermost} repeatedly reduces to the case of $3$ chords.
\end{proof}

This lemma has another formulation; recall the definition of \emph{outer} from section \ref{sec_bypasses_3-ball}.
\begin{lem}
\label{lem_MGammaGamma_outer}
Every nontrivial attaching arc on $\Gamma \times \{0\}$ or $\Gamma \times \{1\}$ in $\M(\Gamma, \Gamma)$ is outer.
\qed
\end{lem}

A weaker statement can be proved purely algebraically: if $\Gamma_0, \Gamma_1$ are bypass-related then precisely one of $\M(\Gamma_0, \Gamma_1)$, $\M(\Gamma_1, \Gamma_0)$ is tight, i.e. $m(\Gamma_0, \Gamma_1) + m(\Gamma_1, \Gamma_0) = 1$. Just expand $m(\Gamma_0 + \Gamma_1, \Gamma_0 + \Gamma_1) = 1$ and use lemma \ref{MGammaGamma}.

\subsubsection{What bypasses and chord diagrams exist in a cobordism?}

\label{sec_what_bypasses_exist}

From the foregoing, it is now straightforward to describe what bypasses and chord diagrams ``exist'' in a cobordism. By this we mean the following.

\begin{defn}[Existence of chord diagram in cobordism]
\label{def_existence_chord_diagram}
Let $\M(\Gamma_0, \Gamma_1)$ be a tight cobordism. The chord diagram $\Gamma$ \emph{exists} or \emph{occurs} in $\M(\Gamma_0, \Gamma_1)$ if there exists an embedded convex disc $D'$ in the unique tight contact $\M(\Gamma_0, \Gamma_1)$ such that:
\begin{enumerate}
\item the boundary $\partial D'$ lies on $\partial D \times I$ and intersects the dividing set of $\partial \M(\Gamma_0, \Gamma_1)$ in precisely $2n$ points (i.e. as efficiently as possible);
\item inheriting a basepoint from $\Gamma_0$ (equivalently $\Gamma_1$), $D'$ has dividing set $\Gamma$.
\end{enumerate}
\end{defn}

For a tight $\M(\Gamma_0, \Gamma_1)$ and an attaching arc $c$ on $\Gamma_0$, there exists a bypass inside $\M(\Gamma_0, \Gamma_1)$ along $c$ iff $c$ is inner (section \ref{sec_bypasses_3-ball}), iff $m(\Up_c(\Gamma_0), \Gamma_1) = 1$. And, any cobordism is constructed from bypass attachments (lemma \ref{lem_cobs_constructed_bypasses}). Thus, $\Gamma$ exists in $\M(\Gamma_0, \Gamma_1)$ iff there is a sequence of inner bypasses which we successively ``dig out'', until we find $\Gamma$ as a boundary of a smaller ``excavated'' manifold.
\begin{lem}[Criterion for existence of $\Gamma$ in a cobordism]
\label{lem_existence_of_chord_diagram}
A chord diagram $\Gamma$ exists in $\M(\Gamma_0, \Gamma_1)$ iff there exists a sequence of chord diagrams
\[
\Gamma_0 = G_0, \; G_1, \; \ldots, \; G_k = \Gamma
\]
and attaching arcs $c_0, \ldots, c_{k-1}$, with $c_i$ on $G_i$, such that:
\begin{enumerate}
\item for $i = 0, \ldots, k-1$, we have $G_{i+1} = \Up_{c_i} G_i$.
\item $c_i$ is inner on $\M(G_i, \Gamma_1)$, or equivalently, $m(G_{i+1}, \Gamma_1) = 1$.
\end{enumerate}
\qed
\end{lem}
There is of course a similar result ``excavating'' from $\Gamma_1$ rather than $\Gamma_0$. As all attaching arcs on $\M(\Gamma, \Gamma)$ are outer (lemma \ref{lem_MGammaGamma_outer}), we immediately have:
\begin{lem}
\label{lem_CbGammaGamma}
The only chord diagram existing in $\M(\Gamma, \Gamma)$ is $\Gamma$.
\qed
\end{lem}

This lemma allows us to prove a classification result for tight contact structures on solid tori with longitudinal bypasses (in the spirit of \cite{Hon02}). We cited this without proof in section \ref{sec_SFH_contact_intro} above. A slight detour, but it illustrates the use of these methods.
\begin{prop}
\label{prop_tight_ct_str_solid_torus}
Tight contact structures up to isotopy on $(T,n)$, i.e. the solid torus $D^2 \times S^1$ with boundary dividing set $F \times S^1$, $F \subset \partial D^2$, $|F| = 2n$, are in bijective correspondence with chord diagrams of $n$ chords.
\end{prop}

\begin{proof}
Suppose we have a tight contact structure $\xi$ on $(T,n)$. Take a convex meridional disc $D$ intersecting the boundary dividing set in precisely $2n$ points; its dividing set is a chord diagram $\Gamma$ of $n$ chords. Cutting along $D$ gives $\M(\Gamma, \Gamma)$ with the unique tight contact structure, hence an $I$-invariant one. Thus $\xi$ is $S^1$-invariant on $D^2 \times S^1$ and is universally tight; in fact any finite cover of $\xi$ is contactomorphic to $\xi$.

Let $D'$ be another convex meridional disc intersecting the boundary dividing set in $2n$ points. After taking (if necessary) a (tight) finite cover of $(T,n)$, we may consider $D'$ disjoint from $D$. Cutting along $D$ we still obtain a (thicker!) $\M(\Gamma, \Gamma)$ with tight contact structure, which contains an embedded $D'$; by lemma \ref{lem_existence_of_chord_diagram} the chord diagram on $D'$ is also $\Gamma$. 

This gives a well-defined map from isotopy classes of tight contact structures to chord diagrams. Conversely, given a chord diagram $\Gamma$, take the tight contact cobordism $\M(\Gamma, \Gamma)$ and glue the ends together, giving an $S^1$-invariant universally tight contact structure on $(T,n)$ with meridional disc dividing set $\Gamma$.
\end{proof}

\subsubsection{Elementary cobordisms and generalised bypass triples}

\label{sec_elementary_cobordisms}

Another simple type of cobordism is one obtained by attaching bypasses along a bypass system.
\begin{defn}[Elementary cobordism]
\label{def_elementary_cob}
A tight cobordism $\M(\Gamma_0, \Gamma_1)$ is \emph{elementary} if its tight contact structure can be constructed by attaching bypasses above $D \times \{0\}$ along a bypass system $c$ on $\Gamma_0$.
\end{defn}
We may think of an elementary cobordism as a ``generalised bypass cobordism''. Note the result of attaching bypasses need not be tight; there may be pinwheels.

An elementary $\M(\Gamma_0, \Gamma_1)$ arising from a bypass system $c_0$ on $\Gamma_0$ without upwards pinwheels gives rise to a triple
\[
 \Gamma_0, \quad \Gamma_1 = \Up_c (\Gamma), \quad \Gamma_{-1} = \Down_c(\Gamma)
\]
which we call a \emph{generalised bypass triple}. Extend the notation $\Up_c$ and $\Down_c$ to bypass systems in the obvious way.  In view of figure \ref{fig:49} and the symmetry of bypass moves as ``local $60^\circ$ rotations'', note there are also corresponding bypass systems $c_{\pm 1}$ on $\Gamma_{\pm 1}$. However $\Gamma_{-1}$ need not be a chord diagram: there may be closed loops; $c_0$ may contain downwards but not upwards pinwheels.

A bypass triple sums to zero; a generalised bypass triple usually does not. However we may expand every downwards bypass move as a sum of null and upwards bypass moves, and vice versa. This gives a sum with $2^k$ terms, where $|c| = k$, a sum over subsets of $c$.
\begin{lem}[Expanding down over up]
\label{lem_expanding_down_over_up}
For any bypass system $c$ on $\Gamma$,
\[
 \Down_c (\Gamma) = \sum_{c' \subseteq c} \Up_{c'} (\Gamma) \quad \text{and} \quad 
 \Up_c (\Gamma) = \sum_{c' \subseteq c} \Down_{c'} (\Gamma).
\]
\qed
\end{lem}

Is every tight cobordism elementary? No: such optimism is crushed by the following.
\begin{lem}[Not every cobordism is elementary]
\label{lem_not_every_cob_elementary}
With $\Gamma_0, \Gamma_1$ as in figure \ref{fig:56}, $\M(\Gamma_0, \Gamma_1)$ is tight but not elementary.
\end{lem}

\begin{figure}
\centering
\includegraphics[scale=0.4]{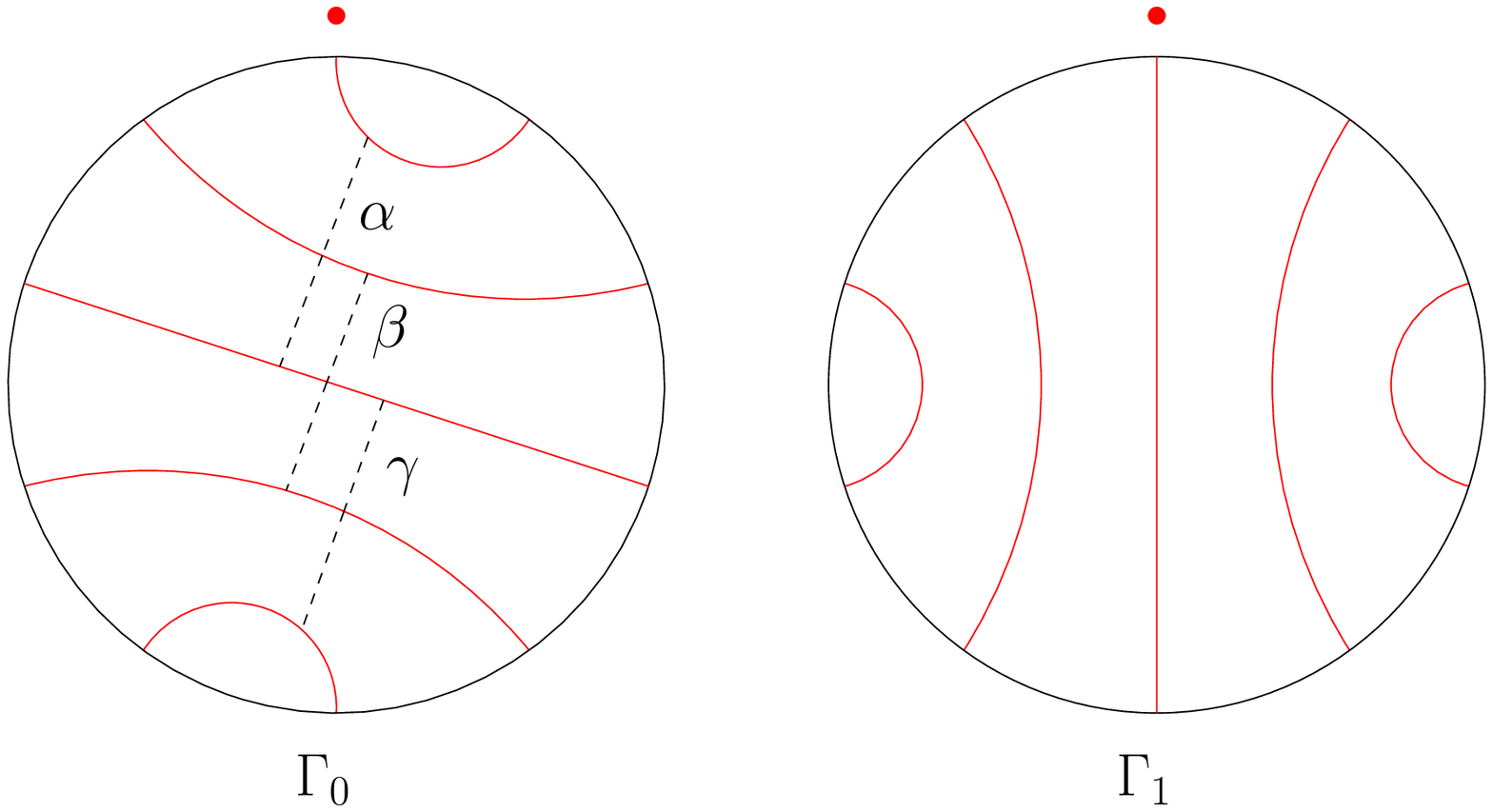}
\caption{Tight, non-elementary cobordism.} \label{fig:56}
\end{figure}

\begin{proof}
Rounding corners, $\M(\Gamma_0, \Gamma_1)$ is tight. Suppose it were elementary, so $\Gamma_1 = \Up_c  \Gamma_0$ for some bypass system $c$ on $\Gamma_0$. On $\Gamma_0$ there are only three nontrivial attaching arcs, namely $\alpha, \beta, \gamma$ as shown in figure \ref{fig:56}. Now $\M(\Up_\alpha \Gamma_0, \Gamma_1)$ and $\M(\Up_\gamma \Gamma_0, \Gamma_1)$ are overtwisted, while $\M(\Up_\beta \Gamma_0, \Gamma_1)$ is tight. Hence $c$ must consist of copies of $\beta$ and trivial attaching arcs; hence $\Up_c \Gamma_0 = \Up_\beta \Gamma_0$. However $\Up_\beta \Gamma_0 \neq \Gamma_1$.
\end{proof}

\subsubsection{The contact category}

\label{sec_The_contact_category}

Given a surface $\Sigma$ and a finite set $F \subset \partial \Sigma$, there is a category $\mathcal{C}(\Sigma, F)$, due to Honda; although \cite{HonCat} remains unpublished we give a definition for our purposes.
\begin{defn}[Contact category]
\label{defn_contact_category}
The category $\mathcal{C}(\Sigma, F)$ is as follows.
\begin{itemize}
\item
The \emph{objects} are isotopy classes of tight dividing sets $\Gamma$ on $\Sigma$ with $\partial \Gamma = F$.
\item
The \emph{morphisms} $\Gamma_0 \To \Gamma_1$ are:
\begin{enumerate}
\item 
isotopy classes of tight contact structures on $\Sigma \times I$, with dividing set $\Gamma_i$ on $\Sigma \times \{i\}$, and a vertical dividing set on $\partial \Sigma \times I$; and
\item
a single morphism, denoted $*$, referring to overtwisted structures on $\Sigma \times I$ with the same boundary conditions. (``The zero morphism.'')
\end{enumerate}
\item
\emph{Composition} of morphisms is given by gluing cobordisms.
\end{itemize}
\end{defn}

Honda \cite{HonCat} shows that this category obeys many of the properties of a \emph{triangulated category} (see e.g. \cite{Gelfand-Manin}.) In particular, there are distinguished triangles, arising from bypass additions, and these obey an octahedral axiom. TQFT properties imply that it behaves functorially with respect to $SFH$.

When $\Sigma = D^2$ and $|F| = 2n$, write $\mathcal{C}(\Sigma, F) = \mathcal{C}(D^2, n)$. Objects of $\mathcal{C}(D^2, n)$ are chord diagrams of $n$ chords; morphisms $\Gamma_0 \To \Gamma_1$ are contact structures on $\M(\Gamma_0, \Gamma_1)$. There is always the overtwisted morphism $*: \Gamma_0 \To \Gamma_1$. If $m(\Gamma_0, \Gamma_1) = 0$ this is all; otherwise there is precisely one other morphism, the unique tight contact structure. Any composition involving $*$ is $*$. If $\Gamma_0 \To \Gamma_1 \To \Gamma_2$ are both tight, then the composition is tight iff $m(\Gamma_0, \Gamma_2) = 1$ and $\Gamma_1$ exists in $\M(\Gamma_0, \Gamma_2)$; otherwise it is $*$.

By relative euler class orthogonality (proposition \ref{lem_euler_orthogonality}), a nontrivial morphism $\Gamma_0 \To \Gamma_1$ exists only when $\Gamma_0, \Gamma_1$ have the same $e$. Define $\mathcal{C}(D^2,n,e)$ to be the full subcategory of $\mathcal{C}(D^2,n)$ on the chord diagrams of relative euler class $e$.

\subsubsection{The bounded contact category}

\label{sec_The_bdd_ct_cat}

A given tight cobordism $\Gamma_0 \stackrel{\xi}{\rightarrow} \Gamma_1$ contains certain objects and morphisms of $\mathcal{C}(D^2, n)$; restricting to these leads to a notion of \emph{bounded contact category}, tenuously analogous to a bounded category, where:
\begin{itemize}
\item objects are dividing sets which occur in $\M(\Gamma_0, \Gamma_1)$ (definition \ref{def_existence_chord_diagram}); and
\item morphisms are ``cobordisms which exist'' in $\M(\Gamma_0, \Gamma_1)$.
\end{itemize}
In particular, there are no overtwisted morphisms. A ``cobordism which exists'' inside another can be defined precisely; the intuitive meaning is hopefully clear.
\begin{defn}[Existence of cobordism inside cobordism]
\label{def_existence_cobordism}
Let $\Gamma_0 \stackrel{\xi}{\rightarrow} \Gamma_1$ be a tight cobordism. The cobordism $\Gamma \stackrel{\xi'}{\rightarrow} \Gamma'$ \emph{exists} or \emph{occurs} in $\M(\Gamma_0, \Gamma_1)$ if:
\begin{itemize}
\item $\Gamma, \Gamma'$ both exist in $\M(\Gamma_0, \Gamma_1)$ on discs $D,D'$;
\item $D,D'$ are disjoint and $D$ is below $D'$; and
\item the restriction of $\xi$ to this embedded $\M(\Gamma, \Gamma')$ is isotopic to $\xi'$.
\end{itemize}
\end{defn}
Note if $\xi$ is tight then so is $\xi'$. So we may say $\M(\Gamma, \Gamma')$ occurs in $\M(\Gamma_0, \Gamma_1)$ without reference to the contact structure.

However, defining composition is a problem. Suppose two morphisms $\Gamma \rightarrow \Gamma'$, $\Gamma' \rightarrow \Gamma''$ occur in $\M(\Gamma_0, \Gamma_1)$; both have discs with dividing set $\Gamma'$, but there is no reason why we should be able to glue them together inside $\M(\Gamma_0, \Gamma_1)$. However, we can avoid this problem --- indeed avoid geometry altogether --- using the following lemma.
\begin{lem}[``Gluing cobordisms'']
\label{lem_gluing_cobordisms}
Suppose that two cobordisms $\M(\Gamma, \Gamma')$ and $\M(\Gamma', \Gamma'')$ occur in the tight cobordism $\M(\Gamma_0, \Gamma_1)$. Then $\M(\Gamma, \Gamma'')$ also occurs in this $\M(\Gamma_0, \Gamma_1)$, and the chord diagram $\Gamma'$ occurs in this tight $\M(\Gamma, \Gamma'')$.
\end{lem}

\begin{proof}
Immediate from the criterion (lemma \ref{lem_existence_of_chord_diagram}) for the existence of chord diagrams in a cobordism.
\end{proof}

\begin{defn}[Bounded contact category]
\label{defn_bonded_ct_cat}
Suppose $m(\Gamma_0, \Gamma_1) = 1$. The \emph{bounded contact category} $\mathcal{C}^b (\Gamma_0, \Gamma_1)$ is as follows.
\begin{itemize}
\item 
The \emph{objects} are the chord diagrams $\Gamma$ which exist in the tight $\M(\Gamma_0, \Gamma_1)$.
\item
There is one \emph{morphism} $\Gamma \To \Gamma'$ precisely when $\M(\Gamma, \Gamma')$ exists in the tight $\M(\Gamma_0, \Gamma_1)$.
\item
\emph{Composition of morphisms} is given by ``gluing cobordisms'' as in lemma \ref{lem_gluing_cobordisms}: $\Gamma \To \Gamma' \To \Gamma''$ composes to the unique $\Gamma \To \Gamma''$.
\end{itemize}
\end{defn}

The following is now immediate from lemma \ref{lem_gluing_cobordisms} and the definition.
\begin{lem}
\label{lem_Cb_is_category}
For any tight $\M(\Gamma_0, \Gamma_1)$, $\mathcal{C}^b (\Gamma_0, \Gamma_1)$ is a category.
\qed
\end{lem}

The notion of ``bounding'' cobordisms in this way can be nested: if the tight $\M(\Gamma, \Gamma')$ exists in the tight $\M(\Gamma_0, \Gamma_1)$, there is a covariant fully faithful functor $\mathcal{C}^b (\Gamma, \Gamma') \To \mathcal{C}^b (\Gamma_0, \Gamma_1)$, injective on objects and morphisms; $\mathcal{C}^b(\Gamma, \Gamma')$ is isomorphic to the full sub-category on the image of its objects in $\mathcal{C}^b(\Gamma_0, \Gamma_1)$.

\subsubsection{The bounded contact category is partially ordered}

\label{sec_Cb_partially_ordered}

A partially ordered set can be considered a category, morphisms being given by the ordering relation: the \emph{category of a partially ordered set}. When does a given category arise from a partial ordering?  The following lemma is clear.
\begin{lem}[When a category is a partial order]
\label{lem_category_partial_order}
Let $\mathcal{C}$ be a category such that:
\begin{enumerate}
\item for every pair $A,B$ of objects of $\mathcal{C}$, there is at most one morphism $A \rightarrow B$;
\item if there are morphisms $A \rightarrow B$ and $B \rightarrow A$, then $A = B$.
\end{enumerate}
Then $\mathcal{C}$ is the category of the partially ordered set given by $A \preceq B$ iff there is a morphism $A \rightarrow B$.
\qed
\end{lem}

Proposition \ref{lem_Cb_partial_order}, that the bounded contact category $\mathcal{C}^b (\Gamma_0, \Gamma_1)$ is partially ordered, is now straightforward.
\begin{proof}[Proof of proposition \ref{lem_Cb_partial_order}]
Verify the conditions of lemma \ref{lem_category_partial_order}. The first is true by definition. The second follows from lemma \ref{lem_CbGammaGamma}.
\end{proof}

Write $\preceq$ for the partial order on $\mathcal{C}^b(\Gamma_0, \Gamma_1)$. Note $\Gamma \preceq \Gamma'$ implies $m(\Gamma, \Gamma') = 1$. We do not know if the converse is true: if $\Gamma, \Gamma'$ are objects in $\mathcal{C}^b(\Gamma_0, \Gamma_1)$ and $m(\Gamma, \Gamma') = 1$, is $\Gamma \preceq \Gamma'$?

We can regard the tight $\M(\Gamma_0, \Gamma_1)$ as a ``geometric realisation'' (morally, not technically!) of $\mathcal{C}^b(\Gamma_0, \Gamma_1)$. Slicing this cylinder along convex surfaces geometrically realises the objects; cutting it into pieces realises the morphisms.

\subsubsection{Functorial properties of elementary cobordisms}
\label{sec_functorial}

Consider $\mathcal{C}^b (\Gamma_0, \Gamma_1)$, where $\M(\Gamma_0, \Gamma_1)$  is an elementary cobordism obtained by attaching bypasses above $\Gamma_0$ along a bypass system $c$. The power set $\mathcal{P}(c)$ is partially ordered under inclusion; attaching bypasses along any subset $c'$ of $c$ gives a morphism $\Gamma_0 \rightarrow \Up_{c'} \Gamma_0$ in $\mathcal{C}^b (\Gamma_0, \Gamma_1)$; two subsets $c' \subset c''$ give a composition $\Gamma_0 \rightarrow \Up_{c'} \Gamma_0 \rightarrow \Up_{c''} \Gamma_0$. Phrasing this categorically we have the following.
\begin{lem}[Up functor]
\label{lem_up_down}
There is a covariant functor $\Up_c: \mathcal{P}(c) \To \mathcal{C}^b (\Gamma_0, \Gamma_1)$ given by
\[
c' \mapsto \Up_{c'} (\Gamma_0), \quad 
(c' \subseteq c'') \mapsto \left( \Up_{c'} (\Gamma_0) \rightarrow \Up_{c''} (\Gamma_0) \right).
\]
\qed
\end{lem}
Similarly, for an elementary $\M(\Gamma_0, \Gamma_1)$ obtained from bypass attachments below $\Gamma_1$ along a bypass system $d$, there is a contravariant functor $\Down_d: \mathcal{P}(d) \To \mathcal{C}^b (\Gamma_0, \Gamma_1)$.

Note $\Up_c$ need not be injective or surjective on objects. Non-injectivity indicates that some attaching arc of $c$ is trivial, or \emph{becomes} trivial after other bypass attachments.

\subsubsection{Other categorical structures}

\label{sec_other_cat_str}

Honda \cite{HonCat} notes that bypass triples can be considered as the distinguished triangles of a triangulated category. But if bypass triples are distinguished triangles, the contact category fails to be triangulated: not every morphism extends to a distinguished triangle.

An analogy also exists with generalised bypass triples (defined in section \ref{sec_elementary_cobordisms}). Suppose we have a generalised bypass triple $\Gamma$, $\Gamma' = \Up_c \Gamma$, $\Gamma'' = \Down_c \Gamma$, for a bypass system $c$ on $\Gamma$ without pinwheels (upwards or downwards). We have corresponding bypass systems on $c', c''$ on $\Gamma', \Gamma''$ and upwards bypass attachments give a triangle of morphisms in $\mathcal{C}(D,n)$. The composition of any two morphisms contains several obvious overtwisted discs. 
\[
\xymatrix{
	      & \Gamma \ar[dr] & \\
\Down_c(\Gamma) \ar[ur] && \Up_c \Gamma \ar[ll] }
\]
With generalised bypass triples for distinguished triangles, every elementary cobordism arising from a bypass system without (up or down) pinwheels includes into a distinguished triangle; however the cone depends on the choice of bypass system, which is unsatisfactory. (We will describe in section \ref{ch_bypass_systems_basis_diagrams} and discuss in section \ref{sec_categorical_meaning} certain ``canonical'' bypass systems, but these only exist for basis chord diagrams.) Moreover, by lemma \ref{lem_not_every_cob_elementary} not all tight cobordisms are elementary. We may consider $\Down_c (\Gamma)$ to be ``a cone'' of the morphism $\Gamma \To \Up_c (\Gamma)$: ``The cone of $\Up_c$ is $\Down_c$''.

\section{The basis of contact elements}
\label{ch_basis_interpretations}

We now examine \emph{basis chord diagrams} in detail: how to construct them (section \ref{sec_basis_construction}); how to decompose in terms of them (section \ref{sec_decomposition}); then the partial order $\preceq$ and its relation to stackability (section \ref{sec_contact_partial_order}).

\subsection{Construction of basis elements}
\label{sec_basis_construction}

\subsubsection{An example}

Consider the basis element $v_{-+-++}$ and its chord diagram. By definition $v_{-+-++} = B_- (v_{+-++})$. Hence there is an outermost chord which encloses a negative region to the immediate ``left'' of the base point; one of its endpoints is the base point. After removing this outermost chord, and adjusting the base point appropriately, we then have  $v_{+-++}$, and repeat. Eventually we arrive at the vacuum $v_\emptyset$. See figure \ref{fig:18}.

\begin{figure}
\centering
\includegraphics[scale=0.4]{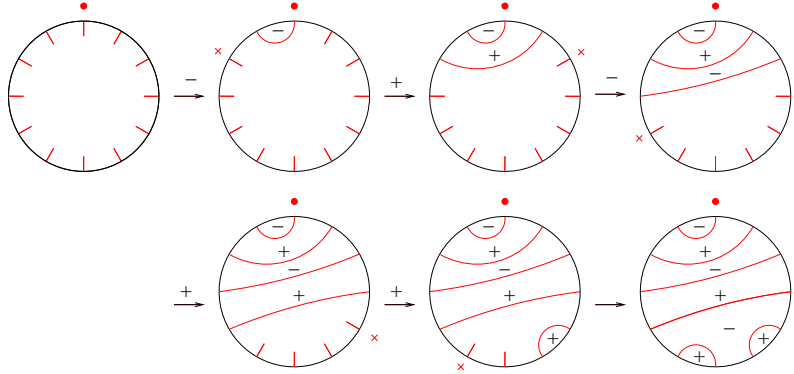}
\caption[Construction of $v_{-+-++}$.]{Construction of $v_{-+-++}$. The temporary base point at each stage is denoted by a red cross.} \label{fig:18}
\end{figure}

\subsubsection{The base point construction algorithm}
\label{sec_base_point_construction}

We formalise the above construction of $v_w$ as an algorithm. This algorithm starts from the base point; hence the name. Throughout this section, $w \in W(n_-, n_+)$ and we consider a disc with $2n+2$ points marked on the boundary, one of them a base point. 
\begin{alg}[Base point construction algorithm]
\label{alg_base_point_construction}
Read $w$ from left to right. For each symbol $s$, draw a chord and move the base point ``temporarily'' as follows. Once there is a chord ending at a marked point, it is called \emph{used}.
\begin{enumerate}
\item
If $s=-$, draw a chord from the current temporary base point to the next unused marked point anticlockwise (``left'') from it. After drawing this chord, move the temporary base point to the next unused marked point in the anticlockwise direction (``left''). (I.e., immediately anticlockwise of the new chord.)
\item
If $s=+$, draw a chord from the current temporary base point to the next unused marked point clockwise (``right'') from it. After drawing this chord, move the temporary base point to the next unused marked point in the clockwise direction. (I.e., immediately clockwise of the new chord.)
\end{enumerate}
This constructs $n$ chords connecting $2n$ marked points. Finally, connect the remaining two marked points with a chord. 
\end{alg}
Figure \ref{fig:18} depicts the construction of $v_{-+-++}$. The terminology ``left'' and ``right'' for directions around a circle is very bad, but there is a reason for it.

That this algorithm is well-defined and produces $v_w$ is easy to check; however we need some details later, and give a lemma describing it precisely.

Label the $2n+2$ marked points by integers modulo $2n+2$, with the (permanent) base point $0$ and the numbering increasing clockwise. As the ``temporary base point'' moves, the numbering of marked points does not change. Note that a chord connecting $(2j-1, 2j)$ (resp. $(2j, 2j+1)$) encloses an outermost negative (resp. positive) region.

A \emph{discrete interval} of integers $[a,b]$ is $\{a, a+1, \ldots, b\}$. A \emph{substring} of a word/string $w$ is a set of adjacent symbols of $w$. (So $-+$ is a substring of $-++--+$ but $---+$ is not.) A \emph{block} of a word is a maximal substring of identical symbols. A \emph{leading} symbol is the first in its block (read left to right). A non-leading symbol is \emph{following}.
\begin{lem}[Mechanics of base point construction algorithm]\
\label{construction_mechanics}
\begin{enumerate}
\item 
Consider the stage of the base point algorithm which processes the $i$'th $-$ sign in $w$. Let $i_+$ be the number of $+$ signs processed up to this point. At this stage:
\begin{enumerate}
\item  A chord is drawn with an endpoint at $1-2i$, precisely:
\begin{align*}
(1-2i, 2 i_+) & \quad \text{if the $i$'th $-$ sign is leading} \\
(1-2i, 2-2i) & \quad \text{if the $i$'th $-$ sign is following}
\end{align*}
\item The temporary base point then moves to the marked point $-2i$.
\item The set of used marked points is now $[1 - 2i, 2 i_+]$.
\end{enumerate}
\item 
Consider the stage of the base point algorithm which processes the $j$'th $+$ sign in $w$. Let $j_-$ be the number of $-$ signs processed up to this point. At this stage:
\begin{enumerate}
\item A chord is drawn with an endpoint at $2j-1$, precisely:
\begin{align*}
(-2j_-, 2j - 1) & \text{if the $j$'th $+$ sign is leading} \\
(2j - 2, 2j - 1) & \text{if the $j$'th $+$ sign is following}
\end{align*}
\item The temporary base point then moves to the marked point $2j$.
\item The set of used marked points is now $[-2j_-, 2j - 1]$.
\end{enumerate}
\end{enumerate}
\qed
\end{lem}

\begin{lem}[Existence of root point]
The position of the final chord drawn in the base point construction algorithm only depends on $n,e$ (equivalently $n_-,n_+$) and the final symbol $s$ of $w$. It encloses an outermost region of sign $s$, and has an endpoint at
\[
2n_+ + 1 = -2n_- - 1 = e+n+1 = e-(n+1).
\]
\qed
\end{lem}

\begin{defn}[Root point]
The point numbered $e\pm(n+1)$ is the \emph{root point}.
\end{defn}

\begin{rem}
The root point is always denoted by a hollow red dot.
\end{rem}

We now see that our terminology of ``left'' and ``right'' is not so bad after all, if the base point is ``north'' and the root point is ``south''.
\begin{defn}[Left/west \& right/east]
The marked points in  $[e-n-1, 0]$ form the \emph{left} or \emph{westside}. The marked points $[0, e+n+1]$ form the \emph{right} or \emph{eastside}.
\end{defn}

The algorithm numbers chords and regions in $\Gamma_w$ in a way that will be useful.
\begin{defn}[Base-$\pm$ numbering of chords and regions] \
\label{def_base_numbering}
The chord created in the base point construction algorithm by processing the $i$'th $-$ sign (resp. $j$'th $+$ sign) of $w$ is called the \emph{base-$i$'th $-$ chord} (resp. \emph{base-$j$'th $+$ chord}). It encloses a $-$ (resp. $+$) region, which is also a region of the completed chord diagram $\Gamma_w$, which we call the \emph{base-$i$'th $-$ region} (resp. \emph{base-$j$'th $+$ region}).
\end{defn}
Note every chord has a base-$\pm$ numbering, except the chord at the root point; every region has a base-$\pm$ numbering, except the two regions adjacent to the root point. Every $\pm$ sign creates a $\pm$ region, with two regions left at the end; $e = n_+ - n_-$.

\subsubsection{The root point construction algorithm}
\label{sec_root_point_construction}

The above algorithm takes $w$ and constructs $\Gamma_w$, starting from the base point, reading $w$ left to right. We can also construct the chord diagram from the root point, reading $w$ right to left. The algorithm is essentially identical.
\begin{alg}[\bf Root point construction algorithm]
Read $w$ from right to left. For each symbol $s$, draw a chord and move the root point ``temporarily'' as follows.
\begin{enumerate}
\item
If $s=-$, draw a chord from the current root point to the next unused marked point \emph{clockwise}/left from it. After drawing this chord, move the temporary root point to the next unused marked point in the \emph{clockwise}/left direction.
\item
If $s=+$, draw a chord from the current root point to the next unused marked point \emph{anticlockwise}/right from it. After drawing this chord, move the temporary root point to the next unused marked point in the \emph{anticlockwise}/right direction.
\end{enumerate}
This constructs $n$ chords. Finally, connect the remaining two marked points.
\end{alg}
It's easy to see that this algorithm constructs $\Gamma_w$. We also have root-numberings of chords and regions, which we will need later, in analogy to base-numberings.
\begin{defn}[Root-$\pm$ numbering of chords and regions]\
\label{def_root_numbering}
The chord created in the root point construction algorithm by processing the $i$'th $-$ sign (resp. $j$'th $+$ sign) of $w$ (from the \emph{left}) is called the \emph{root-$i$'th $-$ chord} (resp. \emph{root-$j$'th $+$ chord}). It encloses a $-$ (resp. $+$) region, which is also a region of the completed chord diagram $\Gamma_w$, which we call the \emph{root-$i$'th $-$ region} (resp. \emph{root-$j$'th $+$ region}).
\end{defn}
Note that the root point construction algorithm processes $w$ from \emph{right to left}: but when we speak of the root-$i$'th $\pm$ region we are reading $w$ from \emph{left to right}.

Every chord has a root-$\pm$ numbering, except the chord at the base point. And every region has a root-$\pm$ numbering, except the two regions adjacent to the base point. Thus, when $n>1$ \emph{every chord} has some numbering, whether from the base or the root.

\subsection{Decomposition into basis elements}

\label{sec_decomposition}

We now show algorithmically how to decompose a chord diagram into basis diagrams. This formalises the discussion in section \ref{sec_basis} and figure \ref{fig:15}.

From $\Gamma$, we successively obtain sets of chord diagrams
\[
 \{\Gamma\} = \Upsilon_0 \rightsquigarrow \Upsilon_1 \rightsquigarrow \cdots \rightsquigarrow \Upsilon_n
\]
where $\Upsilon_k$ is the set of all diagrams obtained at the $k$'th stage. The final set $\Upsilon_n$ contains precisely the basis decomposition of $\Gamma$. In particular, $|\Upsilon_k| \geq |\Upsilon_{k-1}|$ and 
\[
 \Gamma = \sum_{\Gamma' \in \Upsilon_0} \Gamma' = \sum_{\Gamma' \in \Upsilon_1} \Gamma' = \cdots = \sum_{\Gamma' \in \Upsilon_n} \Gamma'.
\]

Note that the first $k$ steps of the base (resp. root) point construction algorithm depend only on the $k$ leftmost (resp. rightmost) symbols of the word $w$.
\begin{defn}[Partial chord diagram]
Let $w$ be a word of length $k \leq n$. The \emph{partial chord diagram for $w \cdot$} (resp. $\cdot w$) consists of the first $k$ chords drawn in processing any word of length $n$ beginning with (resp. ending in) $w$ in the base (resp. root) point construction algorithm.
\end{defn}
The dots in $w\cdot$ and $\cdot w$ describe ``where the rest of the word goes''. If we cut the disc along the chords of a partial chord diagram, all the unused marked points lie in a single component; call this the \emph{unused disc}.

We will label the elements of $\Upsilon_k$ as $\Gamma_{w \cdot}$ (resp. $\Gamma_{\cdot w}$), where $w$ varies over words of length $k$. The chord diagram $\Gamma_{w \cdot}$ (resp. $\Gamma_{\cdot w}$) will be the sum of all basis elements of $\Gamma$ whose words begin (resp. end) with $w$, and it will contain the partial chord diagram for $w \cdot$ (resp. $\cdot w$).  Consider $\Gamma$ itself as corresponding to the empty word, $\Gamma = \Gamma_{\emptyset \cdot} = \Gamma_{\cdot \emptyset}$. 
\begin{alg}[Base point decomposition algorithm]
\label{alg_base_point_decomposition}
Begin with $\Upsilon_0 = \{\Gamma\} = \{\Gamma_{\emptyset \cdot}\}$. At the $k$'th step, consider each $\Gamma_{w\cdot} \in \Upsilon_{k-1}$, with $w$ of length $k-1$.
\begin{enumerate}
\item If there exists a word $w' = w +$ or $w-$ such that $\Gamma_{w\cdot}$ contains the partial chord diagram for $w' \cdot$, then place $\Gamma_{w \cdot}$ in $\Upsilon_k$ and name it $\Gamma_{w' \cdot}$.

\item Otherwise, on the unused disc of $\Gamma_{w \cdot}$, there is no outermost chord at the temporary base point after $k-1$ steps of the base point construction algorithm. Consider an attaching arc which runs close to the boundary of the unused disc, which is centred on the chord emanating from the temporary base point, and which has its two ends on the two chords emanating from the marked points adjacent to the temporary base point on the unused disc. Perform the two possible bypass moves, obtaining two distinct chord diagrams, respectively containing the partial chord diagrams for $w \pm \cdot $. Label them $\Gamma_{w \pm \cdot }$ and place them in $\Upsilon_k$.
\end{enumerate}
This constructs $\Upsilon_k$ from $\Upsilon_{k-1}$. 
\end{alg}

It's clear that the $\Upsilon_k$ have the desired properties. The elements of each $\Upsilon_k$ can be grouped so as to be summable, and they sum to $\Gamma$; the decomposition process gives a directed binary tree of chord diagrams, equivalent to a bracketing.

We may apply the same idea from the root point, obtaining a \emph{root point decomposition algorithm}. Each $\Upsilon_k$ contains elements $\Gamma_{\cdot w}$ where $w$ has length $k$. The analogous properties are clear.

\subsection{Contact interpretation of the partial order $\preceq$}
\label{sec_contact_partial_order}

We now consider stackability of basis chord diagrams, i.e. $m(\Gamma_{w_0}, \Gamma_{w_1})$ for words $w_0, w_1$. Tightness requires $w_0, w_1$ in the same $W(n_-, n_+)$ (proposition \ref{lem_euler_orthogonality}). Write $\M(w_0, w_1) = \M(\Gamma_{w_0}, \Gamma_{w_1})$ and $m(w_0, w_1) = m(\Gamma_{w_0}, \Gamma_{w_1})$. In this section we prove proposition \ref{contact_interp_partial_order}: $\M( w_0, w_1 )$ is tight iff $w_0 \preceq w_1$. Proposition \ref{general_stackability} is then clear by expanding over basis elements.

\begin{lem}
\label{easy_direction}
If $w_0 \npreceq w_1$, then $\M(w_0, w_1)$ is overtwisted.
\end{lem}

\begin{proof}
Using the ``baseball interpretation'', there is some inning $m$ where team $0$ first takes the lead; so in $w_0$, there are $i$ minus signs and $j$ plus signs up to the $m$'th position, but in $w_1$ there are $i+1$ minus signs and $j-1$ plus signs up to the $m$'th position. Moreover, the $m$'th symbol in $w_0$ is a $+$, while in $w_1$ the $m$'th symbol is a $-$.

By lemma \ref{construction_mechanics}, at this stage of the base point algorithm, in $\Gamma_0$ the discrete interval of used marked points is $[-2i, 2j-1]$, while in $\Gamma_1$ it is $[-2i-1, 2j-2]$. After rounding corners, the chords with endpoints in these intervals precisely match. Since the $m$'th stage is not the final stage, the rounded $\M(w_0, w_1)$ has disconnected sutures.
\end{proof}

\begin{lem}
\label{see_overtwisted_early}
If $\M(w_0, w_1)$ is overtwisted, then a separate component of the sutures can be observed in constructing the basis chord diagrams $\Gamma_{w_0}, \Gamma_{w_1}$ with the base point construction algorithm, before the final step.
\end{lem}

\begin{proof}
After rounding, we have disconnected sutures on $S^2$ and, by the argument of proposition \ref{lem_euler_orthogonality} (section \ref{sec_Mandm}), the total euler class is $0$; hence there are at least $3$ components. Thus some component intersects neither of the two the root points.
\end{proof}

\begin{lem}
\label{reduce_word}
Suppose $w_0, w_1$ begin with the same symbol, $w_0 = s w'_0$, $w_1 = s w'_1$, $s \in \{-, +\}$. Proposition \ref{contact_interp_partial_order} holds for $w_0, w_1$ iff it holds for $w'_0, w'_1$.
\end{lem}

\begin{proof}
By lemma \ref{cancel_outermost} (cancelling outermost chords), $\M(w_0, w_1)$ is equivalent to $\M(w'_0, w'_1)$, rounding and re-folding. Clearly $w_0 \preceq w_1$ iff $w'_0 \preceq w'_1$.
\end{proof}

\begin{proof}[Proof of proposition \ref{contact_interp_partial_order}]
It remains to show if $w_0 \preceq w_1$ then $\M(w_0, w_1)$ is tight. Proof by induction on the length $n$ of the words. It is clearly true by inspection for short words; now assume it is true for all lengths less than $n$.

By lemma \ref{see_overtwisted_early}, if $\M(w_0, w_1)$ is overtwisted, then we see a closed loop before the final stage of the construction algorithm. So we show that no closed loop appears at the $m$'th stage of the base point construction algorithm, for all $m$ before the final step. By lemmas \ref{reduce_word} and \ref{easy_direction}, we can assume $w_0$ begins with $-$ and $w_1$ begins with $+$; so the claim is true for $m=1$. At the $m$'th stage of the algorithm, let $[a_m, b_m]$, $[c_m, d_m]$ denote respectively the discrete intervals of used marked points on $\Gamma_0, \Gamma_1$. The hypothesis $w_0 \preceq w_1$ means that for all $m$, $a_m - 1  < c_m$ and (equivalently) $b_m - 1 < d_m$. Suppose that a closed loop appears first at stage $m$; this loop must contain the chord added on either $\Gamma_0$ or $\Gamma_1$ at this stage. Assume that the chord added on $\Gamma_0$ is part of the new closed loop, and call this chord $\gamma_m$; the $\Gamma_1$ case is similar.

Now $\gamma_m$ cannot include $a_m$, since $a_m$ on $D \times \{0\}$ connects to $a_m - 1$ on $D \times \{1\}$, and $a_m - 1 < c_m$, left of all used points of $\Gamma_1$ at this stage, so cannot form a closed loop. Thus, $\gamma_m$ is $(b_{m-1}+1, b_{m-1}+2) = (b_m - 1, b_m)$, and it forms part of a closed loop. Moreover, it encloses an outermost region on the eastside of $\Gamma_0$, and is constructed by processing a following $+$ sign in $w_0$. Let this be the $j$'th $+$ sign in $w$, so using lemma \ref{construction_mechanics}, $(b_m - 1, b_m) = (2j-2, 2j-1)$ Thus $w_0 = u++v$, where $u$ (possibly empty) contains $j-2$ plus signs, and has length $m-2$.

The endpoints of $\gamma_m$ on $\Gamma_0$ connect to the two marked points $\{b_m - 2, b_m -1 \} = \{2j-3, 2j-2\}$ on $\Gamma_1$. We have $d_m > b_m - 1$, so these are not the rightmost points among the used points on $\Gamma_1$, at this $m$'th stage. However, the closed loop we have just created cannot involve any of the points right of $2j-2 = b_m - 1$ on $\Gamma_1$, since these points connect to marked points right of $b_m$ on $\Gamma_0$, which have not been used yet.

Thus, the chord emanating from $2j-2$ on $\Gamma_1$ must go to the westside, enclosing a $-$ region, and must be created by processing a leading $-$ symbol in $w_1$. And the chord emanating from $2j-3$ on $\Gamma_1$, by lemma \ref{construction_mechanics}, is created by processing the $(j-1)$'th $+$ sign in $w_1$. Thus $w_1 = y +-z$, where $y$ (possibly empty) contains $j-2$ plus signs.

Now, rounding the ball and refolding, we may perform a ``finger move'', pushing the whole new chord $\gamma_m$ off $D \times \{0\}$ and up to $D \times \{1\}$, which has the effect of removing $\gamma_m$ from $D \times \{0\}$, and closing off the marked points labelled $2j-3, 2j-2$ on $D \times \{1\}$. See figure \ref{fig:41}.

\begin{figure}
\centering
\includegraphics[scale=0.6]{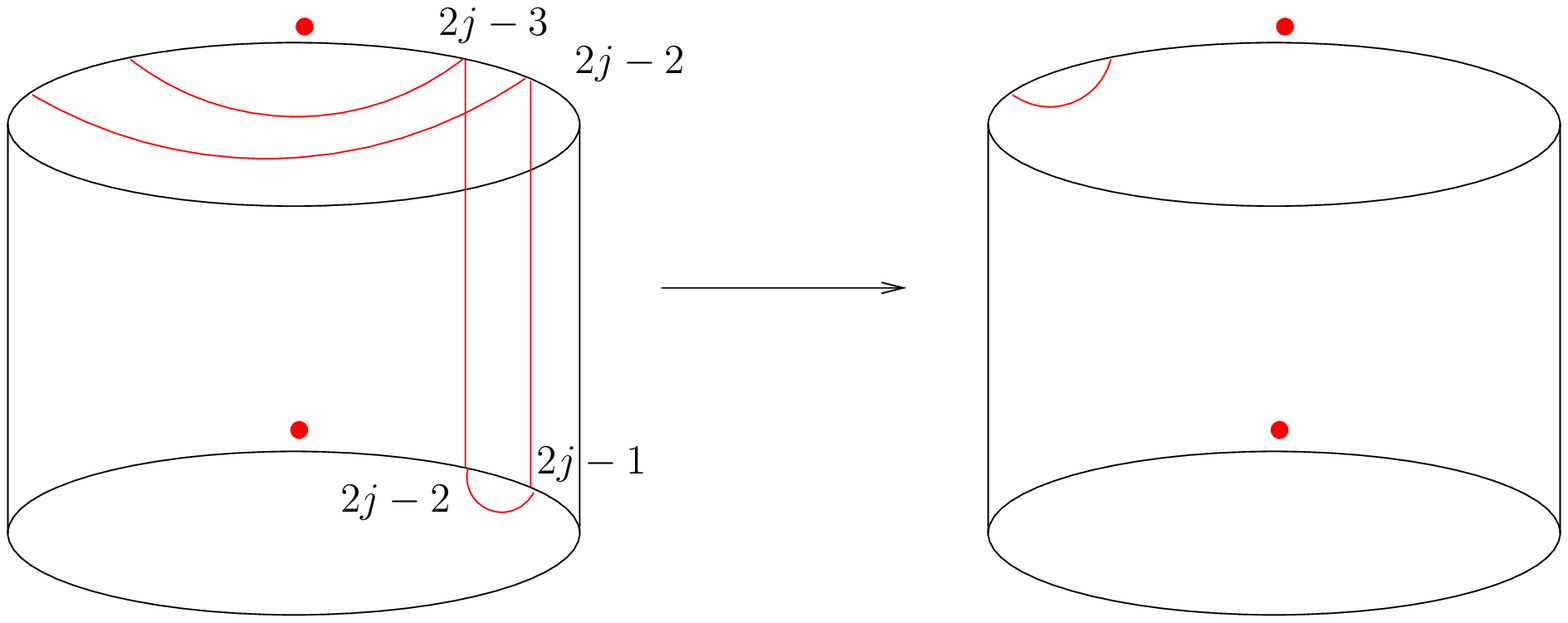}
\caption{Finger move on $\M(\Gamma_0, \Gamma_1)$.} \label{fig:41}
\end{figure}

The chord diagram on $D \times \{0\}$ then reduces to the chord diagram for $w'_0 = u+v$, deleting the $(j-1)$'th $+$ sign from $w_0$. The chord diagram on $D \times \{1\}$ reduces to the chord diagram for $w'_1 = y-z$, also deleting the $(j-1)$'th $+$ sign. Thus the situation reduces to $\M(w'_0, w'_1)$ for two smaller words obtained from deleting the $(j-1)$'th $+$ signs from both $w_0$ and $w_1$; since we deleted the same numbered $+$ signs, $w'_0 \preceq w'_1$. By induction $\M(w'_0, w'_1)$ is tight; so there cannot be any closed loop, a contradiction. Hence we never see a closed loop, and $\M(w_0, w_1)$ is tight.
\end{proof}

\label{sec_stackable}

\section{Bypass systems on basis chord diagrams}

\label{ch_bypass_systems_basis_diagrams}

This section contains the bulk of this paper: we construct concrete bypass systems on basis chord diagrams (section \ref{sec_bypass_systems}); then turn to contact categories (section \ref{sec_ct_cat_comp}).

\subsection{Concrete combinatorial constructions}
\label{sec_bypass_systems}

We now embark upon the proof of proposition \ref{bypass_system_one_to_other}, constructing bypass systems between $\Gamma_1$ and $\Gamma_2$ when $\Gamma_1 \preceq \Gamma_2$ are basis chord diagrams; and showing how performing the same bypass moves in the opposite direction gives a chord diagram with a prescribed minimum and maximum in its basis decomposition. The construction will be explicit. As mentioned in section \ref{sec_moves_diags_words}, we develop a series of increasingly involved analogies between ``word-processing'' and bypass systems. This will take some time. We first give some examples illustrating various phenomena observed when performing multiple bypass moves on a basis chord diagram.

\subsubsection{A menagerie of examples}
\label{sec_menagerie}

First, to go from $\Gamma_{---++++}$ to $\Gamma_{--++-++}$, we ``move the third $-$ sign past the first two $+$ signs''. This is a \emph{forwards elementary move} and is obtained by a single upwards bypass. See figure \ref{fig:19}.

\begin{figure}[tbh]
\centering
\includegraphics[scale=0.35]{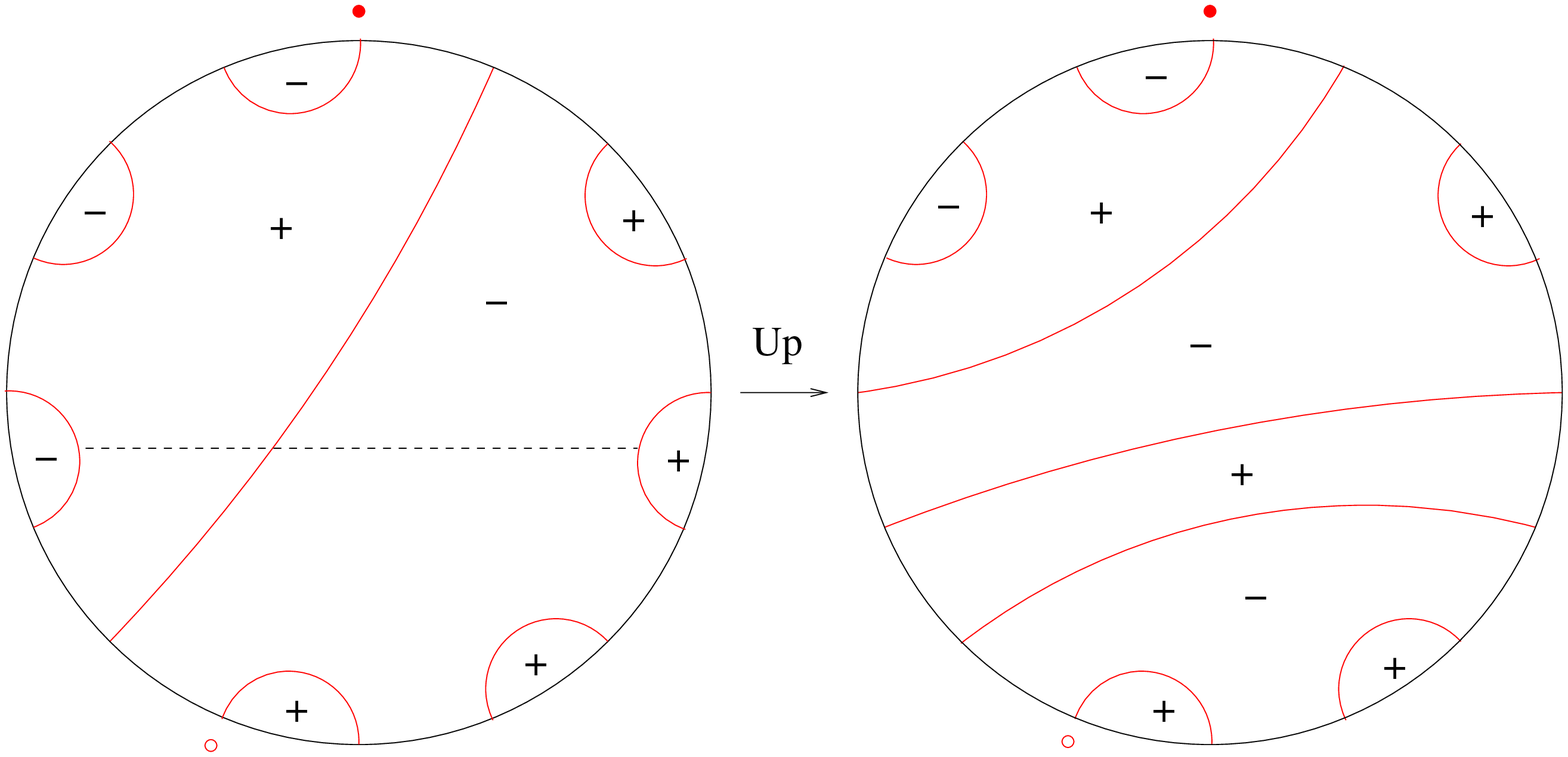}
\caption{Upwards move from $\Gamma_{---++++}$ to $\Gamma_{--++-++}$.} \label{fig:19}
\end{figure}

Next, to go from $\Gamma_{---++++}$ to $\Gamma_{-++--++}$, we ``move the second $-$ sign past the second $+$ sign''. In moving the second $-$ sign, the third $-$ sign is ``brought along for the ride''. This is also a forwards elementary move, also obtained by a single upwards bypass. See figure \ref{fig:20}.

\begin{figure}[tbh]
\centering
\includegraphics[scale=0.35]{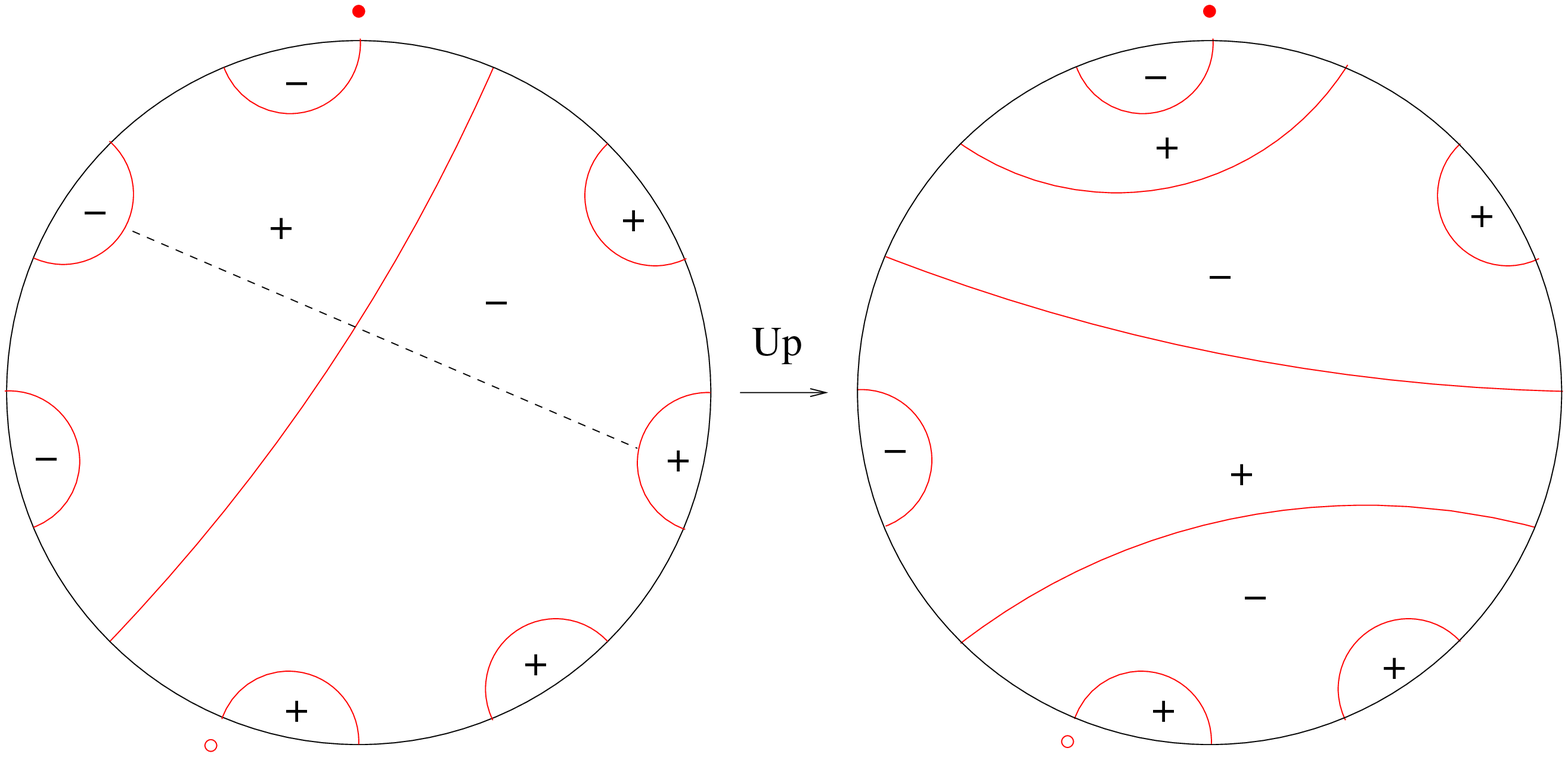}
\caption{Upwards move from $\Gamma_{---++++}$ to $\Gamma_{-++--++}$.} \label{fig:20}
\end{figure}

Alternatively, we could ``treat the two $-$ signs individually'', and perform one bypass move for each, respectively encoding the instruction to move them past the second $+$ sign. See figure \ref{fig:21}. It gives the same result: this is bypass ``redundancy'' or ``rotation'' (section \ref{sec_bypasses_exist}, figure \ref{fig:22}(left) ).

\begin{figure}[tbh]
\centering
\includegraphics[scale=0.35]{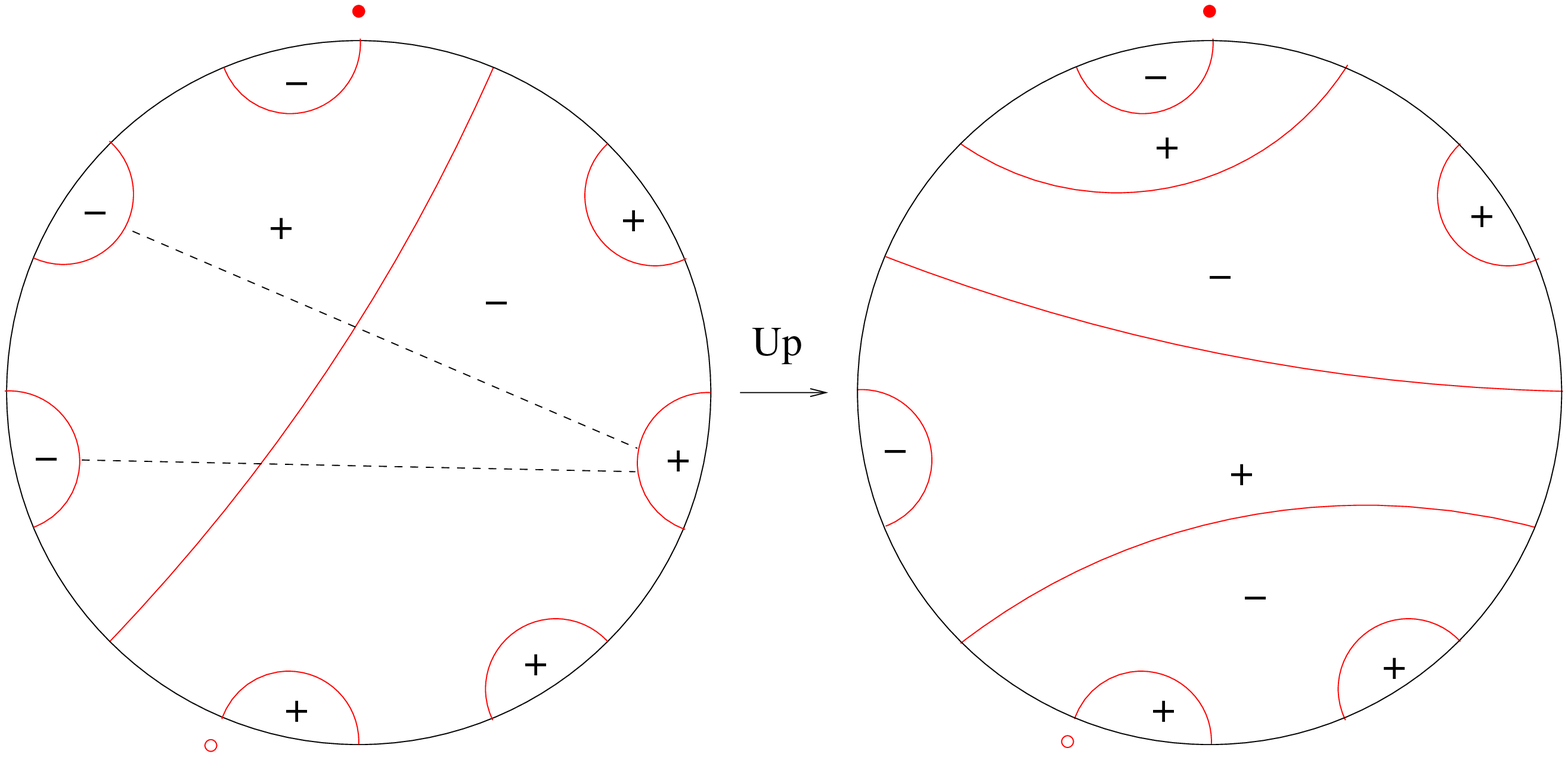}
\caption{Upwards move from $\Gamma_{---++++}$ to $\Gamma_{-++--++}$, another way.} \label{fig:21}
\end{figure}

Next, to go from $\Gamma_{--++--++}$ to $\Gamma_{++--++--}$, we move the first $-$ sign past the second $+$ sign, and the third $-$ sign past the fourth $+$ sign. There are two forwards elementary moves involved, obtained by two upwards bypasses. See figure \ref{fig:23}.

\begin{figure}[tbh]
\centering
\includegraphics[scale=0.35]{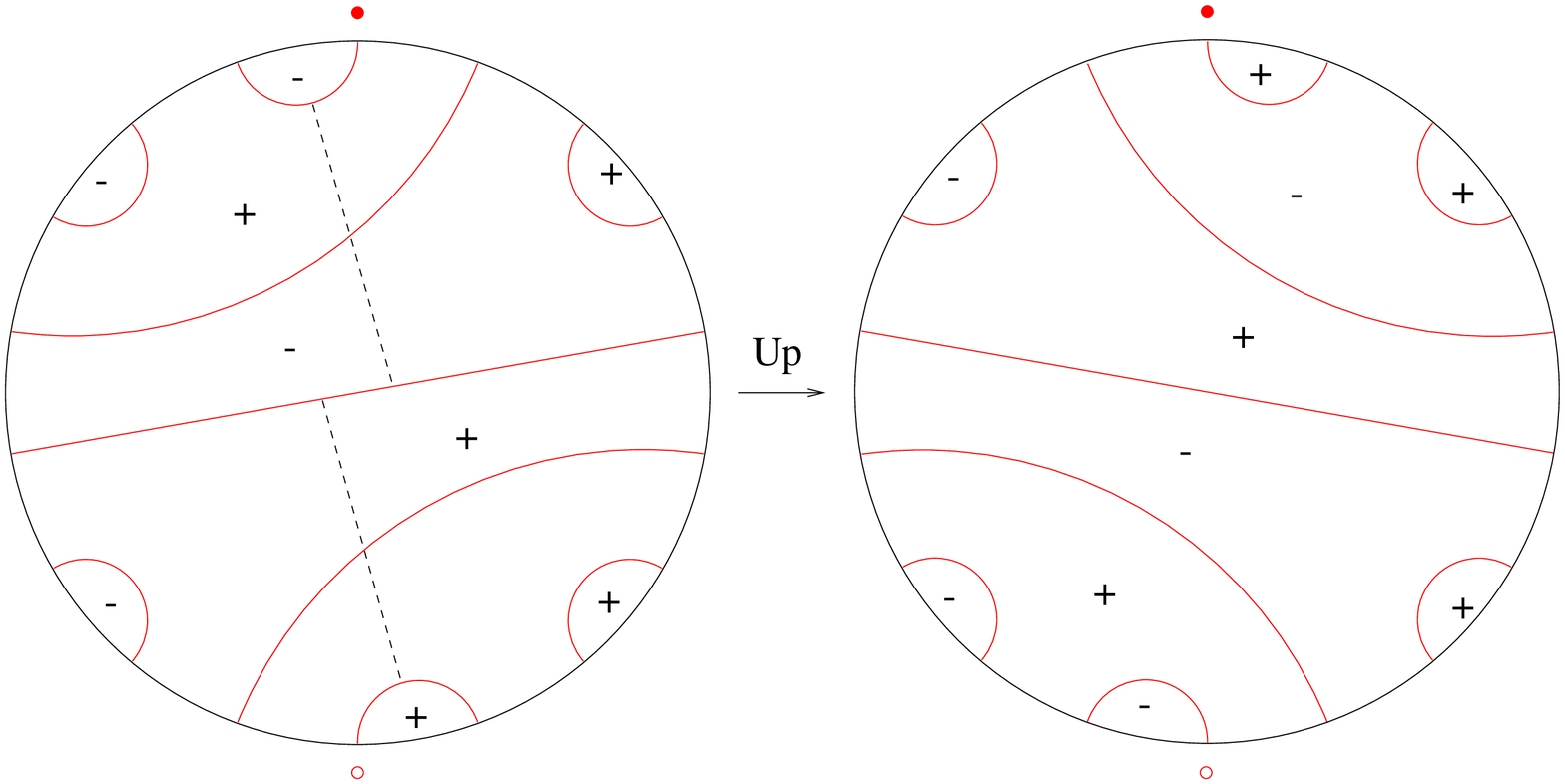}
\caption{Upwards moves from $\Gamma_{--++--++}$ to $\Gamma_{++--++--}$.} \label{fig:23}
\end{figure}

While the position of each of the attaching arcs might be clear from the foregoing, there are \emph{two distinct} ways to place them relative to each other. If we consider these attaching arcs in the other possible arrangement, we obtain the drastically different result $\Gamma_{++++----}$. See figure \ref{fig:24}. This positioning of arcs encodes ``move the first $-$ sign past the fourth $+$ sign'': a \emph{generalised elementary move}.

\begin{figure}[tbh]
\centering
\includegraphics[scale=0.35]{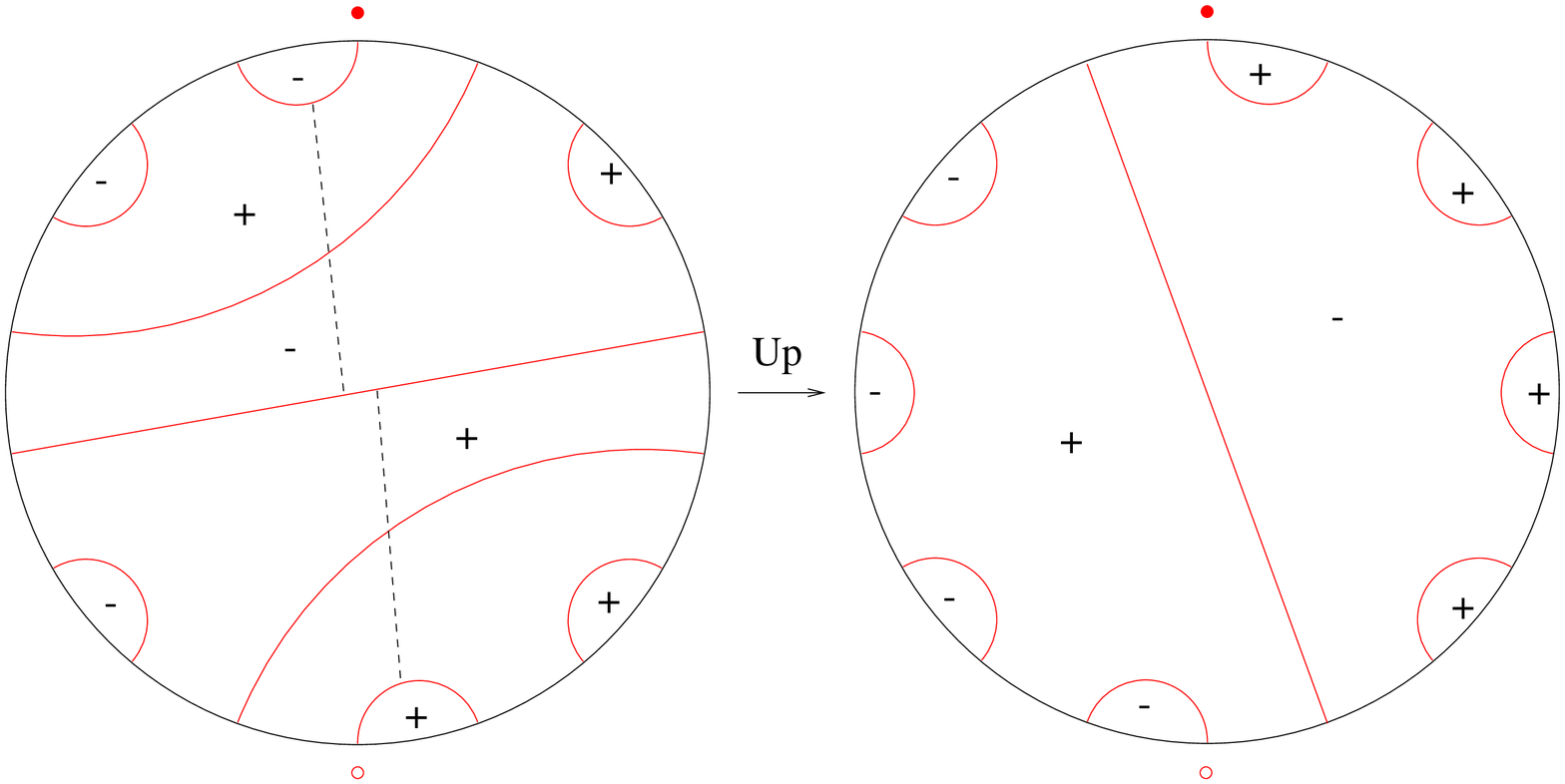}
\caption{Upwards moves from $\Gamma_{--++--++}$ to $\Gamma_{++++----}$.} \label{fig:24}
\end{figure}

Alternatively, the ``treat each $-$ sign individually'' approach to $\Gamma_{--++--++} \rightarrow \Gamma_{++++----}$ requires six bypass arcs: 2 for the first $-$ sign, 2 for the second, 1 for the third, and 1 for the fourth. See figure \ref{fig:25}.

\begin{figure}[tbh]
\centering
\includegraphics[scale=0.35]{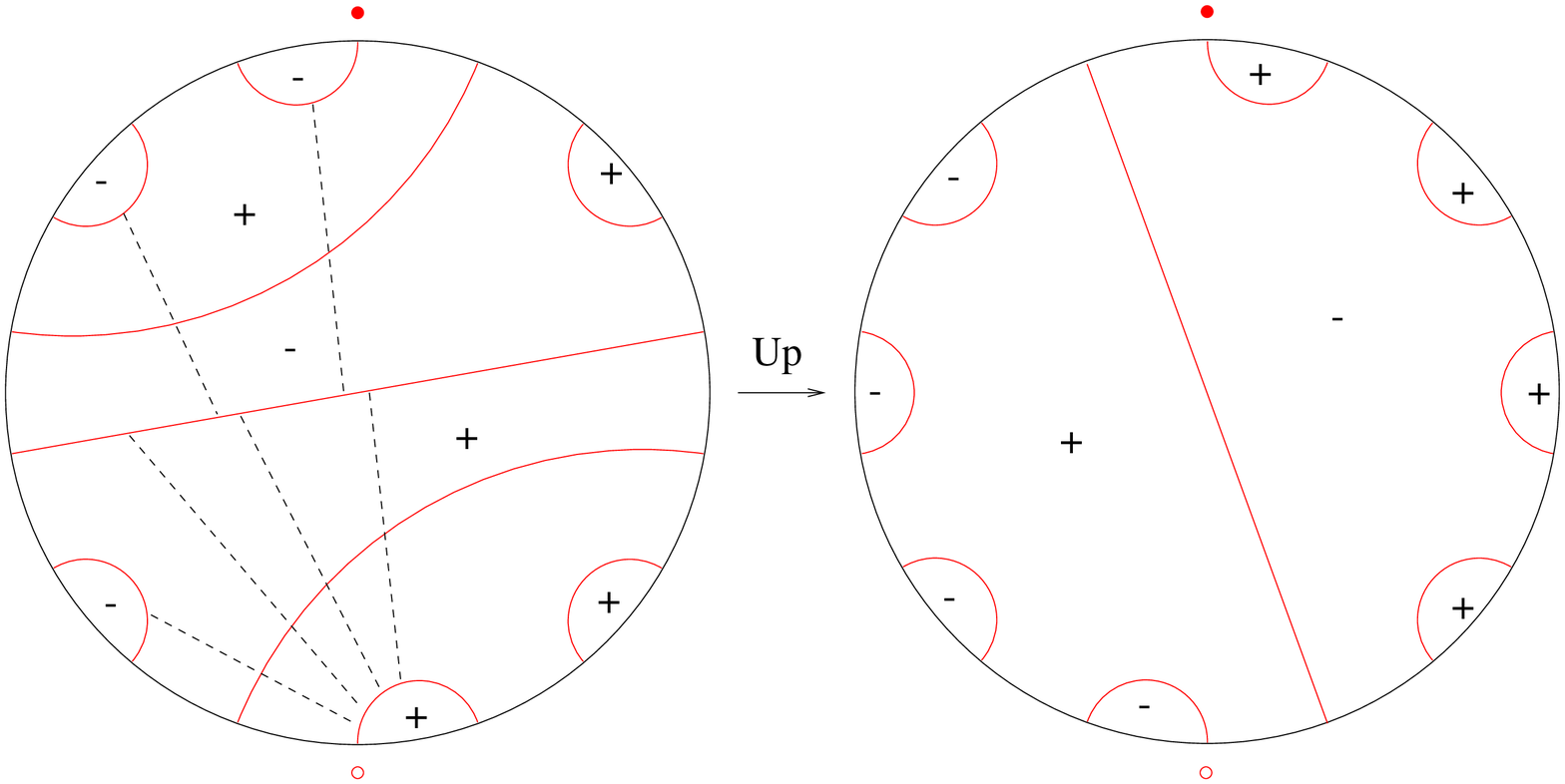}
\caption{``Individual care'' approach to $\Gamma_{--++--++} \rightarrow \Gamma_{++++----}$.} \label{fig:25}
\end{figure}

In general, in the following we apply the ``individual care'' approach, being easier to formalise, even though the sets of bypass moves so obtained often contain  redundancy. This will lead to the notion of \emph{coarse bypass system of a pair} of comparable basis chord diagrams, which we will then refine to a ``minimal'' \emph{bypass system of a pair}.

\subsubsection{Elementary moves on words}
\label{sec_el_moves_words}

Given a word $w$, group it into \emph{blocks} of $\pm$ symbols and write $w = (-)^{a_1} (+)^{b_1} \cdots (-)^{a_k} (+)^{b_k}$.
Possibly $k=1$; possibly $a_1$ or $b_k$ is $0$; but every other $a_i, b_i$ is nonzero.

\begin{defn}[Elementary moves on words]
A \emph{forwards elementary move} on $w$ takes a substring $(-)^a (+)^b$ and replaces it with $(+)^b (-)^a$. A \emph{backwards elementary move} on $w$ takes a substring $(+)^b (-)^a$ and replaces it with $(-)^a (+)^b$.
\end{defn}
Collectively these are \emph{elementary moves}; the direction, forwards or backwards, refers to the partial order $\preceq$.

\begin{defn}[Denoting elementary moves]
$FE(i,j)$ denotes the forwards elementary move taking the $i$'th $-$ sign to the position immediately right of the $j$'th $+$ sign. $BE(i,j)$ denotes the backwards elementary move taking the $j$'th $+$ sign to the position immediately right of the $i$'th $-$ sign.
\end{defn}
Clearly $1 \leq i \leq n_-$, $1 \leq j \leq n_+$ here; but $FE(i,j)$ is not always defined. The move $FE(i,j)$ (resp. $BE(i,j)$) is defined iff the block of the $i$'th $-$ sign is immediately left (resp. right) of the block of the $j$'th $+$ sign.

\subsubsection{Anatomy of attaching arcs on basis chord diagrams}

\label{sec_anatomy_attaching_arcs}

We now give a complete description of attaching arcs on basis chord diagrams.  For $w = (-)^{a_1} (+)^{b_1} \cdots (-)^{a_k} (+)^{b_k}$, $\Gamma_w$ is as shown in figure \ref{fig:26}. There is a nice bijection between nontrivial elementary moves on $w$ and nontrivial arcs of attachment on $\Gamma_w$; to formalise this we need several definitions.

\begin{figure}[tbh]
\centering
\includegraphics[scale=0.3]{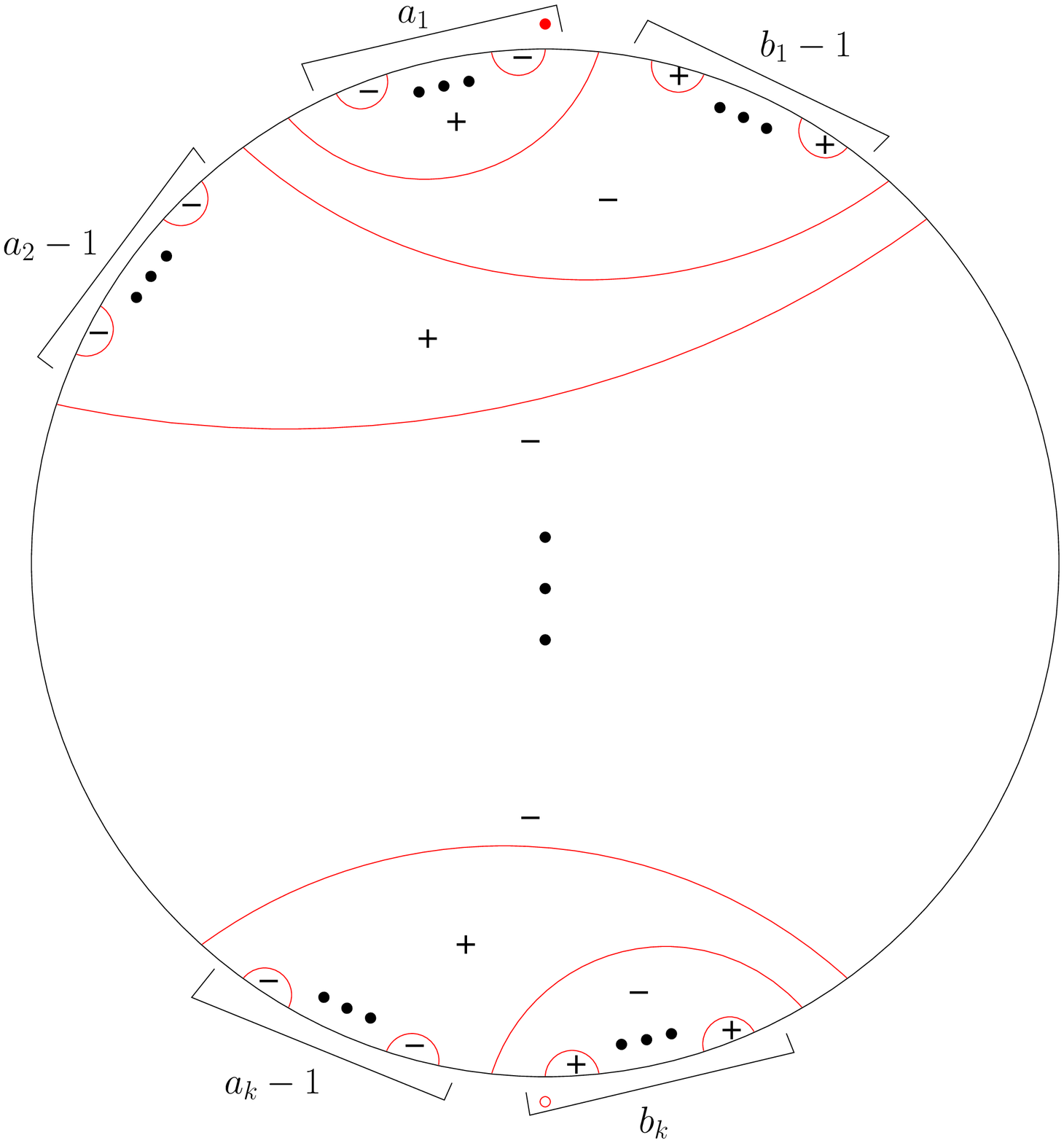}
\caption{General basis chord diagram $(-)^{a_1} \cdots (+)^{b_k}$.} \label{fig:26}
\end{figure}

\begin{defn}[Attaching arc types]
An attaching arc $c$ on a chord diagram $\Gamma$ is:
\begin{enumerate}
\item
\emph{nontrivial} if $c$ intersects three distinct chords of $\Gamma$;
\item
\emph{trivial} if $c$ intersects less than three distinct chords of $\Gamma$;
\begin{enumerate}
\item
\emph{slightly} trivial if $c$ intersects precisely two distinct chords of $\Gamma$.
\item
\emph{super}trivial if $c$ intersects only one chord of $\Gamma$.
\end{enumerate}
\end{enumerate}
\end{defn}

For any trivial arc, performing a bypass move on it in one direction creates a closed curve; in the other direction, the chord diagram is unchanged. If the upwards move produces the same chord diagram (and downwards creates a closed loop), the arc is \emph{upwards}; vice versa for \emph{downwards}.

Supertrivial attaching arcs come in two types. Consider traversing $c$ from one end to the other; let the three intersection points of $c$ with a chord $\gamma$, in order along $c$, be $p_1, p_2, p_3$. If $p_1, p_2, p_3$ lie in order along $\gamma$, $c$ is \emph{direct}; otherwise $c$ is \emph{indirect}. See figure \ref{fig:68}.

\begin{figure}
\centering
\includegraphics[scale=0.4]{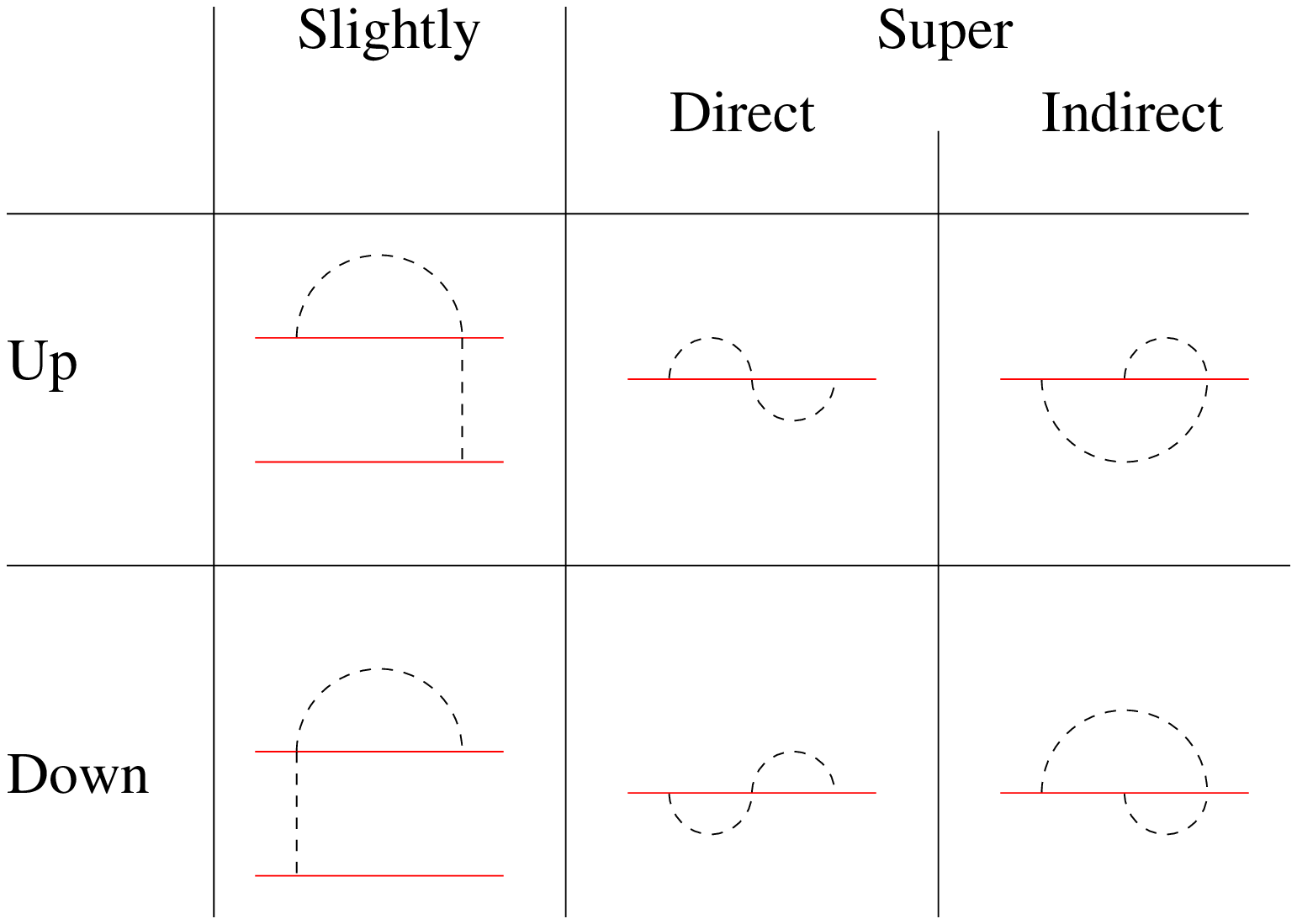}
\caption{Types of trivial attaching arcs.} \label{fig:68}
\end{figure}

A small neighbourhood $U$ of an attaching arc $c$ is cut by $\Gamma$ into $4$ regions. Two of these regions intersect $c$, and two do not. If $c$ is nontrivial, these $4$ regions are distinct; if $c$ is trivial, they are not. The two regions intersecting $c$ lie in components of $D - \Gamma$ called the \emph{inner regions} of $c$; the two regions not intersecting $c$ lie in components  of $D - \Gamma$ called the \emph{outer regions} of $c$. Note the two outer regions of $c$ have opposite sign; similarly for the inner regions. 

The construction algorithm's ordering on regions motivates the following definition.
\begin{defn}[Prior, latter chords]
Let $c$ be a nontrivial or slightly trivial arc of attachment. The endpoints of $c$ lie on two distinct chords; the one created
\begin{enumerate}
\item first in the base point construction algorithm is the \emph{prior chord} of $c$;
\item later is the \emph{latter chord} of $c$.
\end{enumerate}
\end{defn}
Similarly, the outer region adjacent to the prior (resp. latter) chord is the \emph{prior (resp. latter) outer region} of $c$. Note prior/latter do not apply to  supertrivial arcs.

\begin{defn}[Forwards/backwards attaching arcs] For a nontrivial attaching arc $c$:
\begin{enumerate}
\item 
if its prior outer region is negative, $c$ is \emph{forwards}.
\item
if its prior outer region is positive, $c$ is \emph{backwards}.
\end{enumerate}
\end{defn}
For slightly trivial arcs of attachment, we have similar forwards/backwards notions; to avoid (hopefully not create) confusion, we call them \emph{quasi-forwards} or \emph{quasi-backwards}.

For nontrivial attaching arcs, the prior outer region is not adjacent to the root point, and the latter outer region is not adjacent to the base point. Hence we introduce the following notation, recalling numberings of regions (definitions \ref{def_base_numbering}, \ref{def_root_numbering}).
\begin{defn}[Denoting arcs of attachment]\
\begin{enumerate}
\item 
$FA(i,j)$ is the nontrivial forwards attaching arc whose prior outer region is the base-$i$'th $-$ region and latter outer region is the root-$j$'th $+$ region.
\item
$BA(i,j)$ is the nontrivial backwards attaching arc whose prior outer region is the base-$j$'th $+$ region and latter outer region is the root-$i$'th $-$ region.
\end{enumerate}
\end{defn}

\subsubsection{Single bypass moves and elementary moves}
\label{sec_single_bypass_el_moves}

The bijection between bypass moves on basis chord diagrams and elementary moves on words is now straightforward.
\begin{lem}[Existence of attaching arcs]
\label{lem_existence_attaching_arcs}
There is a forwards (resp. backwards) attaching arc $FA(i,j)$ (resp. $BA(i,j)$) on $\Gamma_w$ iff there is a forwards (resp. backwards) elementary move $FE(i,j)$ (resp. $BE(i,j)$) on $w$.
\end{lem}

\begin{proof}
We prove the forwards case; backwards is similar. There exists an $FE(i,j)$ iff the $i$'th $-$ sign and $j$'th $+$ sign appear in adjacent blocks $(-)^{a} (+)^{b}$. Considering the base point construction algorithm, this is equivalent to $\Gamma_w$ containing an arrangement as shown in the left of figure \ref{fig:28}, and hence to the existence of an $FA(i,j)$.
\end{proof}

\begin{rem}[Forwards and backwards analogous]
Throughout this section we have arguments which come in analogous ``forwards'' and ``backwards'' versions. We often give arguments, and sometimes statements, for the forwards version only.
\end{rem}

\begin{lem}[Bypass \& elementary moves]
\label{bypass_moves_elementary_moves}
Performing upwards (resp. downwards) bypass moves on $\Gamma_w$ along $FA(i,j)$ (resp. $BA(i,j)$) gives $\Gamma_{w'}$, where $w' = FE(i,j)(w)$ (resp. $BE(i,j)(w)$).
\[
\xymatrix{
w \ar@{|->}[rr]^{FE(i,j)} \ar@{<~>}[dd] && w' \ar@{<~>}[dd] && w \ar@{|->}[rr]^{BE(i,j)} \ar@{<~>}[dd] && w' \ar@{<~>}[dd] \\
&&& 
 &&& \\
\Gamma_w \ar@{|->}[rr]^{\Up(FA(i,j))} && \Gamma_{w'} && \Gamma_w \ar@{|->}[rr]^{\Down(BA(i,j))} && \Gamma_{w'} }
\]
A downwards bypass move along $FA(i,j)$ (resp. upwards along $BA(i,j)$) gives $\Gamma_w + \Gamma_{w'}$.
\end{lem}

\begin{proof}
We prove the forwards case. Consider $FE(i,j)$ on $w$ and $FA(i,j)$ on $\Gamma_w$; by lemma \ref{lem_existence_attaching_arcs}, one exists iff the other does.  So the $i$'th $-$ and $j$'th $+$ in $w$ occur in adjacent blocks, and $\Gamma_w$ is as in figure \ref{fig:28}. Let the last $-$ sign in the block of the $i$'th $-$ sign be the $l$'th (so $l \geq i$), and let the first $+$ sign in the block with the $j$'th $+$ sign be numbered $m$ (so $m \leq j$). An upwards bypass move along $FA(i,j)$ then has the effect shown; it produces the basis chord diagram for the word $w'$, where $w'$ is obtained from $w$ by swapping the string of $i$'th thru $l$'th $-$ signs with the string of $m$'th thru $j$'th $+$ signs, $(-)^{l-i+1} (+)^{j-m+1} \mapsto (+)^{j-m+1} (-)^{l-i+1}$. Thus $w' = FE(i,j)(w)$. The bypass relation gives the final statement.
\end{proof}
In particular, a bypass move on a basis diagram yields either a basis diagram or a sum of two basis diagrams.

\begin{figure}
\centering
\includegraphics[scale=0.3]{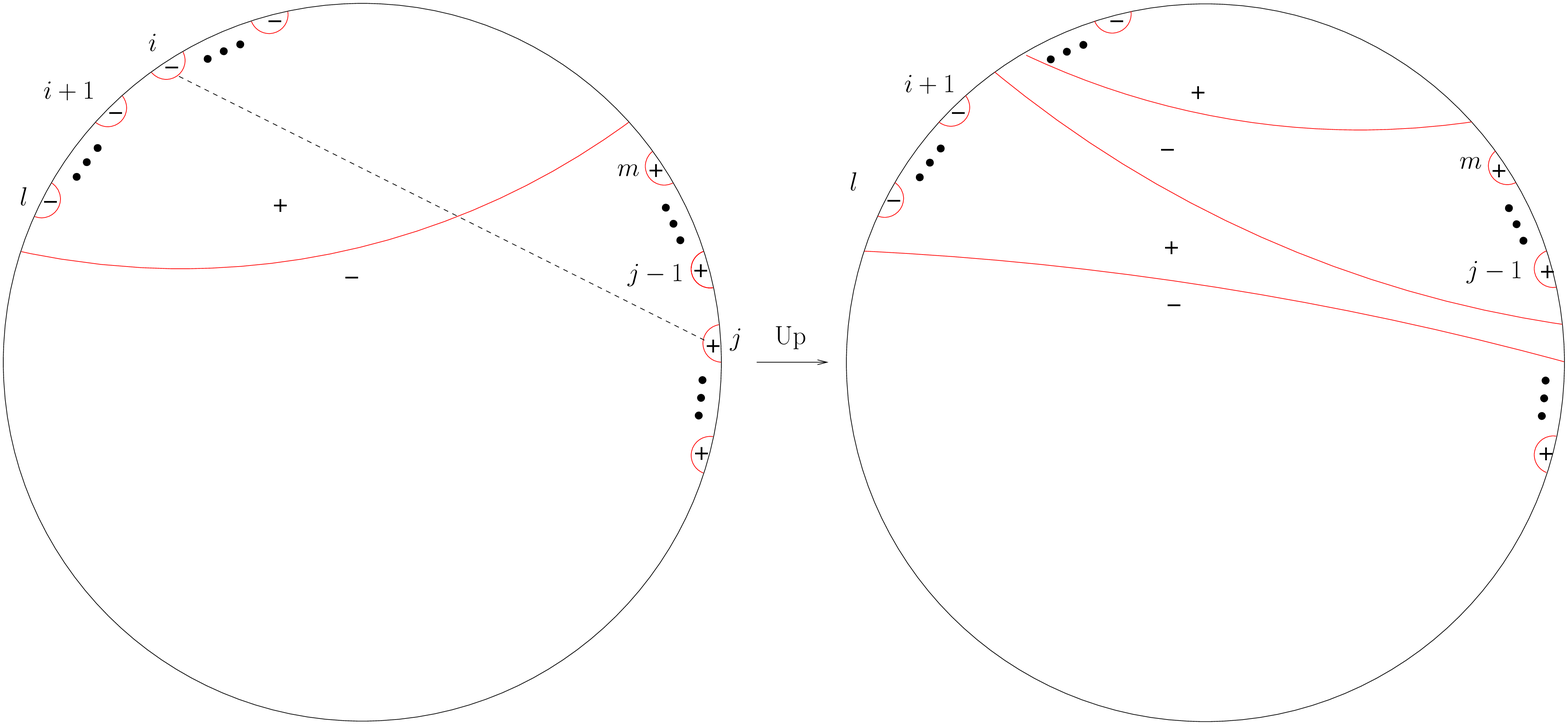}
\caption{Effect of bypass move along $FA(i,j)$.} \label{fig:28}
\end{figure}

\subsubsection{Stability of basis diagrams}

\label{sec_bypass_systems_general}

Certain bypass moves on basis chord diagrams always yield again basis diagrams.
\begin{prop}[Stability of basis diagrams]
\label{upwards_moves_forwards}
Performing upwards (resp. downwards) bypass moves on a basis chord diagram $\Gamma_{w_1}$ along a bypass system consisting entirely of forwards (resp. backwards) nontrivial attaching arcs yields a basis chord diagram $\Gamma_{w_2}$, with $w_1 \preceq w_2$ (resp. $w_2 \preceq w_1$).
\end{prop}

In performing such bypass moves, various phenomena may occur: surgery on an arc $c_i$ may convert a nontrivial arc $c_j$ into a trivial one, or a slightly trivial $c_j$ into a supertrivial one, or vice versa; surgery on a trivial arc $c_i$ might not change the chord diagram but might change the locations of other arcs.

We therefore have the following lemma; it is for this reason that the anatomical terminology ``supertrivial'', ``quasi-forwards'', ``upwards trivial'' and ``direct supertrivial'' has been introduced. There are several cases to check, but it is not difficult to verify.
\begin{lem}
\label{arc_still_okay}
Consider a bypass system on a basis chord diagram consisting of:
\begin{enumerate}
\item nontrivial forwards attaching arcs;
\item slightly trivial, quasi-forwards, upwards attaching arcs;
\item supertrivial, direct, upwards attaching arcs.
\end{enumerate}
After performing an upwards bypass move along one attaching arc, we still have a basis chord diagram, and each remaining attaching arc is of one of the above three types.
\qed
\end{lem}
There is also a backwards version. Proposition \ref{upwards_moves_forwards} follows obviously.

\subsubsection{Generalised elementary moves and attaching arcs}

\label{sec_gen_el_words}

The elementary move $FE(i,j)$ moves the $i$'th $-$ sign right, past the $j$'th $+$ sign: \emph{provided they are in adjacent blocks}. The attaching arc $FA(i,j)$ connects the base-$i$'th $-$ region to the root-$j$'th $+$ region: \emph{provided one can draw an attaching arc between them}. We now generalise, removing the provisos in italics.

\begin{defn}[Generalised elementary move]
Let $w$ be a word.
\begin{enumerate}
\item
If the $i$'th $-$ sign in $w$ is left of the $j$'th $+$ sign, the \emph{forwards generalised elementary move} $FE(i,j)$ takes the $i$'th $-$ sign, and all $-$ signs between it and the $j$'th $+$ sign, and moves them to a position immediately after the $j$'th $+$ sign.
\item
If the $j$'th $+$ sign in $w$ is left of the $i$'th $-$ sign, the \emph{backwards generalised elementary move} $BE(i,j)$ takes the $j$'th $+$ sign, and all $+$ signs between it and the $i$'th $-$ sign, and moves them to a position immediately after the $i$'th $-$ sign.
\end{enumerate}
\end{defn}
Clearly for any $(i,j)$ precisely one of these moves exists. As clear generalisations of elementary moves, we use the same notation without contradiction. Forwards/backwards moves still move forwards/backwards in $\preceq$.

\label{sec_generalised_arcs}
\begin{defn}[Generalised arc of attachment]
A \emph{generalised arc of attachment} $c$ in a chord diagram $\Gamma$ is an arc intersecting $\Gamma$ in an odd number of points, including both its endpoints. A generalised attaching arc is \emph{nontrivial} if all its intersection points with $\Gamma$ lie on different components of $\Gamma$. Two generalised attaching arcs are \emph{equivalent} if they are homotopic through generalised attaching arcs.
\end{defn}
For any two chords in a chord diagram $\Gamma$, there is at most one nontrivial generalised arc of attachment between them, up to equivalence.

We have notions of prior and latter, outer regions, and forwards and backwards, for a nontrivial generalised attaching arc $c$, entirely analogously: $c$'s endpoints lie on two chords, \emph{prior} and \emph{latter} according to order of construction; adjacent to these chords but not intersecting $c$ are $c$'s prior and latter \emph{outer regions}; the two outer regions have opposite signs. Given two outer regions, there is at most one nontrivial generalised attaching arc between them. A generalised attaching arc with negative (resp. positive) prior outer region is \emph{forwards} (resp. \emph{backwards}).

The forwards generalised attaching arc with prior outer region the base-$i$'th $-$ region, and latter outer region the root-$j$'th $+$ region, is called $FA(i,j)$. The backwards generalised attaching arc with prior outer region the base-$j$'th $+$ region, and latter outer region the root-$i$'th $-$ region, is called $BA(i,j)$. Clearly this generalises previous notation. Every nontrivial generalised attaching arc is some $FA(i,j)$ or $BA(i,j)$. For any $(i,j)$, $1 \leq i \leq n_-$, $1 \leq j \leq n_+$ precisely one of $FA(i,j)$ or $BA(i,j)$ exists. The lemma generalising \ref{lem_existence_attaching_arcs} should now be clear.
\begin{lem}[Existence of generalised attaching arcs]\
\label{existence_generalised_arcs}
There is an $FA(i,j)$ in $\Gamma_w$ iff the $i$'th $-$ sign in $w$ occurs before the $j$'th $+$ sign, iff there is an $FE(i,j)$ on $w$.
There is a $BA(i,j)$ in $\Gamma_w$ iff the $j$'th $+$ sign in $w$ occurs before the $i$'th $-$ sign, iff there is a $BE(i,j)$ on $w$.
\qed
\end{lem}

\subsubsection{Bypass system of a generalised attaching arc} 

\label{sec_byp_sys_gen_arc}

We cannot perform a bypass move on a generalised attaching arc. But from it, we can obtain a bypass system: since it intersects chords at an odd number of points, it may be broken into several bona fide attaching arcs, overlapping only at endpoints; we perturb these endpoints to become disjoint.

Precisely, let $c$ be a nontrivial generalised attaching arc in a basis chord diagram $\Gamma_w$, with $|c \cap \Gamma_w| = 2m+1$.  Let $p$ be the intersection of $c$ with a chord $\gamma$ of $\Gamma_w$, an interior point of $c$. Then there are ``prior'' and ``latter'' directions along $c$ from $p$, towards prior and latter chords. Note $\gamma$ cannot be an outermost chord; by the classification of chords in basis chord diagrams (lemma \ref{construction_mechanics}), $\gamma$ runs from the westside to eastside of $\Gamma_w$. So from $p$, there is a well-defined ``west'' and ``east'' direction along $\gamma$.

We split $c$ into a series of attaching arcs $c_1, \ldots, c_m$, labelled from prior chord to latter chord, intersecting only at endpoints. If $c$ is forwards (resp. backwards) then so are the $c_i$. The \emph{bypass system of $c$} is obtained by perturbing the $c_i$ as follows:
\begin{enumerate}
\item 
If $c$ is forwards, then at the intersection point $p$ of $c_i$ and $c_{i+1}$ on a non-outermost chord $\gamma$ of $\Gamma_w$, move the endpoint of $c_i$ slightly west of $p$ along $\gamma$, and the endpoint of $c_{i+1}$ slightly east of $p$ along $\gamma$.
\item
If $c$ is backwards, then at the intersection point $p$ of $c_i$ and $c_{i+1}$ on a non-outermost chord $\gamma$ of $\Gamma_w$, we move the endpoint of $c_i$ slightly east of $p$ along $\gamma$, and the endpoint of $c_{i+1}$ slightly west of $p$ along $\gamma$.
\end{enumerate}
See figure \ref{fig:30}. We now generalise lemma \ref{bypass_moves_elementary_moves}.
\begin{lem}[Generalised attaching arcs and elementary moves]\
\label{generalised_arc_generalised_move}
Performing upwards (resp. downwards) bypass moves on $\Gamma_w$ along the bypass system of $FA(i,j)$ (resp. $BA(i,j)$) gives $\Gamma_{w'}$, where $w' = FE(i,j)(w)$ (resp. $BE(i,j)(w)$).
\[
\xymatrix{
w \ar@{|->}[rr]^{FE(i,j)} \ar@{<~>}[dd] && w' \ar@{<~>}[dd] && w \ar@{|->}[rr]^{BE(i,j)} \ar@{<~>}[dd] && w' \ar@{<~>}[dd] \\
&&& &&& \\
\Gamma_w \ar@{|->}[rr]^{\Up(FA(i,j))} && \Gamma_{w'} && \Gamma_w \ar@{|->}[rr]^{\Down(BA(i,j))} && \Gamma_{w'} }
\]
\end{lem}

\begin{proof}
From lemma \ref{existence_generalised_arcs}, $FA(i,j)$ exists iff $FE(i,j)$ does, iff the $i$'th $-$ sign occurs before the $j$'th $+$ sign in $w$. Let the substring of $w$ between the $i$'th $-$ sign and the $j$'th $+$ sign be $(-)^{a_1} (+)^{b_1} \cdots  (+)^{b_{k-1}} (-)^{a_k} (+)^{b_k}$. Then the situation appears as shown in figure \ref{fig:30}. Performing upwards bypass moves along the arcs of attachment produces the result shown, which corresponds to replacing this substring with $(+)^{b_1 + \cdots + b_k} (-)^{a_1 + \cdots + a_k}$. That is, $FE(i,j)$ is performed on $w$.
\end{proof}

\begin{figure}
\centering
\includegraphics[scale=0.3]{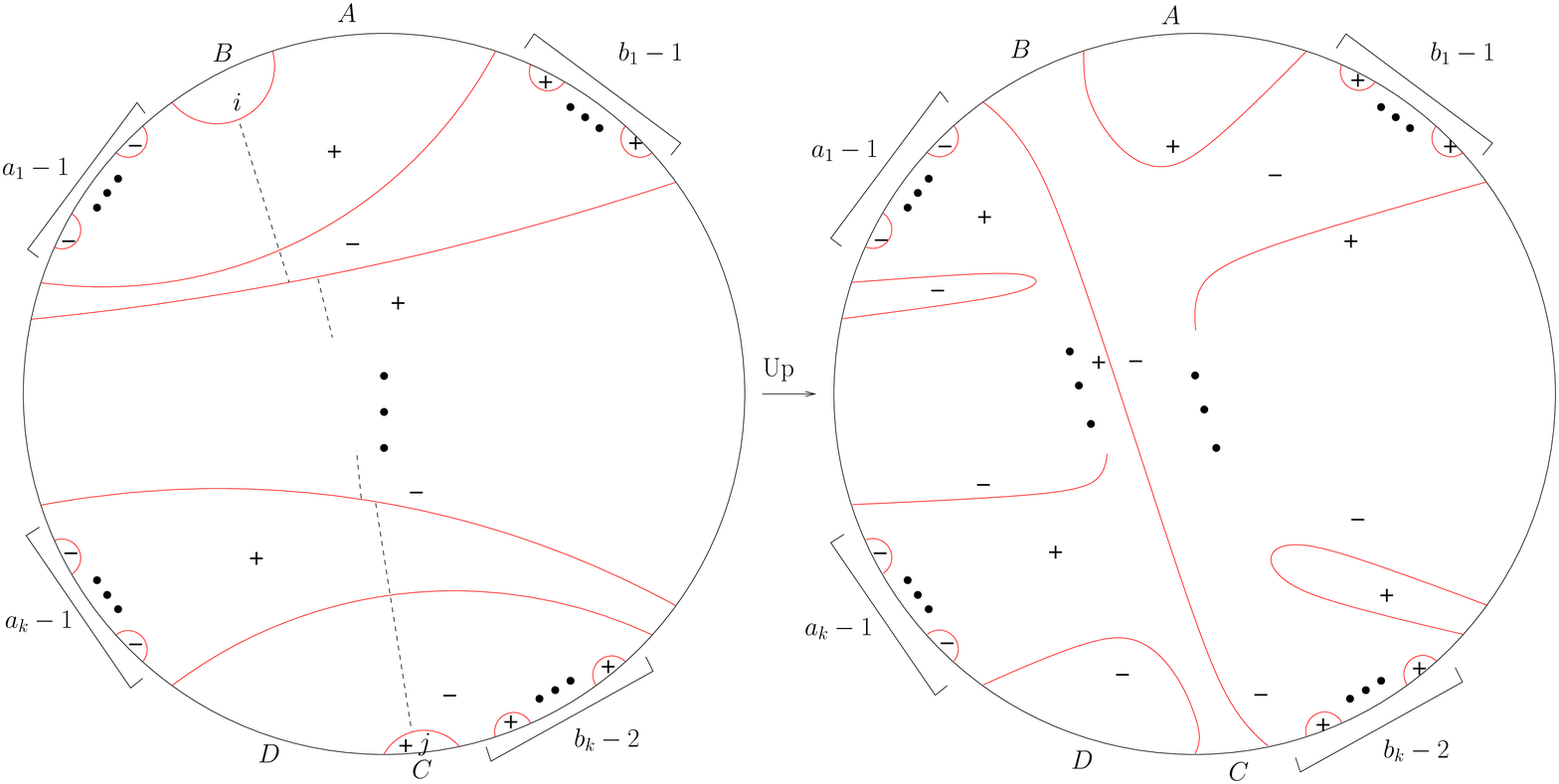}
\caption{Effect of upwards bypass moves on a forwards generalised attaching arc.} \label{fig:30}
\end{figure}

Loosely, bypass moves along the bypass system of a generalised attaching arc create a ``long chord'' running along the arc, and ``close off'' chords on either side of it.

\subsubsection{Anatomy of multiple generalised attaching arcs}

\label{sec_anatomy_multiple_gen_arcs}

We now consider multiple generalised attaching arcs. Two generalised arcs $FA(i_1, j_1)$, $FA(i_2, j_2)$ might intersect; even if they do not intersect, it might be that having placed $FA(i_1, j_1)$, we can place $FA(i_2, j_2)$ on either side of it, leading to different outcomes. We will specify precisely how to make such placements. 

We therefore make some definitions. Note that a forwards arc $FA(i,j)$, taken together with its prior and latter chords, splits the disc $D$ into four regions; see figure \ref{fig:31}(left):
\begin{enumerate}
\item The piece containing the prior outer region of $FA(i,j)$.
\item The piece which contains the marked points on the eastside immediately anticlockwise/right of the latter chord of $FA(i,j)$
\item The piece containing the latter outer region of $FA(i,j)$.
\item The piece which contains the marked points on the westside immediately anticlockwise/left of the prior chord of $FA(i,j)$.
\end{enumerate}
Regions (i) and (ii) are the \emph{northeast} of $FA(i,j)$; regions (iii) and (iv) the \emph{southwest}. All chords and regions constructed in the base point construction algorithm prior to the $i$'th $-$ region lie in the northeast; all chords and regions created after the root-$j$'th $+$ region lie to the southwest. 

Similar definitions of \emph{northwest} and \emph{southeast} exist in the backwards case.

\begin{figure}[tbh]
\centering
\mbox{
\includegraphics[scale=0.3]{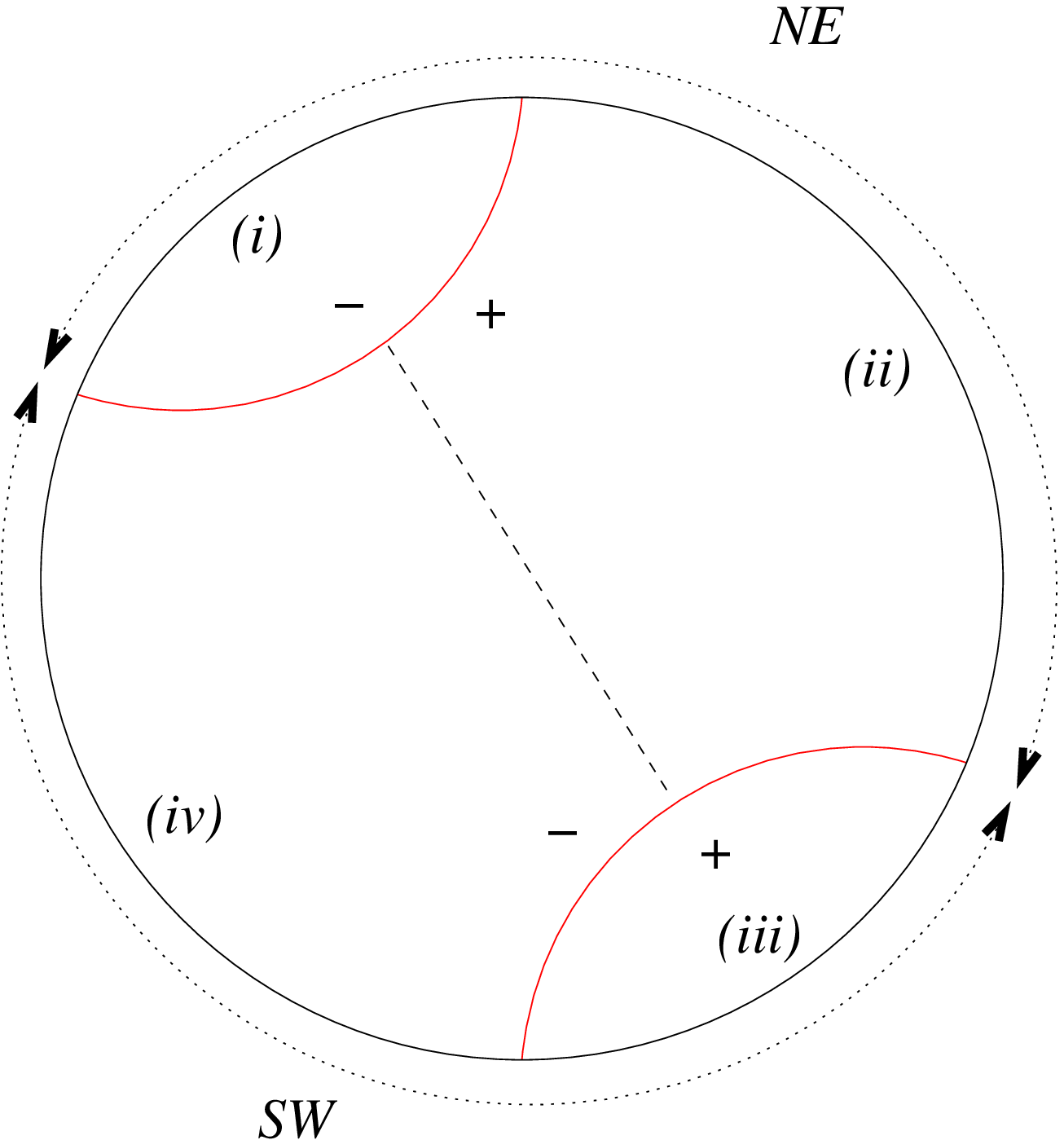}
\quad \quad \quad \quad
\includegraphics[scale=0.25]{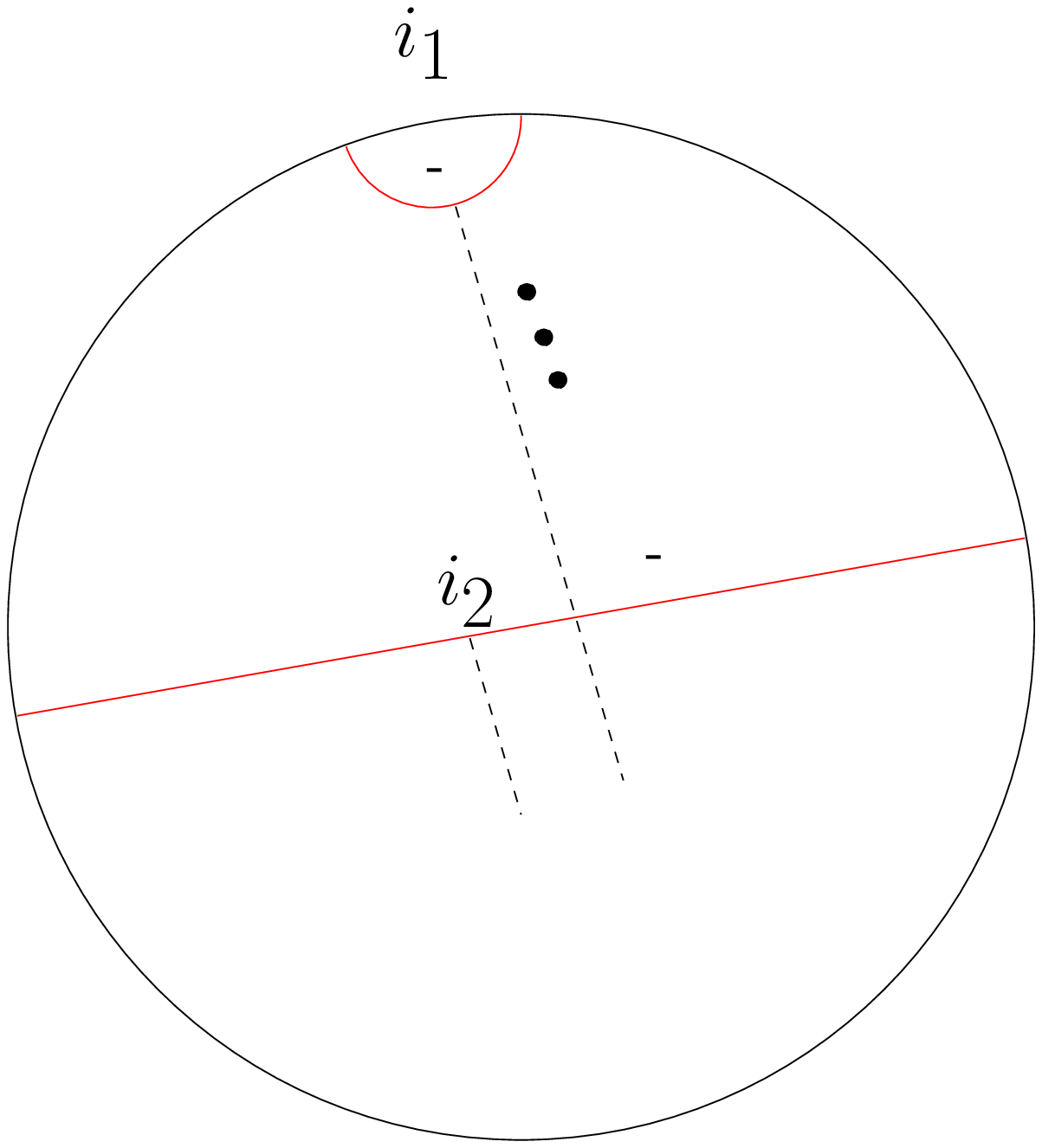}
}
\caption{Left: Compass points. Right: Placing two arcs; $i_1 < i_2$.} 
\label{fig:31}
\label{fig:32}
\end{figure}

\label{sec_placing_two_gen_arcs}

We use these ``compass points'' to place two generalised attaching arcs $FA(i_1, j_1)$ and $FA(i_2, j_2)$, \emph{under the assumptions $i_1 < i_2$, $j_1 \leq j_2$}. From the base point construction algorithm, since $i_1 < i_2$, the base-$i_2$'th $-$ region either lies entirely in the southwest of $FA(i_1, j_1)$, or in both the southwest and northeast regions of $FA(i_1, j_1)$. In the first case the prior endpoint of $FA(i_2, j_2)$ is determined up to equivalence. In the second case there is a choice: the prior endpoint of $FA(i_2, j_2)$ may be southwest or northeast of $FA(i_1, j_1)$. We choose southwest. See figure \ref{fig:32}(right).

Similarly, since $j_1 \leq j_2$, the root-$j_2$'th $+$ region is either identical to the root-$j_1$'th $+$ region (i.e. $j_1 = j_2$); or lies southwest of $FA(i_1, j_1)$ in region (iii) of figure \ref{fig:31}(left); or lies southwest of $FA(i_1, j_1)$ in region (iv). In the last case there is no choice; in the other cases there is a choice: we always choose southwest.

The conditions $i_1 < i_2$ and $j_1 \leq j_2$, with our ``southwest choices'', give a well-defined way to draw $FA(i_2, j_2)$ after $FA(i_1, j_1)$. These choices ensure $FA(i_2, j_2)$ lies disjoint from $FA(i_1, j_1)$, and entirely in its southwest; in fact these conditions determine the positions of $FA(i_1, j_1)$ and $FA(i_2, j_2)$ up to equivalence.

\label{well_placed_two}
Similar considerations apply to backwards generalised attaching arcs.

\label{sec_placing_nicely_multiple_gen_arcs}

Now consider multiple generalised arcs of attachment, similarly ``nicely ordered''.
\begin{defn}[Nicely ordered generalised arcs of attachment]\
\label{nicely_ordered_generalised_arcs}
\begin{enumerate}
\item
A sequence of forwards generalised attaching arcs $FA(i_1, j_1), \ldots, FA(i_m, j_m)$ on $\Gamma_w$ is \emph{nicely ordered} if
\[
 i_1 < i_2 < \cdots < i_m \quad \text{and} \quad j_1 \leq j_2 \leq \cdots \leq j_m.
\]
\item
A sequence of backwards generalised attaching arcs $BA(i_1, j_1), \ldots, BA(i_m, j_m)$ on $\Gamma_w$ is \emph{nicely ordered} if
\[
 j_m < j_{m-1} < \cdots < j_1 \quad \text{and} \quad i_m \leq i_{m-1} \leq \cdots \leq i_1.
\]
\end{enumerate}
\end{defn}

We place the generalised attaching arcs in a nicely ordered sequence by repeating the procedure for two arcs. Start from $FA(i_1, j_1)$ and proceed through to $FA(i_m, j_m)$, choosing ``southwest'' whenever we face a choice. Starting from $FA(i_1, j_1)$, place $FA(i_2, j_2)$ ``southwest'' of it, as above. Then place $FA(i_3, j_3)$ similarly ``southwest'' of $FA(i_2, j_2)$. To see that this defines $FA(i_3, j_3)$ uniquely, consider a chord $\gamma$ of $\Gamma_w$ that intersects $FA(i_3, j_3)$. The nicely ordered condition implies that if $FA(i_1, j_1)$ intersects $\gamma$, then $FA(i_2, j_2)$ does also. Thus, placing $FA(i_3, j_3)$ southwest of $FA(i_2,j_2)$ ensures it is southwest of $FA(i_1, j_1)$ also. Proceeding inductively, all generalised arcs of attachment are placed uniquely: we obtain the following lemma.
\begin{lem}[Placing multiple nicely ordered arcs]
\label{well_placed_many}
Let $FA(i_1, j_1)$, $\ldots$, $FA(i_m, j_m)$ be a nicely ordered sequence of forwards generalised arcs of attachment. They can be placed on $\Gamma_w$, so that they are disjoint, and so that for $u<v$, $FA(i_v, j_v)$ lies entirely in the southwest of $FA(i_u, j_u)$. Moreover, there is only one way to place the arcs satisfying these conditions, up to equivalence (i.e. homotopy through generalised attaching arcs).
\qed
\end{lem}
This bypass system is called the \emph{bypass system of the nicely ordered sequence} of forwards generalised arcs of attachment. For backwards arcs the procedure is similar, from $BA(i_m, j_m)$ to $BA(i_1, j_1)$, always taking the ``northwest'' choice. 
\label{bypass_system_nicely_ordered_generalised_arcs}

\subsubsection{Nicely ordered sequences of generalised moves}

\label{sec_nicely_ordered_seq_gen_el_moves}

We extend the notion of ``nicely ordered'' to generalised elementary moves. A sequence of forwards generalised elementary moves $FE(i_1, j_1), \ldots, FE(i_m, j_m)$ on $w$ is \emph{nicely ordered} if $i_1 < \cdots < i_m$ and $j_1 \leq \cdots \leq j_m$, and a sequence of backwards generalised elementary moves $BE(i_1, j_1), \ldots, BE(i_m, j_m)$ is \emph{nicely ordered} if $j_m < \cdots < j_1$ and $i_m \leq \cdots \leq i_1$.

For convenience, extend the definition of $FE(i,j)$ a little: if $FE(i,j)$ is not well-defined, then regard it as the ``null move'', with trivial effect. The following result is then clear.
\begin{lem} [Commutativity of generalised moves]
\label{generalised_moves_commute}
The generalised elementary moves of a nicely ordered sequence commute.
\qed
\end{lem}
Hence we may speak of applying a nicely ordered set of generalised forwards elementary moves to a word, without regard to their order. There are forwards and backwards versions.

\label{sec_gen_el_moves_comp_pair}

Consider two comparable words $w_1 \preceq w_2$ in $W(n_-, n_+)$. Then, for every $1 \leq i \leq n_-$, let the $i$'th $-$ sign have $\alpha_i$ $+$ signs to its left in $w_1$, and $\beta_i$ $+$ signs to its left in $w_2$; so $\alpha_i \leq \beta_i$. Then the following inequalities describe $w_1 \preceq w_2$.
\[
\begin{array}{ccccccc}
\alpha_1 & \leq & \alpha_2 & \leq & \cdots & \leq & \alpha_{n_-} \\
\begin{turn}{90} $\geq$ \end{turn} && \begin{turn}{90} $\geq$ \end{turn} &&&& \begin{turn}{90} $\geq$ \end{turn} \\
\beta_1 & \leq & \beta_2 & \leq & \cdots & \leq & \beta_{n_-}
\end{array}
\]

\begin{defn}[Elementary moves of comparable pair] 
\label{def_gen_fwd_el_moves_comp_pair}
The \emph{generalised forwards elementary moves of the pair $w_1 \preceq w_2$} are
\[
 FE(1, \beta_1), \; FE(2, \beta_2), \; \ldots, \; FE(n_-, \beta_{n_-}),
\]
where $\beta_i$ denotes the number of $+$ signs to the left of the $i$'th $-$ sign in $w_2$. 
\end{defn}
This is clearly a nicely ordered sequence; by lemma \ref{generalised_moves_commute} the moves commute. Their effect is not difficult to see.
\begin{lem}[Elementary moves between comparable words]\
\label{elementary_moves_one_to_other}
The result of applying the generalised forwards elementary moves of the pair $w_1 \preceq w_2$, to $w_1$, is $w_2$.
\qed
\end{lem}
Note that while $w_1 \preceq w_2$ means ``all $-$ signs move right'', some may not move at all; and if the $i$'th $-$ sign does not move, then $FE(i, \beta_i)$ is trivial. We could delete such moves from the sequence; but it is easier to ``take individual care'' of each symbol in this way.

There is also a backwards version: letting $\delta_j$ be the number of $-$ signs to the left of the $j$'th $+$ sign in $w_1$, the \emph{generalised backwards elementary moves of the pair $w_1 \preceq w_2$} are $BE(\delta_1, 1), \ldots, BE(\delta_{n_+}, n_+)$; the result of applying them to $w_2$ is $w_1$.

Thus we can go from $w_1$ to $w_2$ (and vice versa) by a canonical nicely ordered sequence of generalised elementary moves. It remains to do the same with bypass moves.

\subsubsection{Bypass systems of nicely ordered sequences}

\label{sec_mechanics_byp_sys_nicely_ordered}

We now generalise lemma \ref{generalised_arc_generalised_move} to multiple (nicely ordered) generalised attaching arcs and multiple (nicely ordered) generalised elementary moves. Note commutativity of bypass moves is obvious --- bypass moves on disjoint attaching arcs obviously commute --- in analogy to lemma \ref{generalised_moves_commute}.

Let $FA(i_1, j_1), FA(i_2, j_2)$ be nicely ordered forwards generalised attaching arcs on $\Gamma_w$ (placed as in lemma \ref{well_placed_many}). Write $FA_w(i_1, j_1), FA_w(i_2, j_2)$ to emphasise they refer to $\Gamma_w$. Performing upwards surgery along $FA_w(i_1, j_1)$ gives $\Gamma_{w'}$, where $w' = FE(i_1, j_1)(w)$ (lemma \ref{generalised_arc_generalised_move}). 

On $\Gamma_{w'}$, $FA_w (i_2, j_2)$ may no longer be a nontrivial generalised attaching arc; but it is \emph{equivalent} to the generalised attaching arc $FA_{w'} (i_2, j_2)$ on $\Gamma_{w'}$, as we now describe. If on $\Gamma_{w'}$ there is no forwards generalised arc $FA_{w'}(i_2, j_2)$, then by lemma \ref{existence_generalised_arcs} $j_1 = j_2$. The bypass system of $FA_{w}(i_2, j_2)$ then consists entirely of trivial attaching arcs: see figure \ref{fig:34}. Otherwise, recall the effect of bypass surgery along the bypass system of a forwards generalised attaching arc, creating a ``long chord'', closing off outermost negative regions to the southwest and positive regions to the northeast. Some of these outermost negative regions now contain segments of $FA_w(i_2, j_2)$; they may be ``pushed off'' as in figure \ref{fig:35}. Thus we homotope $FA_{w}(i_2, j_2)$ to $FA_{w'}(i_2, j_2)$: a homotopy through generalised arcs, combined with finitely many local ``pushing off'' moves as in figure \ref{fig:35}. As shown in figure \ref{fig:36}, upwards bypass surgery along the bypass system of $FA_{w} (i_2, j_2)$ or of $FA_{w'} (i_2, j_2)$ gives the result.

\begin{figure}[tbh]
\centering
\includegraphics[scale=0.3]{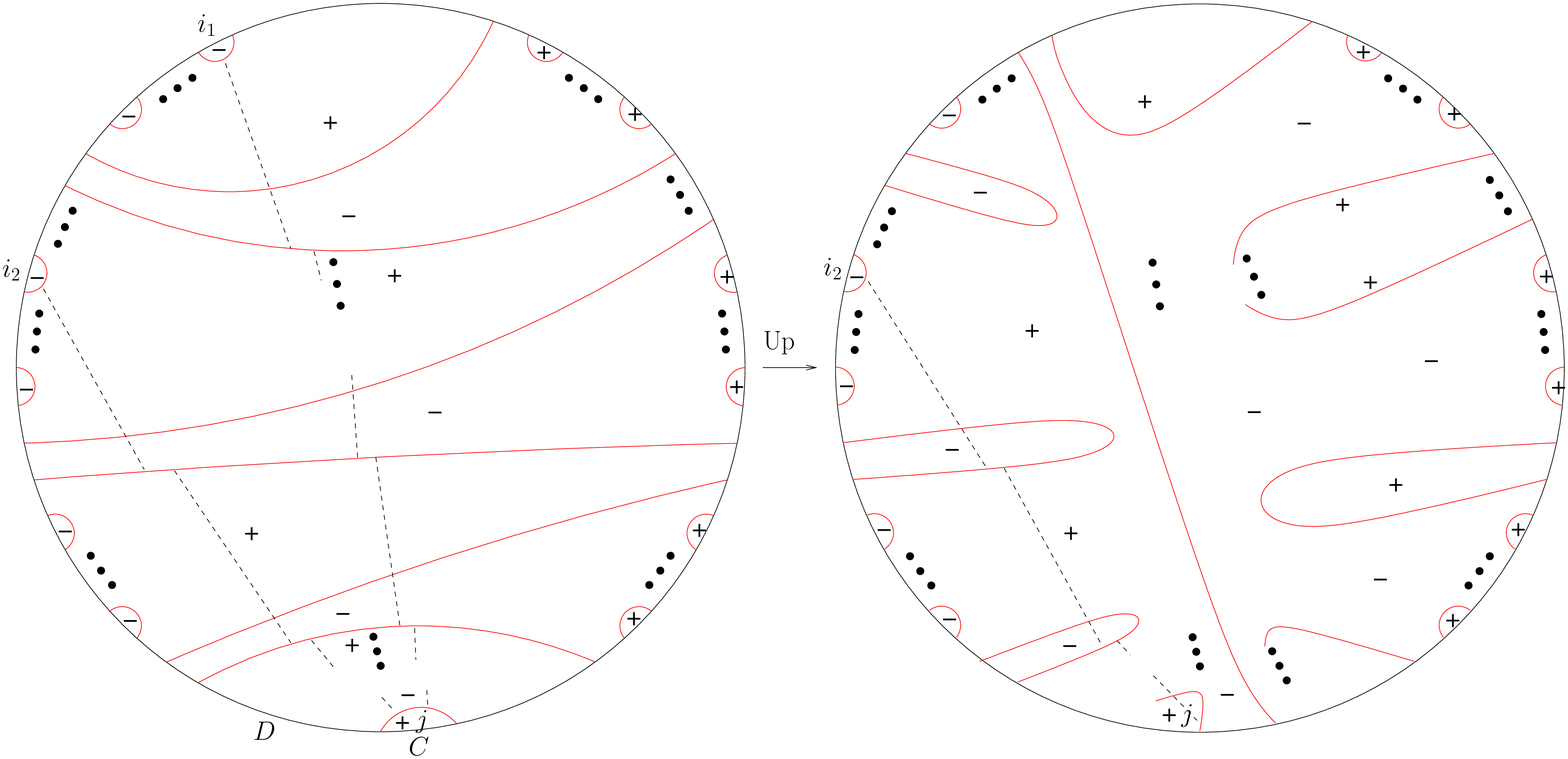}
\caption{The case $j_1 = j_2$.} \label{fig:34}
\end{figure}

\begin{figure}[tbh]
\centering
\includegraphics[scale=0.5]{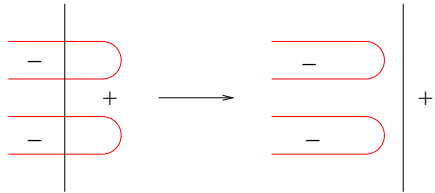}
\caption{Pushing off trivial parts of a generalised attaching arc.} \label{fig:35}
\end{figure}

\begin{figure}[tbh]
\centering
\includegraphics[scale=0.4]{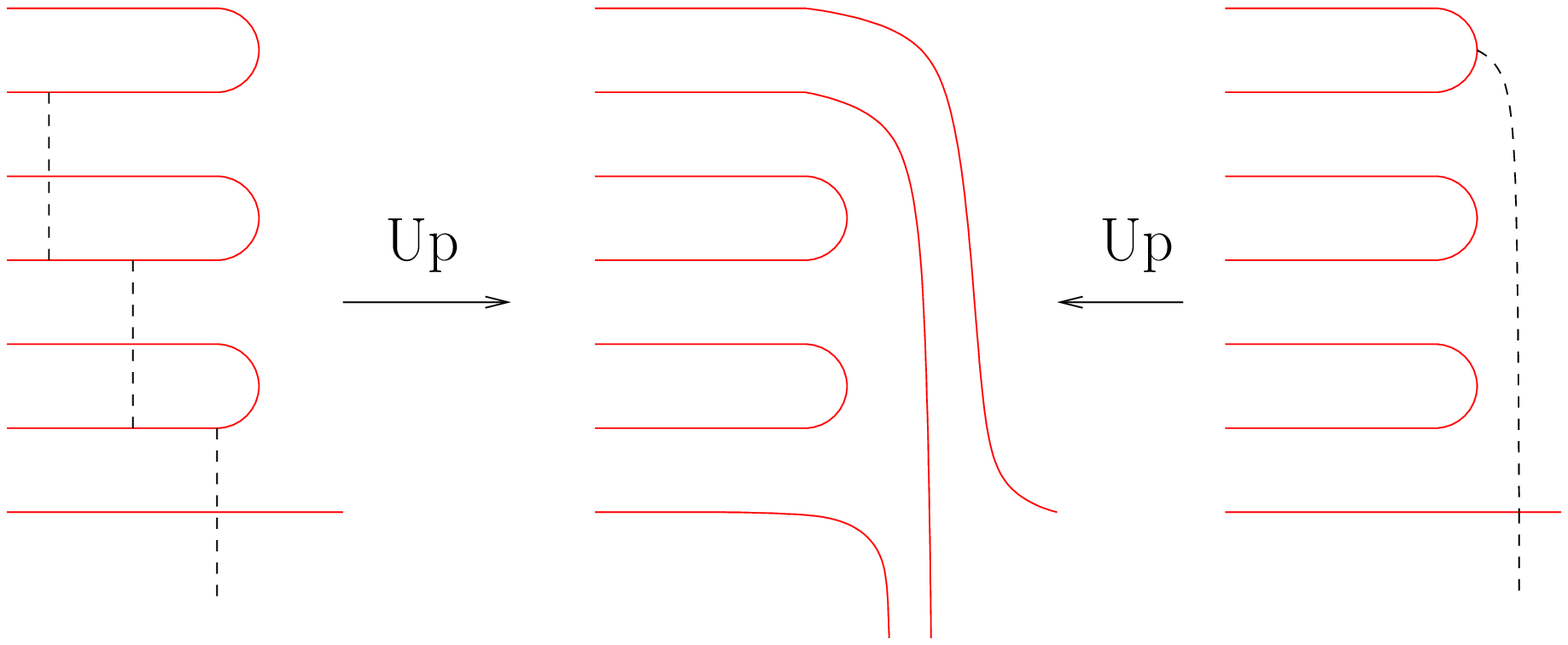}
\caption{Pushing off makes no difference to effect of bypass moves.} \label{fig:36}
\end{figure}

The same applies to a nicely ordered sequence $FA_w(i_1, j_1), \ldots, FA_w(i_m, j_m)$ of arbitrary length on $\Gamma_w$. Performing upwards bypass moves along the bypass system of $FA_w(i_1, j_1)$ yields $\Gamma_{w'}$ where $w' = FE(i_1, j_1)(w)$. On $w'$, each of $FA_w(i_2, j_2), \ldots, FA_w(i_m, j_m)$ may no longer be a nontrivial generalised attaching arc: if $FE(i_k, j_k)$ does not exist on $w'$ then the bypass system of $FA_w(i_k, j_k)$ consists entirely of trivial arcs of attachment, as discussed above; other arcs are \emph{equivalent} to the nontrivial generalised attaching arcs among $FA_{w'}(i_2, j_2), \ldots, FA_{w'} (i_m, j_m)$ on $\Gamma_{w'}$, in the following sense. As above, we may ``push them off'', but now we may have to push off several $FA_w(i_k, j_k)$ simultaneously, as in figure \ref{fig:37}. Thus we may homotope the nontrivial arcs among $FA_{w}(i_2, j_2), \ldots, FA_{w}(i_m, j_m)$ simultaneously, rel endpoints, to $FA_{w'}(i_2, j_2), \ldots, FA_{w'}(i_m, j_m)$, placed ``northeast to southwest'' as in section \ref{bypass_system_nicely_ordered_generalised_arcs}, via a homotopy through disjoint generalised arcs, combined with finitely many ``pushing off'' moves. Upwards bypass surgery on $\Gamma_{w'}$ along either $\{ FA_{w}(i_2, j_2), \ldots, FA_{w}(i_m, j_m) \}$ or $\{ FA_{w'}(i_2, j_2)$, $\ldots$, $FA_{w'}(i_m, j_m) \}$, has the same effect: see figure \ref{fig:38}.

\begin{figure}[tbh]
\centering
\includegraphics[scale=0.4]{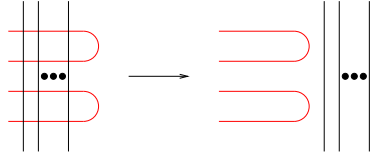}
\caption{Pushing off trivial parts of multiple generalised attaching arcs.} \label{fig:37}
\end{figure}

\begin{figure}[tbh]
\centering
\includegraphics[scale=0.4]{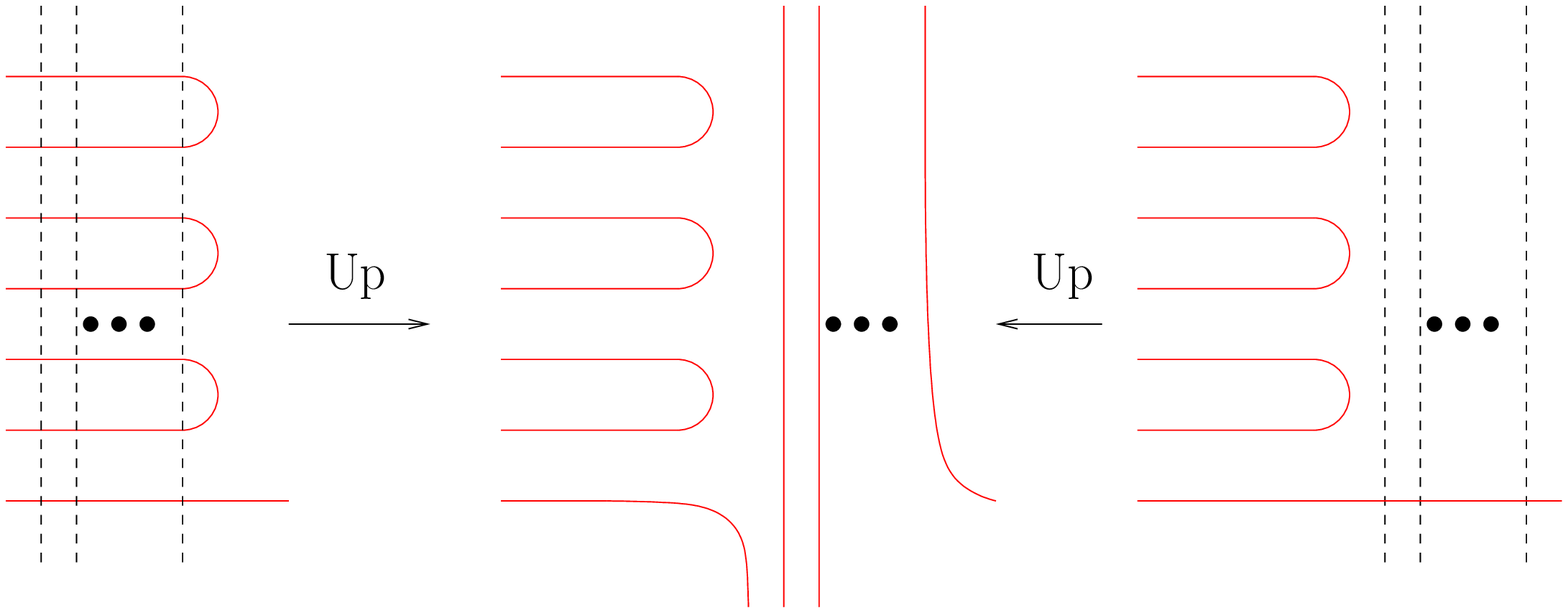}
\caption[Pushing off several generalised attaching arcs makes no difference.]{Pushing off several generalised attaching arcs makes no difference to effect of bypass moves.} \label{fig:38}
\end{figure}

This gives the complete analogy desired. There is also a backwards version.
\begin{lem}[Multiple moves \& bypass systems]
\label{multiple_generalised_elementary_moves_multiple_generalised_arcs}
Consider a nicely ordered sequence of forwards generalised elementary moves $FE(i_1, j_1), \ldots, FE(i_m, j_m)$ on $w$, and a nicely ordered sequence of forwards generalised attaching arcs $FA(i_1, j_1), \ldots, FA(i_m, j_m)$ on $\Gamma_w$. Upwards bypass surgery along the bypass system of this nicely ordered sequence of forwards generalised attaching arcs gives $\Gamma_{w'}$, where $w'$ is obtained from $w$ by performing the forwards generalised elementary moves.
\[
\xymatrix{
w \ar@{|->}[rrrr]^{FE(i_1, j_1) \circ \cdots \circ FE(i_m, j_m)} \ar@{<~>}[dd] &&&& w' \ar@{<~>}[dd] \\
\\
\Gamma_w \ar@{|->}[rrrr]^{\Up(FA (i_1, j_1), \ldots, FA(i_m,j_m))} &&&& \Gamma_{w'} }
\]
\qed
\end{lem}

\subsubsection{Bypass system of a comparable pair}

\label{sec_byp_sys_comp_pair}

We have now built so much superstructure that we can almost use it. For a pair $w_1 \preceq w_2$, we have a nicely ordered sequence of generalised elementary moves
(definition \ref{def_gen_fwd_el_moves_comp_pair}); there is a corresponding \emph{nicely ordered sequence of forwards generalised attaching arcs of the pair $w_1 \preceq w_2$} $FA(1, \beta_1), \ldots, FA(n_-, \beta_{n_-})$. Again $\beta_i$ is the number of $+$ signs to the left of the $i$'th $-$ sign in $w_2$. Similarly we have a backwards version.  If the $i$'th $-$ sign does not move (so $FE(i, \beta_i)$ is trivial), then $FA(i, \beta_i)$ is considered a null arc. From these generalised attaching arcs, we obtain a bypass system:
\begin{defn}[Coarse bypass system of pair]
The \emph{coarse forwards (resp. backwards) bypass system} $CFBS(w_1, w_2)$ (resp. $CBBS(w_1, w_2)$) of the pair $w_1 \preceq w_2$ is the bypass system of the nicely ordered sequence of forwards (resp. backwards) generalised attaching arcs of $w_1 \preceq w_2$.
\end{defn}

We call these systems ``coarse'' because they may contain redundancy, with their ``individual care'' approach. From lemma \ref{multiple_generalised_elementary_moves_multiple_generalised_arcs}, we immediately have
\begin{equation}
\label{eqn_coarse_system_effect}
\Up_{CFBS(w_1, w_2)} \Gamma_{w_1} = \Gamma_{w_2}, \quad \Down_{CBBS(w_1, w_2)} \Gamma_{w_2} = \Gamma_{w_1}.
\end{equation}

Upwards bypass surgery along the various attaching arcs of $CFBS(w_1, w_2)$ takes $\Gamma_{w_1}$ to various basis chord diagrams (proposition \ref{upwards_moves_forwards}), corresponding to words in $W(n_-, n_+)$, always moving forwards in $\preceq$. That is, we have a covariant functor
\begin{align*}
 \Up_{CFBS(w_1, w_2)} \; : \; \mathcal{P} \left( CFBS(w_1, w_2) \right) &\rightarrow W(n_-, n_+) \\
		c' 	&\mapsto	\text{word of basis diagram } \Up_{c'} (\Gamma_1). 
\end{align*}
Here the power set is partially ordered by $\subseteq$, and $W(n_-, n_+)$ by $\preceq$; recall section \ref{sec_functorial}. (In section \ref{sec_functorial} we considered a functor $\Up_c$ to a bounded contact category; $W(n_-, n_+)$ is a bounded contact category; see section \ref{sec_ct_cat_U} and proposition \ref{prop_Cb_basis_cob}.) There is also a contravariant ``down'' functor from $CBBS(w_1, w_2)$. 

Now $CFBS$ is ``coarse'' in the sense that some proper subset may map to $w_2$ under this functor. We take a \emph{minimal subsystem} of the bypass system $CFBS (w_1, w_2)$ (resp. $CBBS(w_1, w_2)$ ). By this we mean a subset of these attaching arcs, such that:
\begin{enumerate}
\item 
it contains no trivial attaching arcs;
\item
upwards (resp. downwards) bypass surgery along it gives $\Gamma_{w_2}$ (resp. $\Gamma_{w_1}$);
\item
upwards (resp. downwards) bypass surgery along along any proper subset of it does not give $\Gamma_{w_2}$ (resp. $\Gamma_{w_1}$). 
\end{enumerate}
Condition (i) may appear redundant, especially given condition (iii). In this case it is, since by definition $CFBS(w_1, w_2)$ contains no trivial attaching arcs. However, in general, performing the reverse of a ``pushing off move'' might well result in a bypass system satisfying (ii) and (iii) but not (i). Not every bypass system has a minimal sub-system, but a bypass system \emph{without trivial attaching arcs} does. We say nothing about the uniqueness of this minimal bypass system, only that one exists. Hence the indefinite article in the following.
\begin{defn}[Bypass systems of a comparable pair]
A \emph{forwards (resp. backwards) bypass system} $FBS(w_1, w_2)$ (resp. $BBS(w_1, w_2)$) of the pair $w_1 \preceq w_2$ is a minimal sub-system of $CFBS(w_1, w_2)$ (resp. $CBBS(w_1, w_2)$).
\end{defn}

At last, we prove proposition \ref{bypass_system_one_to_other}. Parts (1) and (2) are immediate from equation (\ref{eqn_coarse_system_effect}) above and the definition of minimal bypass system. Backwards is similar to forwards.

\begin{proof}[Proof of proposition \ref{bypass_system_other_way}]
Let $FBS(w_1, w_2) = \{c_1, \ldots, c_m\}$ where the $c_i$ are attaching arcs. Let $\Gamma = \Down_{FBS(w_1,w_2)} \Gamma_{w_1}$. ``Expanding down over up'' (lemma \ref{lem_expanding_down_over_up}), $\Gamma$ is a sum of $2^m$ chord diagrams, obtained by performing upwards bypass moves on each of the $2^m$ subsets of $\{c_1, \ldots, c_m\}$. Each of these $2^m$ chord diagrams is a basis chord diagram, by ``stability'' (proposition \ref{upwards_moves_forwards}). They are not necessarily distinct, but by minimality, $\Gamma_{w_1}$ and $\Gamma_{w_2}$ each appear only once. Any other basis chord diagram appearing $\Gamma_w$ is obtained from $\Gamma_{w_1}$ by upwards bypass surgery along forwards attaching arcs; further upwards surgery on it yields $\Gamma_{w_2}$; thus $w_1 \preceq w \preceq w_2$.
\end{proof}

\subsection{Contact categorical computations}
\label{sec_ct_cat_comp}

In this section, we compute the bounded contact category $\mathcal{C}^b (\Gamma_{w_0}, \Gamma_{w_1})$ for basis chord diagrams, and for bypass cobordisms.

\subsubsection{Bounded contact categories for basis chord diagrams}
\label{sec_ct_cat_U}

We prove proposition \ref{prop_Cb_basis_cob}, computing $\mathcal{C}^b (\Gamma_{w_0}, \Gamma_{w_1})$. Recall the definition of $\mathcal{U}(n_-, n_+)$ from section \ref{sec_computation_Cb}. We have some lemmata.

\begin{lem}
\label{lem_only_basis_in_universal}
If the chord diagram $\Gamma$ exists in $\mathcal{U}(n_-, n_+)$, then $\Gamma = \Gamma_w$ for some word $w \in W(n_-, n_+)$.
\end{lem}

\begin{proof}(\# 1, by combinatorial skiing)
Such a $\Gamma$ must satisfy $m(\Gamma_{(-)^{n_-} (+)^{n_+}}, \Gamma) = 1$ and $m(\Gamma, \Gamma_{(+)^{n_+} (-)^{n_-}}) = 1$. After edge rounding, both conditions are equivalent to the condition that, when $\Gamma$ is placed in figure \ref{fig:58}, the curve obtained must be connected. The disc is drawn as a rectangle, with base point at the top, root point at the bottom, and other points on the right or left. For any such $\Gamma$, there can be no nesting of arcs on either side; hence arcs from base and root points must be outermost; and every other arc must either be outermost on the left side, outermost on the right side, or run from left to right. Such a $\Gamma$ fits into figure \ref{fig:58} to form a ``slalom course'' --- think of the top of the diagram as the top of a ski slope --- successively rounding obstacles on left or right sides. The sequence in which the skier rounds obstacles on the left $(-)$ or right $(+)$ precisely gives the word of which $\Gamma$ is the basis element.
\end{proof}

\begin{figure}
\centering
\includegraphics[scale=0.3]{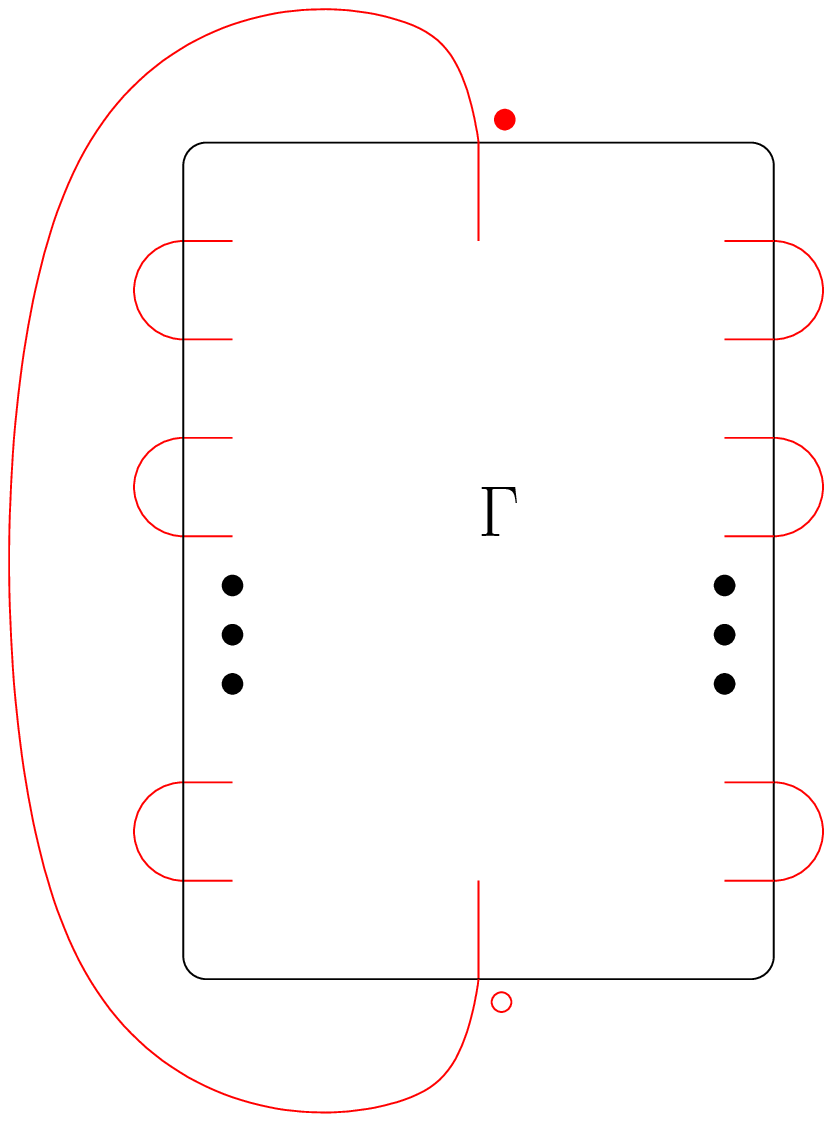}
\caption{Chord diagram ski slope.} \label{fig:58}
\end{figure}

\begin{proof}(\#2, by bypasses)
If $\Gamma$ exists in $\mathcal{U}(n_-, n_+)$, then by lemma \ref{lem_existence_of_chord_diagram} there is a sequence of upwards bypasses from $\Gamma_{(-)^{n_-} (+)^{n_+}} = G_0$, through $G_1, \ldots, G_k$ to $G_k = \Gamma$, where each $G_i$ satisfies $m(G_i, \Gamma_{(+)^{n_+} (-)^{n_-}}) = 1$. As every word is $\preceq (+)^{n_+} (-)^{n_-}$, each $G_i$ must have an odd number of basis elements in its decomposition.

From section \ref{sec_single_bypass_el_moves}, we know all bypass moves on basis chord diagrams. An upwards bypass move along a forwards attaching arc gives another basis diagram; and along a backwards attaching arc, gives a chord diagram which is a sum of two basis diagrams. Of these only the former satisfies $m( \cdot, \Gamma_{(+)^{n_+} (-)^{n_-}} ) = 1$; any bypass attachment upwards must move from a basis diagram to another basis diagram. Hence every diagram existing in $\mathcal{U}(n_-, n_+)$ is a basis diagram.
\end{proof}

Out of this, we serendipitously obtain proposition \ref{even_number_decomposition}: every non-basis chord diagram has an even number of elements in its basis decomposition. We also prove this directly in section \ref{sec_how_many_basis_elts_directly}.
\begin{proof}[Proof of proposition \ref{even_number_decomposition}]
Returning to the ski slope, a non-basis chord diagram $\Gamma$ satisfies $m(\Gamma, \Gamma_{(+)^{n_+} (-)^{n_-}}) = 0$.
\end{proof}

\begin{lem}
\label{lem_CbU}
$\mathcal{C}^b \left( \mathcal{U}(n_-, n_+) \right) \cong W(n_-, n_+).$
\end{lem}

\begin{proof}
Lemma \ref{lem_only_basis_in_universal} shows $\Ob  \mathcal{C}^b(\mathcal{U}(n_-, n_+)) \subseteq \Ob W(n_-, n_+)$. To see $\Ob  \mathcal{C}^b(\mathcal{U}(n_-, n_+)) \supseteq \Ob W(n_-, n_+)$, i.e. every basis chord diagram $\Gamma_w$, $w \in W(n_-, n_+)$, exists in $\mathcal{U}(n_-, n_+)$: take the bypass system $FBS( (-)^{n_-} (+)^{n_+}, w )$ on $\Gamma_{(-)^{n_-} (+)^{n_+}}$. Upwards bypass surgery gives $\Gamma_w$. Each successive bypass attachment gives a basis diagram $\Gamma_{w'}$, where $w' \preceq w$ and $m(\Gamma_{w'}, \Gamma_{(+)^{n_+} (-)^{n_-}}) = 1$, hence (lemma \ref{lem_existence_of_chord_diagram}) $\Gamma_{w'}$ exists in $\mathcal{U}(n_-, n_+)$.

A similar argument shows that for any $w \preceq w'$ in $W(n_-, n_+)$, the cobordism $\M(\Gamma_w, \Gamma_{w'})$ exists in $\mathcal{U}(n_-, n_+)$: this is $\Mor W(n_-, n_+) \subseteq \Mor \mathcal{C}^b(\mathcal{U}(n_-, n_+))$. To see $\Mor W(n_-, n_+) \supseteq \Mor \mathcal{C}^b(\mathcal{U}(n_-, n_+))$, note that for any $w \npreceq w'$, $\M(\Gamma_w, \Gamma_{w'})$ is not tight, hence cannot exist in the tight $\mathcal{U}(n_-, n_+)$ Morphism compositions clearly agree.
\end{proof}

\begin{proof}[Proof of proposition \ref{prop_Cb_basis_cob}]
The cobordism $\M(\Gamma_{w_0}, \Gamma_{w_1})$ with tight contact structure exists in $\mathcal{U}(n_-, n_+)$ (lemma \ref{lem_CbU}). Hence its bounded contact category is the full sub-category on those objects $\Gamma_w$ with $w_0 \preceq w \preceq w_1$.
\end{proof}

\subsubsection{Bounded contact category of a bypass cobordism}

\label{sec_Cb_bypass_cob}

We now compute $\mathcal{C}^b (\Gamma_0, \Gamma_1)$ for any $\Gamma_0, \Gamma_1$ where $\Gamma_1 = \Up_c \Gamma_0$ for an attaching arc $c$. We prove $\mathcal{C}^b (\Gamma_0, \Gamma_1) \cong W(n_-, n_+)$ for some $n_-, n_+$ that we now describe; $n_-, n_+$ are essentially largest possible so that $\mathcal{U}(n_-, n_+)$ embeds into $\M(\Gamma_0, \Gamma_1)$.

\begin{figure}[tbh]
\centering
\includegraphics[scale=0.35]{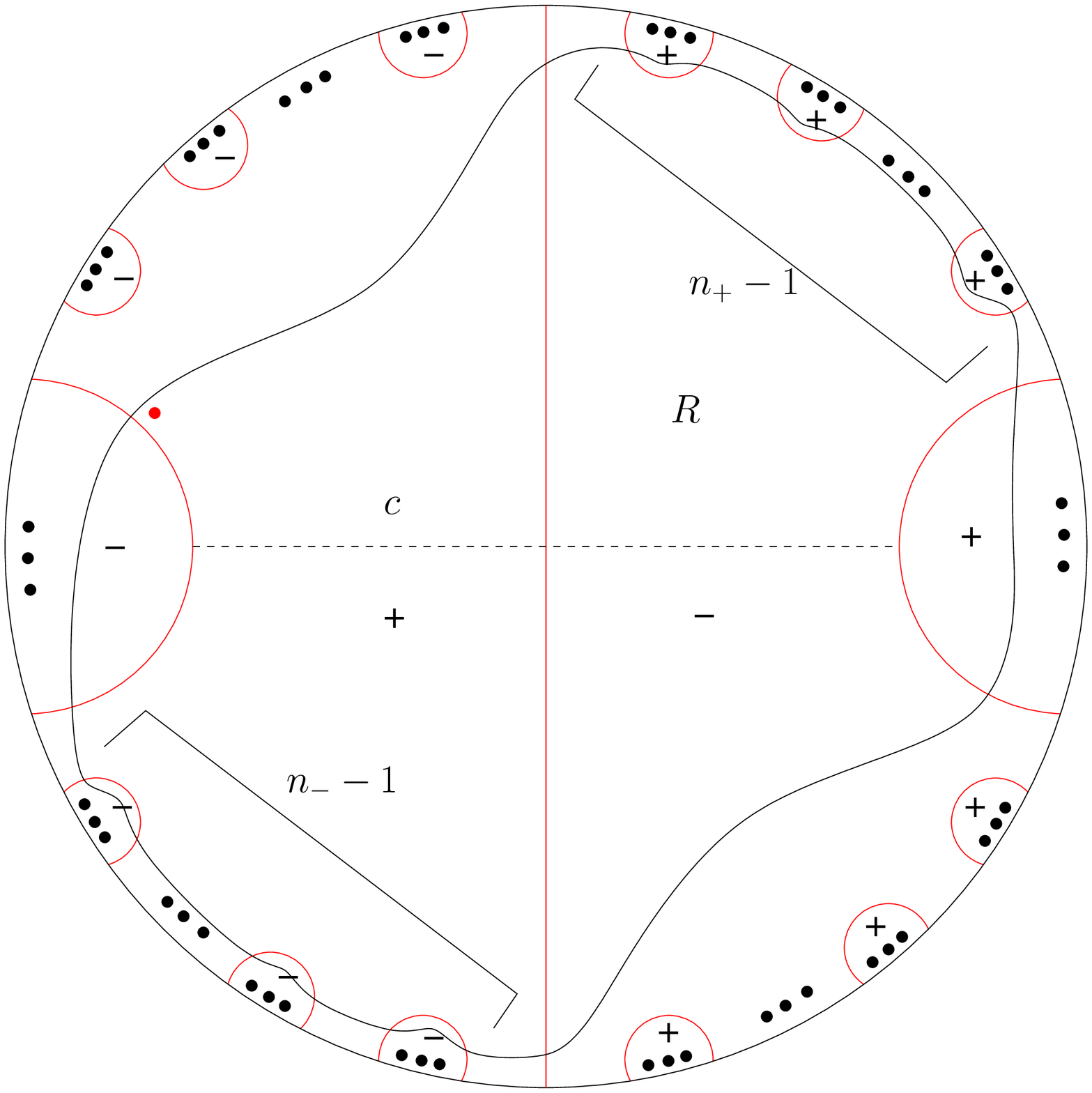}
\caption[A general bypass attachment.]{A general bypass attachment. Possible locations of other chords are denoted by black dots ($\ldots$). The region $R$ is enclosed by the black curve.} \label{fig:59}
\end{figure}

Figure \ref{fig:59} depicts a general bypass attachment on a chord diagram $\Gamma$ along an attaching arc $c$. For either of $c$'s inner regions (recall section \ref{sec_anatomy_attaching_arcs}), its boundary consists of several arcs: some on the boundary of the disc; two chords of $\Gamma$ which intersect $c$; and several other chords of $\Gamma$. Of these other chords, let the number bounding the inner $+$ region which lie anticlockwise of the outer $-$ region and clockwise of the inner $-$ region be $n_- - 1$; and let the number bounding the inner $-$ region which lie anticlockwise of the outer $+$ region and clockwise of the inner $+$ region be $n_+ - 1$.

\begin{thm}[Bounded contact category of bypass cobordism]
\label{thm_Cb_bypass_cob}
Let $\Gamma_1 = \Up_c \Gamma_0$ for an attaching arc $c$. Let $n_-, n_+$ be defined as above. Then $\mathcal{C}^b (\Gamma_0, \Gamma_1) \cong W(n_-, n_+)$.
\end{thm}

In figure \ref{fig:59}, the region $R$ contains sub-arcs of $n_- + n_+ + 1$ chords from $\Gamma_0$: three chords intersecting $c$; and arcs of the $(n_- - 1) + (n_+ - 1)$ other chords described above bounding the inner regions of $c$. With base point as shown, $R$ encloses the chord diagram $\Gamma_{(-)^{n_-} (+)^{n_+}}$; and on $\Gamma_1$, $R$ encloses $\Gamma_{(+)^{n_+} (-)^{n_-}}$. We roughly think of $R \times I$ as an ``embedded $\mathcal{U}(n_-, n_+)$'' in $\M(\Gamma_0, \Gamma_1)$.

The point is that the bounded contact category of $\M(\Gamma_0, \Gamma_1)$ is precisely that of $\mathcal{U}(n_-, n_+)$; all nontrivial bypasses attached in $\M(\Gamma_0, \Gamma_1)$ upwards from $\Gamma_0$ can be attached within $R$; even after some bypasses are attached, any existing bypass can be attached along $R$. Recalling lemma \ref{lem_existence_of_chord_diagram}, the theorem is the consequence of the following two lemmas.

For $w \in W(n_-, n_+)$, let $G_w$ denote the chord diagram which consists of taking $\Gamma_0$, and within $R$, replacing $\Gamma_{(-)^{n_-} (+)^{n_+}}$ with $\Gamma_w$.
\begin{lem}
Let $\M(\Gamma_0, \Gamma_1)$ be a bypass cobordism as above. For any word $w$ in $W(n_-, n_+)$, the chord diagram $G_w$ exists in $\M(\Gamma_0, \Gamma_1)$.
\end{lem}

\begin{proof} 
By section \ref{sec_ct_cat_U}, $\Gamma_w$ exists in $\mathcal{U}(n_-, n_+)$; the same bypasses which on $\Gamma_{(-)^{n_-} (+)^{n_+}}$ give $\Gamma_w$, added along $R$ give $G_w$; at each stage, the bypasses exist inside $\M(\Gamma_0, \Gamma_1)$ since the extra arcs outside $R$ can be removed using lemma \ref{cancel_outermost}.
\end{proof}

\begin{lem}
Consider the cobordism $\M(G_w, \Gamma_1)$ within $\M(\Gamma_0, \Gamma_1)$. Let $c$ be a nontrivial attaching arc in $G_w$, such that a bypass exists upwards along $c$ in the tight $\M(G_w, \Gamma_1)$. Then $c$ is isotopic to an attaching arc lying entirely in the region $R$, and $\Up_c G_w = G_{w'}$ for some $w'$ where $w \preceq w'$.
\end{lem}

\begin{proof}
We consider all the possible locations of the nontrivial attaching arc $c$. 

First suppose $c$ is equivalent to an attaching arc lying entirely in $R$. Then, applying lemma \ref{cancel_outermost}, $c$ can be taken as an attaching arc in $\Gamma_w$ within $\mathcal{U}(n_-, n_+)$. Hence an upwards bypass exists iff $c$ is forwards; the result is $G_{w'}$, where $w \preceq w'$.

Thus we may assume $c$ is not isotopic to an attaching arc in $R$. Hence $c$ intersects chords of $G_w$ not entering $R$; indeed has an endpoint on a chord not entering $R$.

\begin{figure}
\centering
\includegraphics[scale=0.35]{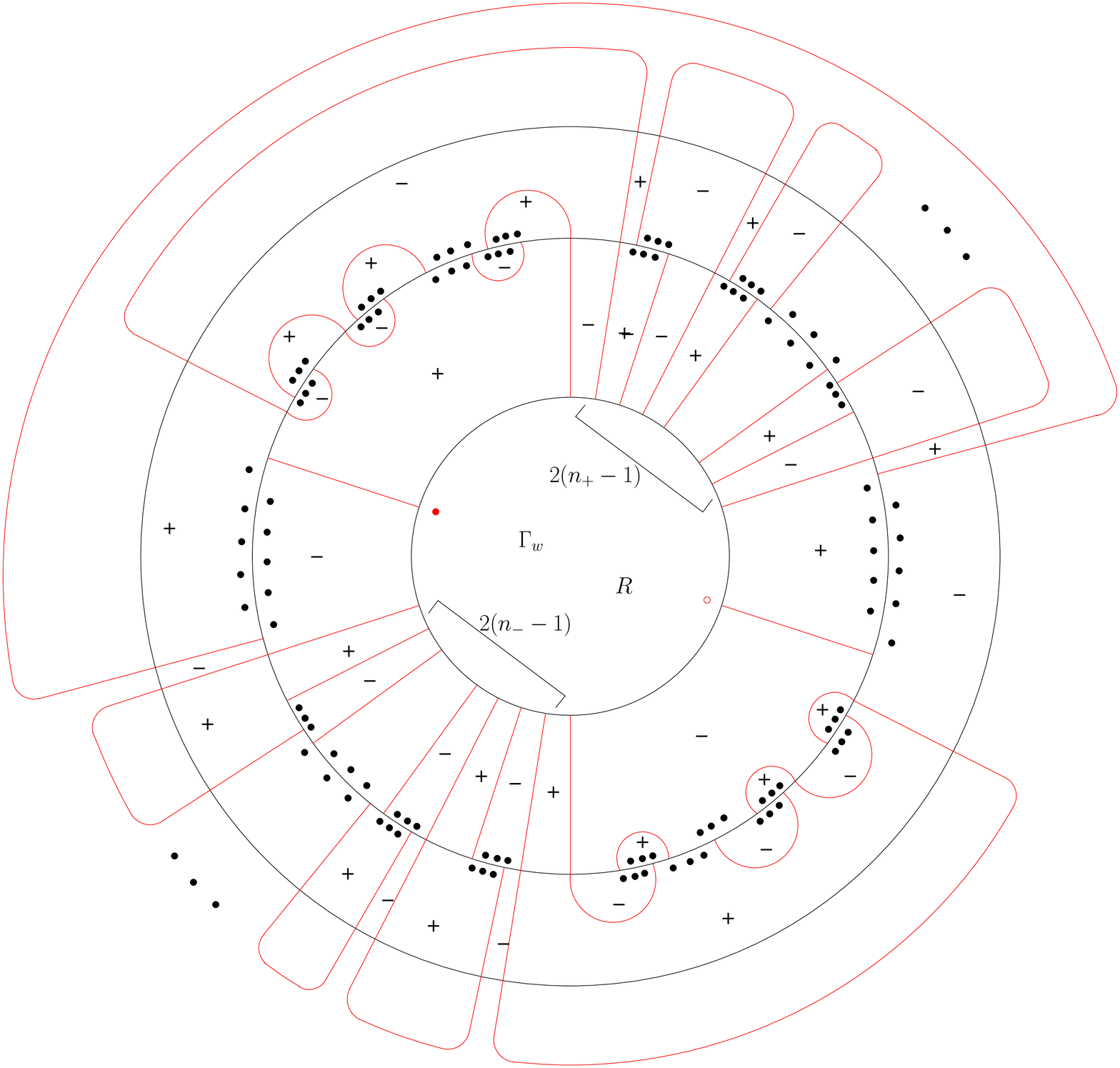}
\caption[General dividing set on $\M(G_w, \Gamma_1)$ within a general bypass cobordism.]{Dividing set on $\M(G_w, \Gamma_1)$ within a general bypass cobordism. Each set of black dots ($\ldots$) represents extra chords; corresponding sets of black dots contain copies of the same arrangements of chords. The four concentric regions are respectively (from inside to out): the disc with dividing set $\Gamma_w$; the two annuli with identical (but relatively shifted) dividing sets; and then the disc with $\Gamma_{(+)^{n_+} (-)^{n_-}}$ (although in ``flipping'' this disc to draw it in our diagram, the signs of regions are reversed).
} \label{fig:60}
\end{figure}

Consider now the arrangement of dividing curves on the whole boundary $S^2$ of $\M(G_w, \Gamma_1)$. We regard this $S^2$ as consisting of four regions: two discs arising from the region $R$ (containing $\Gamma_w$ and $\Gamma_{(+)^{n_+} (-)^{n_-}}$ as dividing sets respectively) separated by two annuli containing identical dividing sets; although in rounding corners, the two identical dividing sets on the annuli meet each other relatively shifted by one marked point. Taking figure \ref{fig:59} and drawing $\Gamma_1$ ``on the outside'', we obtain the dividing set on $S^2$: see figure \ref{fig:60}.

Suppose the middle intersection point of $c$ lies on a chord not entering $R$; hence half of $c$ can be isotoped to lie entirely outside $R$. Observe from the arrangement of figure \ref{fig:60}, arising from the clockwise rotation (as depicted in the diagram) of $\Gamma_1$ relative to $G_w$, that we may slide an endpoint of $c$ along the dividing set on the sphere $S^2$, until it approaches the middle intersection point of $c$, and the result is as in figure \ref{fig:61}. Upwards bypass surgery along $c$ would disconnect the dividing set, hence no such bypass exists.

\begin{figure}[tbh]
\centering
\includegraphics[scale=0.25]{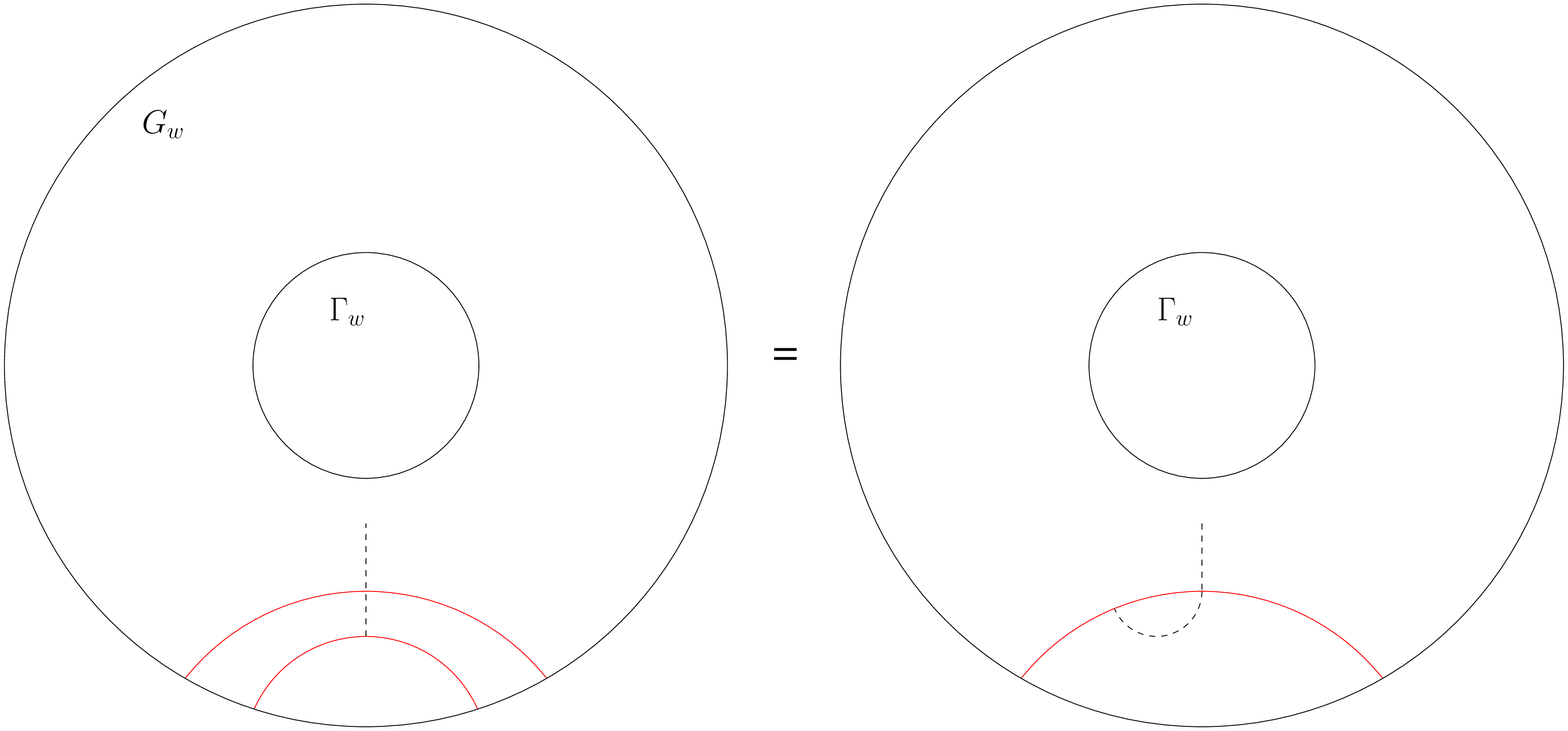}
\caption{Middle intersection point of $c$ lies outside $R$.} \label{fig:61}
\end{figure}

Thus we may assume the middle intersection point of $c$ lies on a chord entering $R$. There are two cases: (i) both endpoints of $c$ lie on chords outside $R$, or (ii) precisely one endpoint lies on a chord which enters $R$. See figure \ref{fig:62}. 

\begin{figure}[tbh]
\centering
\includegraphics[scale=0.25]{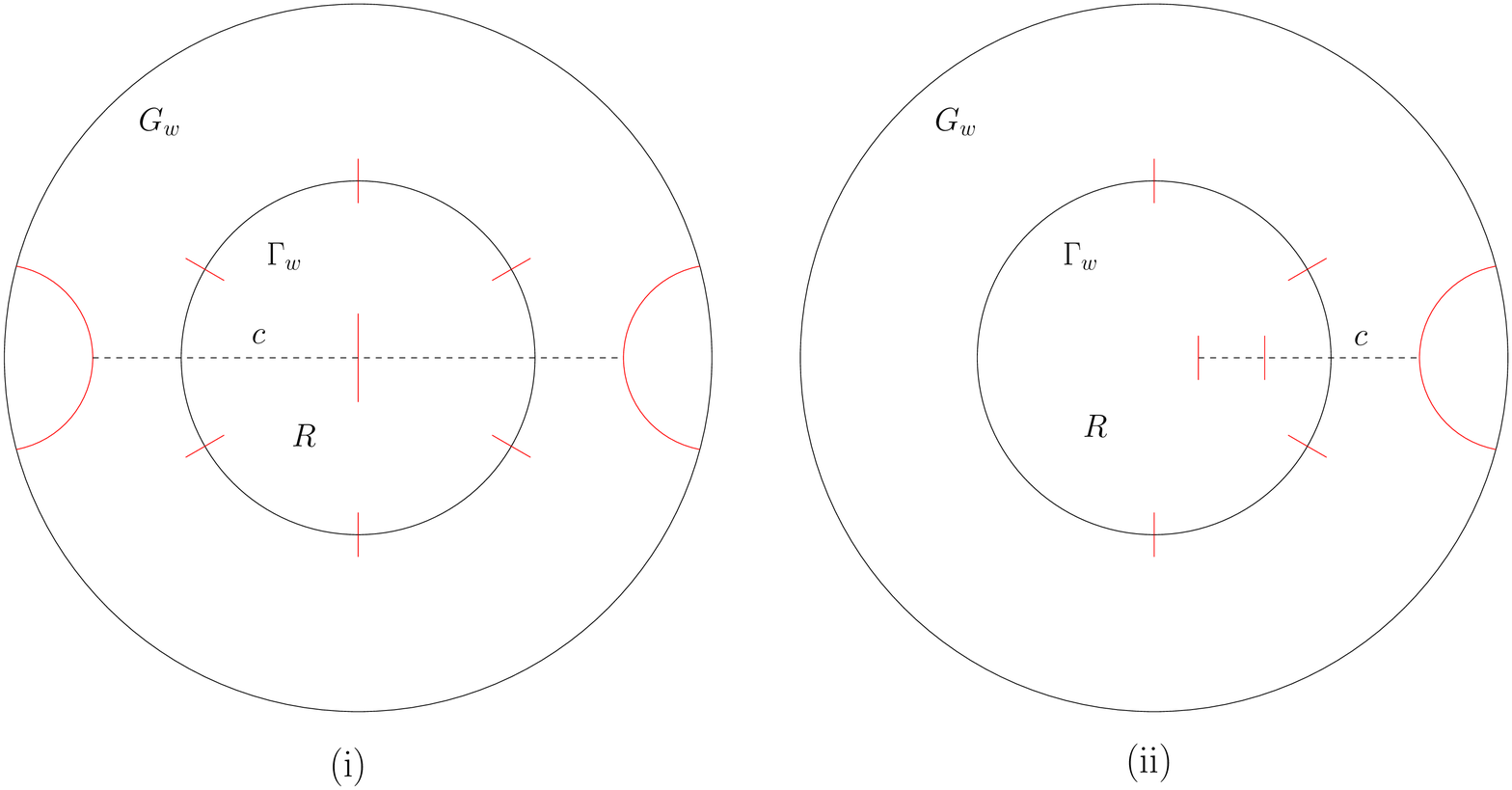}
\caption{Middle intersection point of $c$ lies inside $R$: cases (i) and (ii).} \label{fig:62}
\end{figure}

Consider again figure \ref{fig:60}. If $c$ exits $R$ and then intersects a chord of $G_w$ outside $R$, then $c$ exits $R$ either through a positive region on the eastside of $\Gamma_w$, or through a negative region on the westside of $\Gamma_w$. In case (i), therefore, $c$ exits $R$ at one end through a positive region $p$ on the eastside; and at the other end through a negative region $n$ on the westside. Again slide endpoints of $c$ along the dividing set on $S^2$, until they lie in $R$; the result is figure \ref{fig:63}(a). If either $p$ or $n$ is an outermost region (i.e. enclosed by an outermost chord) in $\Gamma_w$, then it is clear upwards bypass surgery along $c$ creates a dividing set with multiple components. Thus we may assume neither $p$ nor $n$ is outermost. Hence the middle intersection point of $c$ lies on a non-outermost chord $\gamma$ of $\Gamma_w$. Since $\Gamma_w$ is a basis diagram, $\gamma$ runs from the eastside to the westside; there are two possibilities, \ref{fig:63}(b) and (c). In \ref{fig:63}(b), upwards surgery along $c$ disconnects the dividing set. In \ref{fig:63}(c), $c$ has become a backwards attaching arc on $\Gamma_w$; and hence no bypass exists above it in $\M(\Gamma_w, \Gamma_{(+)^{n_+} (-)^{n_-}})$, or equivalently, in $\M(G_w, \Gamma_1)$.

\begin{figure}[tbh]
\centering
\includegraphics[scale=0.25]{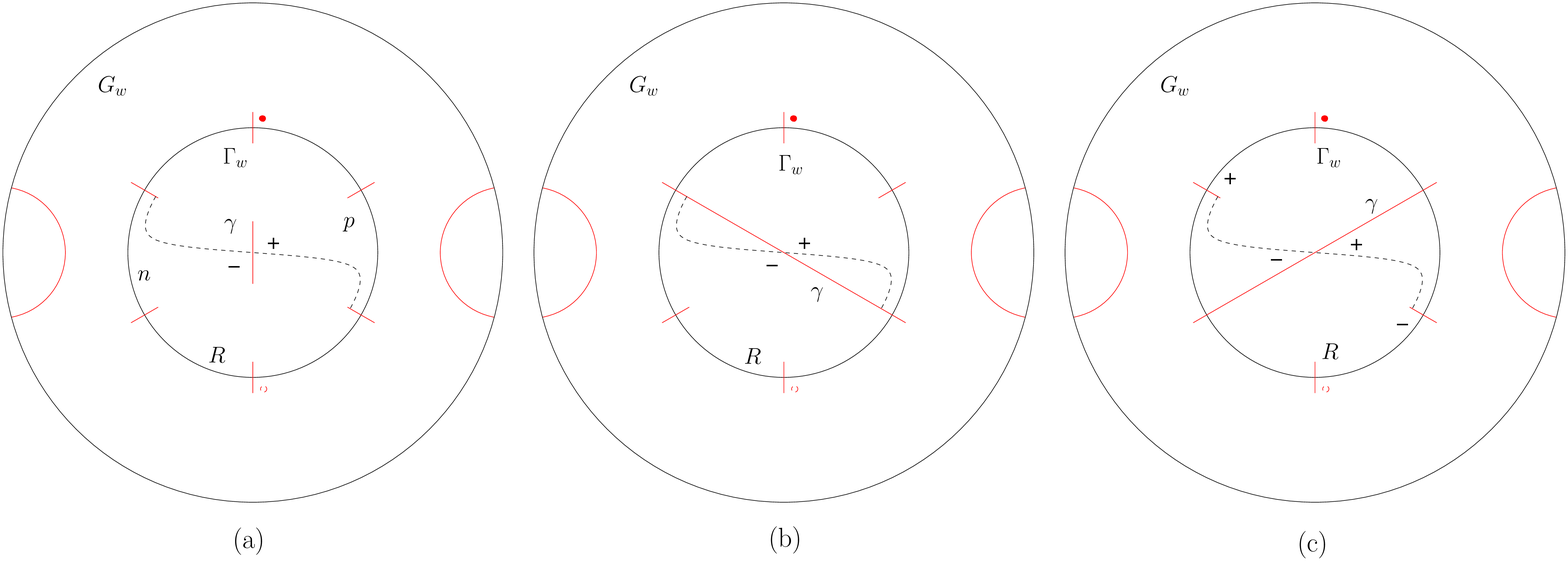}
\caption{Arrangements in case (i).} \label{fig:63}
\end{figure}

In case (ii), we assume $c$ exits $R$ through a positive region $p$ on the eastside; exiting through a negative westside region is similar. Again slide the endpoint of $c$ outside $R$ along the dividing set on the sphere $S^2$ until it lies in $R$; the result is figure \ref{fig:64}(a). If $p$ is outermost in $\Gamma_w$, then an upwards bypass move along $c$ disconnects the dividing set. So the two chords $\gamma_1, \gamma_2$ of $\Gamma_w$ adjacent to the exit point of $c$ are non-outermost, and proceed to the westside. The region $p$ may have several components of $\Gamma_w$ on its boundary; since $\Gamma_w$ is a basis chord diagram, in addition to $\gamma_1, \gamma_2$ on the boundary of $p$, there may also be outermost chords of $\Gamma_w$ on the westside. However if $c$ intersects any of these then its final intersection point must also lie on the same chord, contradicting nontriviality. Thus the middle intersection point of $c$ lies on $\gamma_1$ or $\gamma_2$; the two possibilities are depicted in figure \ref{fig:64}(b); $c$ has become either trivial or backwards; in neither case can a bypass exist above it.
\end{proof}

\begin{figure}[tbh]
\centering
\includegraphics[scale=0.25]{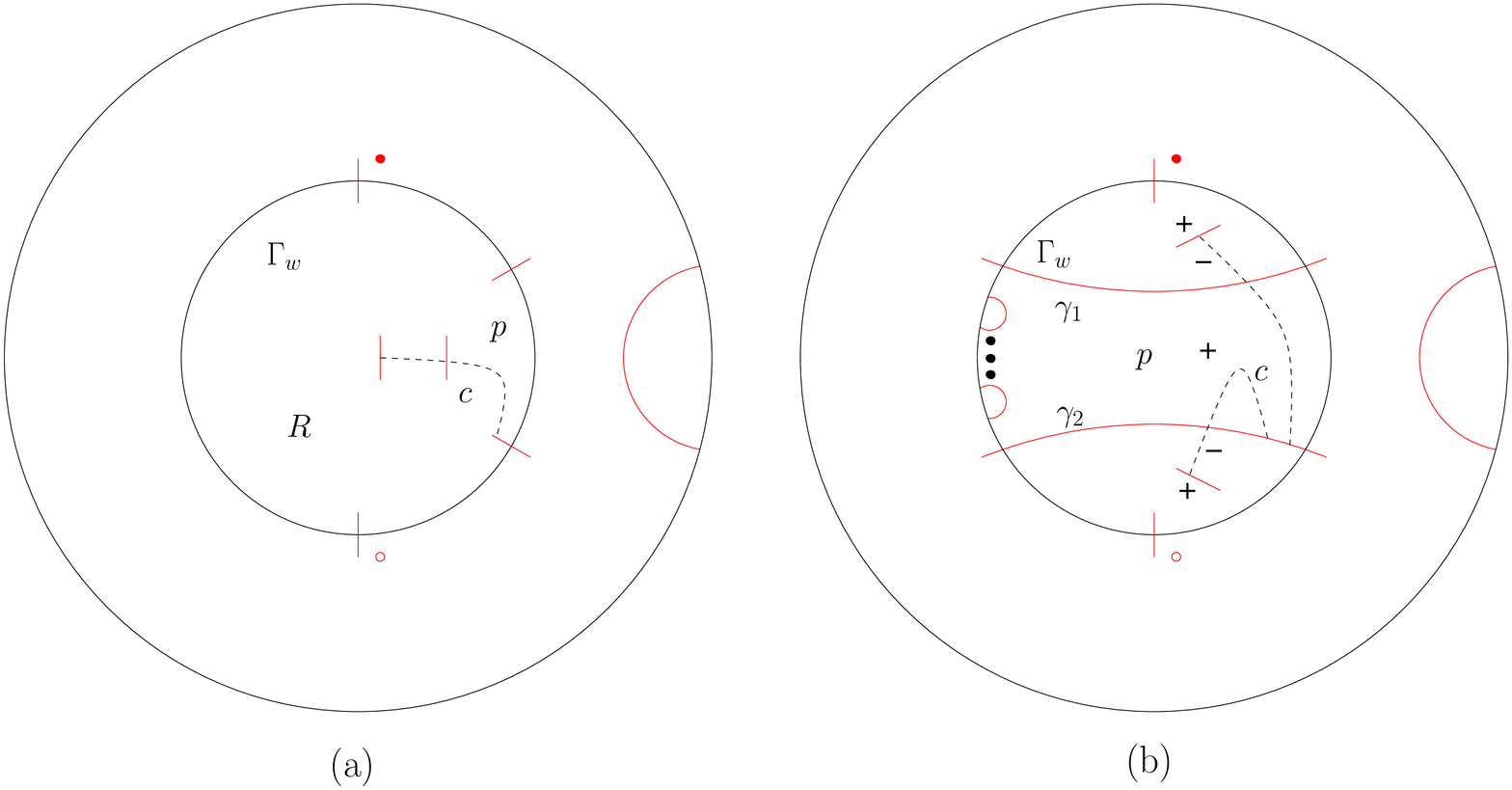}
\caption{Arrangements in case (ii).} \label{fig:64}
\end{figure}

\section{Main results and consequences}

\label{ch_ct_elts}

\subsection{Proof of main theorem}
\label{sec_main_proof}

We prove our main theorem \ref{main_theorem}: the bijection between chord diagrams and pairs of comparable words, taking a chord diagram to the first and last elements $w_-, w_+$ occurring in its decomposition; for every $w$ occurring, $w_- \preceq w \preceq w_+$.

Proposition \ref{bypass_system_other_way} constructs, for any $w_1 \preceq w_2$, a chord diagram whose basis decomposition contains $w_1$, $w_2$, and other words $w$ such that $w_1 \preceq w \preceq w_2$. This map
\[
 \left\{ \begin{array}{c} \text{Comparable pairs of} \\ \text{words $w_1 \preceq w_2$} \end{array} \right\} \To \left\{ \begin{array}{c} \text{Chord} \\ \text{Diagrams} \end{array} \right\}
\]
is clearly injective. By proposition \ref{enumerative_bijection} these two sets have the same cardinality. Thus we have the desired bijection. In particular, every chord diagram is of the form described in proposition \ref{bypass_system_other_way}. This proves theorem \ref{main_theorem}.

\subsection{Consequences of main results}

\label{sec_conseq_main_results}

\subsubsection{Up and Down}

\label{sec_up_down}

The idea of proposition \ref{bypass_system_other_way} gives the following corollary, which we need later.
\begin{cor}[Upwards vs. downwards bypass moves]\
\label{general_bypass_system_other_way}
Suppose there is a minimal bypass system $B$ on $\Gamma_{w_1}$ (resp. $\Gamma_{w_2}$) such that attaching bypasses above $\Gamma_{w_1}$ (resp. below $\Gamma_{w_2}$) along $B$ gives a tight $\M(\Gamma_{w_1}, \Gamma_{w_2})$. Minimality here means $B$ has no trivial attaching arcs, and no proper subset $B' \subset B$ satisfies $\Up_{B'}(\Gamma_{w_1}) = \Gamma_{w_2}$ (resp. $\Down_{B'}(\Gamma_{w_2}) = \Gamma_{w_1}$). 

Then $\Down_B (\Gamma_{w_1})$ (resp. $\Up_B (\Gamma_{w_2})$) is $[\Gamma_{w_1}, \Gamma_{w_2}]$.
\end{cor}

\begin{proof}
We prove the upwards case; downwards is similar. From tightness $w_1 \preceq w_2$. By proposition \ref{prop_Cb_basis_cob}, all chord diagrams in $\M(\Gamma_{w_1}, \Gamma_{w_2})$ are basis diagrams; so every attaching arc of $B$ is forwards. Expanding $\Down_B \Gamma_{w_1}$ over subsets of upwards bypasses (lemma \ref{lem_expanding_down_over_up}) and using stability proposition \ref{upwards_moves_forwards}, we have a sum of basis chord diagrams, where $w_1$ is the minimal element occurring in this sum and $w_2$ the maximum; and $w_1$ and $w_2$ occur only once by minimality, hence do not cancel.
\end{proof}

\subsubsection{Generalised bypass triple $\Gamma_-$, $\Gamma_+$, $[\Gamma_-, \Gamma_+]$}

\label{sec_ct_interp_gamma+-}
\label{sec_generalised_byp_triple_GammaGammaGamma}

We now consider further the chord diagrams $\Gamma_-$, $\Gamma_+$ and $\Gamma = [\Gamma_-, \Gamma_+]$.  By the constructions of the main theorem, these chord diagrams form a generalised bypass triple, analogous to an ``exact triangle'' in the contact category (sections \ref{sec_elementary_cobordisms}, \ref{sec_other_cat_str}). There is a bypass system $c_- = FBS(w_-, w_+)$ on  $\Gamma_-$ such that $\Up_{c_-} (\Gamma_-) = \Gamma_+$ and $\Down_{c_-} (\Gamma_-) = \Gamma$; and as bypass moves are local $60^\circ$ rotations, there are corresponding bypass systems $c_+$ on $\Gamma_+$ and $c$ on $[\Gamma_-, \Gamma_+]$. Similarly, there is $d_+ = BBS(w_-, w_+)$ on $\Gamma_+$, and corresponding $d_-$ on $\Gamma_-$ and $d$ on $\Gamma$. Note $c_-$ and $d_+$ are minimal, but the other bypass systems need not be.
\[
\xymatrix{
&& \Gamma \ar@/^/[ddrr]^{\Down(c)} \ar@/^/[ddll]^{\Up(c)} && 		&&&& \Gamma \ar@/^/[ddrr]^{\Down(d)} \ar@/^/[ddll]^{\Up(d)} \\
\\
\Gamma_-  \ar@/^/[uurr]^{\Down(c_-)}  \ar@/^/[rrrr]^{\Up(c_-)} &&&& \Gamma_+ \ar@/^/[uull]^{\Up(c_+)} \ar@/^/[llll]^{\Down(c_+)} && \Gamma_-  \ar@/^/[uurr]^{\Down(d_-)}  \ar@/^/[rrrr]^{\Up(d_-)} &&&& \Gamma_+ \ar@/^/[uull]^{\Up(d_+)} \ar@/^/[llll]^{\Down(d_+)} 
}
\]
\begin{prop}[Contact generalised bypass triple] 
\label{MGamma+-}
The contact structure on:
\begin{enumerate}
\item
$\M(\Gamma, \Gamma_-)$ obtained from performing upwards bypass attachments on $c$ or downwards bypass attachments on $c_-$ is tight.
\item
$\M(\Gamma_-, \Gamma_+)$ obtained from performing upwards bypass attachments on $c_-$ or $d_-$ or downwards bypass attachments on $c_+$ or $d_+$ is tight.
\item
$\M(\Gamma_+, \Gamma)$ obtained from performing upwards bypass attachments on $d_+$ or downwards bypass attachments on $d$ is tight.
\end{enumerate}
\end{prop}
Note that our computation of $\mathcal{C}^b(\Gamma_-, \Gamma_+)$ implies that the contact structure on $\M(\Gamma_-, \Gamma_+)$ arising from $c_-$ or $d_+$ is tight; we give an independent proof using pinwheels.
\begin{proof}
We show that $c_- = FBS(w_-, w_+)$ has no pinwheels, upwards or downwards. Suppose there is an upwards pinwheel $P$; downwards is similar. Recall (definition \ref{def_pinwheel}) the boundary of $P$ consists of arcs $\alpha_i$ and $\gamma_i$, where each $\gamma_i$ runs along $\Gamma_- = \Gamma_{w_-}$, and each $\alpha_i$ runs along the attaching arcs of $c_-$. Let $\gamma_i$ be the base-$s_i$'th chord (definition \ref{def_base_numbering}). Each attaching arc of $FBS(w_-,w_+)$, being forwards, has negative prior outer region. If $P$ is a negative region, this implies $s_k < s_{k-1} < \cdots < s_1 < s_k$, a contradiction. Similarly, if $P$ is positive, $s_1 < \cdots < s_k < s_1$.

So $c_-$ is pinwheel-free; performing upwards (resp. downwards) bypass attachments along $c_-$ gives a tight contact structure on $\M(\Gamma_-, \Gamma_+)$ (resp. $\M(\Gamma, \Gamma_-)$). Since $c_+$ corresponds to $c_-$, it gives the same tight contact structure on $\M(\Gamma_-, \Gamma_+)$. Similar arguments apply to the remaining desired tight contact structures.
\end{proof}

The weaker result that $m(\Gamma, \Gamma_-) = m(\Gamma_+, \Gamma) = m(\Gamma_-, \Gamma_+) = 1$ can be proved by other means. For instance, simply expanding out over basis elements. More geometrically, consider $\M(\Gamma, \Gamma_-)$. We may round and un-round corners; each chord created in the base point construction algorithm for $\Gamma_-$ can be successively pushed down the cylinder into the bottom disc, where it simplifies $\Gamma$ to the chord diagram on the unused disc of some $\Gamma_{w \cdot}$. In this way we see $m(\Gamma, \Gamma_-) = 1$ directly; similarly for $m(\Gamma_+, \Gamma)$.

\subsubsection{Categorical meaning of main theorem}

\label{sec_categorical_meaning}

We now interpret the main theorem and above remarks in the language of the contact category, proving the results announced in section \ref{sec_intro_categorical_meaning}.

Proposition \ref{prop_tight_basis_cob_elementary}, that tight basis cobordisms are elementary, now follows immediately from proposition \ref{MGamma+-}.

We regard the following as an ``exact triangle'' in $\mathcal{C}(D^2,n)$:
\[
 \To \Gamma_1 \To [\Gamma_0, \Gamma_1] \To \Gamma_0 \To.
\]
With the (unsatisfactory) notion of ``cone'' of section \ref{sec_other_cat_str}, $[\Gamma_0, \Gamma_1]$ is the ``cone'' of the morphism $\Gamma_0 \To \Gamma_1$ arising from the bypass system $FBS(\Gamma_0,\Gamma_1)$. However, although $FBS(\Gamma_0, \Gamma_1)$ can be any minimal subsystem of $CFBS(\Gamma_0, \Gamma_1)$, any choice gives $\Down_{FBS(\Gamma_0, \Gamma_1)} \Gamma_0 = [\Gamma_0, \Gamma_1]$. So $FBS(\Gamma_0, \Gamma_1)$ has a canonical effect, and we define $[\Gamma_0, \Gamma_1]$ to be \emph{the cone} of $\Gamma_0 \To \Gamma_1$. This defines the cone of any tight morphism between basis chord diagrams.

Thus, chord diagrams $[\Gamma_0, \Gamma_1]$, the objects of $\mathcal{C}(D^2,n+1)$, correspond precisely, via this cone construction, to nontrivial morphisms $\Gamma_0 \To \Gamma_1$ of basis elements, which are precisely the morphisms of $\mathcal{C}^b ( \mathcal{U} (n_-, n_+) )$. This gives the following proposition.
\begin{prop}[Chord diagrams as cones]
\label{prop_chord_diagrams_cones}
Chord diagrams with $n+1$ chords and euler class $e$ are in bijective correspondence with the morphisms of $\mathcal{C}^b (\mathcal{U}(n_-, n_+))$:
\[
\text{Mor} \left( \mathcal{C}^b \left( \mathcal{U} (n_-, n_+) \right) \right) \cong \text{Ob} \left( \mathcal{C} (D^2,n+1,e) \right).
\]
Moreover, under the inclusion $\iota \; : \; \mathcal{C}^b \left( \mathcal{U}(n_-, n_+) \right) \hookrightarrow \mathcal{C}(D^2,n+1,e)$,
every morphism of $\mathcal{C}^b( \mathcal{U}(n_-, n_+) )$ has a well-defined cone in $\mathcal{C}(D^2,n+1,e)$;
\begin{align*}
 \text{Cone} \circ \iota \; : \; \text{Mor} \left( \mathcal{C}^b \left( \mathcal{U}(n_-, n_+) \right) \right) & \stackrel{\cong}{\To} \text{Ob} \left( \mathcal{C}(D^2,n+1,e) \right) \\
 \left( \Gamma_{w_0} \rightarrow \Gamma_{w_1} \right) &\mapsto \text{Cone} \left( \iota ( \Gamma_{w_0} \rightarrow \Gamma_{w_1} ) \right) = [\Gamma_{w_0}, \Gamma_{w_1}]
\end{align*}
gives the bijection explicitly.
\qed
\end{prop}

The following lemma is not particularly profound, but perhaps of categorical interest.
\begin{lem}[``Snake lemma'']
\label{lem_snake}
Consider a tight cobordism $\M(\Gamma_1, \Gamma_2)$, where $\Gamma_1 = [\Gamma_1^-, \Gamma_1^+]$, $\Gamma_2 = [\Gamma_2^-, \Gamma_2^+]$. Then there is a tight morphism $\Gamma_1^- \To \Gamma_2^+$.
\end{lem}
The name ``snake lemma'' arises from the following diagram, regarding the cobordism as a map of ``exact triangles''. All arrows represent tight cobordisms.
\[
\xymatrix{
\ar[r] & \Gamma_2^+ \ar[r] & [\Gamma_2^-, \Gamma_2^+] \ar[r] & \Gamma_2^- \ar[r] & \\
\ar[r] & \Gamma_1^+ \ar[r] & [\Gamma_1^-, \Gamma_1^+] \ar[r] \ar@{=>}[u] & \Gamma_1^- \ar[r] \ar@{.>}[ull] |!{[l];[ul]}\hole &
}
\]

\begin{proof}
Since $m(\Gamma_1, \Gamma_2) = 1$, by proposition \ref{general_stackability}, the number of pairs $w_1 \preceq w_2$ with $\Gamma_{w_1} \in \Gamma_1$ and $\Gamma_{w_2} \in \Gamma_2$ is odd, hence not zero. But (theorem \ref{main_theorem_maxmin}) $\Gamma_1^-$ is a total minimum for $\Gamma_1$, and $\Gamma_2^+$ is a total maximum for $\Gamma_2$. Hence $\Gamma_1^- \preceq \Gamma_2^+$.
\end{proof}

\subsection{Properties of contact elements}
\label{sec_properties_contact_elements}

\subsubsection{How many basis elements in a decomposition?}
\label{sec_how_many_basis_elts_directly}

We prove proposition \ref{even_number_decomposition}: a non-basis contact element has an even number of basis elements in its decomposition. (We proved this result by ``skiing'' in section \ref{sec_ct_cat_U}; we now give a more direct proof.)

\begin{proof}
Consider the decomposition algorithm \ref{alg_base_point_construction} applied to a non-basis chord diagram $\Gamma$. Each basis diagram in the decomposition of $\Gamma$ first appears at some stage of this algorithm, from decomposing a non-basis chord diagram related to it by a bypass move. By lemma \ref{bypass_moves_elementary_moves}, a non-basis chord diagram related by a bypass move to a basis chord diagram is a sum of two basis chord diagrams. So basis elements appear in pairs.
\end{proof}
In fact, we have proved a little more: writing the basis elements of $\Gamma$ in lexicographic order, the $(2j-1)$'th and $2j$'th are bypass-related. Using the root point decomposition algorithm instead, a similar result holds for the right-to-left lexicographic order.

We immediately obtain a criterion for whether a contact element is a basis element:
\[
m \left(  \Gamma_{(-)^{n_-} (+)^{n_+}}, \Gamma \right) = m \left( \Gamma, \Gamma_{(+)^{n_+} (-)^{n_-}} \right) 
= \left\{ \begin{array}{cl} 1 & \text{if $\Gamma$ is a basis diagram} \\ 0 & \text{otherwise} \end{array} \right.
\]
Unsurprisingly, this is identical to the criterion for existing in a universal cobordism (see section \ref{sec_ct_cat_U}); it can also be proved by skiing.

\subsubsection{Symbolic interpretation of outermost regions}
\label{symbolic_outermost}

The appearance of certain symbols in both $w_-, w_+$ implies the appearance of symbols in all basis elements of $[w_-, w_+]$, and means that $\Gamma$ has an outermost chord in a certain place.
\begin{lem} [Outermost regions at base point]
\label{symbolic_basepoint}
Let $\Gamma = [\Gamma_-, \Gamma_+] = [w_-, w_+]$.
The following are equivalent.
\begin{enumerate}
\item 
$\Gamma$ has an outermost chord enclosing a $-$ (resp. $+$) region at the base point.
\item
Every word $w$ in the basis decomposition of $\Gamma$ begins with a $-$ (resp. $+$), i.e. $\Gamma_w$ has an outermost chord enclosing a $-$ (resp. $+$) region at the base point.
\item
$w_-, w_+$ both begin with a $-$ (resp. $+$), i.e. $\Gamma_-, \Gamma_+$ both have outermost chords enclosing a $-$ (resp. $+$) at the base point.
\end{enumerate}
\end{lem}

\begin{proof}
That (1) implies (2) follows immediately from the decomposition algorithm. That (2) implies (3) is obvious. That (3) implies (1) follows immediately from $B_- [\Gamma_{w_1}, \Gamma_{w_2}] = [B_- \Gamma_{w_1}, B_- \Gamma_{w_2}] = [ \Gamma_{-w_1}, \Gamma_{-w_2}]$.
\end{proof}

A similar result at the root point gives an equivalence between: an outermost chord on $\Gamma$ enclosing a $\pm$ region at the root point; every word $w$ in the decomposition of $\Gamma$ ending with a $\pm$; and $w_-, w_+$ both ending with a $\pm$.

In fact, we can detect any outermost chords enclosing negative regions on the westside or positive regions on the eastside similarly. To this end, define some more creation and annihilation operators (which can be defined anywhere; recall section \ref{sec_creation_annihilation}).
\begin{defn}[Eastside/westside creation operators]\
\begin{enumerate}
\item
For each $i=0, \ldots, n_-$, the operator
\[
 B_-^{west, i}: SFH(T, n+1, e) \rightarrow SFH(T, n+2, e-1)
\]
takes a chord diagram with $n+1$ chords and relative euler class $e$, and produces a chord diagram with $n+2$ chords and relative euler class $e-1$, adding an outermost chord between points $(-2i-3, -2i-2)$ (as labelled on the chord diagram with $n+2$ chords).
\item
For each $j=0, \ldots, n_+$, the operator
\[
 B_+^{east, j}: SFH(T, n+1, e) \rightarrow SFH(n+2, e+1)
\]
takes a chord diagram with $n+1$ chords and relative euler class $e$, and produces a chord diagram with $n+2$ chords and relative euler class $e+1$, adding an outermost chord between points $(2j+2,2j+3)$.
\end{enumerate}
\end{defn}

Note that $B_-^{west,0}$ adds an outermost negative region west of the base point (``$B_- = B_-^{west,-1}$''); and then the various $B_-^{west,i}$ add outermost regions further anticlockwise, until $B_-^{west,n_-}$ adds an outermost region ``east'' of the original root point, to create a new root point further anticlockwise of the original one. Similarly for the $B_+^{east,j}$.

\begin{defn}[Eastside/westside annihilation operators]\
\begin{enumerate}
\item
For each $i = 0, \ldots, n_-$, the operator
\[
 A_+^{west, i}: SFH(T, n+1, e) \To SFH(T, n, e+1)
\]
takes a chord diagram with $n+1$ chords and relative euler class $e$, and produces a chord diagram with $n$ chords and relative euler class $e+1$, by joining the chords at positions $(-2i-2,-2i-1)$.
\item
For each $j=0, \ldots, n_+$, the operator
\[
 A_-^{east, j}: SFH(T, n+1, e) \To SFH(T,n,e-1)
\]
takes a chord diagram with $n+1$ chords and relative euler class $e$, and produces a chord diagram with $n$ chords and relative euler class $e-1$, by joining the chords at positions $(2j+1,2j+2)$.
\end{enumerate}
\end{defn}
Note $A_+^{west,0}$ joins chords west of the base point (``$A_- = A_-^{west,-1}$''); the various $A_+^{west,i}$ join chords further anticlockwise; $A_+^{west,n_-}$ joins chords on the ``east'' of the original chord diagram, namely those at the root point and immediately east of it. 

As with our original annihilation and creation operators, 
\[
 A_+^{west,j} \circ B_-^{west,j} = 1, \quad A_-^{east,j} \circ B_+^{east,j} = 1.
\]
We investigate further relations in section \ref{sec_simplicial}.

From lemma \ref{construction_mechanics} we see that $B_-^{west,j}$ has the effect on $\Gamma_w$ of producing $\Gamma_{w'}$, where $w'$ is obtained from $w$ as follows. If $0 \leq j \leq n_- - 1$, then we insert a $-$ sign in $w$ immediately before (or after) the $(j+1)$'th $-$ sign, ``splitting the $(j+1)$'th $-$ sign into two $-$ signs''. If $j=n_-$, then we add a $-$ sign at the end of $w$. $B_+^{east,j}$ is analogous.

As for annihilation operators, $A_+^{west,j}$ deletes the $(j+1)$'th $-$ sign, for $0 \leq j \leq n_- - 1$; for $j=n_-$, it deletes the $-$ sign at the end of the word, if there is one, or returns $0$ if the word ends in a $+$. And $A_-^{east,j}$ is analogous with signs reversed..

The proof of lemma \ref{symbolic_basepoint}, applied to westside operators, immediately gives the following.
\begin{lem} [Outermost negative regions on westside]
\label{symbolic_westside}
Let $\Gamma = [\Gamma_-, \Gamma_+]$, and let $j$ be an integer from $1$ to $n_- - 1$.
The following are equivalent.
\begin{enumerate}
\item 
$\Gamma$ has an outermost chord enclosing a $-$ region between $(-2j-1, -2j)$.
\item
Every word $w$ in the basis decomposition of $\Gamma$ has $(j+1)$'th $-$ sign following (i.e. not the first in its block), i.e. $\Gamma_w$ has an outermost chord enclosing a $-$ region between $(-2j-1, -2j)$.
\item
$w_-, w_+$ both have $(j+1)$'th $-$ sign following, i.e. $\Gamma_-, \Gamma_+$ both have outermost chords enclosing a $-$ region between $(-2j-1, -2j)$.
\end{enumerate}
\qed
\end{lem}

There is an analogous lemma on the eastside. All the lemmas in this section say that, if a chord diagram has an outermost region in a specific place, then so do all the basis chord diagrams in its decomposition. As we proceed through the decomposition algorithm, there is no decomposition there. This can be seen explicitly from the algorithm.

\subsubsection{Relations from bypass systems}

\label{sec_relations_ct_elts}

We prove theorem \ref{not_much_comparability}: the only basis elements in $\Gamma$ comparable to all others are $\Gamma_{\pm}$.

First, we observe that in addition to the tight contact cylinders of proposition \ref{MGamma+-}, we have many more tight cylinders within them. As in section \ref{sec_generalised_byp_triple_GammaGammaGamma}, take $\Gamma = [w_-, w_+]$ and bypass systems $c_- = FBS(w_-, w_+)$, $d_+ = BBS(w_-, w_+)$ on $\Gamma_-, \Gamma_+$ respectively.
\begin{lem}[More tight cylinders]\
\label{lem_tight_cyls}
\begin{enumerate}
\item
For every $\Gamma_w$ obtained by performing upwards bypass moves on $\Gamma_-$ along some subset $a_-$ of $c_-$,
\[
m(\Gamma_-, \Gamma_w) = m(\Gamma_w, \Gamma_+) = m(\Gamma, [\Gamma_-, \Gamma_w]) = m([\Gamma_-, \Gamma_w], \Gamma_-) = 1,
\]
and tight contact structures on these cylinders can be obtained by bypass attachments from $c_-$.
\item
For every $\Gamma_w$ obtained by performing downwards bypass moves on $\Gamma_+$ along some subset $b_+$ of $d_+$,
\[
m(\Gamma_-, \Gamma_w) = m(\Gamma_w, \Gamma_+) = m(\Gamma_+, [\Gamma_w, \Gamma_+]) = m([\Gamma_w, \Gamma_+], \Gamma) = 1,
\]
and tight contact structures on these cylinders can be obtained by bypass attachments from $d_+$.
\end{enumerate}
\end{lem}
Note any $\Gamma_w \in \Gamma$ satisfies the hypotheses of both halves of this lemma; the lemma is more general, because the $\Gamma_w = \Up_{a_-} \Gamma_-$ may come in even multiplicity and cancel.

\begin{proof}
We prove part (1); (2) is similar. The contact structure on $\M(\Gamma_-, \Gamma_+)$ obtained by attaching bypasses above $\Gamma_-$ along $c_-$ is tight (proposition \ref{MGamma+-}). Attaching bypasses along the subset $a_-$ thus gives a tight $\M(\Gamma_-, \Gamma_w)$; above this is a tight $\M(\Gamma_w, \Gamma_+)$.

To construct a tight contact structure on $\M([\Gamma_-, \Gamma_w], \Gamma_-)$, we first take a \emph{minimal} sub-system $a_-^0$ of $a_-$ (as in section \ref{sec_byp_sys_comp_pair}). That is, $a_-^0$ contains no trivial attaching arcs, $\Up (a_-^0) \Gamma_- = \Gamma_w$, but upwards surgery on any proper subset of $a_-^0$ does not give $\Gamma_w$. As noted in section \ref{sec_byp_sys_comp_pair}, not every bypass system has a minimal sub-system; but a bypass system with no trivial attaching arcs, such as $a_-$, does.

Corollary \ref{general_bypass_system_other_way} then gives $\Down (a_-^0) \Gamma_- = [\Gamma_-, \Gamma_w]$. Moreover, from proposition \ref{MGamma+-}, attaching downwards bypasses to $\Gamma_-$ along all of $c_-$ gives a tight contact structure on $\M(\Gamma, \Gamma_-)$. Attaching only along the subset $a_-^0$ gives a tight $\M([\Gamma_-, \Gamma_w], \Gamma_-)$; below this is a tight $\M(\Gamma, [\Gamma_-, \Gamma_w])$.
\end{proof}

Now we can take these bypass system shenanigans a little further.
\begin{prop}
\label{MGammaGammaw}
Let $\Gamma_w$ be a basis chord diagram obtained by either performing upwards bypass moves on $\Gamma_-$ along a subset $a_-$ of $c_-$, or downwards bypass moves on $\Gamma_+$ along a subset $b_+$ of $d_+$. Then
\[
 m(\Gamma, \Gamma_w) = \left\{ \begin{array}{cl} 1 & \Gamma_w = \Gamma_- \\ 0 & \text{otherwise} \end{array} \right.
\quad \text{and} \quad
 m(\Gamma_w, \Gamma) = \left\{ \begin{array}{cl} 1 & \Gamma_w = \Gamma_+ \\ 0 & \text{otherwise}. \end{array} \right. 
\]
\end{prop}
In particular, the above proposition applies to any $\Gamma_w \in \Gamma$.

\begin{proof}
We consider $m(\Gamma, \Gamma_w)$ and $a_- \subseteq c_-$; other cases are similar. As $a_-$ contains no trivial attaching arcs, it has a minimal sub-system; hence we may assume $a_-$ minimal. Thus $a_-$ is empty iff $\Gamma_w = \Gamma_-$. Expanding ``up over down'' (lemma \ref{lem_expanding_down_over_up}),
\[
\Gamma_w = \Up_{a_-} \Gamma_- = \sum_{b_- \subseteq a_-} \Down_{b_-} \Gamma_-.
\]
For each $b_- \subseteq a_-$, we have
\[
 \Down_{b_-} \Gamma_- = \sum_{e_- \subseteq b_-} \Up_{e_-} \Gamma_-.
\]
By ``stability'' proposition \ref{upwards_moves_forwards}, each $\Up_{e_-} \Gamma_-$ is a basis element; the least term is $\Gamma_-$, which appears precisely once, when $e_-$ is the empty set. If $b_-$ is not empty, then $\Down_{b_-} \Gamma_- \neq \Gamma_-$; hence
\[
\Down_{b_-} \Gamma_- = [\Gamma_-, \Gamma_x] \quad \text{for some word $x$,} \quad \Gamma_- \prec \Gamma_x \preceq \Gamma_+.
\]
Moreover, $\Gamma_x$ is obtained from upwards bypass moves from $\Gamma_-$ along some subset of $c_-$. (Note $\Gamma_x$ need not be $\Up_{b_-} \Gamma_-$, which may appear several times and cancel; minimality of $a_-$ does not imply minimality of each $b_-$.)

Now taking $m(\Gamma, \cdot)$ we obtain
\[
 m(\Gamma, \Gamma_w) = \sum_{b_- \subseteq a_-} m \left( \Gamma, \Down_{b_-} \Gamma_- \right).
\]
For $b_-$ empty, the term is $m(\Gamma, \Gamma_-) = 1$. For $b_-$ nonempty, each term is of the form $m(\Gamma, [\Gamma_-, \Gamma_x])$, where $\Gamma_x$ is obtained from upwards bypass moves from $\Gamma_-$ along some subset of $c_-$; by lemma \ref{lem_tight_cyls}, this term is $1$. So the sum is
\[
 m(\Gamma, \Gamma_w ) = \sum_{b_- \subseteq a_-} 1 = 2^{|a_-|}
\]
which (mod 2) is $0$ when $a_-$ is nonempty, and $1$ when $a_-$ is empty.
\end{proof}

\begin{prop}
\label{prop_prec_follow_even}
For every $\Gamma_w \in \Gamma$, other than $\Gamma_\pm$, the number of basis elements of $\Gamma$ which precede $\Gamma_w$ with respect to $\preceq$ is even, and the number of basis elements which follow $\Gamma_w$ with respect to $\preceq$ is also even.
\end{prop}

\begin{proof}
Expand $m(\Gamma, \Gamma_w) = 0$ and $m(\Gamma_w, \Gamma) = 0$ over basis elements of $\Gamma$.
\end{proof}

\begin{proof}[Proof of theorem \ref{not_much_comparability}]
For a basis element, it is clear. Otherwise, the number of elements comparable to $\Gamma_w$ is $m(\Gamma, \Gamma_w) + m(\Gamma_w, \Gamma) + 1$. (We overcount $\Gamma_w$ in the sum, so correct by adding $1$.) If $\Gamma_w$ is comparable to every basis element in $\Gamma$ then this number must be even, since $\Gamma$ contains an even number of basis elements (proposition \ref{even_number_decomposition}). But by proposition \ref{MGammaGammaw} it is odd unless $\Gamma_w = \Gamma_\pm$.
\end{proof}

\section{Further considerations}

\label{ch_further_considerations}

\subsection{The rotation operator}
\label{sec_rotation}

We consider the operation of rotating chord diagrams, or equivalently, moving the base point. To maintain our base point sign convention, we move the base point by two marked points. Such rotation corresponds to an inclusion $(T,n) \hookrightarrow (T,n)$ with $S^1$-invariant contact structure on the intermediate $(\text{annulus}) \times S^1$ given by the dividing set in figure \ref{fig:43}.

\begin{figure}[tbh]
\centering
\includegraphics[scale=0.4]{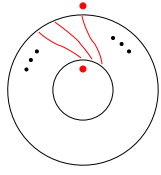}
\caption{The rotation operator.} \label{fig:43}
\end{figure}

TQFT-inclusion then gives a linear operator $R$ on $SFH$. Refining by $e$ gives a map
\[
 R: SFH(T, n+1, e) \To SFH(T, n+1, e), \quad \text{i.e.} \quad \Z_2^{\binom{n}{k}} \To \Z_2^{\binom{n}{k}}.
\]
Clearly: $m(\Gamma_0, \Gamma_1) = m(R \Gamma_0, R \Gamma_1)$; $R^{n+1} = 1$; and $R$ is a bijection on contact elements. When we wish to refer to a particular $\binom{n}{k}$ we write $R_{n,k}$ for the map on $\Z_2^{\binom{n}{k}}$

\subsubsection{Small cases}

For $SFH(T,1,0) = \Z_2$, obviously $R=1$. Similarly for an extremal euler class $SFH(T, n+1, e= \pm n) = \Z_2^{\binom{n}{0}}$ or $\Z_2^{\binom{n}{n}} = \Z_2$, again $R=1$. In the smallest non-identity case $SFH(T,3,0) = \Z_2^{\binom{2}{1}} = \Z_2^2$, the three chord diagrams form a bypass triple; we easily obtain $v_{-+} \mapsto v_{+-} \mapsto v_{-+} + v_{+-} \mapsto v_{-+}$. Writing matrices using the lexicographically ordered basis, we obtain:
\begin{align*}
R_{2,1} = \begin{bmatrix} 0 & 1 \\ 1 & 1 \end{bmatrix}, \quad
R_{3,1} = R_{3,2} = \begin{bmatrix} 0 & 1 & 0 \\ 0 & 0 & 1 \\ 1 & 1 & 1 \end{bmatrix}, \quad 
R_{4,1} = R_{4,3} = \begin{bmatrix} 0 & 1 & 0 & 0 \\ 0 & 0 & 1 & 0 \\ 0 & 0 & 0 & 1 \\ 1 & 1 & 1 & 1 \end{bmatrix},
\end{align*}

\[
R_{5,3} \quad = \quad
\begin{BMAT}{rccc}{cccc}
& \underbrace{--+ \cdot} & \underbrace{-+ \cdot} & \underbrace{+ \cdot}
\\
--+ \cdot \} &
0 & \begin{BMAT}{c:cc}{c} 1 & 0 & 0 \end{BMAT} & \begin{BMAT}{cccccc}{c} 0 & 0 & 0 & 0 & 0 & 0 \end{BMAT}
\\
\left. \begin{array}{c} -+ \cdot \end{array} \right\} &
\begin{BMAT}{c}{ccc} 0 \\ 0 \\ 0 \end{BMAT} & \begin{BMAT}{ccc}{ccc} 0 & 0 & 0 \\ 0 & 0 & 0 \\ 0 & 0 & 0 \end{BMAT} & \begin{BMAT}{ccc:ccc}{ccc} 0 & 1 & 0 & 0 & 0 & 0 \\ 0 & 0 & 1 & 0 & 0 & 0 \\ 1 & 1 & 1 & 0 & 0 & 0 \end{BMAT}
\\
\left. \begin{array}{c} + \cdot \end{array} \right\} &
\begin{BMAT}{c}{cccccc} 0 \\ 0 \\ 0 \\ 0 \\ 0 \\ 1 \end{BMAT} & \begin{BMAT}{ccc}{cccccc} 0 & 1 & 0 \\ 0 & 0 & 0 \\ 0 & 0 & 0 \\ 0 & 0 & 1 \\ 0 & 0 & 0 \\ 1 & 1 & 1 \end{BMAT} & \begin{BMAT}{cccccc}{cccccc} 0 & 1 & 0 & 0 & 0 & 0 \\ 0 & 0 & 0 & 0 & 1 & 0 \\ 0 & 0 & 0 & 1 & 1 & 0 \\ 0 & 0 & 1 & 0 & 1 & 0 \\ 0 & 0 & 0 & 0 & 0 & 1 \\ 1 & 1 & 1 & 1 & 1 & 1 \end{BMAT}
\addpath{(2,0,0)uuu}
\addpath{(3,0,0)uuu}
\addpath{(1,1,0)rrr}
\addpath{(1,2,0)rrr}
\end{BMAT}
\]
(It might appear that $R_{n,k} = R_{n,n-k}$, but this is not the case in general: e.g. $R_{5,2} \neq R_{5,3}$.)

\subsubsection{Computation of $R$}

\label{sec_computation_R}

We can compute $R$ recursively. If $x$ is a word in $\{-,+\}$ let \emph{$x$-basis elements} denote those $v_w$ where $w$ begins with the string $x$. Let \emph{$x$-rows} or \emph{$x$-columns} be rows or columns corresponding to $x$-basis elements. For two words $x$ and $y$, the \emph{$x \times y$ minor} of $R$ is the submatrix consisting of the intersection of the $x$-rows and $y$-columns.

For example, $R_{5,3}$ is written above to suggest a decomposition in terms of such minors.
\[
R_{5,3} \quad = \quad
\begin{BMAT}{rccc}{cccc}
& \underbrace{--+ \cdot} & \underbrace{-+ \cdot} & \underbrace{+ \cdot} \\
--+ \cdot \} & 0 & \begin{BMAT}{c:c}{c} R_{2,2} & 0 \end{BMAT} & 0 \\
-+ \cdot \} & 0 & 0 & \begin{BMAT}{c:c}{c} R_{3,2} & 0 \end{BMAT} \\
\left. \begin{array}{c} + \cdot \end{array} \right\} &
\begin{array}{c} (-- \cdot)\text{-cols} \\ \text{of } R_{4,2} \end{array} & \begin{array}{c} (- \cdot)\text{-cols} \\ \text{of } R_{4,2} \end{array} & R_{4,2} 
\addpath{(2,0,0)uuu}
\addpath{(3,0,0)uuu}
\addpath{(1,1,0)rrr}
\addpath{(1,2,0)rrr}
\end{BMAT}
\]
We now prove that something similar occurs for all $R_{n,k}$.

\begin{prop}[Recursive computation of $R$]
\label{prop_recursive_rotation}
$R_{n,k}$ is given by:
\begin{enumerate}
\item The $(+) \times (+)$ minor of $R_{n,k}$ consists of $R_{n-1,k-1}$.
\item The $(+) \times (-+)$ minor of $R_{n,k}$ contains the $(-)$-columns of $R_{n-1,k-1}$. More generally, the $(+) \times ((-)^j +)$ minor of $R_{n,k}$ contains the $((-)^{j})$-columns of $R_{n-1,k-1}$, for any $j=1, \ldots, n-k$.
\item The $(-+) \times (+-)$ minor of $R_{n,k}$ consists of $R_{n-2,k-1}$. More generally, for any $j = 0, \ldots, n-k-1$, the $((-)^j - +) \times ((-)^j +-)$ minor of $R_{n,k}$ consists of $R_{n-j-2,k-1}$.
\item All other entries are zero. To write these remaining entries out exhaustively (with some overlap):
\begin{enumerate}
\item (``Below and on the diagonal, in the $(-)$ rows.'') The $(-+) \times (-)$ minor of $R_{n,k}$ is zero. More generally, the $((-)^j +) \times ((-)^j)$ minor of $R_{n,k}$ is zero, for any $j=1, \ldots, n-k$.
\item (``Above the diagonal and the submatrices $R_{n-j-2,k-1}$.'') The $(--) \times (+)$ minor of $R_{n,k}$ is zero. More generally, the $((-)^j --) \times ((-)^j +)$ minor of $R_{n,k}$ is zero, for any $j=0, \ldots, n-k-2$.
\item (The pieces in the $(-)$ rows just to the right of the submatrices $R_{n-j-2,k-1}$.) The $(-) \times (++)$ minor of $R_{n,k}$ is zero. More generally, the $(-) \times ((-)^j ++)$ minor of $R_{n,k}$ is zero, for any $j = 0, \ldots, n-k$.
\end{enumerate}
\end{enumerate}
\end{prop}

\begin{proof}
We simply verify all these conditions. The conditions given are equivalent to the following equations on operators:
\begin{enumerate}
\item $A_- R B_+ = R$.
\item $A_- R (B_-)^j B_+ = R (B_-)^j$, for $j=1, \ldots, n-k$.
\item $A_- A_+ (A_+)^j R (B_-)^j B_+ B_- = R$, for $j=0, \ldots, n-k-1$.
\item \begin{enumerate}
\item $A_- (A_+)^j R (B_-)^j = 0$, for $j=1, \ldots, n-k$.
\item $A_+ A_+ (A_+)^j R (B_-)^j B_+ = 0$, for $j=0, \ldots, n-k-2$.
\item $A_+ R (B_-)^j B_+ B_+ = 0$, for $j=0, \ldots, n-k$.
\end{enumerate}
\end{enumerate}
These are now easily proved by examining the corresponding chord diagrams.
\end{proof}

These matrices have interesting combinatorial properties: for instance, for every row, there is precisely one column which has its highest nonzero element in that row. The above recursive form also gives a recursive formula:
\[
 R = \sum_{n=0}^\infty B_+ R B_-^n A_- A_+^n + B_-^{n+1} B_+ R A_+ A_- A_+^n 
 = \sum_{n=0}^\infty \left[ B_+ R A_+, B_-^{n+1} \right] A_- A_+^n.
\]

\subsubsection{An explicit description}

\label{sec_explicit_R}

We now describe $R$ explicitly on basis vectors $v_w$. Write $w = (-)^{a_1} (+)^{b_1} \cdots (-)^{a_k} (+)^{b_k}$ where possibly $a_1 = 0$ or $b_k = 0$, but all other $a_i, b_i$ are nonzero. Interpreting the formula above as a set of instructions for operating on $w$, removing or adding $-,+$ signs, and proceeding by induction, we obtain the following.

\begin{prop}[Explicit computation of $R$]
\label{prop_explicit_rotation}
If $k \geq 2$ then $R(v_w)$ is given by taking
\[
\begin{array}{c}
 (+)^{b_1 - 1} (-)^{a_1 + 1} (+)^{b_2} (-)^{a_2} \cdots (+)^{b_{k-1}} (-)^{a_{k-1}} (+)^{b_k + 1} (-)^{a_k - 1} \\
 = (+)^{b_1 - 1} (-)^{a_1 + 1} \left( \prod_{j=2}^{k-1} (+)^{b_j} (-)^{a_j} \right) (+)^{b_k + 1} (-)^{a_k - 1}
\end{array}
\]
and then, for each possible way of grouping $(1, 2, \ldots, k)$ into the form
\[
  ((1,2,\ldots,l_1), (l_1 + 1, l_1 + 2, \ldots, l_2), \ldots, (l_{T-1}+1, l_{T-1} + 1, \ldots, l_T = k)),
\]
(including the trivial grouping $((1), (2), \ldots, (k))$, taking the expression
\[
\begin{array}{c}
(+)^{b_1 + \cdots + b_{l_1} - 1} (-)^{a_1 + \cdots + a_{l_1} +1} (+)^{b_{l_1 + 1} + \cdots + b_{l_2}} (-)^{a_{l_1 + 1} + \cdots + a_{l_2}} \cdots \\
\cdots 
(+)^{b_{l_{T-2} + 1} + \cdots + b_{l_{T-1}}} (-)^{b_{a_{T-2} + 1} + \cdots + a_{l_{T-1}}}
(+)^{b_{l_{T-1} + 1} + \cdots + b_{l_T} + 1} (-)^{a_{l_{T-1} + 1} + \cdots + a_{l_T} - 1} \\ 
\\ =
(+)^{b_1 + \cdots + b_{l_1} - 1} (-)^{a_1 + \cdots + a_{l_1} +1} 
\left( \prod_{m=2}^{T-1} (+)^{b_{l_{m-1} + 1} + \cdots + b_{l_m}} (-)^{a_{l_{m-1} + 1} + \cdots + a_{l_m}} \right) \\
(+)^{b_{l_{T-1} + 1} + \cdots + b_{l_T} + 1} (-)^{a_{l_{T-1} + 1} + \cdots + a_{l_T} - 1}
\end{array}
\] 
obtained by grouping factors of the first expression accordingly, and summing all the corresponding basis elements.

If $k=1$, so that $w$ is of the form $(-)^a$ or $(+)^b$ or $(-)^a (+)^b$, then $R(v_w)$ is given by a single term $v_{w'}$ where $w'$ is given by:
\begin{enumerate}
 \item for $w = (-)^a$, $w' = (-)^a$ also;
 \item for $w = (+)^a$, $w' = (+)^a$ also;
 \item for $w = (-)^a (+)^b$, $w' = (+)^b (-)^a$;
\end{enumerate}
\qed
\end{prop}
Note every chord diagram has an outermost region; after some rotation, an outermost region comes to the base point; and a diagram with an outermost region at the base point is of the form $B_\pm \Gamma$. Thus, rotation matrices give a quick way to compute all the contact elements in $SFH(T,n+1,e)$ from those in $SFH(T,n,e \pm 1)$.

\subsection{Simplicial structures}
\label{sec_simplicial}

Recall $B_-^{west,i}, A_+^{west,i}, B_+^{east,j}, A_-^{east,j}$ from section \ref{symbolic_outermost}: for $0 \leq i \leq n_-$ and $0 \leq j \leq n_+$,
\begin{enumerate}
\item  $B_-^{west,i}$ inserts a chord $(-2i-3,-2i-2)$,
\item  $A_+^{west,i}$ joins the chords at positions $(-2i-2,-2i-1)$,
\item  $B_+^{east,j}$ inserts a chord $(2i+2,2i+3)$,
\item  $A_-^{east,j}$ joins the chords at positions $(2j+1,2j+2)$.
\end{enumerate}
Taking $i,j=-1$ recovers the original $A_\pm, B_\pm$. We have seen $B_-^{west,j} \circ A_+^{west,j} = B_+^{east,j} \circ A_-^{east,j} = 1$. The following lemma gives further relations; the proof is clear, considering the effect on words, or on chord diagrams.
\begin{lem}[Westside simplicial structure]
For all $0 \leq i,j \leq n_-$:
\begin{align*}
A_+^{west,i} \circ A_+^{west,j} &= \begin{array}{cl} A_+^{west,j-1} \circ A_+^{west,i} & i<j \end{array} \\
A_+^{west,i} \circ B_-^{west,j} & = \left\{ \begin{array}{cl} B_-^{west,j-1} \circ A_+^{west,i} & i<j \\ 1 & i=j,j+1 \\ B_-^{west,j} \circ A_+^{west,i-1} & i>j+1 \end{array} \right. \\
B_-^{west,i} \circ B_-^{west,j} &= \begin{array}{cl} B_-^{west,j+1} \circ B_-^{west,i} & i \leq j \end{array}
\end{align*}\
\qed
\end{lem}
Hence there is a \emph{simplicial structure} on $SFH(T,n+1)$, with face maps $d_i^+ = A_+^{west,i}$ and degeneracy maps $s_j^+ = B_-^{west,j}$ for $0 \leq i,j \leq n_-$. The the (mod 2) ``boundary map'' $d^+ = \sum_{i=0}^{n_-} d_i^+ = \sum_{i=0}^{n_-} A_+^{west,i}$ satisfies $(d^+)^2=0$, and we obtain chain complexes
\[
 SFH(T,n+1,e) \stackrel{d^+}{\To} SFH(T,n,e+1) \stackrel{d^+}{\To} \cdots \stackrel{d^+}{\To} SFH \left( T, \frac{n+e}{2}+1, \frac{n+e}{2} \right)
\]
along which $n_+$ is constant and $n_-$ decreases to $0$; hence $n_-$ can be regarded as the ``dimension''. This is a ``northeast--southwest'' diagonal of Pascal's triangle. We call the chain complex $C^{+,n_+}_*$.

Recall (section \ref{symbolic_outermost}) the effect of $d_i^+ = A_+^{west,i}$ on words:
\begin{enumerate}
\item
For $0 \leq i \leq n_- - 1$, $A_+^{west,i}$ deletes the $(i+1)$'th $-$ sign in a word.
\item
For $i = n_-$, $A_+^{west,n_-}$ deletes the final $-$ sign, if possible; else returns $0$.
\end{enumerate}
Hence the effect of $d^+$ on a word $w$ is to give a sum over all of the above, and we easily obtain the following lemma. We write $w = (-)^{a_1} (+)^{b_1} \cdots (-)^{a_k} (+)^{b_k}$ as usual.
\begin{lem}[Effect of $d^+$]
If $b_k > 0$, i.e. $w$ ends in a $+$, then
\[
d^+ w = a_1 (-)^{a_1 - 1} (+)^{b_1} \cdots (-)^{a_k} (+)^{b_k} + \cdots + a_k (-)^{a_1} (+)^{b_1} \cdots (-)^{a_k - 1} (+)^{b_k}.
\]
If $b_k = 0$, so that $w$ ends in a $-$, then
\[
\begin{array}{rcl}
d^+ w &=& a_1 (-)^{a_1 - 1} (+)^{b_1} \cdots (-)^{a_k} + \cdots + a_{k-1} (-)^{a_1} \cdots (-)^{a_{k-1} - 1} (+)^{b_{k-1}} (-)^{a_k} \\
&& + (a_k + 1) (-)^{a_1} (+)^{b_1} \cdots (-)^{a_{k-1}} (+)^{b_{k-1}} (-)^{a_k - 1}.
\end{array}
\]
\qed
\end{lem}
So when $w$ ends in $+$, $d^+$ behaves just like (non-commutative) ``partial differentiation by $-$, $d^+ = \frac{\partial}{\partial -}$''. When $w$ ends in $-$, there is an extra term. From this it is easy to see that $(d^+)^2 = 0$ directly.

All this applies analogously on the eastside. We have a simplicial structure with: face maps $d_i^- = A_-^{east,i}$; degeneracy maps $s_j^- = B_+^{east,j}$; a boundary operator $d^- = \sum_{i=0}^{n_+} d_i^-$ which is ``$d^- = \frac{\partial}{\partial +}$'' on words ending in $-$, else there is an extra term; and a chain complex where $n_-$ is constant and $n_+$ can be regarded as ``dimension'', which is a ``northwest--southeast'' diagonal of Pascal's triangle.
\[
\xymatrix{ 
&&& \Z_2^{\binom{0}{0}} &&& \\
&& \Z_2^{\binom{1}{0}} \ar[ur]^{d^+} && \Z_2^{\binom{1}{1}} \ar[ul]_{d^-} && \\
& \Z_2^{\binom{2}{0}} \ar[ur]^{d^+} && \Z_2^{\binom{2}{1}} \ar[ul]_{d^-} \ar[ur]^{d^+} && \Z_2^{\binom{2}{2}} \ar[ul]_{d^-}&
}
\]
The chain complex groups $C^{+,n_+}_{n_-} = C^{-,n_-}_{n_+} = SFH(T,n+1,e)$. It is not difficult to see that the two boundary operators $d^-, d^+$ commute; they are essentially partial differentiation by different variables, though some consideration must be paid to the final term. Thus we obtain a double complex structure on the categorified Pascal's triangle. The homology of the chain complexes is rather uninteresting.
\begin{prop}[Westside/eastside homology]
For all $i$, the homology of the complex $\left( C^{+,n_+}_*, d^+ \right)$ or $\left( C^{-,n_-}_*, d^- \right)$ is zero.
\end{prop}

We argue similarly to Frabetti in \cite{FraSimplicial}, in the context of planar binary trees. (Planar binary trees have a nice bijective correspondence with chord diagrams.)
\begin{proof}
We consider $C^{+,n_+}_*$; the other case is similar. Note that the ``original'' creation operator $B_- = B_-^{west,-1}: C^{+,n_+}_* \To C^{+,n_+}_{*+1}$ satisfies $A_+^{west,0} B_- = 1$ and $A_+^{west,i} B_- = B_- A_+^{west,i-1}$ for $i>0$. Hence for a word $w \in W(n_-, n_+)$, we have $B_- d^+ + d^+ B_- = 1$:
\begin{align*}
\left( B_- d^+ + d^+ B_- \right) w &= B_- \sum_{i=0}^{n_-} A_+^{west,i} w + \sum_{i=0}^{n_- + 1} A_+^{west,i} B_- w \\
&= A_+^{west,0} B_- w + \sum_{i=1}^{n_- + 1} \left( A_+^{west,i} B_- + B_- A_+^{west,i-1} \right) w.
\end{align*}
So $B_-$ is a chain homotopy from the chain maps $1$ to $0$.
\end{proof}
More directly, if $w$ is a cycle, $d^+ w = 0$, then $B_- d^+ w + d^+ B_- w = d^+ B_- w = w$, so $w = d^+ (B_- w)$ is a boundary. We have now proved proposition \ref{prop_pascal_double_complex}.

\subsection{QFT and higher categorical considerations}
\label{sec_QFT_categorical}

\subsubsection{A contact 2-category}

\label{sec_2-category}

We have shown that the objects of $\mathcal{C}(D^2,n)$ are chord diagrams $\Gamma$; such $\Gamma$ are naturally described by pairs $\Gamma_- \preceq \Gamma_+$, hence can be considered as morphisms in $\mathcal{C}(\mathcal{U}(n_-,n_+))$, or as cones thereof (proposition \ref{prop_chord_diagrams_cones}). That is, ``objects are morphisms''; so morphisms become ``morphisms between morphisms''. 

This leads us to a 2-category. The objects of Honda's category become our 1-morphisms; and its 1-morphisms become our 2-morphisms.
\begin{defn}[Contact 2-category]
The \emph{contact 2-category} $\mathcal{C}(n+1,e)$ is as follows.
\begin{enumerate}
\item The objects are words in $W(n_-, n_+)$, i.e. with $n_-$ $-$ signs and $n_+$ $+$ signs.
\item The 1-morphisms $w_0 \rightarrow w_1$ and their composition are defined by the partial order $\preceq$. Equivalently, a 1-morphism is a chord diagram $\Gamma = [\Gamma_{w_0}, \Gamma_{w_1}]$ with $n+1$ chords and relative euler class $e$; composition of chord diagrams $\Gamma = [\Gamma_{w_0}, \Gamma_{w_1}]$ and $\Gamma' = [\Gamma_{w_1}, \Gamma_{w_2}]$ is $\Gamma' \circ \Gamma = [\Gamma_{w_0}, \Gamma_{w_2}]$.
\item The 2-morphisms $\Gamma_0 \rightarrow \Gamma_1$ are the tight contact structures on $\M(\Gamma_0, \Gamma_1)$, along with one extra 2-morphism $\{*\}$ for overtwisted contact structures. There are two types of composition of 2-morphisms.
\begin{itemize}
\item 
Given two 2-morphisms $\Gamma_0 \stackrel{\xi_0}{\rightarrow} \Gamma_1 \stackrel{\xi_1}{\rightarrow} \Gamma_2$, their \emph{vertical composition} $\xi_0 \cdot \xi_1$ is the 2-morphism $\Gamma_0 \stackrel{\xi_0 \cup \xi_1}{\rightarrow} \Gamma_2$ which is the contact structure on $\M(\Gamma_0, \Gamma_2)$ obtained by stacking $\xi_0, \xi_1$.
\item
Given three objects $w_0, w_1, w_2$, two pairs of 1-morphisms between them
\[
 w_0 \stackrel{\Gamma_0, \Gamma'_0}{\To} w_1,
\quad w_1 \stackrel{\Gamma_1, \Gamma'_1}{\To} w_2,
\]
and two $2$-morphisms
\[
 \Gamma_0 \stackrel{\xi_0}{\To} \Gamma'_0, \quad \Gamma_1 \stackrel{\xi_1}{\To} \Gamma'_1,
\quad \text{i.e.} \quad 
\xymatrix{
w_0 \ar@/_2pc/[rr]^{\Gamma_0}="g0" \ar@/^2pc/[rr]_{\Gamma'_0}="gp0" && w_1 \ar@/_2pc/[rr]^{\Gamma_1}="g1" \ar@/^2pc/[rr]_{\Gamma'_1}="gp1" && w_2,
\ar@{=>}_{\xi_0} "g0";"gp0" 
\ar@{=>}_{\xi_1} "g1";"gp1"
}
\]
the \emph{horizontal composition} $\xi_0 \xi_1$ is a morphism $(\Gamma_1 \circ \Gamma_0) \rightarrow (\Gamma'_1 \circ \Gamma'_0)$ as follows. Since 1-morphisms arise from $\preceq$, $\Gamma_0 =\Gamma'_0$ and $\Gamma_1 = \Gamma'_1$. Thus each $\xi_i$ is a contact structure on $\M(\Gamma_i, \Gamma_i)$. If these are both the unique tight contact structures, then $\xi_0 \xi_1$ is the unique tight contact structure on $\M(\Gamma_1 \circ \Gamma_0, \Gamma_1 \circ \Gamma_0)$. Otherwise $\xi_0 \xi_1 = \{*\}$.
\end{itemize}
\end{enumerate}
\end{defn}
As a 1-category, $\mathcal{C}(n+1,e) \cong W(n_-, n_+) \cong \mathcal{C}^b (\mathcal{U}(n_-, n_+))$.
\begin{lem}
$\mathcal{C}(n+1,e)$ is a 2-category.
\end{lem}

\begin{proof}
We verify the axioms of a 2-category as stated in \cite{BaezIntronCat}; we clearly already have a 1-category. That vertical composition is associative is clear, being a union of contact structures. Note that $\{*\}$ acts as a zero for this composition.

That horizontal composition is associative is also clear: if any $\xi_i$ is overtwisted then the horizontal composition is $\{*\}$; else associativity follows immediately since 1-morphisms arise from a partial order. Again $\{*\}$ acts as a zero.

There is an identity 2-morphism for each 1-morphism $\Gamma$, namely the tight contact structure on $\M(\Gamma, \Gamma)$. As a thickened convex surface, its vertical composition is indeed the identity; since it is not $\{*\}$, its horizontal composition is also the identity.

The ``interchange law''
\[
 (\xi_1 \cdot \xi_2)(\xi_3 \cdot \xi_4) = (\xi_1 \xi_3) \cdot (\xi_2 \xi_4),
\]
perhaps best understood from the diagram,
\[
\xymatrix{
w_0 \ar@/^3pc/[rr]_{\Gamma}="gtl" \ar[rr]_{\Gamma}="gcl" \ar@/_3pc/[rr]^{\Gamma}="gbl"
&&
w_1 \ar@/^3pc/[rr]_{\Gamma'}="gtr" \ar[rr]_{\Gamma'}="gcr" \ar@/_3pc/[rr]^{\Gamma'}="gbr"
&&
w_2, \ar@{=>}_{\xi_1} "gbl";"gcl" \ar@{=>}_{\xi_2} "gcl";"gtl" \ar@{=>}_{\xi_3} "gbr";"gcr" \ar@{=>}_{\xi_4} "gcr";"gtr"
}
\]
is only defined when the 2-morphisms $\xi_1, \xi_2$ are contact structures on some $\M(\Gamma, \Gamma)$, where $\Gamma = [\Gamma_{w_0}, \Gamma_{w_1}]$; and similarly the 2-morphisms $\xi_3, \xi_4$ are contact structures on some $\M(\Gamma', \Gamma')$, where $\Gamma' = [\Gamma_{w_1}, \Gamma_{w_2}]$. If any of these is $\{*\}$, we have $\{*\}$ on both sides. If not, then $\xi_1 = \xi_2 = 1_\Gamma$ and $\xi_3 = \xi_4 = 1_{\Gamma'}$, being standard neighbourhoods of chord diagrams; thus both sides are equal to $1_{\Gamma' \circ \Gamma}$.
\end{proof}

We have now proved proposition \ref{contact_2_category}.

\label{sec_improving_2-category}

Our contact 2-category is specific to an $n$ and $e$; in over all $n$ and $e$, we obtain a family of 2-categories indexed by words on $\{-,+\}$. But these words can be regarded as \emph{paths} on Pascal's triangle, suggesting the existence of a 3-category.

\subsubsection{Dimensionally-reduced TQFT}

\label{sec_dimensionally-reduced_TQFT}
\label{sec_QFT}

For sutured manifolds $(\Sigma \times S^1, F \times S^1)$, the TQFT-like properties of $SFH$ can be regarded as ``dimensionally reduced'', $(1+1)$-dimensional (see \cite{HKM08}). Here $\Sigma$ is a surface with boundary, and $F \subset \partial \Sigma$ is finite; $(T,n)$ is precisely the case $\Sigma = D^2$.

In \cite{HKM08}, it is computed that, if $|F| = 2n$,
\[
 SFH(\Sigma \times S^1, F \times S^1) = \Z_2^{2^{n-\chi(\Sigma)}}.
\]
As in the case $\Sigma = D^2$, contact structures on $(\Sigma \times S^1, F \times S^1)$ correspond bijectively to dividing sets $K$ drawn on $\Sigma$ without any contractible components \cite{Hon00II, Gi02}. However, on higher genus surfaces, $K$ may be tight with closed components.

The dimensionally-reduced TQFT has the following properties: see \cite{HKM08} for further details. To each $(\Sigma, F)$ we associate the vector space $V(\Sigma, F) = \Z_2^{2^{n - \chi(\Sigma)}}$. To every properly embedded 1-manifold $K \subset \Sigma$ with $\partial K = F$, dividing $\Sigma$ into positive and negative regions, consistent with the signs on $\partial \Sigma - F$, we associate an element $c(K) \in V(\Sigma, F)$. Moreover, $V(\Sigma, F)$ is generated by contact elements. We have seen these properties (and much more) for discs. 

TQFT-inclusion gives gluing maps in the dimensionally-reduced case. Let $\gamma, \gamma'$ be disjoint arcs of $\partial \Sigma$, with endpoints not in $F$. Let $\tau$ be a map $\gamma \stackrel{\sim}{\to} \gamma'$ preserving positive and negative regions. Gluing $\gamma$ to $\gamma'$ by $\tau$ produces a $(\Sigma', F')$, and we obtain a map $\Phi_\tau: V(\Sigma, F) \To V(\Sigma',F')$. This map takes contact elements $c(K) \mapsto c(\bar{K})$, where $\bar{K}$ is obtained by gluing $K$ by $\tau$.

In fact \cite[lemma 7.9]{HKM08}, if $\gamma, \gamma'$ each intersect $F$ precisely once, then $\Phi_\tau$ is an isomorphism. Such a gluing decreases $n$ by $1$ and decreases $\chi$ by $1$, so the dimension $2^{n - \chi(\Sigma)}$ of the vector spaces is preserved. In particular, for any $(\Sigma, F)$, there exists an isomorphism $V(D^2, F') \cong V(\Sigma, F)$ obtained from a composition of such gluing maps. Such an isomorphism takes our basis of contact elements in $V(D^2, F')$, with all of its rich structure developed in the foregoing, to a basis of contact elements for $V(\Sigma, F)$.

However, although there may be an isomorphism between any $V(\Sigma, F)$ and some $V(D^2, F')$, this need not not give a bijection between contact elements; or  between nonzero contact elements. Every contact element in $V(D^2, F')$ gives a corresponding contact element in $V(\Sigma, F)$; but not all contact elements in $V(\Sigma, F)$ arise in this way; only those arising from dividing sets which intersect every gluing arc precisely once. That is, the isomorphism $V(D^2, F') \rightarrow V(\Sigma, F)$ induces an injective but not surjective map on contact elements.

A simple case of  such an isomorphism is when $\partial \Sigma$ has components with two points of $F$.  We may simply glue up such a boundary component of $\Sigma$ and obtain a surface with one fewer boundary component. In a standard topological quantum field theory picture, this is a good reason why a chord diagram with 1 chord can be regarded as ``the vacuum''. It may be glued up, or ``filled in'', or ``capped off'', without any effect. It is, effectively, not there. A cobordism from vacua is equivalent to a cobordism from the empty set, in this TQFT.

We also remark that our chord diagrams are bijective, in an explicit fashion, with \emph{planar binary trees}, and the vector space generated by such objects has been considered previously; they have also been considered in physical contexts. See, e.g., \cite{ChaFra, FraSimplicial, FraGroups, HNT, Loday-Ronco, Novelli-Thibon}. The bypass relation translates into a similar linear relation on trees, which appears not to have been considered previously, so far as the author could find.

%
%
%
%

\bibliographystyle{amsplain}
\bibliography{danbib}

\providecommand{\bysame}{\leavevmode\hbox to3em{\hrulefill}\thinspace}
\providecommand{\MR}{\relax\ifhmode\unskip\space\fi MR }
\providecommand{\MRhref}[2]{%
  \href{http://www.ams.org/mathscinet-getitem?mr=#1}{#2}
}
\providecommand{\href}[2]{#2}
\begin{thebibliography}{10}

\bibitem{Aigner}
Martin Aigner, \emph{Enumeration via ballot numbers}, Discrete Math.
  \textbf{308} (2008), no.~12, 2544--2563. \MR{MR2410460 (2009b:05012)}

\bibitem{BaezIntronCat}
John~C. Baez, \emph{An introduction to {$n$}-categories}, Category theory and
  computer science ({S}anta {M}argherita {L}igure, 1997), Lecture Notes in
  Comput. Sci., vol. 1290, Springer, Berlin, 1997, pp.~1--33. \MR{MR1640335
  (99h:18008)}

\bibitem{Benchekroun}
S.~Benchekroun and P.~Moszkowski, \emph{A bijective proof of an enumerative
  property of legal bracketings}, Discrete Math. \textbf{176} (1997), no.~1-3,
  273--277. \MR{MR1477296 (98k:05012)}

\bibitem{Bona-Sagan}
Mikl{\'o}s B{\'o}na and Bruce~E. Sagan, \emph{On divisibility of {N}arayana
  numbers by primes}, J. Integer Seq. \textbf{8} (2005), no.~2, Article 05.2.4,
  5 pp. (electronic). \MR{MR2152284 (2006d:11010)}

\bibitem{ChaFra}
Fr\'{e}d\'{e}ric Chapoton and Alessandra Frabetti, \emph{From quantum
  electrodynamics to posets of planar binary trees},
  {\small{\url{http://arxiv.org/abs/0811.4712}}}, 2008.

\bibitem{DSV}
Tomislav Do{\v{s}}li{\'c}, Dragutin Svrtan, and Darko Veljan, \emph{Enumerative
  aspects of secondary structures}, Discrete Math. \textbf{285} (2004),
  no.~1-3, 67--82. \MR{MR2074841 (2005d:05013)}

\bibitem{Et02}
John~B. Etnyre, \emph{Introductory lectures on contact geometry},  (2002).

\bibitem{Fomin-Reading}
Sergey Fomin and Nathan Reading, \emph{Root systems and generalized
  associahedra}, Geometric combinatorics, IAS/Park City Math. Ser., vol.~13,
  Amer. Math. Soc., Providence, RI, 2007, pp.~63--131. \MR{MR2383126}

\bibitem{FraSimplicial}
Alessandra Frabetti, \emph{Simplicial properties of the set of planar binary
  trees}, J. Algebraic Combin. \textbf{13} (2001), no.~1, 41--65. \MR{MR1817703
  (2001m:55048)}

\bibitem{FraGroups}
\bysame, \emph{Groups of tree-expanded series}, J. Algebra \textbf{319} (2008),
  no.~1, 377--413. \MR{MR2378077}

\bibitem{Gelfand-Manin}
S.~I. Gelfand and Yu.~I. Manin, \emph{Homological algebra}, Springer-Verlag,
  Berlin, 1999, Translated from the 1989 Russian original by the authors,
  Reprint of the original English edition from the series Encyclopaedia of
  Mathematical Sciences [{{\i}t Algebra, V}, Encyclopaedia Math. Sci., 38,
  Springer, Berlin, 1994; MR1309679 (95g:18007)]. \MR{MR1698374 (2000b:18016)}

\bibitem{Gi91}
Emmanuel Giroux, \emph{Convexit\'e en topologie de contact}, Comment. Math.
  Helv. \textbf{66} (1991), no.~4, 637--677. \MR{MR1129802 (93b:57029)}

\bibitem{Gi00}
\bysame, \emph{Structures de contact en dimension trois et bifurcations des
  feuilletages de surfaces}, Invent. Math. \textbf{141} (2000), no.~3,
  615--689. \MR{MR1779622 (2001i:53147)}

\bibitem{GiBundles}
\bysame, \emph{Structures de contact sur les vari\'et\'es fibr\'ees en cercles
  audessus d'une surface (contact structures on manifolds that are
  circle-bundles over a surface)}, Comment. Math. Helv. \textbf{76} (2001),
  no.~2, 218--262. \MR{MR1839346 (2002c:53138)}

\bibitem{Gi02}
\bysame, \emph{G\'eom\'etrie de contact: de la dimension trois vers les
  dimensions sup\'erieures}, Proceedings of the {I}nternational {C}ongress of
  {M}athematicians, {V}ol. {II} ({B}eijing, 2002) (Beijing), Higher Ed. Press,
  2002, pp.~405--414. \MR{MR1957051 (2004c:53144)}

\bibitem{HNT}
Florent Hivert, Jean-Christophe Novelli, and Jean-Yves Thibon,
  \emph{Commutative combinatorial {H}opf algebras}, J. Algebraic Combin.
  \textbf{28} (2008), no.~1, 65--95. \MR{MR2420780}

\bibitem{HonCat}
Ko~Honda, \emph{Contact structures, {H}eegaard {F}loer homology and
  triangulated categories}, in preparation.

\bibitem{Hon00I}
\bysame, \emph{On the classification of tight contact structures. {I}}, Geom.
  Topol. \textbf{4} (2000), 309--368 (electronic). \MR{MR1786111 (2001i:53148)}

\bibitem{Hon00II}
\bysame, \emph{On the classification of tight contact structures. {II}}, J.
  Differential Geom. \textbf{55} (2000), no.~1, 83--143. \MR{MR1849027
  (2002g:53155)}

\bibitem{Hon01}
\bysame, \emph{Factoring nonrotative {$T\sp 2\times I$} layers. {E}rratum:
  ``{O}n the classification of tight contact structures. {I}'' [{G}eom.
  {T}opol. {\bf 4} (2000), 309--368; mr1786111]}, Geom. Topol. \textbf{5}
  (2001), 925--938. \MR{MR2435298}

\bibitem{Hon02}
\bysame, \emph{Gluing tight contact structures}, Duke Math. J. \textbf{115}
  (2002), no.~3, 435--478. \MR{MR1940409 (2003i:53125)}

\bibitem{Hon3Dim}
\bysame, \emph{3-dimensional methods in contact geometry}, Different faces of
  geometry, Int. Math. Ser. (N. Y.), vol.~3, Kluwer/Plenum, New York, 2004,
  pp.~47--86. \MR{MR2102994 (2005k:53169)}

\bibitem{HonDimThree}
\bysame, \emph{The topology and geometry of contact structures in dimension
  three}, International {C}ongress of {M}athematicians. {V}ol. {II}, Eur. Math.
  Soc., Z\"urich, 2006, pp.~705--717. \MR{MR2275619 (2007k:53143)}

\bibitem{HKMContClass}
Ko~Honda, William~H. Kazez, and Gordana Mati\'{c}, \emph{On the contact class
  in {H}eegaard {F}loer homology}.

\bibitem{HKM01}
Ko~Honda, William~H. Kazez, and Gordana Mati{\'c}, \emph{Tight contact
  structures on fibered hyperbolic 3-manifolds}, J. Differential Geom.
  \textbf{64} (2003), no.~2, 305--358. \MR{MR2029907 (2005b:53140)}

\bibitem{HKMPinwheel}
\bysame, \emph{Pinwheels and bypasses}, Algebr. Geom. Topol. \textbf{5} (2005),
  769--784 (electronic). \MR{MR2153107 (2006f:53134)}

\bibitem{HKM06ContClass}
Ko~Honda, William~H. Kazez, and Gordana Mati\'{c}, \emph{The contact invariant
  in sutured {F}loer homology}, {\small{\url{http://arxiv.org/abs/0705.2828}}},
  2007.

\bibitem{HKM05}
Ko~Honda, William~H. Kazez, and Gordana Mati{\'c}, \emph{Right-veering
  diffeomorphisms of compact surfaces with boundary}, Invent. Math.
  \textbf{169} (2007), no.~2, 427--449. \MR{MR2318562 (2008e:57028)}

\bibitem{HKM08}
Ko~Honda, William~H. Kazez, and Gordana Mati\'{c}, \emph{Contact structures,
  sutured {F}loer homology and {TQFT}},
  {\small{\url{http://arxiv.org/abs/0807.2431}}}, 2008.

\bibitem{HKM06RV}
Ko~Honda, William~H. Kazez, and Gordana Mati{\'c}, \emph{Right-veering
  diffeomorphisms of compact surfaces with boundary. {II}}, Geom. Topol.
  \textbf{12} (2008), no.~4, 2057--2094. \MR{MR2431016 (2009i:57057)}

\bibitem{Hwang-Mallows}
F.~K. Hwang and C.~L. Mallows, \emph{Enumerating nested and consecutive
  partitions}, J. Combin. Theory Ser. A \textbf{70} (1995), no.~2, 323--333.
  \MR{MR1329396 (96e:05014)}

\bibitem{Ju06}
Andr{\'a}s Juh{\'a}sz, \emph{Holomorphic discs and sutured manifolds}, Algebr.
  Geom. Topol. \textbf{6} (2006), 1429--1457 (electronic). \MR{MR2253454
  (2007g:57024)}

\bibitem{Ju08}
\bysame, \emph{Floer homology and surface decompositions}, Geom. Topol.
  \textbf{12} (2008), no.~1, 299--350. \MR{MR2390347}

\bibitem{Kaz}
William~H. Kazez, \emph{A cut-and-paste approach to contact topology}, Bol.
  Soc. Mat. Mexicana (3) \textbf{10} (2004), no.~Special Issue, 1--42.
  \MR{MR2199337 (2006j:57054)}

\bibitem{Loday-Ronco}
Jean-Louis Loday and Mar{\'{\i}}a~O. Ronco, \emph{Hopf algebra of the planar
  binary trees}, Adv. Math. \textbf{139} (1998), no.~2, 293--309. \MR{MR1654173
  (99m:16063)}

\bibitem{Novelli-Thibon}
Jean-Christophe Novelli and Jean-Yves Thibon, \emph{Hopf algebras and
  dendriform structures arising from parking functions}, Fund. Math.
  \textbf{193} (2007), no.~3, 189--241. \MR{MR2289770 (2007i:16064)}

\bibitem{OS04Prop}
Peter Ozsv{\'a}th and Zolt{\'a}n Szab{\'o}, \emph{Holomorphic disks and
  three-manifold invariants: properties and applications}, Ann. of Math. (2)
  \textbf{159} (2004), no.~3, 1159--1245. \MR{MR2113020 (2006b:57017)}

\bibitem{OS04Closed}
\bysame, \emph{Holomorphic disks and topological invariants for closed
  three-manifolds}, Ann. of Math. (2) \textbf{159} (2004), no.~3, 1027--1158.
  \MR{MR2113019 (2006b:57016)}

\bibitem{OSContact}
\bysame, \emph{Heegaard {F}loer homology and contact structures}, Duke Math. J.
  \textbf{129} (2005), no.~1, 39--61. \MR{MR2153455 (2006b:57043)}

\bibitem{OS06}
\bysame, \emph{Heegaard diagrams and {F}loer homology}, International
  {C}ongress of {M}athematicians. {V}ol. {II}, Eur. Math. Soc., Z\"urich, 2006,
  pp.~1083--1099. \MR{MR2275636 (2008h:57048)}

\bibitem{Sulanke}
Robert~A. Sulanke, \emph{Moments, {N}arayana numbers, and the cut and paste for
  lattice paths}, J. Statist. Plann. Inference \textbf{135} (2005), no.~1,
  229--244. \MR{MR2202349 (2006j:05013)}

\bibitem{Williams05}
Lauren~K. Williams, \emph{Enumeration of totally positive {G}rassmann cells},
  Adv. Math. \textbf{190} (2005), no.~2, 319--342. \MR{MR2102660 (2005i:05197)}

\bibitem{Yano-Yoshida}
Fujine Yano and Hiroaki Yoshida, \emph{Some set partition statistics in
  non-crossing partitions and generating functions}, Discrete Math.
  \textbf{307} (2007), no.~24, 3147--3160. \MR{MR2370117}

\end{thebibliography}


\end{document}